\documentclass[number,sort&compress]{elsarticle} 

\usepackage{amssymb}

\usepackage{amsthm}
\usepackage{amsmath}
\usepackage{hyperref}
\usepackage{xypic}
\usepackage{rotating}
\xyoption{all}

\usepackage{stmaryrd}

\usepackage{color}

\newtheorem{SATZ}{Theorem}[section]
\newtheorem{PROP}[SATZ]{Proposition}
\newtheorem*{SATZO}{Main Theorem \ref{RESULTGLOBAL}}
\newtheorem{HAUPTSATZ}[SATZ]{Main Theorem}
\newtheorem{LEMMA}[SATZ]{Lemma}
\newtheorem{DEF}[SATZ]{Definition}
\newtheorem{DEFSATZ}[SATZ]{Definition/Theorem}
\newtheorem{DEFPROP}[SATZ]{Definition/Proposition}
\newtheorem{DEFLEMMA}[SATZ]{Definition/Lemma}
\newtheorem{BEISPIEL}[SATZ]{Example}
\newtheorem{FRAGE}[SATZ]{Question}

\newtheorem{VERMUTUNG}[SATZ]{Conjecture}
\newtheorem{ASSERTION}[SATZ]{Assertion}

\newtheorem{KOR}[SATZ]{Corollary}

\newtheorem{BEM}[SATZ]{Remark}

\newtheoremstyle{bare}        
  {}            
  {}            
  {\normalfont}                 
  {}                            
  {\bfseries}                   
  {}                            
  {.0em}                           
  {\thmnumber{#2}. \thmnote{ \normalfont\textsc{(#3)}}} 

\theoremstyle{bare}
\newtheorem{PAR}[SATZ]{}

\newcommand{\M}[1]{ \left(\begin{matrix} #1 \end{matrix} \right) }

\newcommand{\F}{ \mathbb{F} }
\newcommand{\R}{ \mathbb{R} }
\newcommand{\C}{ \mathbb{C} }
\newcommand{\Q}{ \mathbb{Q} }
\newcommand{\Qp}{ {\mathbb{Q}_p} }

\newcommand{\Z}{ \mathbb{Z} }
\newcommand{\Zh}{ {\widehat{\mathbb{Z}}} }
\newcommand{\Zp}{ {\mathbb{Z}_p} }
\newcommand{\Fp}{ {\mathbb{F}_p} }
\newcommand{\Zpp}{ {\mathbb{Z}_{(p)}} }
\newcommand{\Zl}{ {\mathbb{Z}_l} }
\newcommand{\Gm}{ {\mathbb{G}_m} }
\newcommand{\Ga}{ {\mathbb{G}_a} }

\newcommand{\N}{ \mathbb{N} }
\newcommand{\HH}{ \mathbb{H} }

\newcommand{\G}{ \mathbb{G} }
\newcommand{\A}{ \mathbb{A} }
\newcommand{\Af}{ {\mathbb{A}^{(\infty)}} }
\newcommand{\Afp}{ {\mathbb{A}^{(\infty, p)}} }

\newcommand{\PP}{ \mathbb{P} }
\newcommand{\SSS}{ \mathbb{S} }
\newcommand{\OOO}{\text{\footnotesize$\mathcal{O}$}}
\newcommand{\OO}{ {\cal O} }
\newcommand{\D}{ \mbox{d} }
\newcommand{\where}{\enspace | \enspace }

\newcommand{\cat}[1]{ {[\textbf{ #1 }]} }

\newcommand{\tensor}{\otimes}

\DeclareMathOperator{\supp}{supp}
\DeclareMathOperator{\dd}{d}
\newcommand{\ddc}{ \dd \dd^c }

\DeclareMathOperator*{\dotsop}{\dots}

\DeclareMathOperator{\Ad}{Ad}
\DeclareMathOperator{\Lie}{Lie}
\DeclareMathOperator{\Int}{int}
\DeclareMathOperator{\im}{im}

\DeclareMathOperator{\ind}{ind}

\DeclareMathOperator{\Frob}{Frob}
\DeclareMathOperator{\Aut}{Aut}
\DeclareMathOperator{\Iso}{Iso}
\DeclareMathOperator{\Model}{M}
\DeclareMathOperator{\Princ}{P}
\DeclareMathOperator{\Cycle}{Z}
\DeclareMathOperator{\Cliff}{C}

\DeclareMathOperator{\id}{id}

\DeclareMathOperator{\Hom}{Hom}

\DeclareMathOperator{\Sym}{Sym}
\DeclareMathOperator{\End}{End}
\DeclareMathOperator{\Stab}{Stab}

\DeclareMathOperator{\Spin}{Spin}
\DeclareMathOperator{\GSpin}{GSpin}
\DeclareMathOperator{\Isome}{I}

\DeclareMathOperator{\BP}{BP}

\DeclareMathOperator{\Weil}{Weil}

\DeclareMathOperator{\ModForm}{ModForm}
\DeclareMathOperator{\HolModForm}{HolModForm}
\DeclareMathOperator{\PowSer}{PowSer}
\DeclareMathOperator{\Laur}{Laur}
\DeclareMathOperator{\Sing}{Sing}
\DeclareMathOperator{\rec}{rec}
\DeclareMathOperator{\res}{res}
\DeclareMathOperator{\Res}{Res}

\DeclareMathOperator{\qpar}{qpar}

\DeclareMathOperator{\GSp}{GSp}
\DeclareMathOperator{\USp}{USp}
\DeclareMathOperator{\PSp}{PSp}
\DeclareMathOperator{\WSp}{WSp}

\DeclareMathOperator{\SO}{SO}
\DeclareMathOperator{\Grass}{Grass}

\DeclareMathOperator{\CH}{CH}
\newcommand{\aCH}{\widehat{\CH}}

\DeclareMathOperator{\hght}{ht}
\DeclareMathOperator{\vol}{vol}
\newcommand{\avol}{{\widehat{\vol}}}

\DeclareMathOperator{\Gal}{Gal}

\DeclareMathOperator{\type}{type}
\DeclareMathOperator{\ftype}{ftype}

\DeclareMathOperator{\tr}{tr}

\DeclareMathOperator{\Sp}{Sp}
\DeclareMathOperator{\SL}{SL}
\DeclareMathOperator{\GL}{GL}
\DeclareMathOperator{\PGL}{PGL}
\DeclareMathOperator{\Mp}{Mp}
\DeclareMathOperator{\adeg}{\widehat{deg}}
\DeclareMathOperator{\Sh}{Sh}
\DeclareMathOperator{\chern}{c}
\DeclareMathOperator{\achern}{\widehat{c}}
\DeclareMathOperator{\chernch}{ch}
\DeclareMathOperator{\achernch}{\widehat{ch}}

\DeclareMathOperator{\spec}{spec}
\DeclareMathOperator{\spf}{spf}

\DeclareMathOperator{\gr}{gr}
\DeclareMathOperator{\Div}{div}

\newcommand{\nGa}{\mathbb{W}}

\newcommand{\nSD}{\mathbf{X}}
\newcommand{\nSDp}{{\mathbf{X}'}}
\newcommand{\nSDi}{\mathbf{Y}}
\newcommand{\nSDB}{\mathbf{B}}
\newcommand{\nSDBp}{{\mathbf{B}'}}
\newcommand{\nP}{P}
\newcommand{\nG}{G}
\newcommand{\nX}{\mathbb{D}}
\newcommand{\nh}{h}
\newcommand{\nw}{w}
\newcommand{\nW}{W}
\newcommand{\nU}{U}
\newcommand{\nV}{V}
\newcommand{\nC}{C}

\newcommand{\nZ}{\Cycle}
\newcommand{\nSh}{\Model}
\newcommand{\nShD}{\Model^\vee}
\newcommand{\nSPB}{\Princ}
\newcommand{\nMorphD}{\Pi}
\newcommand{\nMorphSt}{\Xi}

\newcommand{\WeilGamma}{\Upsilon}

\newcommand{\nRPCD}{\Delta}

\newcommand{\nH}{\mathbf{H}}

\newcommand{\nL}{L}
\newcommand{\nM}{M}

\newcommand{\nI}{I}

\newcommand{\nS}{\mathbf{S}}
\newcommand{\nO}{\mathbf{O}}

\newcommand{\nQ}{Q}

\newcommand{\nIsome}{\Isome}

\journal{?}

\begin{document}

\begin{frontmatter}



\title{The geometric and arithmetic volume of Shimura varieties of orthogonal type}


\author{Fritz H\"ormann}

\address{Mathematisches Institut, 
Albert-Ludwigs-Universit\"at Freiburg \\
\emph{E-mail: fritz.hoermann@math.uni-freiburg.de}}

\begin{abstract}
We apply the theory of Borcherds products to calculate arithmetic volumes (heights)
of Shimura varieties of orthogonal type up to contributions from very bad primes. The approach is analogous
to the well-known computation of their geometric volume by induction, using special cycles. 
A functorial theory of integral models of toroidal compactifications of those varieties and 
a theory of arithmetic Chern classes of integral automorphic vector bundles with singular metrics are used.
We obtain some evidence in the direction of Kudla's conjectures on
relations of heights of special cycles on these varieties to special derivatives of Eisenstein series.
\end{abstract}


\end{frontmatter}




\tableofcontents

\section{Introduction}

A Shimura variety of orthogonal type arises from the Shimura datum $\nO(\nL)$ consisting of the orthogonal group $\SO(\nL_\Q)$ of a quadratic space $\nL_\Q$ of signature $(m-2,2)$ and the set
\[ \nX := \{ N \subseteq \nL_\R \ |\ N \text{ negative definite plane} \}, \]
which has the structure of an Hermitian symmetric domain and can be interpreted as
a conjugacy class of morphisms $\SSS=\Res_\R^\C \Gm \rightarrow \SO(\nL_\R)$. For any compact open subgroup $K\subseteq \SO(\nL_\Af)$, we can form the Shimura variety (orbifold):
\[ [ \SO(\nL_\Q) \backslash \nX \times \SO(\nL_\Af) / K ].  \]
It is a smooth manifold if $K$ is sufficiently small,
has a canonical algebraic model over $\Z_{(p)}$\footnote{if $m>2$} for $p \not| 2D$, where
$D$ is the discriminant of a lattice $\nL_\Z$, and $K$ is equal to the corresponding maximal subgroup at $p$.
In the first part of this article (sections \ref{SECTIONGROUPSCHEMES}--\ref{SECTIONINTMODELS}), we recall 
the functorial theory of
\begin{itemize}
\item canonical integral models of toroidal compactifications of mixed Shimura varieties of Abelian type,
\item arithmetic automorphic vector bundles on them,
\end{itemize}
developed in the thesis of the author \cite{Thesis}. This theory, for general Hodge or even Abelian type, still relies on an assumption regarding the stratification of the compactification (\ref{MAINCONJECTURE}) which remains unproven. Even though therefore this theory has to be seen as work in progress, we think that it is of independent interest to publish this application. 
 
The orthogonal Shimura varieties are interesting because they easily enable to define algebraic cycles of arbitrary codimension which turn out to have amazing arithmetic properties:
For an isometry $x: \nM \hookrightarrow \nL$, where $\nM$ is positive definite, we can simply form
\[ \nX_x := \{ N \in \nX \ |\ N \perp x(\nM) \}. \]
If ($K$-stable) lattices $\nL_\Z$ and $\nM_\Z$ or more generally a $K$-invariant Schwartz function $\varphi \in S(\nM_\Af^* \otimes \nL_\Af)$ is chosen, 
we can form cycles $\nZ(\nM, \nL, \varphi; K)$ on the Shimura variety by taking the quotient of the union of the $\nX_x$ over all integral isometries (resp. all isometries in the support of $\varphi$ in a weighted way) as above. Also for singular lattices $\nM$ analogous cycles can be defined.
Consider a model $\nSh({}^K_\nRPCD \nO(\nL))$ of a toroidal compactification of the Shimura variety. 
We consider its algebraic Chow groups $\CH^n(\nSh({}^K_\nRPCD \nO(\nL)))$ and the
generating series
\[ \Theta_n(\nL) = \sum_{Q \in \Sym^2((\Z^{n})^*)} [\nZ(\Z^n_Q, \nL, \varphi; K)] \exp(2\pi i Q \tau)   \] 
with values in the Chow group (tensored with $\C$). 
Assume now that $\nSh({}^K_\nRPCD \nO(\nL))$ is even defined over $\Z$ in a ``reasonably canonical'' way.
Kudla proposes a ``reasonably canonical'' way of defining arithmetic cycles $\widehat{\nZ}(\nM, \nL, \varphi; K, \nu)$, depending on the
imaginary part $\nu$ of $\tau$, too, and also for indefinite $\nM$, such that
\[ \widehat{\Theta}_n(\nL) = \sum_{Q \in \Sym^2((\Z^{n})^*)} [\widehat{\nZ}(\Z^n_Q, \nL, \varphi; K, \nu)] \exp(2\pi i Q \tau),  \]
should have values in a suitable Arakelov Chow group $\aCH^n(\nSh({}^K_\nRPCD \nO(\nL)))_\C$. 
He proposes specific Greens functions, which have singularities at the boundary.
The orthogonal Shimura varieties come equipped with a natural Hermitian automorphic line bundle $\nMorphSt^*\overline{\mathcal{E}}$ (whose metric also has singularities along the boundary). Multiplication with a suitable power of its first Chern  class and taking push-forward provides us with geometric (resp. arithmetic) degree maps $\deg: \CH^p(\cdots) \rightarrow \Z$ (resp. $\adeg:\aCH^p(\cdots) \rightarrow \R$). Assuming that an Arakelov theory can be set up to deal with all different occurring singularities, Kudla conjectures (for the geometric case this goes back to Siegel, Hirzebruch, Zagier, Kudla-Millson, Borcherds, to name a few)\footnote{If $m-r \le n+1$, the statement has to be modified. Here $r$ is the Witt-rank of $\nL$.}
\begin{enumerate}
\item[(K1)] $\Theta_n$ and $\widehat{\Theta}_n$ are (holomorphic, resp. non-holomorphic) Siegel modular forms of weight $\frac{m}{2}$ and genus $n$.
\item[(K2)] $\deg(\Theta_n)$ and $\adeg(\widehat{\Theta}_n)$ are equal to a special value of a normalized version of a standard Eisenstein series of weight $\frac{m}{2}$ associated with the Weil representation of $\nL$ \cite[section 4]{Paper1}, resp. its special derivative at the same point.
\item[(K3)] $\Theta_{n_1}(\tau_1) \cdot \Theta_{n_2}(\tau_2) = \Theta_{n_1+n_2}(\M{\tau_1&\\&\tau_2})$ and similarly for $\widehat{\Theta}$.
\item[(K4)] $\widehat{\Theta}_{m-1}$ can be defined, with coefficients being zero-cycles on the arithmetic model and satisfies the properties above.
\end{enumerate}
Kudla shows (see \cite{Kudla6} for an overview) that this implies almost formally vast generalizations of the formula of Gross-Zagier \cite{GZ}.
In full generality the conjectures are known only for Shimura curves \cite{KRY3}. See \cite[section 11]{Paper1}, the
introduction to \cite{Thesis}, and especially Kudla's article \cite{Kudla6} for more information on what is known in other cases. In the {\em geometric case}
Borcherds \cite{Borcherds2} shows (K1) for $n=1$. (K2) follows either from Kudla-Millson theory and the Siegel-Weil formula, or directly from the knowledge of the Tamagawa number, or from Borcherds' theory. All 3 proofs (for non-singular coefficients) are sketched in section 13. 

In this work, we show --- in a sense --- the good reduction part of the {\em arithmetic case} of (K2) for {\em all} Shimura
varieties of orthogonal type. The approach uses only information from the ``Archimedean fibre'' to attack the problem. It is a generalization of work of
Bruinier, Burgos, and K\"uhn \cite{BBK}, which dealt with the case of Hilbert modular surfaces. It uses
Borcherds' construction of modular forms on these Shimura varieties \cite{Borcherds1} and a computation of the integral of their norm \cite{BK, Kudla4}.
More precisely, we show the following:

\begin{SATZO}Assume \ref{MAINCONJECTURE}.

Let $D'$ be the product of primes $p$ such that $p^2|D$, where $D$ is the discriminant of $\nL_\Z$.
We have
\[
\begin{array}{rrcl}
\text{1.} \qquad & \vol_{E}(\nSh({}^{K}_\nRPCD\nO(\nL_\Z))) &=& 4 \widetilde{\lambda}^{-1}(\nL_\Z; 0)  \\
\text{2.} \qquad & \widehat{\vol}_{\overline{\mathcal{E}}}(\nSh({}^{K}_\nRPCD \nO(\nL_\Z))) &\equiv& \frac{d}{ds} 4 \widetilde{\lambda}^{-1}(\nL_{\Z};s)|_{s=0} \qquad \text{ in $\R_{2D'}$}
\end{array}
\]

Let $\nM_\Z$ be a lattice of dimension $n$ with positive definite $Q_\nM \in \Sym^2(\nM_\Q^*)$. 
Let $D''$ be the product of primes $p$ such that $\nM_\Zp \not\subseteq \nM_\Zp^*$ or $\nM_\Zp^*/\nM_\Zp$ is not cyclic. 
Assume 
\begin{itemize}
\item $m-n \ge 4$, {\em or} 
\item $m=4, n=1$, and $\nL_\Q$ has Witt rank 1. 
\end{itemize}
Then we have
\[
\begin{array}{rrcl}
\text{3.} &  \vol_{\mathcal{E}}(\nZ(\nL_\Z, \nM_\Z, \kappa; K)) &=& 4 \widetilde{\lambda}^{-1}(\nL_\Z; s) \widetilde{\mu}(\nL_\Z, \nM_\Z, \kappa; 0),  \\
\text{4.} &  \hght_{\overline{\mathcal{E}}}(\nZ(\nL_\Z, \nM_\Z, \kappa; K)) &\equiv& \frac{\dd}{\dd s} \left. { 4  \left( \widetilde{\lambda}^{-1}(\nL_\Z; s) \widetilde{\mu}(\nL_\Z, \nM_\Z, \kappa; s) \right) } \right|_{s=0} \\
& & & \text{in } \R_{2D D''}.
\end{array}
\]
3. is true without the restriction on $m$ and $n$. Note that for $n=1$ and integral $Q_\nM$, we have trivially always $D''=1$. 
\end{SATZO}

Here $\R_N$ is $\R$ modulo rational multiples of $\log(p)$ for $p|N$, and
the $\widetilde{\lambda}$ and $\widetilde{\mu}$ are
functions in $s \in \C$, given by certain Euler products (\ref{DEFLAMBDAMU}) associated with representation densities of $\nL_\Z$ and $\nM_\Z$. $\pm\widetilde{\mu}$ appears as the ``holomorphic part'' of a Fourier coefficient of the standard Eisenstein series associated with the Weil representation of $\nL_\Z$. $K$ is the discriminant kernel and $\nSh({}^{K}_\nRPCD\nO(\nL_\Z))$ is any
toroidal compactification of the orthogonal Shimura variety (see below). $\overline{\mathcal{E}}$ is the integral tautological bundle on the compact dual equipped with a metric on the restriction
of its complex fibre to $\nX$; the volumes/heights are computed
w.r.t. the associated arithmetic automorphic line bundle (\ref{DEFAUTOMORPHICVB})
 $\nMorphSt^*\overline{\mathcal{E}}$ on $\nSh({}^K_\nRPCD \nO(\nL_\Z))$. 

For a more detailed introduction to the method of proof, we refer the reader to section 11 of \cite{Paper1}, in which
the results of this paper are announced and explained. They have their origin in the thesis of the author \cite{Thesis}.

In the first sections we explain in detail the functorial properties integral models of toroidal compactifications of orthogonal Shimura varieties ought to satisfy. These properties have been established in the thesis of the author \cite{Thesis}
assuming a certain technical assumption (\ref{MAINCONJECTURE}) that is conjectured to be always satisfied. This is proven if
$m\le 5$ by Lan \cite{Lan} (the Shimura varieties --- Spin-version --- are then of P.E.L. type) 
and was announced in general. 
For the rest of this work, we assume that the assumption holds true, and we have hence models with the required properties. These are constructed locally (i.e. over an extension of $\Zpp$).
Input data for the theory are $p$-integral mixed Shimura data ($p$-MSD) $\nSD=(\nP_\nSD, \nX_\nSD, \nh_\nSD)$ consisting
of an affine group scheme $\nP_\nSD$ over $\spec(\Zpp)$ of a certain rigid type (P) (see \ref{DEFTYPEP}) and a set $\nX_\nSD$ which
comes equipped with a finite covering $\nh_\nSD: \nX_\nSD \rightarrow \Hom(\SSS_\C, \nP_{\nSD,\C})$ onto a conjugacy class in the latter group, subject to some axioms, which are roughly Pink's mixed extension \cite{Pink} of Delige's axioms for a pure Shimura datum. Then call a compact open subgroup $K \subseteq \nP_\nSD(\Af)$ admissible, if it is of the form 
$K^{(p)} \times \nP_\nSD(\Zp)$ for a compact open subgroup $K^{(p)} \subseteq \nP_\nSD(\Afp)$. For the toroidal compactification a certain rational polyhedral cone decomposition $\nRPCD$ is needed. We call the collection
${}^K \nSD$ (resp. ${}^K_\nRPCD \nSD$) extended (compactified) $p$-integral mixed Shimura data ($p$-EMSD, resp. $p$-ECMSD). These form 
categories where morphisms ${}^K_\nRPCD \nSD \rightarrow {}^{K'}_{\nRPCD'} \nSDi$ are pairs 
$(\alpha, \rho)$ of a morphism $\alpha$ of Shimura data and $\rho \in \nP_\nSDi(\Afp)$ satisfying compatibility with the $K$'s and $\nRPCD$'s. The theory is then a functor $\nSh$ from $p$-ECMSD to the category of Deligne-Mumford stacks over reflex rings (above $\Zpp$) which over $\C$ and restricted to $p$-MSD becomes naturally isomorphic to the one given by the analytic mixed Shimura variety construction. It is characterized uniquely by Deligne's canonical model condition, Milne's extension property (integral canonicity) and a stratification of the boundary into mixed Shimura varieties, together with boundary isomorphisms of the formal completions along the latter with similar completions of more mixed Shimura varieties.  These boundary isomorphisms, for the case of the symplectic Shimura varieties, are given by Mumford's construction \cite[Appendix]{FC}. There is also a functor `compact' dual from the category of $p$-MSD to the category of schemes over reflex rings. The duals come equipped with an action of the group scheme $\nP_\nSD$, and we have morphisms 
\[ \nMorphSt_\nSD: \nSh({}^K_\nRPCD \nSD) \rightarrow \left[ \nShD(\nSD)/\nP_{\nSD, \OOO_\nSD} \right], \]
which form a pseudo-natural transformation of functors with values in Artin stacks over reflex rings, which are
a model of the usual construction over $\C$ if $\nRPCD$ is trivial.
This is the theory of integral automorphic vector bundles. It is compatible with boundary isomorphisms. For more information on the `philosophy' of these objects in terms of motives see the introduction to \cite{Thesis}.

In section 7 and 8 a precise general $q$-expansion principle is derived from these abstract properties. 

In section 9 the structure of the models of orthogonal Shimura varieties is investigated and special cycles are defined. 

In section 10 we define integral Hermitian automorphic vector bundles, define the notions of arithmetic and
geometric volume. Furthermore we set up an Arakelov theory which has enough properties to deal
with singularities of the natural Hermitian metrics on automorphic vector bundles. This uses work
of Burgos, Kramer, and K\"uhn \cite{BKK1, BKK2}. 

In section 11 we use the general $q$-expansion principle to prove that Borcherds products with their natural norm yield {\em integral} sections of an appropriate integral Hermitian line bundle. Among other things the product 
expansions of Borcherds' are adelized and their Galois properties investigated.

In section 12, we prove that the bundle of vector valued modular forms for the Weil representation (which appear as input forms in the construction of Borcherds products) has a rational structure. Then we use this to construct input forms with special properties which will be needed in the application.

In section 13 we develop the main ingredient of the proof of the Main Theorem which connects Borcherds' theory and Arakelov geometry on the orthogonal Shimura varieties.

In section 14 the Main Theorem is proven. 

In 3 appendices additional material on ``lacunarity of modular forms'', on quadratic forms, and on semi-linear representations is provided. 

Finally, it is a pleasure to thank the Department of Mathematics and Statistics at McGill University and especially Eyal Goren, Jayce Getz, and Henri Darmon for providing a very inspiring working environment during the preparation of this article.

\section{Group schemes of type (P)}\label{SECTIONGROUPSCHEMES}

\begin{PAR}
Our definition of integral mixed Shimura data involves a special kind of group schemes over
$\spec(\Zpp)$ which are --- in a sense --- as rigid as reductive group schemes. We call them group schemes of type (P), and they are a slight generalization of those of type (R) and (RR) considered in \cite[XXII, 5.1]{SGAIII}, which occur as parabolics of reductive group schemes. We use extensively the notation and results of [loc. cit.]. Proofs of the theorems in this section can be found in \cite[section 1]{Thesis}.
\end{PAR}

\begin{PAR}
For a coherent sheaf on a scheme $S$, we denote
by $\nGa(E)$ its additive group scheme.
For a group scheme $X$ over $S$ we denote the functor $D(X):=\underline{\Hom}(X, \Gm)$. For an ordinary Abelian group $M$, denote $M_S$ the associated constant group scheme. 
\end{PAR}
 
\begin{DEF}\label{DEFTYPEP}
Let $P$ be a group scheme of finite type
over $S$. We call $P$ of type (P), if the following conditions are satisfied:

There exists a closed unipotent normal subgroup scheme $W$
(called unipotent radical).

$S$ is covered by etale neighborhoods $S' \rightarrow S$ for which there exist
\begin{enumerate}
\item
a closed reductive subgroup scheme $G$ of $P_{S'}$ such that
$P_{S'}$ is isomorphic to the semi-direct product of $W_{S'}$ with $G$,

(Then $P$ is smooth and affine over $S$, and there
exist maximal split tori locally in the etale topology \cite[XII, 1.7]{SGAIII}),
\item
a split maximal torus $T \cong D(M_{S'})$ of $G$, where $M$ is some lattice, 
\item
a system of roots \cite[XIX, 3.2, 3.6]{SGAIII} $R = R_G \dot{\cup} R_W \subset M$, such that:
\begin{eqnarray*}
 \Lie(G) &=& \Lie(T_{S'}) \oplus \bigoplus_{r \in R_G} \Lie(G)^r \\
 \Lie(W_{S'}) &=& \bigoplus_{r \in R_W} \Lie(W_{S'})^r,
\end{eqnarray*}

\item closed embeddings
\[ \exp_r: \nGa(\Lie(W_{S'})^r) \rightarrow W \]
for all $r \in R_W$,
inducing via $\Lie$ the inclusion and satisfying for all $S'' \rightarrow S'$, $t \in T(S'')$, $X \in \Lie(G_{S''})^r$:
\[ \Int(t) \exp_r(X) = \exp_r(r(t)X). \]
\end{enumerate}
with the property that any two different $r_1 \in R$, $r_2 \in R_W$ are linearly independent.
\end{DEF}

\begin{BEM}This property is obviously stable under base change. If the objects required in (1--4) above are already defined over $S$ we call $(P, T, M, R)$ a split group of type (P).
\end{BEM}

Group schemes of type (P) have a nice theory of (quasi-)parabolics.

\begin{DEF}[{\cite[XV, 6.1]{SGAIII}}]\label{DEFPARABOLIC}
Let $P$ be a group scheme of finite type over $S$.
A {\bf parabolic} subgroup scheme of $P$ is a smooth subgroup scheme $Q$ of $G$, such that for each
$s \in S$, $Q_{\overline{s}}$ is a parabolic subgroup of $G_{\overline{s}}$, (i.e. such that
$G_{\overline{s}}/Q_{\overline{s}}$ is proper).
\end{DEF}

\begin{PROP}\label{EXISTENCEQUASIPARABOLIC}
Let $S$ be a {\em reduced} scheme, and let $(P,T,M,R)$ be a split group of type (P) over $S$.
If $R' \subset R$ is a closed subset, there exists a unique smooth subgroup scheme $P' \subset P$,
containing $T$, such that
\[ \Lie(P') = \Lie(T) \oplus \bigoplus_{r \in R'} \Lie(P)^r \]
$P'$ is of type (P) and closed in $P$.
In particular each parabolic is of this form.
\end{PROP}

\begin{DEF}
Let $\nP$ be of type (P).
A closed smooth subgroup scheme $Q$ of $P$ is called a {\bf quasi-parabolic} group, if etale locally, say on $S'\rightarrow S$, there is a splitting
$(P, T, M, R)$ over $S'$, $Q_{S'}$ is of the form given in Proposition \ref{EXISTENCEQUASIPARABOLIC} for
a $R' \subset R$ which contains a set of positive roots (but not necessarily $R_W$!).
\end{DEF}

\begin{SATZ}\label{QUASIPARABOLICS}
Let $S$ be reduced and $P$ be a group scheme over $S$ of type (P).
\begin{enumerate}
\item
The functor $S' \mapsto \{\text{(quasi-)parabolic subgroups of $P_{S'}$}\}$ is representable by a smooth (quasi-)projective $S$-scheme $(\mathcal{Q})\mathcal{PAR}$.

There is an open and closed embedding $\mathcal{PAR}$ into $\mathcal{QPAR}$.

\item
There is an etale sheaf $\mathcal{TYPE}$ of finite sets over $S$ and a surjective morphism
\[ \type: \mathcal{QPAR} \rightarrow \mathcal{TYPE}, \]
with the property $Q$ and $Q'$ are locally conjugated in the etale topology, if and only if $\type(Q) = \type(Q')$.

\item
If $(P, T, M, R)$ is split
\begin{eqnarray*}
 \mathcal{TYPE} &\cong& \left\{ \begin{array}{c}
 \text{$W(R_G)$-orbits of closed subsets of $R$} \\ { \text{containing a set of positive roots} } \end{array} \right\}_S. \\
\end{eqnarray*}

\item
Let $Q$ be a quasi-parabolic of $P$.
The morphism $P/Q \rightarrow \mathcal{QPAR}$,
$g \mapsto \Int(g)Q$ is an open and closed immersion onto a connected component of $\mathcal{QPAR}$.
It is hence the fibre above $\type(Q)$.

\item
$\underline{\Hom}(\Gm, P)$ is representable by a smooth affine scheme over $S$.

There is an etale sheaf $\mathcal{FTYPE}$ of (infinite) sets over $S$ and a surjective morphism
\[ \ftype: \underline{\Hom}(\Gm, P) \rightarrow \mathcal{FTYPE}, \]
with the property $\alpha$ and $\alpha'$ are locally conjugated in the etale topology, if and only if $\ftype(\alpha) = \ftype(\alpha')$.

\item
If $(P, T, M, R)$ is split:
\begin{eqnarray*}
 \mathcal{FTYPE} &\cong& \{ \text{ $W(R_G)$-orbits in $M^*$ } \}_S. \\
\end{eqnarray*}

\item
There is a surjective morphism
\[ \qpar: \underline{\Hom}(\Gm, P) \rightarrow \mathcal{QPAR}, \]
characterized by the properties
\begin{enumerate}
\item Let $\alpha: \G_{m,S'} \rightarrow G_{S'}$ be some cocharacter.
For each etale $S''\rightarrow S'$, where there is a splitting $(G_{S''}, T, M, R)$ such that
$\alpha_{S''}: \G_{m,S''} \rightarrow T \subset G_{S''}$, we have
\[ \type(\qpar(\alpha)) = W(R_G) \{r \in R \where r \circ \alpha \ge 0\}. \]

\item For each $\alpha: \G_{m,S'} \rightarrow G_{S'}$, $\alpha$ factors through $\qpar(\alpha)$.
\end{enumerate}

There is a commutative diagram
\[ \xymatrix{
\underline{\Hom}(\Gm, P) \ar[r]^-{\qpar} \ar[d]^{\ftype} & \mathcal{QPAR} \ar[d]^{\type}  \\
\mathcal{FTYPE} \ar[r] & \mathcal{TYPE}
} \]
If $(P, T, M, R)$ is split, the morphism on the bottom is induced by
\[
 M^* \ni m \mapsto \{ r \in R \where \langle r,m \rangle \ge 0 \}.
\]

\end{enumerate}
\end{SATZ}

\section{Mixed Shimura data}

In this section we define $p$-integral mixed Shimura data. The notion is basically an integral version of the one defined in Pink's thesis \cite{Pink}. Proofs of the theorems in this section can be found in \cite[section 2]{Thesis}.
Recall the Deligne torus $\SSS := \Res_{\R}^{\C}(\Gm)$, with weight morphism $w: \G_{m,\R} \rightarrow \SSS$. 

\begin{DEF}\label{DEFADMISSIBLEH}
Let $\nP$ be a group scheme of type (P) defined over $\spec(\Zpp)$, and let $\pi: \nP \rightarrow \nP/\nW$ the 
projection, where $W$ is its unipotent radical (\ref{DEFTYPEP}).

A morphism
$\nh: \SSS_\C \rightarrow \nP_{\C}$
is called {\bf admissible} if
\begin{itemize}
\item $\pi \circ \nh$ is defined over $\R$,
\item $\pi \circ \nh \circ \nw$ is a cocharacter of the center of $\nP_{\nSD,\Q}/\nW_{\nSD,\Q}$ defined over $\Q$.
\item Under the weight filtration on $\Lie(\nP_{\nSD,\Q})$ defined by $\Ad \circ \nh \circ \nw$, we have:
\[ W_{-1}(\Lie(\nP_{\nSD,\Q})) = \Lie(\nW_{\nSD,\Q}). \]
\end{itemize}
\end{DEF}

Recall from \cite[Prop. 1.5]{Pink} that
a morphism $\nh: \SSS \rightarrow \GL(\nL_\R)$ is associated with a mixed Hodge structure if and only if $\nh$ factors through a subgroup $\nP_{\nSD,\Q} \subset \GL(\nL_\Q)$, and the induced morphism is admissible.

\begin{LEMMA}\label{LEMMADISCRETE}
 Let $T$ be a $\Q$-torus.
 The following are equivalent conditions:
\begin{enumerate}
 \item $T(\Q)$ is discrete in $T(\Af)$,
 \item $T$ is an almost direct product of a $\Q$-split torus with a torus $T'$, such that $T'(\R)$ is compact.
\end{enumerate}
\end{LEMMA}

\begin{DEF}\label{DEFSHIMURADATA}\label{DEFPINTEGRALMIXEDSHIMURADATA}
Let $p$ be a prime.

A {\bf $p$-integral mixed Shimura datum ($p$-MSD)} $\nSD$ consists of
\begin{enumerate}
 \item a group scheme $\nP_\nSD$ of type (P) over $\spec(\Z_{(p)})$,
 \item a closed unipotent subgroup scheme $\nU_\nSD$,
 \item a homogeneous space $\nX_\nSD$ under $\nP_\nSD(\R)\nU_\nSD(\C)$,
 \item a $\nP_\nSD(\R)\nU_\nSD(\C)$-equivariant finite-to-one morphism
 $\nh_\nSD: \nX_\nSD \rightarrow \Hom(\SSS_\C, \nP_{\nSD,\C})$,
 such that the image consists of {\em admissible} morphisms (\ref{DEFADMISSIBLEH}),
\end{enumerate}

subject to the following condition:
For (one, hence for all) $h_x, x \in \nX_\nSD$,
\begin{enumerate}
\item $\Ad_\nP \circ h_x$ induces on $\Lie(\nP)$ a mixed Hodge structure of type
\[ (-1,1),(0,0),(1,-1) \qquad (-1,0),(0,-1) \qquad (-1,-1),  \]
\item the weight filtration on $\Lie(\nP_\Q)$ is given by
\[ W_i(\Lie(\nP_\Q)) = \begin{cases} \Lie(\nP_\Q) & \text{if $i \ge 0$,} \\ \Lie(\nW_\Q) & \text{if $i = -1$,} \\ \Lie(\nU_\Q) & \text{if $i =-2$,} \\ 0 & \text{if $i < -2$}, \end{cases} \]
where $\nW_\Q$ is the unipotent radical of $\nP_\Q$ and $\nU_\Q=\nU_{\nSD,\Q}$ is a central subgroup,
\item $\Int(\pi(h_x(i)))$ induces a Cartan involution on $\nG^{ad}_\R$, where $\nG = \nP/\nW$,
\item $\nG^{ad}_\R$ possesses no nontrivial factors of compact type that are defined over $\Q$,
\item the center $Z$ of $\nP$ satisfies the properties of Lemma \ref{LEMMADISCRETE}.
\end{enumerate}
$\nSD$ is called {\bf pure}, if $\nW_\nSD=1$.
A {\bf morphism of $p$-MSD} $\nSD \rightarrow \nSDi$
is a pair of a homomorphism of group schemes $\nP_\nSD \rightarrow \nP_\nSDi$ and a homomorphism
$\nX_\nSD \rightarrow \nX_\nSDi$ respecting the maps $\nh$.

We call a morphism an {\bf embedding}, if the morphism of group schemes is a {\em closed} embedding, and
the map $\nX_\nSD \rightarrow \nX_\nSDi$ is {\em injective}.

If we have a $p$-MSD $\nSD$, then an {\bf admissible} compact open subgroup of $\nP_\nSD(\Af)$
is a group of the form $K^{(p)} \nP_\nSD(\Z_p)$, where $K^{(p)}$ is a compact open subgroup of $\nP_\nSD(\Afp)$.
\end{DEF}
see \cite[2.1]{Pink}\footnote{The property 5. in the definition is stated in [loc. cit.] in a weaker form, namely it is required that the action on $W$ is through a torus of type \ref{LEMMADISCRETE}. In \cite{Milne4} this property is called (SV5) for pure Shimura varieties.}
 for the rational case. The subgroup scheme $\nU_\nSD$ is determined by the triple
 $(\nP_\nSD, \nX_\nSD, \nh_\nSD)$.

\begin{DEF}\label{DEFPEMSD}
 A pair ${}^K \nSD$ is called {\bf $p$-integral extended mixed Shimura data ($p$-EMSD)}, where
 $\nSD$ is $p$-integral Shimura data, $K$ is an admissible compact open subgroup of $\nP_\nSD(\Af)$.
 
 A {\bf morphism of $p$-EMSD} ${}^{K'} \nSDi \rightarrow {}^{K} \nSD$
 is a pair $(\gamma, \rho)$, where $\gamma$ is a morphism of $p$-integral mixed Shimura data and $\rho \in \nP_\nSD(\Afp)$ such that $\gamma(K')^\rho \subset K$,

 If $\gamma$ is an automorphism, then we call the morphism a {\bf Hecke operator}.
 If $\gamma$ is an embedding and $\gamma(K_1)^\rho = K_2 \cap P_1(\Af)$, and such that the map
\[ [ \nP_\nSDi(\Q) \backslash \nX_\nSDi \times (\nP_\nSDi(\Af) / K')] \rightarrow [\nP_\nSD(\Q) \backslash \nX_\nSD \times (\nP_\nSD(\Af) / K)] \]
 is a closed embedding (compare also \ref{EMBEDDINGEMSD}), then we call the morphism an {\bf embedding}.
\end{DEF}
\begin{LEMMA}\label{STABILIZER}
$\nP_\nSD(\Q) \cap \Stab(x,\nP_\nSD(\R)\nU_\nSD(\C))K$ is finite for every compact open subgroup $K \subset \nP_\nSD(\Af)$ and trivial for
sufficiently small $K$.
\end{LEMMA}

\begin{DEF}\label{DEFKN}
Let $M$ be an integer.
Let $\nP_\nSD$ be a group scheme over $\Z[1/M]$ of type (P) and $\nX_\nSD$ such that
they define {\em rational} mixed Shimura data $\nSD$. For each $p \nmid M$, $\nP_\nSD \times_{\Z[1/M]} \Z_{(p)}$
will then define {\em $p$-integral} mixed Shimura data which we equally denote by $\nSD$.
Let $\nL_{\Z}$ be a lattice with a faithful
representation of $\nP_\nSD$, i.e. a closed embedding $\nP_\nSD \hookrightarrow \GL(\nL_{\Z[1/M]})$.

For each integer $N$, we define the following compact open subgroup
of $\nP_\nSD(\Af)$:
\[ K(N) := \{ g \in \nP_\nSD(\Af) \where g \nL_{\widehat{\Z}} = \nL_{\widehat{\Z}}, g \equiv \id \mod N \}. \]
If $p \nmid N$ then $K(N)$ is admissible at $p$.
\end{DEF}

\begin{LEMMA}
For each (admissible) compact open subgroup $K \subset \nP_\nSD(\Af)$,
there is a $\gamma \in \nP_\nSD(\Af)$ (resp. $\in \nP_\nSD(\Afp)$) such that
\[ \nP_\nSD^\gamma \subseteq K(1). \]
\end{LEMMA}

\begin{LEMMA}\label{EMBEDDINGEMSD}
\begin{enumerate}
\item
For each ${}^{K_1}\nSD_1$ and an embedding $\alpha: \nSD_1 \hookrightarrow \nSD_2$,
there is an admissible $K_2 \subset \nP_{\nSD_2}(\Af)$ such that ${}^{K_1}\nSD_1 \hookrightarrow {}^{K_2}\nSD_2$ is an embedding (\ref{DEFPEMSD}).

\item
If ${}^{K_1}\nSD_1 \hookrightarrow {}^{K_2}\nSD_2$ is an embedding, and if $K_2$ is neat, then for each $K_2' \subset K_2$, the map ${}^{K_2' \cap \nP_{\nSD_1}(\Af)}\nSD_1 \hookrightarrow {}^{K_2'}\nSD_2$
is an embedding.

\item
Let $\nP_{\nSD_2}$ operate linearly on a $V_\Zpp$, and choose a lattice $V_\Z \subset V_\Zpp$.
There is an integer $N$ such that for all $(M,p)=1, N|M$,
${}^{K(M)_1}\nSD_1 \hookrightarrow {}^{K(M)_2}\nSD_2$ is an embedding.
\end{enumerate}
\end{LEMMA}

We have the following structure theorem for $p$-integral mixed Shimura data
\begin{PROP}\label{FUNCTORIALITYW}\label{UNIPOTENTFIBRE}
Let $\nSD$ be $p$-MSD.

There are smooth closed subgroup schemes of $\nP_\nSD$:
$\nW_\nSD$ (the unipotent radical) and $\nU_\nSD = \nGa(W_{-2}(\Lie(\nP_\nSD)))$. There is a
smooth group scheme $\nV_\nSD \cong \nGa(\gr_{-1}(\Lie(\nP_\nSD))$ and an exact sequence
\[ \xymatrix{ 0 \ar[r] & \nU_\nSD \ar[r] & \nW_\nSD \ar[r] & \nV_\nSD \ar[r] & 0. } \]
There is a closed reductive subgroup scheme $\nG_\nSD$, such that $\nP_\nSD = \nW_\nSD \rtimes \nG_\nSD$.
Any two such subgroup schemes are conjugated by an element in $\nW_\nSD(\Zpp)$.

For an morphism of $p$-MSD $\alpha: \nSD \rightarrow \nSDi$,
we have a diagram
\[ \xymatrix{ 0 \ar[r] & \nU_\nSD \ar[r] \ar[d] & \nW_\nSD \ar[r] \ar[d] & \nV_\nSD \ar[r] \ar[d] & 0 \\
  0 \ar[r] & \nU_\nSDi \ar[r] & \nW_\nSDi \ar[r] & \nV_\nSDi \ar[r] & 0. } \]
If $\alpha$ is an embedding, the vertical maps are closed embeddings, and the outer ones are
given by
saturated inclusions (i.e. inducing an inclusion mod $p$ as well) of the corresponding modules over $\Zpp$.
Furthermore, there are closed reductive subgroup schemes $\nG_\nSD$ and $\nU_\nSDi$ of
$\nP_\nSD$ and $\nP_\nSDi$, respectively, such that $\nP_\nSD = \nW_\nSD\rtimes \nG_\nSD$,
$\nP_\nSDi = \nW_\nSDi \rtimes \nG_\nSDi$ and such that $\nG_\nSD$ is mapped to $\nG_\nSDi$.
\end{PROP}

\begin{DEFPROP}\label{UNIPOTENTFIBRECONVERSE}
Let $\nSD$ be $p$-MSD. We have the following converse to \ref{UNIPOTENTFIBRE}:

Given two $\Zpp$-modules $V$ and $U$ acted on by $\nP_\nSD$ with unimodular {\em invariant}
symplectic form $\Psi: \nV \times \nV \rightarrow \nU$ (i.e. inducing an isomorphism $\nV \cong \nV^*$)
This defines a group scheme $\nW_0$ sitting in an exact sequence:
\[ \xymatrix{ 0 \ar[r] & \nU_0 \ar[r] & \nW_0 \ar[r] &  \nV_0 \ar[r] & 0,  } \]
where $\nU_0:=\nGa(U)$ and $\nV_0:=\nGa(V)$. Using the action of $\nP_\nSD$ we may form a
semi-direct product $\nP_\nSDp:= \nW_0 \rtimes \nP_\nSD$. Assume that
every subquotient of $\Lie(W_{0,\R})$ is of type
\[ \{(-1,0), (0, -1)\} \text{ or } \{(-1,-1)\}. \]

Define $\nX_\nSDp$ as
\begin{gather*} 
\{(x, k) \in \nX_\nSD \times \Hom(\SSS_\C, P'_\C) \where \\
h_x=\phi \circ k; \enspace \pi' \circ k: \SSS_\C \rightarrow (P'/U')_\C \text{ is defined over } \R \}.
\end{gather*}

$\nSDp$ is then $p$-MSD, called a {\bf unipotent extension} of $\nSD$, denoted by $\nSD[U, V]$.

$\nX_\nSDp \rightarrow \nX_\nSD$ it a torsor under $\nW_0 (\R)(\nW_0 \cap \nU_\nSDp)(\C)$.

We have $\nSD[U, V]/W_0 \cong \nSD$.

\end{DEFPROP}

Proposition \ref{UNIPOTENTFIBRE} may be strengthened as follows. Every $p$-MSD satisfies
\[ \nSD \cong (\nSD/\nW_\nSD)[\Lie(\nU_\nSD), \Lie(\nV_\nSD)], \]
the symplectic form and action of $\nG_\nSD$ being determined by $\nSD$.

\begin{DEF}\label{DEFADMISSIBLEPARABOLIC}
Let $\nSD$ be (rational) MSD. For $\nG_{\nSD,\Q}:=\nP_{\nSD,\Q}/\nW_{\nSD,\Q}$ every
$\Q$-parabolic subgroup of $\nP_{\nSD,\Q}$ is the inverse image of a $\Q$-parabolic subgroup of $\nG_{\nSD,\Q}^{ad}$.
Let $\nG_{\nSD,\Q}^{ad} = \nG_{1,\Q} \times \cdots \nG_{r,\Q}$ be the decomposition into $\Q$-simple factors.
Choose $\Q$-parabolic subgroup $Q_{i,\Q} \subseteq \nG_{i,\Q}$ for every $i$, and let $Q_\Q$ be the inverse image of
$Q_{1,\Q}\times \cdots Q_{r,\Q}$ in $\nP_{\nSD,\Q}$. We call $Q_\Q$ an {\bf admissible}
$\Q$-parabolic subgroup of $\nP_{\nSD,\Q}$, if every $Q_{i,\Q}$ is either equal to $\nG_{i,\Q}$ or a
maximal proper $\Q$-parabolic subgroup of $\nG_{i,\Q}$.
\end{DEF}

\begin{LEMMA}\label{LEMMAPARABOLIC_S}
If $\nSD$ is $p$-MSD and $Q_\Q$ is a parabolic subgroup of $\nP_{\nSD,\Q}$, then
there is a parabolic subgroup scheme $\nQ$ of $\nP_\nSD$ (\ref{DEFPARABOLIC}), such that
$Q_\Q = Q \times_S \spec(\Q)$.
\end{LEMMA}

\begin{PAR}\label{LONGROOTS}
Let $S$ be a maximal $\R$-split torus of $(G_\R)^{der}$. Let $R$ be the root system of $S$ (acting on $\Lie(G^{ad})$).
The irreducible components of $R$ are of type (C) or (BC) \cite[Corollaire 3.1.7]{Deligne7}, cf. also \cite[\S 2, Proposition 4]{AMRT}.
The {\em long roots} \cite[loc. cit.]{Deligne7}, cf. also \cite[p. 185]{AMRT}. $L=\{\alpha_1,\dots,\alpha_g \}$
form a basis of mutually orthogonal roots of $S$.
An admissible $\R$-parabolic (\ref{DEFADMISSIBLEPARABOLIC}) $Q_\R \supset S$ is determined (in the sense of \ref{QUASIPARABOLICS}) by a homomorphism $\lambda: \G_{m,\R} \hookrightarrow S$, where
\[ \langle \lambda, \alpha \rangle = \begin{cases} 2 & \alpha \in L', \\ 0 & \alpha \notin L'. \end{cases} \]
where $L'$ is any subset of $L$. If the corresponding cocharacter can be defined over $\Q$, then $Q$ is defined over $\Q$.
\end{PAR}

Recall the group $H_0 := \{(z, \alpha)\in\SSS \times \GL_{2,\R} \where z \overline{z} = \det(\alpha) \}$. Let $h_0: \SSS_\R \rightarrow H_{0,\R}$ be the morphism given by $z \mapsto (z, h(z))$, where
$h \in \nX_{\nH_1}$ is any element. Here $\nX_{\nH_1}$ is the set of representations of type $(-1,0),(0,-1)$ (see \ref{GSP} for the notation).
Let $h_\infty: \SSS_\C \rightarrow H_{0,\C}$ be the morphism $z \mapsto (z, h_m(z))$, where $h_m: \SSS_\C \rightarrow \M{*&*\\&*}$ describes the mixed Hodge structure having the same
Hodge filtration as the one described by $h$, but weight filtration corresponding to (i.e. fixed by) $\M{*&*\\&*}$.

\begin{PROP}\label{BOUNDARYCOMPONENTS}
Let $Q_\Q$ be a $\Q$-parabolic subgroup of $\nP_{\nSD,\Q}$. Let $\pi'$ be the projection $\nP_\nSD \rightarrow \nP_\nSD/\nU_\nSD$.
The following are equivalent:
\begin{enumerate}
\item $Q_\Q$ is admissible
\item
For every $x \in \nX_\nSD$ there is a unique homomorphism
\[ \omega_x: H_{0,\C} \rightarrow \nP_{\nSD,\C} \]
such that
\begin{enumerate}
\item $\pi'\circ \omega_x: H_{0,\C} \rightarrow (\nP_\nSD/\nU_\nSD)_\C$ is already defined over $\R$,
\item $h_x = \omega_x \circ h_0$,
\item $\omega_x \circ h_\infty \circ w: \G_{m,\C} \rightarrow Q_\C$ is conjugated to $\mu \cdot \lambda$, where $\lambda$ is the morphism constructed in \ref{LONGROOTS}, $\mu=h_x \circ w$, and
$\Lie(Q_\C)$ is the direct sum of all nonnegative weight spaces in $\Lie(\nP_{\nSD,\C})$ under $\Ad_P \circ \omega_x \circ h_\infty \circ w$.
\end{enumerate}
\item There exists an $x \in \nX_\nSD$ and a homomorphism $\omega_x$ such that the three conditions in 2. are satisfied.
\end{enumerate}
\end{PROP}

If $\nP_\nSD$ is reductive, $Q$ corresponds to a boundary component in the sense of [AMRT, III, p. 220, no. 2].
The morphism $\omega_x \circ h_\infty$ is independent of the choice of $h_0$.

\begin{DEFSATZ}\label{DEFBOUNDARYCOMP}\label{INTEGRALBOUNDARYCOMP}
Assume that $\nSD$ is $p$-MSD. Choose an admissible $Q_\Q$ as above.
We will define mixed Shimura data $\nSDB$ as follows:

Let $\nP_{\nSDB,\Q}$ be the smallest normal $\Q$-subgroup of $Q$, such that $\omega_x \circ h_\infty$ factorizes through it.

Consider the map
\begin{eqnarray}\label{bdrsymdomain}
 \nX_\nSD &\rightarrow& \pi_0(\nX_\nSD) \times \Hom(\SSS_\C, \nP_{\nSDB,\C}) \\
 x &\mapsto & ([x], \omega_x \circ h_\infty). \nonumber
\end{eqnarray}

Choose a $\nP_\nSDB(\R)\nU_\nSDB(\C)$-orbit $\nX_\nSDB$ containing an $([x], \omega_x \circ h_\infty)$ in the above image. The
image is contained in the union of finitely many such. Each $\nX_\nSDB$ is a finite covering of
the corresponding $\nP_\nSDB(\R)\nU_\nSDB(\C)$-orbit $h(\nX_\nSDB)$ in $\Hom(\SSS_\C, \nP_{\nSDB,\C})$.
Let $\nX_{\nSDB \Longrightarrow \nSD}$ be its inverse image in $\nX_\nSD$.

The closure $\nP_\nSDB$ of $\nP_{\nSDB,\Q}$ in $\nP_\nSD$ is of type (P) \cite[Theorem 2.4.5]{Thesis}, and hence
$\nSDB$ is $p$-integral mixed Shimura data and called a {\bf boundary component} of $\nSD$. It is called {\bf proper}, if $Q$ is a proper parabolic, otherwise
{\bf improper}.

A {\bf boundary map} $\nSDB' \Longrightarrow \nSD$ is an isomorphism of $\nSDB'$ with one of the boundary components $\nSDB$.
\end{DEFSATZ}

For each $Q$ there are finitely many choices of $\nX_\nSDB$'s and accordingly, finitely many boundary components.

\begin{PROP}\label{IMAGINARYPART}
There is a functorial map, called {\bf projection on the imaginary part},
\[ \im: \nX_\nSD \rightarrow \nU_\nSD(\R)(-1) \qquad x \mapsto u_x, \]
where $u_x$ is the unique element, such that $\Int(u_x^{-1})\circ h_x$ is defined over $\R$.
\end{PROP}

\begin{PROP}\label{FUNCTORIALITYBOUNDARYCOMP}
If $\nSDB$ is a boundary component of $\nSD$ and $\alpha: \nSD \rightarrow \nSDp$ is a morphism, there is a unique
boundary component $\nSDBp$ of $\nSDp$ and a corresponding
map $\widetilde{\alpha}: \nSDB \rightarrow \nSDBp$. If $\alpha$ is an embedding,
$\widetilde{\alpha}$ is either.
\end{PROP}

\begin{PROP}\label{BOUNDARYCOMPONENTS2}
\begin{enumerate}
\item
Let $\nX_{\nSDB \Longrightarrow \nSD}$ be as in Definition \ref{DEFBOUNDARYCOMP}. Let $\nX_\nSD^0$ be a connected component
of $\nX_{\nSDB \Longrightarrow \nSD}$ and $\nX_\nSDB^0$ be corresponding component of $\nX_\nSDB$.
Then
\begin{enumerate}
\item
The map $\nX_{\nSDB \Longrightarrow \nSD} \rightarrow \nX_\nSDB$ is an open embedding.
\item
The image of $\nX_\nSD^0$ in $\nX^0_\nSDB$ is the inverse image of an open complex cone $C:=C(\nX_\nSD^0, P_\nSDB) \subset U_\nSDB(\R)(-1)$ under the map $\im|_{\nX_\nSDB^0}$
\item
The cone is an orbit in $\nU_\nSDB(\R)(-1)$ under translation by $\nU_\nSD(\R)(-1)$ and conjugation by $Q(\R)^\circ$. It is also invariant under translation
by $(\nU_\nSDB \cap \nW_\nSD)(\R)(-1)$.
\item
 Modulo $(\nU_\nSDB \cap \nW_\nSD)(\R)(-1)$ the cone $C$ is a non-degenerate homogeneous self-adjoint cone (in the sense of [AMRT, II, p. 57, \S 1.1]).
\end{enumerate}
\item
Consider a morphism of Shimura data $\iota: \nSD \rightarrow \nSDp$. For each
rational boundary component $\nSDB$ of $\nSD$ there is a unique boundary component $\nSDBp$ of
$\nSDp$ and a morphism  $\widetilde{\iota}: \nSDB \rightarrow \nSDBp$
such that
\[\xymatrix{
\nX_{\nSDB \Longrightarrow \nSD} \ar[r]^\iota \ar@{^{(}->}[d] & \nX_{\nSDBp \Longrightarrow \nSDp} \ar@{^{(}->}[d] \\
\nX_\nSDB \ar[r]^{\widetilde{\iota}} & \nX_\nSDBp \\
 } \]
commutes.
\item Each boundary component $\nSDBp$ of $\nSDB$ is naturally a boundary component of $\nSD$.
This defines a partial order on the set of boundary components of $\nSD$.
\end{enumerate}
\end{PROP}

\begin{DEF}\label{DEFCONICALCOMPLEX}
Let $\nSDB_1$ be a rational boundary component of $\nSD$ and $\nX_\nSD^0$ be a connected
component of $\nX_{\nSDB_1 \Longrightarrow \nSD}$.
Let $C^*(\nX_\nSD^0, \nP_{\nSDB_1}) \subset \nU_{\nSDB_1}(\R)(-1)$ denote the union of the cones $C(\nX_\nSD^0, \nP_{\nSDB_2})$ for all rational
boundary components $\nSDB_2$ such that
$\nSDB_1 \Longrightarrow \nSDB_2 \Longrightarrow \nSD$. It is a convex cone.
Form the following quotient
\[ \nC_\nSD := \coprod_{(\nX_\nSD^0, \nP_{\nSDB_1})} C^*(\nX_\nSD^0, \nP_{\nSDB_1}) / \sim \]
by the equivalence relation generated by the graph of all embeddings
\[ C^*(\nX_\nSD^0, \nP_{\nSDB_2}) \hookrightarrow C^*(\nX^0_\nSD, \nP_{\nSDB_1}) \] 
for $\nSDB_1 \Longrightarrow \nSDB_2 \Longrightarrow \nSD$.
It is called the {\bf conical complex} associated with $\nSD$ (cf. \cite[4.24]{Pink}).
\end{DEF}

Consider the set
\[ \nC_\nSD \times \nP_\nSD(\Af). \]

$\nP_\nSD(\Q)$ acts on this from the left by conjugation of boundary components \cite[4.23]{Pink} and on $\nP_\nSD(\Af)$ by left multiplication.
$\nP_\nSD(\Af)$ acts via multiplication on the right on the second factor. 
Furthermore $\nP_\nSD(\Af)$ acts on the second factor through left multiplication. 

Let a set $\nRPCD$ of subsets of $\nC_\nSD \times \nP_\nSD(\Af)$ be given, such that every $\sigma \in \nRPCD$ is contained in some
$\nC(\nX_\nSD^0, \nP_\nSDB) \times {\rho}$. We denote by $\nRPCD(\nX_\nSD^0, \nP_\nSDB, \rho)$ the subset of $\sigma \in \nRPCD$, such that 
$\sigma \subset \overline{\nC(\nX_\nSD^0, \nP_\nSDB) \times {\rho}}$.

Let $K$ be an (admissible) compact open subgroup of $\nP_\nSD(\Af)$.
For any $\nX_\nSD^0$, $\nSDB$ and $\rho$, we let $\Gamma_\nU \subset \nU_\nSDB(\Q)$ be the image of
\[ \{z \in Z(\nP_\nSD(\Q)) \where z|_{\nX_\nSD} = \id\} \nU_\nSDB(\Q)) \cap {}^\rho K \]
under the projection $Z(\nP_\nSD(\Q)) \times \nU_\nSDB(\Q) \rightarrow \nU_\nSDB(\Q)$.

\begin{DEF}\label{DEFRPCD}
$\nRPCD$ is called a {\bf $K$-admissible (partial) rational polyhedral cone decomposition} for $\nSD$, if
\begin{enumerate}
 \item For each $\nX_\nSD^0, \nSDB$ and $\rho$, $\nRPCD(\nX_\nSD^0, \nP_\nSDB, \rho)$ is a (partial) rational polyhedral cone decomposition of
the closure of $C(\nX_\nSD^0, \nP_\nSDB) \times {\rho}$. {\em We understand a cone as open in its closure}.
 \item $\nRPCD$ is invariant under right multiplication of $K$ and
   under left multiplication of $P_\nSD(\Q)$.
 \item For each $\nSDB$ the set $\bigcup_{\rho \in \nP_\nSD(\Af)}\nRPCD(\nX_\nSD^0, \nP_\nSDB, \rho)$ is invariant
 under left multiplication of $\nP_\nSDB(\Af)$.
\end{enumerate}
It is called {\bf finite}, if the quotient $\nP_\nSD(\Q) \backslash \nRPCD / K$ is finite.

It is called {\bf complete}, if in 1. $\nRPCD(\nX_\nSD^0, \nP_\nSDB, \rho)$ is a complete rational polyhedral cone decomposition.

It is called {\bf projective}, if on each $\nRPCD(\nX_\nSD^0, \nP_\nSDB, \rho)$ there exists a polarization function (\cite[IV, \S 2.1]{AMRT}, cf. \cite[IV, 2.4]{FC}).

It is called {\bf smooth} with respect to $K$, if for all $\nX_\nSD^0$, $\nSDB$ and $\rho$, as above,
$\nRPCD (\nX^0_\nSD, \nP_\nSDB, \rho)$ is smooth with respect to the lattice $\Gamma_\nU$.
\end{DEF}

The condition 3. is called the arithmeticity condition. Without it, the compactification exists over $\C$ but may not
descend to the reflex field or a reflex ring.

\begin{DEF}\label{DEFPECMSD}\label{DEFFUNCTORIALITYRPCD}
A triple ${}^K_\nRPCD \nSD$ is called {\bf $p$-integral extended compactified mixed Shimura data ($p$-ECMSD)}, where everything is as in
Definition \ref{DEFPEMSD}, but $\nRPCD$ is in addition a $K$-admissible (partial) rational polyhedral cone decomposition.

Morphisms of $p$-ECMSD have to satisfy the property that
for each $\sigma_1 \in \nRPCD_1$ there is a $\sigma_2 \in \nRPCD_2$ with $\gamma(\sigma_1)^\rho \subset \sigma_2$.

Let ${}^K_\nRPCD \nSD$ be $p$-ECMSD and $[\alpha, \rho]: {}^{K'} \nSDi \rightarrow {}^{K} \nSD$ be a morphism of $p$-EMSD, such that
$\alpha$ is a {\em closed embedding}.

Set $[\alpha, \rho]^* \nRPCD$ to be the set of all cones
$\{ (u, \rho') \where (\alpha(u), \alpha(\rho') \rho) \in \sigma \}$ for all $\sigma \in \nRPCD$. This is a
$K'$-admissible rational partial cone decomposition for $\nSDi$. It is finite, resp. complete, resp. projective
if $\nRPCD$ is finite, resp. complete, resp. projective.

This association is functorial. If $[\alpha, \rho]$ was an embedding (this includes a condition on $K, K'$, see \ref{DEFPEMSD}), we call
${}^{K'}_{\nRPCD'} \nSDi \rightarrow {}^{K}_{\nRPCD} \nSD$
an {\bf embedding}.
\end{DEF}

Be aware that, in general, smoothness is not inherited by $[\alpha, \rho]^* \nRPCD$, but we have:

\begin{PROP}[{\cite[Theorem 2.4.12]{Thesis}}]\label{PROPERTIESRPCD}
Let $\nSD$ be $p$-MSD. Let $K$ be an admissible compact open. 
\begin{enumerate}
 \item $K$-admissible (complete) rational polyhedral cone decompositions for $\nSD$ exist.
 \item If $\nRPCD$ is a $K$-admissible rational polyhedral cone decomposition for $\nSD$, there is a smooth and projective refinement $\nRPCD'$.
  {\em Any} refinement of $\nRPCD'$ will be projective again.
 \item If $\nRPCD_i$, $i=1,2$ are 2 rational polyhedral cone decompositions for $\nSD$, there is a common refinement $\nRPCD$ 
   (supported on the intersection of their supports).
 \item If $\alpha: \nSD \rightarrow \nSDp$ is an embedding and $\rho \in \nP_\nSDp(\Afp)$ is given, there is $K'$ 
such that we have an embedding $[\alpha, \rho]: {}^K \nSD \rightarrow {}^{K'} \nSDp$,
a $K'$-admissible rational polyhedral smooth and projective cone decomposition $\nRPCD'$ for $\nSDp$, 
with the property that $\nRPCD:= [\alpha, \rho]^* \nRPCD'$ is smooth and projective. For {\em every} smooth refinement of $\widetilde{\nRPCD}$ of $\nRPCD$, there is 
a smooth refinement $\widetilde{\nRPCD}'$ of $\nRPCD'$ with $\widetilde{\nRPCD}:= [\alpha, \rho]^* \widetilde{\nRPCD}'$.
\end{enumerate}
\end{PROP}

\begin{DEF}\label{DEFBOUNDARYCOMPRPCD}
Let ${}^K_\nRPCD \nSD$ be $p$-ECMSD and $\iota: \nSDB_1 \Longrightarrow \nSD$ a boundary map.
For any $\rho \in \nP_\nSD(\Af)$, we define $K' := \nP_{\nSDB_1}(\Af) \cap {}^\rho K$, write
\[ (\iota, \rho): {}^{K'}_{\nRPCD'} \nSDB_1 \Longrightarrow {}^K_{\nRPCD} \nSD, \]
and call this a boundary component of (or boundary morphism to) the $p$-ECMSD ${}^K_{\nRPCD} \nSD$.
$\nRPCD'$ is defined as $([\iota,\rho]^* \nRPCD)|_{\nSDB_1}$, where restriction is characterized by
\[ \nRPCD|_{\nSDB_1}( \nX_{\nSDB_1}^0, \nP_{\nSDB_2}, \rho_1 ) = \nRPCD( \nX_\nSD^0, \nP_{\nSDB_2}, \rho_1 ) \]
for all $\rho_1 \in \nP_{\nSDB_1}(\Af)$, every boundary map $\nSDB_2 \Longrightarrow \nSDB_1$ and every pair
of connected components $\nX^0_\nSD$ and $\nX^0_{\nSDB_1}$ such that $\nX^0_\nSD \hookrightarrow \nX^0_{\nSDB_1} \hookrightarrow \nX_{\nSDB_2}$.
$\nRPCD'$ in general inherits neither completeness nor finiteness. It is $K'$-admissible.

We call two boundary components
\[ (\iota', \rho'): {}^{K'}_{\nRPCD'} \nSDB' \Longrightarrow {}^K_{\nRPCD} \nSD \]
and
\[ (\iota'', \rho''): {}^{K''}_{\nRPCD''} \nSDB'' \Longrightarrow {}^K_{\nRPCD} \nSD \]
{\bf equivalent}, if (the images of) $\nSDB'$ and $\nSDB''$ are conjugated via $\alpha \in \nP_\nSD(\Q)$ and
\[ \alpha \rho' \in \Stab_{Q(\Q)}(\nX_\nSDB) \nP_\nSDB(\Af) \rho'' K. \]
(Here $Q$ is the parabolic defining $\nSDB$.)
In an equivalence class, we may assume $\rho \in \nP_\nSD(\Afp)$. 
\end{DEF}

\begin{DEF}\label{DEFCONCENTRATIONUNIPOTENTFIBRE}
Let ${}^K_\nRPCD \nSD$ be $p$-ECMSD.
Define $\nRPCD^0$ as the set of all $\sigma \in \nRPCD$, such that $\sigma \subset C(\nX_\nSD^0, \nP_\nSDB)\times \nP_\nSD(\Af)$
for some {\em improper} rational boundary component $\nSDB\Longrightarrow\nSD$. We say that $\nRPCD$ is {\bf concentrated in the
unipotent fibre}, if $\nRPCD=\nRPCD^\circ$.
\end{DEF}

\section{Symplectic, Hodge type, and Abelian type Shimura data}

\begin{PAR}\label{GSP}
Let $S$ be a scheme.
Let $\nL$ be a locally free sheaf on $S$ with unimodular alternating form $\langle v, w \rangle = \Psi(v,w)$, i.e.
satisfying $\langle v, v \rangle = 0$ and such that the induced homomorphism $\nL \rightarrow \nL^*$ is an isomorphism.

Let us assume first that $\nL$ is {\em non-zero}. We then have the reductive group schemes $\Sp(\nL)$ (resp. $\GSp(\nL)$) over $S$, the {\bf symplectic group}, resp. the
{\bf group of symplectic similitudes} of $\nL$.

There is an exact sequence
\[ \xymatrix{ 0 \ar[r] & \Sp(\nL) \ar[r] & \GSp(\nL) \ar[r]^\lambda & \Gm \ar[r] & 0, } \]
where $\lambda$ is the similtude factor.
Let $\rho$ denote the standard representation of $\GSp(\nL)$ on $\nL$.

Let now $S=\spec(\Zpp)$. We define $p$-integral pure Shimura data $\nH_g$ associated with $\nL$
(it depends, up to isomorphism, only on the rank $2g$ of $\nL$) by
$\nP_{\nH_g} := \GSp(\nL)$ and $\nX_{\nH_g}$ to be {\em the} conjugacy class of morphisms $h: \SSS \rightarrow \GSp(\nL_\R)$,
such that they give pure Hodge structures of type $(-1,0),(0,-1)$ on $\nL$ and
which are {\bf polarized}, i.e. such that
the form $\langle \cdot, h(i) \cdot \rangle$ is symmetric and (positive or negative) definite.

If $\nL$ is the {\em zero} sheaf on $S$ we define $\Sp(\nL) := 1$ and $\GSp(\nL):=\G_{m,S}$, and we let $\nX_{\nH_0}$ be the 2 point set of isomorphisms $\Z \rightarrow \Z(1)$ with
the nontrivial action of $\Gm(\R)$. This defines a ($p$-integral) Shimura datum
$\nH_0$. We understand the morphism $\lambda: \GSp(0) \rightarrow \Gm$ be the identity.
\end{PAR}

\begin{PAR}\label{PSP}
Let $\nL_0$ be a locally free sheaf on $S$ as before with unimodular sympletic form (possibly 0).

Let $\nI$ be another locally free sheaf on $S$ (also possibly 0).
Let $\nI^*$ be its dual. 
We define the semi-direct product
\[ \PSp(\nL_0, \nI) := \nGa(\nL_0 \otimes \nI) \rtimes \GSp(\nL_0). \]
It is of type (P). Here $\GSp(\nL_0)$ acts via the product of the standard representation on $\nL_0$ with the
trivial one on $\nI$.
It acts on $\nL := \nL_0 \oplus \nI^*$ as follows: The action of $\GSp$ is given by the standard
representation on $\nL_0$ and trivial action on $\nI^*$, $X = v' \otimes u' \in \nL_0 \otimes \nI$ 
acts, considered as element of the Lie algebra, as
\[ X(v, u^*) = ((u^* u) v',0) \]
and $\nGa(\nL_0 \otimes \nI)$ via the exponential $\exp(X)(v, u^*) = (v, u^*) + X(v, u^*)$, which in this case makes sense over any $S$.
This is compatible with the structure of semi-direct product.
Let $S=\spec(\Zpp)$ from now on. The action fixes a weight filtration
\[ W_i(\nL_0 \oplus \nI^*) := \begin{cases} 0 & i \le -2 \\
 \nL_0 & i=-1 \\
 \nL_0 \oplus \nI^*  & i \ge 0.
\end{cases} \]

We have the unipotent extension $\nH_{g_0}[0,\nI\otimes\nL_0]$
of the $p$-integral pure Shimura data $\nH_{g_0}$.
Its underlying $\nP$ is $\PSp(\nL_0, \nI)$.
Its underlying $\nX$ (if $g_0 \not=0$ may be identified with {\em the} conjugacy class of morphisms $h: \SSS \rightarrow \PSp(\nL_{0,\R}, \nI_\R)$,
such that they give mixed Hodge structures of type $(-1,0),(0,-1),(0,0)$ with respect to the weight filtration $W$,
such that on $W_{-1}$ they are polarized.
(We have $\PSp(0, \nI) = \GSp(0)$ and $\PSp(\nL_0, 0) = \GSp(\nL_0)$).
\end{PAR}

\begin{PAR}\label{USP}
Let $\nL_0$ be a locally free sheaf on $S$ with unimodular symplectic form as before (possibly 0).
Consider the following extension of Abelian unipotent groups ({\bf Heissenberg group})
\[ \xymatrix{ 0 \ar[r] & \nGa((\nI \otimes \nI)^s) \ar[r] & \WSp \ar[r] & \nGa(\nL_0 \otimes \nI) \ar[r] & 0 } \]
defined (if $2$ is invertible in $S$) by the following group law 
\[ (u_1 u_2, v u_3) (u_1' u_2', v' u_3') = (u_1 u_2 + u_1' u_2' {\bf + \frac{1}{2}\langle v,v'\rangle (u_3 u_3'+u_3' u_3)
 }, v u_3 + v'u_3'). \]

There is an action of $\GSp(\nL_0)$ on $\WSp$ given by $\lambda$ acting on $(\nI\otimes \nI)^s$ by scalars and standard representation tensored with the trivial one on $\nL_0 \otimes \nI$.

We define the semi-direct product
\[ \USp(\nL_0, \nI) := \WSp \rtimes \GSp(\nL_0). \]
It is again of type (P).

Denote $\mathfrak{I} := \nI \oplus \nI^* $. Choose on $\nL_0 \oplus \mathfrak{I}$ the symplectic form
\[ \langle v_1,u_1,u_1^* ; v_2,u_2,u_2^* \rangle := \langle v_1, v_2 \rangle + u_2^*u_1 - u_1^* u_2.  \]

We define an action of $\USp(\nL_0, \nI)$ on $\nL:=\nL_0 \oplus \mathfrak{I}$ as follows: The action of $\GSp(\nL_0)$ is given by the standard
representation on $\nL_0$, trivial representation on $\nI^*$ and $\lambda$ acting by scalars on $\nI$.
$X = v' \otimes u' \in \nL_0 \otimes \nI)$ acts, considered as element of the Lie algebra, as
\[ X(v, u, u^*) = ( (u^* u') v' , \langle v',v \rangle u', 0), \]
and $X = u_1 \otimes u_2 \in \nI \otimes \nI$ acts by
\[ X(v, u, u^*) = (0,(u^* u_1) u_2,0), \]
and $\nGa(\cdots)$ acts via the exponential $\exp(X)(v, u^*) = (v, u^*) + X(v, u^*) + \frac{1}{2}X^2(v, u^*)$. This is compatible with the group structure given above.

If 2 is not invertible in $S$, we assume that there is an isomorphism $\nL_0 \cong \nL_{00} \oplus \nL_{00}^*$, such
that the alternating form is given by the standard one. In this case we let the groups
$\nGa(\nL_{00} \otimes \nI)$, $\nGa(\nL_{00}^* \otimes \nI)$ and $\nGa((\nI \otimes \nI)^s)$ operate as above via exponential 
(it terminates after the second step). One checks 
that
\[ \nGa((\nI \otimes \nI)^s) \times \nGa(\nL_{00} \otimes \nI) \times \nGa(\nL_{00}^* \otimes \nI) \rightarrow \GSp(\nL) \]
is a closed embedding onto a subgroup scheme.
Explicitly the group law is given by
\[  (X_1,X_2,X_3)(X_1',X_2',X_3') = (X_1+X_1'{\bf+\langle X_3,X_2'\rangle - \langle X_2',X_3\rangle},X_2+X_2',X_3+X_3'). \]

This is in any case compatible with the structure of semi-direct product.
Let $S=\spec(\Zpp)$ from now on. The action fixes a weight filtration
\[ W_i(\nL) := \begin{cases} 0 & i \le -3, \\
 \nI & i=-2, \\
 \nL_0 \oplus \nI & i=-1, \\
 \nL  & i \ge 0.
\end{cases} \]

We have the unipotent extension $\nH_{g_0}[(\nI\otimes\nI)^s,\nI\otimes\nL_0]$
of the $p$-integral pure Shimura data $\nH_{g_0}$.
Its underlying $\nP$ is $\USp(\nL_0, \nI)$.
Its underlying $\nX$ (if $g_0 \not=0$ may be identified with {\em the} conjugacy class of morphisms $h: \SSS \rightarrow \USp(\nL_{0,\R}, \nI_\R)$,
such that they give (polarized)
mixed Hodge structures on $\nL_\R$ of type 
$(-1,-1)$, $(-1,0)$, $(0,-1)$, $(0,0)$ with respect to the weight filtration above.

The action of $\USp(\nL_0, \nI)$ on $\nL=\nL_0 \oplus \mathfrak{I}$ defined above induces a closed embedding
\[ \USp(\nL_0, \nI) \hookrightarrow \GSp(\nL). \]
The image is precisely the subgroup scheme fixing the weight filtration and $\nI^*=\gr_{0}(\nL_0 \oplus \mathfrak{I})$ point-wise.
\end{PAR}

\begin{DEF}\label{SYMPLECTICSHIMURADATA}
For any saturated submodule $\nU' \subseteq (\nI \otimes \nI)^s$ we call the $p$-integral mixed Shimura data
\[ \nH_{g_0}[(\nI \otimes \nI)^s, \nI \otimes \nL_0] / \nGa(\nU') \]
a Shimura datum of {\bf symplectic type}.
\end{DEF}
Note that any of the $p$-integral mixed Shimura data in this section is of this form.

\begin{PROP}[{\cite[Theorem 2.5.4]{Thesis}}]\label{SYMPLECTICBOUNDARYCOMP}
There is  a bijection
\[  \{ \text{ isotropic subspaces $\nI_\Q$ of $\nL_\Q$ }\} \cong \{ \text{ boundary components $\nSDB$ of $\nH_g$} \}. \]

Let $\nI_\Q$ be an isotropic subspace of dimension $g-g_0$, $\nI=\nI_{\Zpp}$ the associated saturated sublattice and $\nSDB \Longrightarrow \nH_g$ the corresponding boundary component.

For any choice of
splitting $\nL \cong \nL_0 \oplus \mathfrak{\nI}$ (where $\mathfrak{\nI} = \nI^* \oplus \nI$, as usual, with natural symplectic form),
there is an isomorphism 
\[ \nH_{g_0}[(\nI \otimes \nI)^s, \nI \otimes \nL_0] \cong \nSDB, \]
whose underlying morphism is given by the representation of $\PSp(\nL_0, \nI)$ on
$\nL = \nL_0 \oplus \nI \oplus \nI^*$ given above (it depends on the choice of splitting).
\end{PROP}

\begin{KOR}\label{SYMPLECTICBOUNDARYCOMP2}
Every boundary component of ($p$-integral) mixed Shimura data of symplectic type is again
of symplectic type.
\end{KOR}

\begin{PAR}\label{FUNCTORIALITYSYMPLECTICSD}

We have quotient maps of mixed Shimura data.
\[ \nH_{g_0}[(\nI\otimes \nI)^s,\nI \otimes \nL_0] \rightarrow \nH_{g_0}[0,\nI \otimes \nL_0] \rightarrow \nH_{g_0}. \]
The first gives an isomorphism
\[ \nH_{g_0}[(\nI\otimes \nI)^s,\nI \otimes \nL_0]/\nU \cong \nH_{g_0}[0,\nI \otimes \nL_0] \]
(where $\nU=\nGa((\nI\otimes \nI)^s)$), and the second gives an isomorphism
\[ \nH_{g_0}[0,\nI \otimes \nL_0]/\nW \cong \nH_{g_0},\]
(where $\nW=\nV=\nGa(\nI \otimes \nL_0)$).
\end{PAR}

\begin{LEMMA}\label{SYMPLECTICGSTRUCTURES}
Let $S$ be a scheme with $H^1_{et}(S, \Gm)=1$ and $\nM$ a locally free sheaf on $S$.
\begin{enumerate}
\item A $(\GSp(\nL), \nL)$-structure on $\nM$ is a unimodular symplectic form on $\nM$ (i.e. inducing an iso $\nM \cong \nM^*$)
up to multiplication by $H^0(S, \OO_S^*)$.

\item A $(\PSp(\nL_0,\nI), \nL)$-structure on $\nM$ is a saturated filtration
\[ W_0=\nM \supset W_{-1} \supset W_{-2}=0 \]
with unimodular symplectic form on $W_{-1}$ up to multiplication by $H^0(S, \OO_S^*)$
and an isomorphism $\rho: W_0/W_{-1} \cong \nI^*_S$.

\item A $(\USp(\nL_0,\nI), \nL)$-structure on $\nM$ is a unimodular symplectic form on $\nM$
up to multiplication by $H^0(S, \OO_S^*)$, a saturated filtration
\[ W_0=\nM \supset W_{-1} \supset W_{-2} \supset W_{-3}=0 \]
such that $W_{-2}$ is isotropic and $W_{-1} = (W_{-2})^\perp$ and an isomorphism $\rho: W_0/W_{-1} \cong \nI^*_S$.
\end{enumerate}
\end{LEMMA}

Similar statements are true for a $(\GSp(\nL)(R),\nL_R)$ (resp. $\dots$) structure on a local system.

\begin{LEMMA}\label{SIMPLEGSTRUCTURES}
Let $S$ be a scheme and $\nM$ a locally free sheaf on $S$.
Let $G = \nGa(\nU_\Zpp) \rtimes \Gm$ acting on $\Zpp \oplus \nU_\Zpp^*$ as follows:
$\Gm$ acts by scalar multiplication on $\Zpp$ and $\Ga(\nU_\Zpp)$ acts via
$u(x, u^*) = (x-u^*u, u^*)$.
Let $S$ be a scheme over $\Zpp$ and $\nM$ be a locally free sheaf on $S$.
\begin{enumerate}
\item A $(G,\Zpp \oplus \nU_\Zpp^*)$-structure on $\nM$ is a 
saturated filtration  
\[ W_0=\nM \supset W_{-1} = W_{-2} \supset W_{-3}=0 \]
with isomorphism
\[ \rho: W_0/W_{-1} \cong \nU^*_S \]
and such that $W_{-1}$ is locally free of rank 1.
\end{enumerate}
\end{LEMMA}

\begin{PAR}
Clearly, given two different representations $\nL, \nL'$ of a group scheme $G$, there is a 1:1 ``change of respresentation correspondence'' between
$G$-structures associated with those representations. For a translation of a certain such change of representation correspondence of $\USp(0, \nI)$ to 1-motives (cf. \ref{H0REPR}), we need a fairly explicit construction of this correspondence 
in terms of linear algebra.

Let $S$ be a $\Zpp$-scheme.
Let $G = \USp(0, \nI)$ and $\nL$ its standard representation. Let $\nM$ be
a locally free sheaf on $S$ with $(G, \nL)$-structure.
We may associate with it a $(G, \Zpp \oplus \nU_\Zpp^*)$-structure $\nM''$ by the following (local) procedure:
Define $\nM'$ to be the quotient of $\nM \otimes \nI^*$ by the kernel of contraction $W_{-1}\otimes \nI^* \cong (\nI \otimes \nI^*)_S \rightarrow \OO_S$ (the isomorphism is defined only up to scalars).  Claim: The set $\Lambda^2 \nI^*$ has a unique lift to $\nM'$ (under the map $\nM \otimes \nI^* \rightarrow (\nM/W_{-1}\nM) \otimes \nI^* = \nI^* \otimes \nI^*$) defined as follows: Let $i_1^* \wedge i_2^*$ be in  $\Lambda^2 \nI^*$. 
We choose lifts $\widetilde{i}_1^*, \widetilde{i}_2^* \in \nM$ with the property that $\langle \widetilde{i}_1^*, \widetilde{i}_2^* \rangle = 0$ .
The projection of the element $\widetilde{i}_1^* \otimes i_2^* - \widetilde{i}_2^* \otimes i_1^*$ in $\nM'$ is then well-defined 
(for this consider a difference $(\widetilde{i}_1^*-\widetilde{\widetilde{i}}_1^*) \otimes i_2^* - (\widetilde{i}_2^*-\widetilde{\widetilde{i}}_2^*) \otimes i_1^*$. 
Both summands lie already in $W_{-1}\nM \otimes \nI^*$ and, since $W_{-1}\nM$ is isotropic, its contraction may be computed by the expression:
$\langle \widetilde{i}_1^*-\widetilde{\widetilde{i}}_1^* , \widetilde{i}_2^* \rangle - \langle \widetilde{i}_2^*-\widetilde{\widetilde{i}}_2^*, \widetilde{\widetilde{i}}_1^* \rangle$. This is zero because 
of the choice of lifts.)
We define $\nM'' = \nM'/\Lambda^2(\nI^*)$ by means of this lift. We get an inherited filtration:
\[ 0 \subseteq W_{-1}\nM'' \subseteq \nM'' \]
such that  $W_{-1}\nM'' \cong \OO_S$ non-canonically and $\nM''/ W_{-1}\nM'' \cong \Sym^2(\nI^*)$ canonically. That is, glueing these local constructions, we get a 
$(G, \Zpp \oplus \nU_\Zpp^*)$-structure, where $\nU_\Zpp = (\nI\otimes \nI)^s$. The construction may be reversed. 
\end{PAR}

\begin{DEF}\label{HODGETYPE}
Let $\nSD$ be $p$-integral mixed Shimura data ($p$-MSD). $\nSD$ is called of {\bf Hodge type}, if
there is an embedding into a $p$-integral mixed Shimura datum of symplectic type (see \ref{SYMPLECTICSHIMURADATA} for the notation)
\[ \nSD \hookrightarrow \nH_{g_0}[(\nI\otimes\nI)^s,\nL_0 \otimes \nI]/\nGa(\nU'). \]

Also all ${}^K \nSD$ and ${}^K_\nRPCD \nSD$ are called of Hodge type.

A mixed Shimura datum $\nSD$ is called of {\bf Abelian type}, 
if there is a mixed Shimura datum of Hodge type $\nSDp$, a central isogeny
$\nP_\nSDp^{der} \rightarrow \nP_\nSD^{der}$ which induces an isomorphism of mixed Shimura data $\nSDp^{ad} \cong \nSD^{ad}$. 
\end{DEF}

\begin{PROP}
Let $\nSD$ be $p$-integral mixed Shimura data and $\nSDB$ a boundary component of
$\nSD$. If $\nSD$ is of Hodge (resp. Abelian) type then $\nSDB$ is either.
\end{PROP}
\begin{proof}An easy extension of \cite[Theorem 2.6.3]{Thesis}.
\end{proof}

\section{Orthogonal Shimura data}

\begin{PAR}
Let $\nL$ be a $\Zpp$-lattice with unimodular quadratic form $Q_\nL$. 
We have the reductive group schemes $\SO(\nL)$ and $\GSpin(\nL)$.
We let $\SO(\nL)$ act on $\nL$ by its natural operation and $\GSpin(\nL)$ by $g v \mapsto g v g'$, where $'$ is the main involution of the Clifford algebra $\Cliff(\nL)$ of $\nL$. Note that these actions are {\em not} compatible
with the projection $l: \GSpin(\nL) \rightarrow \SO(\nL)$. The group $\GSpin(\nL)$ is equipped with a morphism $\lambda: \GSpin(\nL) \rightarrow \Gm$ whose kernel is the Spin group.
\end{PAR}

For the rest of this section we assume $p\not=2$.
Then there exists an orthogonal basis $\{ v_i \}$ of $\nL$, such that all $Q_\nL(v_i)$ are units.
We have the basis
$ v_{i_1} \cdots v_{i_j}, i_1 < \cdots < i_j, j \text{ even} $
of $\Cliff^+(\nL)$ and the basis
$ v_{i_1} \cdots v_{i_j}, i_1 < \cdots < i_j, j \text{ odd} $
of $\Cliff^-(\nL)$.
The trace of any basis element (acting by left multiplication on $\Cliff^+(\nL)$ or $\Cliff^-(\nL)$)
except 1 is 0. The trace of 1 is $2^{m-1}$ in any case.

\begin{LEMMA}\label{LEMMACLIFFORDSYMPLECTIC}
For an element $\delta \in \Cliff^+(\nL)^*$ with $\delta' = -\delta$, the form
\[ \langle x,y \rangle_\delta \mapsto \tr(x \delta y') \]
on $\Cliff^+(\nL)$ is symplectic, unimodular, and $\Spin(Q)$-invariant (resp. $\GSpin(Q)$-invariant up to scalar given by $\lambda$), where these groups act by left multiplication.
\end{LEMMA}
\begin{proof}
The form is invariant because $\tr(AB)=\tr(BA)$. Since $\delta$ is invertible, nondegeneracy is equivalent to that
of $\tr(xy')$ which is given by an invertible diagonal matrix with respect to the basis chosen above.
The form is symplectic because $\tr$ is invariant under $'$ and $\delta' = -\delta$.
\end{proof}

\begin{PAR}\label{ORTHSHIMURADATUM}
Suppose again, $\nL_\Zpp$ is a lattice with unimodular quadratic form.
Suppose that $\nL_\R$ has signature $(m-2,2)$, $m \ge 3$. A polarized Hodge structure
of type $(-2,0),(-1,-1),(0,-2)$ on $\nL_\C$ with $\dim(\nL^{-2,0})=1$
is determined by an isotropic subspace $\nL^{-2,0}$, satisfying $\langle z,\overline{z} \rangle < 0$ for
non-zero $z \in \nL^{-2,0}$. Then the other spaces in the Hodge decomposition 
are determined by $\nL^{0,-2} = \overline{\nL^{-2,0}}$ and $\nL^{-1,-1}=(\nL^{-2,0}+\nL^{0,-2})^\perp$.
We define $\nX_{\nS(\nL)}=\nX_{\nO(\nL)}$ to be the set of these Hodge structures. It is identified
with
\[ \{ <z> \in \PP(\nL_\C) \where \langle z, z \rangle = 0, \langle z,\overline{z} \rangle < 0 \}. \]
It is equipped with a complex structure coming from $\PP(\nL_\C)$.

$\nX_\nS$ has 2 connected components.
There is a 2:1 map
\begin{eqnarray*}
\nX_{\nS(L)} & \rightarrow & \Grass^-(\nL_\R) \\
<z> &\mapsto& D_{<z>}:=(<z,\overline{z}>)^{\Gal(\C|\R)}
\end{eqnarray*}
onto the Grassmannian $\Grass^-$ of negative definite subspaces of $\nL_\R$.
It has a section determined by a choice of $\C$-orientation\footnote{i.e. the choice of a $\R$-linear isometry $D \cong \C$} on any of the subspaces $D$.
The images of these sections are the 2 connected components of $\nX_\nS$.

It will be convenient to fix an orientation on some (hence on all) of the negative definite subspaces $D$. This identifies
$\pi_0(\nX_{\nS(L)})$ with $\nX_{\nH_0} = \Hom(\Z, \Z(1))$ and induces a morphism of $p$-MSD
\[ \nS(\nL) \rightarrow \nH_0. \]
\end{PAR}

The space $D_{<z>}$ has an orthonormal (with respect to $Q_\nL$) basis given by
\begin{equation}\label{x1x2eqn}
 x_1 = \frac{z + \overline{z}}{\sqrt{-\langle z, \overline{z} \rangle}}, \qquad 
 x_2 = \pm \frac{z - \overline{z}}{\sqrt{ \langle z, \overline{z} \rangle}}. 
\end{equation}

\begin{LEMMA}\label{LEMMAGSPINEMBEDDING}
$\nS(L)$ is $p$-integral pure Shimura data.

If $\nL$ has signature $(0,2)$, $\GSpin$ does not operate at all on $\nX_\nS$, if the latter is
defined as above (i.e. consisting of 2 points). Therefore, in this case we get 2 different $p$-integral Shimura data $\nS(\nL)_\pm$.
If $\nL'$ is a saturated sublattice of $\nL$ of signature $(m-n-2,2)$, $0 \le n \le m-2$, it induces an embedding of
$p$-integral Shimura data
\[ \nS(\nL') \hookrightarrow \nS(\nL). \]

For any chosen $\delta \in \Cliff^+(\nL)^*$ with $\delta' = -\delta$ there is an embedding
\[ \nS(\nL) \hookrightarrow \nH(\Cliff^+(\nL), \langle\cdot ,\cdot \rangle_\delta) \]
and accordingly $\nS(\nL)$ is of Hodge type. $\nO(\nL):= \nS(\nL)/\Gm$ is of Abelian type.
\end{LEMMA}
\begin{proof}
To each $<z> \in \nX_{\nS(\nL)}$ the associated morphism $h$ factors through $\GSpin(\nL_\R)$. It is explicitly given by
\begin{eqnarray*}
\SSS &\rightarrow& \Cliff^+(\nL,\R) \\
w=a + bi &\mapsto& a + b i \frac{z \overline{z} - \overline{z} z }{ \langle z, \overline{z} \rangle } = a+b x_1x_2
\end{eqnarray*}
for any choice of $i\in \C$ (note: $x_2$ depends on the choice of $i \in \C$ as well).

It even defines a field isomorphism $\C \cong \Cliff^+(D_{<z>}, \R)$. It acts on $\nL^{-2,0}$ by $w^2$, on $\nL^{0,-2}$
by $\overline{w}^2$ and on $\nL^{-1,-1}$ by $w\overline{w}$.
The second statement is true because the inclusion $\Cliff^+(\nL') \subset \Cliff^+(\nL)$ induces a closed embedding
$\GSpin(\nL') \hookrightarrow \GSpin(\nL)$, and any morphism $h$ of the first datum defines a morphism $h$ for the second
datum of the same type.

For the property of Hodge type,
fix a positive definite $\Zpp$-sublattice $D_0$ which is a direct summand and an 
orthonormal basis $x_1, x_2$ (hence an induced orientation of $D_{0,\R}$).
Consider the element $\delta := x_1 x_2$. It satisfies the requirements of Lemma \ref{LEMMACLIFFORDSYMPLECTIC}.

Let $<z>$ be a point in $\nX_{\nS(\nL)}$.
$P^{-1,0}_z := \frac{z\overline{z}}{\langle z, \overline{z} \rangle}$ resp. its complex conjugate, satisfy
$P^{-1,0}_z+P^{0,-1}_z = \id$, $(P^{i,j}_z)^2=P^{i,j}_z$, and on $P^{i,j}_z \Cliff^+(\nL_\C)$ the morphism $h$ acts as $w^{-i} \overline{w}^{-j}$. 
Furthermore, for $z$ chosen, such that $D_z = D_0$, the form
$\langle \cdot, x_1 x_2 \cdot \rangle_\delta = \langle \cdot, h(i) \cdot \rangle_\delta = \tr(x \delta y' \delta') = 
- \tr(x \delta y' \delta) $ is symmetric and definite.
(Here $i \in \C$ it the root of $-1$ determined by the orientation $x_1 \wedge x_2$ of $D_0 = D_z$, see (\ref{x1x2eqn})).
The substitution $x \mapsto gx$, $y \mapsto gy$ for any $g$ in the spin group changes $x_1 x_2$ in $g(x_1) g(x_2)$ 
but does not affect the properties of symmetry and definiteness. The sign of definiteness, however, is reflected by
the chosen orientation of $D_0$. Hence the map
$<z> \mapsto P^{-1,0}_z\Cliff^+(\nL_\C)$ may be seen as a map $\nX_{\nS(\nL)} \rightarrow \nX_{\nH(\Cliff^+(\nL),\langle\cdot,\cdot \rangle_\delta)}$.
Together with the closed embedding
\[ \GSpin(\nL) \hookrightarrow \GSp(\Cliff^+(\nL),\langle\cdot ,\cdot \rangle_\delta), \]
it induces an embedding $\nS(\nL) \hookrightarrow \nH(\Cliff^+(\nL), \langle\cdot ,\cdot \rangle_\delta)$.
$\nO(\nL)$ is, as quotient of $\nS(\nL)$, clearly of Abelian type.
\end{proof}

We have the following proposition 
(similar to \ref{SYMPLECTICBOUNDARYCOMP}):

\begin{PROP}\label{ORTHBOUNDARYCOMP}
There is  a bijection
\[  \{ \text{ isotropic subspaces $\nI_\Q$ of $\nL_\Q$ }\} \cong \{ \text{ boundary components $\nSDB$ of $\nS(\nL)$ or $\nO(\nL)$} \}. \]

Let $\nI_\Q$ be an isotropic subspace, $\nI=\nI_{\Zpp}$ the associated saturated sublattice and $\nSDB$ the corresponding boundary component.

For any choice of
splitting $\nL = \nI \oplus \nI^* \perp \nL_0$, 
there is an isomorphism 
(in each case, $\nS(\nL)$ or $\nO(\nL)$):
\[ \nSDB \cong \begin{cases}
  \nH_0[\nI \otimes (\nI^\perp/\nI), 0] & \text{if }\dim(\nI)=1 , \\
  \nH_1(\nI)[\Lambda^2 \nI, \nI \otimes (\nI^\perp/\nI)] & \text{if }\dim(\nI)=2,
            \end{cases} \]
where in the first case $\Gm$ acts on $\nI \otimes (\nI^\perp/\nI)$ by scalars and
$\GL(\nI)$ in the second case acts trivial on $\nI^\perp/\nI$ 
(see the definition of unipotent extension \ref{UNIPOTENTFIBRECONVERSE}).
The isomorphisms depend in addition on a common choice of orientation on the maximal negative definite subspaces of $\nL_\R$.
\end{PROP}

\begin{proof}
According to Proposition \ref{BOUNDARYCOMPONENTS} we have to classify admissible parabolic subgroups of $\SO(\nL)$ or $\GSpin(\nL)$ (this amounts to the same). 
We could also restrict ourselves to $\GSpin(\nL_\Q)$ because of the projectivity of $\mathcal{PAR}$ (\ref{QUASIPARABOLICS}).

So let $Q$ be any proper {\em admissible} parabolic \ref{DEFADMISSIBLEPARABOLIC} of $\GSpin(\nL)$. It is well- known that $Q$ corresponds to a saturated filtration
\[ 0 \subset \nI \subseteq \nI^\perp \subset \nL, \]
with $\nI$ isotropic, whose stabilizer it is. We will now determine the boundary component in the sense of \ref{DEFBOUNDARYCOMP} associated with $Q$. First note that
we can identify $\Lie(\SO)=\Lie(\Spin)$ with $\Lambda^2 \nL$ via the action:
\[ (v \wedge w) x := \langle w, x \rangle v - \langle v, x \rangle w. \]
The filtration induces the following saturated filtration
\[ 0 \subseteq \Lambda^2 \nI \subseteq \nI \wedge \nI^\perp \]
consisting of those elements in $\Lambda^2 \nL$ shifting the filtration by at least -2, resp.  -1.
We have $\Lie(W_Q)=\nI \wedge \nI^\perp$, where $W_Q$ is the unipotent radical of $Q$ (exists for parabolics, e.g. by the fact that they are of type (P), cf. \ref{EXISTENCEQUASIPARABOLIC}).

\begin{PAR}
Let $\nL$ be a unimodular quadratic form over a discrete valuation ring $R$ with 2 invertible. 
Decompose $\nL_R = H_{1,R} \perp \cdots \perp H_{k,R} \perp \nL'_R$, where $\nL'$ is
anisotropic over $R$. We get an $R$-split torus of $\GSpin$ as follows:
Let $z_i,z'_i$ be an isotropic basis of $H_{i,R}$ with $\langle z_i, z_i' \rangle=1$.
We get a 1-parameter subgroup $T_i$ consisting of
\[ t_i(\alpha) := \alpha z_i z_i' + z_i'z_i = 1 + (\alpha-1)z_i z_i' \]
in $\GSpin(\nL)$. $l$ induces an isomorphism of it with its image in $\SO(\nL)$.
The product of all $T_i$ is direct and yields a split torus $T$ of $\GSpin(\nL)$.
The product of it with the center of $\GSpin$ is a maximal $R$-split torus.
$l$ induces again an {\em isomorphism} of $T$ with $l(T)$ which is a {\em maximal} $R$-split torus of $\SO$. $t_i(\alpha)$ acts, via the standard representation of $\GSpin(\nL)$ by $z_i \mapsto \alpha^2 z_i$, $z'_i \mapsto z'_i$ and $v \mapsto \alpha v$ for $v \in <z, z'>^\perp$. $l(t_i(\alpha))$ acts, via the standard representation of $\SO(\nL)$ by $z_i \mapsto \alpha z_i$, $z_i' \mapsto \alpha^{-1} z_i'$  and $v \mapsto v$ for $v \in <z,z'>^\perp$.
\end{PAR}

\begin{PAR}\label{cocharorth} Let $\nL_\R$ be a quadratic space over $\R$ with signature $(m-2,2)$.
Decompose $\nL_\R = H_{1,\R} \perp H_{2,\R} \perp \nL'_\R$, where $H_{i,\R}$ are hyperbolic planes.
A maximal $\R$-split torus of $\SO(\nL_\R)$ is $l(T_1\cdot T_2) \cong \G_m^2$. The long roots (\ref{LONGROOTS}) are in this case: $\alpha_1=(1,-1)$ and $\alpha_2=(1,1)$ in $\Z^2 = \Hom(l(T_1T_2), \Gm)$. 
The weight morphism $\mu:=\mu_x=h_x \circ w$ for any $x\in \nX_{\nS(\nL)}$ is given
by the natural inclusion $\Gm \hookrightarrow \GSpin$.  
The homomorphism $\lambda$ satisfying the criterion in \ref{LONGROOTS}, in paricular mapping to $\Spin(\nL)$, and its product with $\mu$ is given in the following table: 
\[ \begin{array}{lll}
\text{subset of long roots} & \lambda & \lambda \cdot \mu \\
\hline
\emptyset & (0; 0,0) & (1; 0,0) \\
\{\alpha_1\} & (0; 1,-1) & (1; 1,-1) \\
\{\alpha_2\} & (-1; 1,1) & (0; 1,1) \\
\{\alpha_1, \alpha_2\} & (-1; 2,0) & (0; 2,0)\\
\end{array}
\]
Here $(l; m, n)$ denotes the cocharacter of $\GSpin$: $\alpha \mapsto \alpha^l t_1(\alpha^m) t_2(\alpha^n)$.
\end{PAR}

\begin{PAR}\label{I1DIM} Case: $\nI$ 1 dimensional (boundary point).

Choose a splitting $\nL = \nI \oplus \nI^* \perp \nL_0$. This is possible because the discriminant is a unit at $p$. We may assume $H_{1,\R} = \nI_\R \oplus \nI^*_\R$. Only the last cocharacter of \ref{cocharorth} is defined over $\Q$ {\em and} yields a proper
parabolic subgroup.

The morphism $\lambda \cdot \mu$ which is 
a weight morphism for the boundary component has the following eigenspaces:
\[ \nL^{-4} = \nI  \qquad \nL^{-2} = \nL_0  \qquad \nL^{0} = \nI^*.  \]

We get a Levi decomposition
\[ Q = \nGa(\nI \otimes (\nI^\perp/\nI)) \rtimes G, \]
where $\nGa(\nI \otimes (\nI^\perp/\nI))$ is the unipotent radical and acts via the exponential (note $\nI \wedge \nI^\perp \cong \nI \otimes (\nI^\perp/\nI)$ in this case).
 
We interpreted the set $\nX_{\nS(\nL)}$ as the set of Hodge structures of weight 2 on $\nL_\R$ with
$\nL_{-2,0}$ isotropic, and therefore $h(\nX_\nSDB)$ can be interpreted as an orbit in the set of 
mixed Hodge structures
with respect to the weight filtration
\[ W_i(\nL) = \begin{cases}
 0 & i<=-5 \\
 \nI & i=-4,-3 \\
 \nI^\perp & i=-2,-1 \\
 \nL & i \ge 0
\end{cases}
\]
containing the same Hodge filtrations as above.

Since the Hodge filtration type stays the same, the induced Hodge structures on $\gr_W(\nL)$ have to be trivial by weight reasons. Also, we see that a mixed Hodge structure
is given by 
\begin{equation}\label{refmhs} 
0 \subseteq F^0 = \nI^*_\C \subseteq F^{-1}  = (\nI^*)^\perp \subseteq \nL_\C .
\end{equation}
The associated morphism $h_0: \SSS_\C \rightarrow \GSpin(\nL_\C)$ is given by $z \mapsto t_1(z\overline{z})$. Its weight morphism is precisely the morphism $\lambda \cdot \mu$ constructed above. The product $\nGa(\nI \otimes (\nI^\perp/\nI)) T_1$ is already a closed normal subgroup scheme of $Q$, hence equal to $\nP_\nSDB$. It is the same as the subgroup scheme underlying $\nH_0[\nI \otimes (\nI^\perp/\nI), 0]$. 

Consider
\[ \pi_0(\nX_{\nS}) \times \Hom(\SSS_\C, \nP_{\nSDB,\C}). \]
$T_1(\R)$ acts by the nontrivial operation on the set $\pi_0(\nX_{\nS})$.

Choose vectors $z, z'$ spanning $\nL$ and $\nL^*$ respectively. 
Via the isomorphism $\nL_0 \rightarrow \nI \otimes \nI^\perp / \nI,$ $k \mapsto z \tensor k$, a vector $k \in \nL_{0,\C}$ maps the mixed Hodge structure (\ref{refmhs}) above to the one determined by 
\[ F^0 = < z' + \langle z, z' \rangle k - \left( \langle k, z' \rangle + \frac{1}{2} \langle z, z' \rangle \langle k, k \rangle \right) z >. \]
If $\Im(k)>0$, this is the filtration associated with a Hodge structure in $\nX_{\nS(\nL)}$.
Hence by Definition \ref{DEFBOUNDARYCOMP} the 
$T_1(\R)\nL_{0,\C}$-orbit $\nX_\nSDB$ in $\pi_0(\nX_{\nS}) \times \Hom(\SSS_\C, \nP_{\nSDB,\C})$ is equal to $\pi_0(\nX_{\nS}) \times \nL_{0,\C} \cdot h_0$.

This is (see Definition \ref{UNIPOTENTFIBRECONVERSE}) isomorphic to
\[ \nX_{\nH_0[\nI \otimes (\nI^\perp/\nI), 0]}, \]
where $\pi_0(\nX_{\nS})$ can be identified (after choice of a common orientation 
of the negative definite subspaces) with the set of isomorphisms $\Z \cong \Z(1)$ (choice of root of $-1$). 
Combining this with the isomorphism of group schemes above, we get an isomorphism
\[ \nSDB \cong \nH_0[\nI \otimes (\nI^\perp/\nI), 0], \]
as required. 
\end{PAR}

\begin{PAR} Case: $\nI$ 2 dimensional (boundary curve).

We may now assume that the whole torus $T_1 T_2$ is defined over $\Q$. We get an associated maximal parabolic by considering the two 
parabolics associated with the morphisms $(1; 1,-1)$ and $(0;1,-1)$.
The morphism $\lambda \cdot \mu$ which is 
a weight morphism for the boundary component has the following eigenspaces:
\[ \nL^{-3} = <z,w'>  \qquad \nL^{-2} = \nL_0  \qquad \nL^{-1} = <z',w>  \]

In the other case the morphism $\lambda \cdot \mu$ which is 
a weight morphism for the boundary component has the following eigenspaces:
\[ \nL^{-3} = <z,w>  \qquad \nL^{-2} = \nL_0  \qquad \nL^{-1} = <z',w'>  \]

If $m=4$ these are {\em not} conjugated, in accordance with the fact that $\Spin(\nL)$ is not simple. 

We get
\[ W_i(\nL) = \begin{cases}
 0 & i<=-4 \\
 \nI & i=-3 \\
 \nI^\perp & i=-2 \\
 \nL & i \ge -1.
\end{cases}
\]

For any isotropic $<z> \subset \nI_\C$ and  corresponding $<z'> = <z>^{\perp,\nI^*_\C}$, we get a filtration
\begin{equation}\label{refmhs2}
 0 \subseteq F^0=<z'> \subseteq F^{-1}=<z'>^\perp \subseteq \nL_\C .
 \end{equation}
If $F^0=<z> \subset \nI_\C$  is a Hodge structure on $\nI_\C$, that is, if
$<\overline{z}> \not= <z>$ then the filtration (\ref{refmhs2}) defines a mixed Hodge structure on $\nL_\C$, together with the weight filtration above.

We have an exact sequence 
\[ \xymatrix{
0 \ar[r] & \Lambda^2 (\nI) \ar[r] & \nI \wedge \nI^\perp \ar[r] & \nI \otimes (\nI^\perp/\nI) \ar[r] & 0
} \]
where $\Lambda^2(\nI)$ are the elements shifting the filtration by 2 and a corresponding central extension
\[ \xymatrix{
0 \ar[r] & \nU_\nSDB:=\nGa(\Lambda^2 (\nI)) \ar[r] & \nW_\nSDB \ar[r] & \nV_\nSDB:=\nGa(\nI \otimes (\nI^\perp/\nI)) \ar[r] & 0 
} \]

If we choose, in a addition to our splitting, a basis $z, w$ (as before) determining
$\nI \cong (\Z_{(p)})^2$ and dual basis $\nI^* \cong (\Z_{(p)})^2$, we get an obvious isomorphism $\det: \Lambda^2(\nI) \cong \Zpp$ and an isomorphism $(\nI \otimes (\nI^\perp/\nI)) \cong L_0^2$. An element
of $\nW_\nSDB(\R)\nU_\nSDB(\C)$ may now be written as $(u, k_1, k_2)$ with $u \in \C$ and
$k_i \in \nL_{0, \R}$.
It sends the Hodge filtration (\ref{refmhs2}) determined by $z=\M{\beta \\ -\alpha}_\nI$
explicitly to the following one:

\[ F^0 = <
\left( \begin{array}{c} \alpha \\ \beta \end{array} \right)_{\nI^*}
+ u \left( \begin{array}{c} \beta \\ -\alpha \end{array} \right)_\nI
+ \alpha k_1 + \beta k_2 - \left( \begin{array}{c} \alpha Q(k_1) \\
\beta Q(k_2) + \alpha \langle k_1, k_2 \rangle \end{array} \right)_\nI>.
\]
(these again contain all filtrations of Hodge structures determined by elements in $h(\nX_{\nS(\nL)})$.)

We also get a morphism $\SL(\nI) \rightarrow \SO(\nL)$, letting
a matrix act on $\nI$ by the standard representation, on $\nI^*$ by its contragredient and
on $\nL_0$ trivially. It lifts to a morphism $\SL(\nI) \rightarrow \Spin(\nL) \subset \GSpin(\nL)$ because
$\SL(\nI)$ is algebraically simply-connected. Note that $t_1(-1)t_2(-1)$ acts the same way than
$-1$ under the above lift. Hence the morphism extends to a morphism
$\rho: \GL(\nI) \rightarrow \GSpin(\nL)$ mapping the diagonal torus to $T_1T_2$.
It hence acts via the standard representation of $\GSpin$ as follows on $\nL$: 
By the standard representation times determinant on $\nI$, by the contragredient times
determinant on $\nI^*$ and by the determinant on $\nL_0$.
The mixed Hodge structure determined by (\ref{refmhs2}) is hence given by 
the morphism $\SSS_\C \rightarrow \GSpin(\nL_\C)$ obtained composing the
morphism $h_0: \SSS \rightarrow \GL(\nI_\R)$ associated with the Hodge structure $F^0=<z> \subseteq \nI_\C$ with $\rho$. Its weight morphism is precisely the morphism $\lambda \cdot \mu$ constructed above.
The product $\nW_\nSDB \GL(\nL)$ is already a closed normal subgroup scheme of $Q$, hence equal to $\nP_\nSDB$.
\end{PAR}

Consider
\[ \pi_0(\nX_{\nS}) \times \Hom(\SSS_\C, \nP_{\nSDB,\C}). \]
The stabilizer in $\GL(\nL_\R)$ of $h_0$ 
is connected (in real topology) and hence acts trivially on the set $\pi_0(\nX_{\nS})$. Hence by Definition \ref{DEFBOUNDARYCOMP} the 
$\nP_\nSDB(\R)\nU_\nSDB(\C)$-orbit $\nX_\nSDB$ in $\pi_0(\nX_{\nS}) \times \Hom(\SSS_\C, \nP_{\nSDB,\C})$ is isomorphic to its projection to $\Hom(\SSS_\C, \nP_{\nSDB,\C})$.

This is (see Definition \ref{UNIPOTENTFIBRECONVERSE}) isomorphic to
\[ \nX_{\nH_1(\nI)[\Lambda^2 \nI, \nI \otimes (\nI^\perp/\nI)]}. \]

Combining this with the isomorphism of group schemes above, we get an isomorphism:
\[ \nSDB \cong \nH_1(\nI)[\Lambda^2 \nI, \nI \otimes (\nI^\perp/\nI)] \]
as required. 
\end{proof}

\section{Integral models of Shimura varieties and their toroidal compactifications}\label{SECTIONINTMODELS}

\begin{PAR}
To prove our main theorems about the arithmetic volume of Shimura varieties of
orthogonal type, we need to assume that the following statements on canonical integral models of them and their toroidal compactifications hold true. We treat exclusively the case of good reduction. In this case the theorems have been proven under a certain hypothesis (\ref{MAINCONJECTURE}) in the thesis of the author \cite{Thesis}.
\end{PAR}

\begin{DEFSATZ}\label{DEFREFLEXRING}
Let $\nSD$ be $p$-integral mixed Shimura data. With $x \in \nX_\nSD$ there is associated $h_x: \SSS_\C \rightarrow \nP_{\nSD,\C}$, hence
a morphism $u_x: \G_{m,\C} \rightarrow \nP_{\nSD,\C}$ via composition with the inclusion of the first factor 
into $\SSS_\C \cong \G_{m,\C} \times \G_{m,\C}$.
The conjugacy class $M$ of these morphisms is defined over a number field $E_\nSD$, called the {\bf reflex field}, unramified over $p$. We call the localization of its ring of integers $\OOO:=\OOO_E \otimes_\Z \Zpp$ at $p$ the {\bf reflex ring}. It is a semi-local ring.
The closure $\overline{M}$ in $\underline{\Hom}(\G_{m,\OOO}, \nP_{\nSD,\OOO})$ maps surjectively onto $\spec(\OOO)$.

We consider the 2-category of Deligne-Mumford stacks over reflex rings as the 2-category of morphisms $(\pi: S \rightarrow \spec(\OOO))$, where $\OOO \subset \C$ is the localization of a ring of integers at $\Z \setminus (p)$ and $\pi: S \rightarrow \spec(\OOO)$ is a morphism of Deligne-Mumford stacks.
Morphisms are commutative diagrams where the map $\spec(\OOO) \rightarrow \spec(\OOO')$ has to come
from the inclusion. 
\end{DEFSATZ}

\begin{PAR}
We begin by describing the functorial theory of (uncompactified) integral models of mixed Shimura varieties in the
case of good reduction. The Main Theorem is essentially due to Kisin \cite{Kisin, Kisin2} or Vasiu \cite{Vasiu1} in the pure case.
\end{PAR}

\begin{DEF}\label{DEFTESTSCHEME}
Let $\OO$ be a discrete valuation ring with fraction field $F$.
Let a {\bf test scheme} $S$ over $\OO$ be as in \cite[Def. 3.5]{Moonen}, i.e.
$S$ has a cover by open affines $\spec(A)$, such that there exist rings $\OO \subseteq \OO' \subseteq A_0 \subseteq A$,
where
\begin{itemize}
\item $\OO \subseteq \OO'$ is a faithfully flat and unramified extension of d.v.r. with $\OO'/(\pi)$
separable over $\OO/(\pi)$.
\item $A_0$ is a smooth $\OO'$-algebra.
\item $A_0 \subseteq A_1 \subseteq \cdots \subseteq A$ is a countable union of etale extensions.
\end{itemize}
\end{DEF}

\begin{HAUPTSATZ}\label{MAINTHEOREM1}
Let $p\not=2$ be a prime. 

There is a {\em unique} (up to unique isomorphism) pseudo-functor $\nSh$ from the category of $p$-EMSD ${}^K \nSD$ of Abelian type to the 2-category of smooth Deligne-Mumford stacks over reflex rings:
\[ {}^K \nSD \mapsto (\nSh({}^K \nSD) \rightarrow \spec(\OO_\nSD)) \]
with a pseudo-natural isomorphism:
\[ [ \nP_\nSD(\Q) \backslash \nX_\nSD \times (\nP_\nSD(\Af) / K) ] \cong (\nSh({}^K \nSD) \times_{\spec(\OOO_\nSD)}\spec(\C))^{an},  \]
satisfying the following properties:

\begin{enumerate}
\item If $(\gamma, \rho): {}^{K'} \nSDi \rightarrow {}^K \nSD$ is an embedding, 
  \[ \nSh(\gamma, \rho) \times \pi_\nSDi: \nSh({}^{K'} \nSDi) \rightarrow \nSh( {}^K \nSD) \times_{\spec(\OOO_\nSD)} \spec(\OOO_\nSDi) \]
   is a normalization followed by a closed embedding.
  If $(\gamma, \rho)$ is a Hecke operator, then $\nSh(\gamma, \rho) \times \pi_\nSDi$ is etale and finite.

\item (rational canonicity) For a pure Shimura datum $\nSDi$, where $\nP_\nSDi$ is a torus, the composition of the
reciprocity isomorphism (normalized as in \cite[11.3]{Pink})
\[ \Gal(\overline{E}|E)^{ab} \cong \pi_0(T_E(\Q) \backslash T_E(\A)), \]
(where $T_E = \Res_\Q^E(\Gm)$) with $\pi_0$ of the reflex norm \cite[11.4]{Pink}
\[ \pi_0(T_E(\Q) \backslash T_E(\A)) \rightarrow \pi_0( \nP_\nSDi(\Q) \backslash \nP_\nSDi(\A) ) \]
defines by means of the natural action of $\pi_0( \nP_\nSDi(\Q) \backslash \nP_\nSDi(\A) )$ on $\nSh({}^K\nSDi)(\C)$ the
rational model $\nSh({}^K\nSDi)_E$.

\item (integral canonicity) The projective limit
\[ \nSh^p(\nSD) := \underleftarrow{\lim}_{K} \nSh({}^{K} \nSD), \]
where the limit is taken over all {\em admissible} compact open subgroups, satisfies the following property:

For any prime $\wp|p$ of $\OOO_\nSD$, each {\em test scheme} (\ref{DEFTESTSCHEME}) 
$T$ over $S=\spec(\OOO_{\nSD,\wp})$, and a given morphism
\[ \alpha_E : T \times_{S} Q \rightarrow \nSh^p(\nSD) \times_{S} Q, \]
where $Q = \spec(E_\nSD)$ is the generic fibre of $S$, there exists a uniquely determined morphism
\[ \alpha: T \rightarrow \nSh^p(\nSD) \]
such that $\alpha_E = \alpha \times_{S} Q$.

\item 
  $M({}^{K} \nSD)$ is a quasi-projective scheme, if $K$ is neat.

\end{enumerate}
\end{HAUPTSATZ}

For $p=2$, it may be necessary to change the notion of `integral canonicity' as formulated above, cf. \cite[3.4]{Moonen}.

The following Main Theorem \ref{MAINTHEOREM1a} on canonical models of {\em toroidal compactifications} was only proven in \cite{Thesis} under the following hypothesis, which remains unproven:

\begin{VERMUTUNG}\label{MAINCONJECTURE}
Let $\nSD$ be pure $p$-integral Shimura data. Let any embedding 
${}^K_\nRPCD \nSD \hookrightarrow {}^{K'}_{\nRPCD'} \nSDp$ in a datum of symplectic type $\nSDp = \nH_{g_0}[(\nI\otimes\nI)^s,\nL_0\otimes \nI]/\nGa(\nU')$ be given, 
such that $\nRPCD$ {\em and} $\nRPCD'$ are smooth projective and $K$, $K'$ are neat.

Consider the toroidal compactification $\nSh({}^{K'}_{\nRPCD'} \nSDp)$ over $\spec(\Z_{(p)})$, constructed in \cite{FC}, and let $D$ be its boundary divisor.
For every prime $\wp|p$ of $\OOO_\nSD$ we have the following:

Let $S=\spec(\OOO_{\nSD,\wp})$ and $\mathcal{M}'$ be the
Zariski closure of (the rational canonical model of)
$\nP_\nSD(\Q) \backslash \nX_\nSD \times (\nP_\nSD(\Af) / K)$ in $\nSh({}^{K'}_{\nRPCD'} \nSDp)_S$.

No {\em connected} component of $N(\mathcal{M}') \cap D_S$ lies entirely in the fibre above $\wp$.
\end{VERMUTUNG}

\begin{PAR}\label{BOUNDARYSTRATA}
The toroidal compactifications will have a stratification, which consists of mixed Shimura varieties itself. For this, we have to set up some notation (cf. \cite[7.3ff]{Pink}).
Let $(\iota, \rho): {}^{K'}_{\nRPCD'} \nSDB \Longrightarrow {}^{K}_{\nRPCD} \nSD$ be a boundary component. 
For each cone 
$\sigma \in \Delta'$ with $\sigma \subset C(\nX^0_\nSD, \nP_\nSDB) \times \rho'$ for some $\rho'$ in the equivalence class of $\rho$,
we define a $p$-ECMSD ${}^{K_{[\sigma]}}_{\nRPCD_{[\sigma]}}\nSDB_{[\sigma]}$ as follows:
$\nSDB_{[\sigma]}$ is defined by
\begin{eqnarray*}
 \nP_{\nSDB_{[\sigma]}} &:=& \nP_\nSDB / <\sigma>, \\
 \nX_{\nSDB_{[\sigma]}} &:=& \nP_{\nSDB_{[\sigma]}}(\R)\nU_{\nSDB_{[\sigma]}}(\C)-\text{orbit generated by } (\nX_\nSD^0/<\sigma>(\C)) \times \{ \sigma \nP_\nSD(\Af) \} \\
  && \text{ in } (\nX_{\nSD}^0/<\sigma>(\C)) \times \{[\sigma] / \nP_\nSD(\Af)\} \text{ for } \sigma \in [\sigma]\cap \nRPCD(\nX_\nSD^0,\nP_\nSDB,1).
\end{eqnarray*}
$K_{[\sigma]}$ is the image of $K'$ under the projection $\nP_\nSDB \rightarrow \nP_{\nSDB_{[\sigma]}}$.
$\nRPCD_{[\sigma]}$ is defined by
\[ \nRPCD_{[\sigma]}(\nX^0_{\nSDB_{[\sigma]}}, \nP_{\nSDB_{[\sigma]}}, \overline{\rho} ) := \{\tau \text{ mod } <\sigma> \where \tau \in \nRPCD(\nX^0_\nSD, \nP_\nSDB, \rho) \text{ such that } \sigma \text{ is a face of } \tau \}. \] 
It is $K_{[\sigma]}$-admissible.

Furthermore, we define
\begin{eqnarray*}
\Gamma &:=& (\Stab_{Q(\Q)}(\nX_\nSDB)\cap (\nP_\nSDB(\Af) ({}^\rho K))) / \nP_\nSDB(\Q),
\end{eqnarray*}
where $Q$ is the parabolic describing $\nSDB$.

$\Sigma(\iota, \rho)$ is defined to be the set of equivalence classes of the action of $\Gamma$ (defined above)
on those $[\sigma] \in \nP_\nSDB(\Q) \backslash (\nRPCD')^0 / \nP_\nSDB(\Af)$, which satisfy $\sigma \subset \nC(\nX_\nSD^0, \nP_\nSDB) \times \rho'$ for some $\nX_\nSD^0$.
\end{PAR}

\begin{PAR}\label{COMPUNIPOTENTFIBRE}
Recall from \ref{UNIPOTENTFIBRE}: Any morphism
\[ \nSD \rightarrow \nSD/\nU \]
is a torsor under the {\em group object}
\[ \nSD/\nW_\nSD[U,0] \rightarrow \nSD/\nW_\nSD \]
(cf. also \cite[6.6--6.8]{Pink}).
The corresponding map of Shimura varieties (assume all $K$'s in the sequel to be neat):
\[ \nSh({}^{K} (\nSD/\nW_\nSD)[U,0]) \rightarrow \nSh({}^{K/U} \nSD/\nW_\nSD) \]
is a split torus with character group (in general non canonically) isomorphic to $U_\Q \cap K$.
If $\nRPCD$ is a $K$-admissible polyhedral cone decomposition for $\nSD$, concentrated in the unipotent fibre (\ref{DEFCONCENTRATIONUNIPOTENTFIBRE}), we can define
an associated torus embedding. Recall: The rational polyhedral cone decomposition is determined by its restriction to 
any $(U_\R)(-1) \times \rho = \nC(\nX^0, \nP_\nSD) \times \rho$ by the
arithmeticity condition (\ref{DEFRPCD}, 3.). Therefore this determines a torus embedding (see [FC, IV, \S 2] for the integral case):
\[ \nSh({}^K \nSD/\nW_\nSD[U,0]) \hookrightarrow \nSh({}^K_\nRPCD \nSD/\nW_\nSD[U,0]) \]
and a corresponding toroidal embedding of the torsor:
\[ \nSh({}^{K'} \nSD) \hookrightarrow \nSh({}^{K'}_\nRPCD \nSD) \]
(here $K'$ and $K$ are related by $K/U = K'/W$ and $K\cap \nU(\Af)=K' \cap \nU(\Af)$.)
\end{PAR}

\begin{HAUPTSATZ}\label{MAINTHEOREM1a}
Let $p\not=2$ be a prime.
{\em Assume Main Conjecture \ref{MAINCONJECTURE}}.

There is a {\em unique} (up to unique isomorphism) extension of the pseudo-functor $\nSh$ to the category of $p$-ECMSD 
of Abelian type --- with sufficiently fine smooth (w.r.t. the respective $K$) and projective $\nRPCD$ --- to the category of Deligne-Mumford stacks over reflex rings, satisfying 

\begin{enumerate}
 \item The morphism $\nSh(\id,1): \nSh({}^K \nSD) \rightarrow \nSh({}^K_\nRPCD \nSD)$ is an open embedding.
 If $(\gamma, \rho): {}^{K'}_{\nRPCD'} \nSDi \rightarrow {}^K_{\nRPCD} \nSD$ is an embedding, 

  \[ \nSh(\gamma, \rho) \times \pi_\nSDi: \nSh({}^{K'}_{\nRPCD'} \nSDi) \rightarrow \nSh( {}^K_\nRPCD \nSD) \times_{\spec(\OOO_\nSD)} \spec(\OOO_\nSDi) \]
 is a normalization map followed by a closed embedding.

\item If $\nRPCD=\nRPCD^0$ (i.e. $\nRPCD$ is concentrated in the unipotent fibre, \ref{DEFCONCENTRATIONUNIPOTENTFIBRE}), 
$\nSh({}^{K}_{\nRPCD} \nSD)$ is the toroidal embedding described in \ref{COMPUNIPOTENTFIBRE}.

\item
  $\nSh({}^{K}_\nRPCD \nSD)$ possesses a stratification into mixed Shimura varieties (stacks)
\[  \underset{ \substack{ [ (\iota, \rho): {}^{K'}_{\nRPCD'} \nSDB \Longrightarrow {}^K_\nRPCD \nSD ] \\
   [\sigma] \in \Sigma(\iota, \rho) } } {\coprod}
[\Stab_{\Gamma}([\sigma]) \backslash \nSh({}^{K_{[\sigma]}}\nSDB_{[\sigma]})].  \]

Here $[(\iota, \rho): {}^{K'}_{\nRPCD'} \nSDB \Longrightarrow {}^K_\nRPCD \nSD]$ in the first line means equivalence classes of boundary
components of ${}^K_\nRPCD \nSD$ (\ref{DEFBOUNDARYCOMPRPCD}).

\item 
For each boundary map (\ref{DEFBOUNDARYCOMPRPCD})
\[ (\iota, \rho): {}^{K'}_{\nRPCD'} \nSDB \Longrightarrow {}^K_\nRPCD \nSD, \]
and $\sigma \in \nRPCD'$ such that $\sigma \subset \nC(\nX_\nSD^0, \nP_\nSDB) \times \rho'$ for some $\nX_\nSD^0$ and $\rho'$, there is a boundary isomorphism
\[ \nSh(\iota, \rho) : \left[\Stab_{\Gamma}([\sigma]) \backslash \widehat{\nSh( {}^{K'}_{\nRPCD'}\nSDB)} \right] \rightarrow \widehat{\nSh({}^K_\nRPCD \nSD)} \]
of the formal completions along the boundary stratum
\[ \left[ \Stab_{\Gamma}([\sigma]) \backslash \nSh( {}^{K_{[\sigma]}}_{\Delta_{[\sigma]}} \nSDB_{[\sigma]} ) \right]. \]

The complexification of $\nSh(\iota, \rho)$ converges in a neighborhood of the boundary stratum and is
in the interior, via the identification with the complex analytic mixed Shimura varieties given by a quotient of the map
\[ \nX_{\nSDB \Longrightarrow \nSD} \times (\nP_\nSDB(\Af) / K')  \rightarrow \nX_{\nSDB \Longrightarrow \nSD} \times (\nP_\nSD(\Af) / K)  \]
induced by the closed embedding $\nP_\nSDB \hookrightarrow \nP_\nSD$, where we considered $\nX_{\nSDB \Longrightarrow \nSD} \subseteq \nX_\nSD$
as a subset of $\nX_{\nSDB}$ via the analytic boundary map (\ref{BOUNDARYCOMPONENTS}).

The stratification is compatible with the functoriality, i.e. the morphisms $\nSh(\cdot)$ induce
morphisms of the strata of the same type again.
\end{enumerate}

We have in addition: 
\begin{itemize}
\item[5.] $\nSh({}^{K}_\nRPCD \nSD)$ is smooth.
\item[6.] $\nSh({}^{K}_\nRPCD \nSD)$ is proper, if $\nRPCD$ is complete.
\item[7.] $\nSh({}^{K}_\nRPCD \nSD)$ is a projective scheme, if $\nRPCD$ is complete and $K$ is neat.
\item[8.] The complement $D$ of the open stratum $\nSh({}^K \nSD)$ in $\nSh({}^K_\nRPCD \nSD)$ ---
the union of all lower dimensional strata in 3. --- is a
smooth divisor of normal crossings on $\nSh({}^K_\nRPCD \nSD)$.
\end{itemize}
\end{HAUPTSATZ}

\begin{BEM}
The stratification in 3. is indexed by pairs of an equivalence class of boundary components $[ (\iota, \rho): {}^{K'}_{\nRPCD'} \nSDB \Longrightarrow {}^K_\nRPCD \nSD ]$
and a class $[\sigma] \in \Sigma(\iota, \rho)$. The set of these pairs is just isomorphic to the set of double cosets $\nP_\nSD(\Q) \backslash \Delta / K$.
The bijection is as follows. Each $\sigma \in \Delta$ is supported on a $\nC(\nX^0_\nSD, \nP_\nSDB) \times \rho$. $\nSDB$ and $\rho$ determine a boundary component 
$(\iota, \rho): {}^{K'}_{\nRPCD'} \nSDB \Longrightarrow {}^K_\nRPCD \nSD$ and the class of the image of $\sigma$ under restriction to $\nRPCD'$ lies in $\Sigma(\iota, \rho)$.

The stratum $\nSh({}^{K_{[\sigma]}} \nSDB_{[\sigma]})$ is contained in the closure of the stratum
$\nSh({}^{K_{[\tau]}} \nSDB_{[\tau]})$, if and only if (up to a change of representatives in $\nP_\nSD(\Q) \backslash \Delta / K$)
$\tau$ is a face of $\sigma$.
\end{BEM}

\begin{PAR}
There exist integral canonical models of the compact duals as well, which are defined as
conjugacy classes of quasi-parabolic subgroup schemes. 
Let $p$ be a prime and $\nSD$ be $p$-integral mixed Shimura data. Let $\OOO$ be the reflex ring of it at $p$ and $S=\spec(\OOO)$. The closure $\overline{M}$ of the conjugacy class of the morphisms $u_x, x \in \nX_\nSD$ is defined over $\OOO$ by Definition \ref{DEFREFLEXRING}. It corresponds to a
section $t' \in \mathcal{FTYPE}(S)$.

The image of $\overline{M}$ via $\qpar$ is a fibre of the morphism `$\type$' above a section $t: S \rightarrow \mathcal{TYPE}$. We define $\nShD(\nSD)$ to be this fibre. It may, as a scheme over $\spec(\Zpp)$, be
seen as an open and closed subscheme of $\mathcal{QPAR}$ itself.
We understand here the action of $\nP_\nSD$ on $\mathcal{QPAR}$ by conjugation on the right (contrary to \ref{QUASIPARABOLICS}).
The action fixes the morphism `$\type$', hence we have an induced operation on $\nShD(\nSD)$.
\end{PAR}

\begin{HAUPTSATZ}\label{MAINTHEOREM2}The construction above yields a functor from $p$-MSD to schemes over reflex rings (defined analogously).
$\nShD(\nSD)$ is a smooth connected quasi-projective right $\nP_{\nSD,\OOO_\nSD}$-homogeneous scheme
over $S$ called {\bf the dual} of $\nSD$ (depending only on $\nP_\nSD$ and $\nh(\nX_\nSD)$) with the following properties:
\begin{enumerate}
\item If $\nSD$ is pure, $\nShD(\nSD)$ is projective.
\item Let $\nL$ be a free $\Zpp$-module of finite dimension, acted on faithfully by $\nP_\nSD$.
$\nShD(\nSD)$ (considered as scheme over $\OO_\nSD$) represents
\[ S' \mapsto \{ \text{ filtrations $\{F^\bullet\}$ on $\nL_{S'}$ compatible with $(\nP_{\nSD,S'},\nL_{S'})$ of type $t'$ } \}, \]
compatible with $\nP_{\nSD, \OO_\nSD}$-action (defined on $\nL$ as $v \cdot g := g^{-1}v$).
\item
For each map of $p$-MSD $\gamma: \nSDi \rightarrow \nSD$,
 \[ \nShD(\gamma) \times \pi_\nSDi: \nShD(\nSDi) \rightarrow \nShD( \nSD) \times_{\spec(\OOO_\nSD)} \spec
 (\OOO_\nSDi) \]
 is homogenous and is a closed embedding, if $\gamma$ is an embedding.
\item
For each $p$-integral boundary map $\iota: \nSDB \Longrightarrow \nSD$, there is a
$\nP_\nSDB$-equivariant open embedding
\[ \nShD(\iota): \nShD(\nSDB) \hookrightarrow \nShD(\nSD). \]
\item
There is an $\nP_\nSD(\R)\nW_\nSD(\C)$-equivariant open embedding (Borel embedding)
\[ h(\nX_\nSD) \hookrightarrow (\nShD(\nSD) \times_{\spec(\OO_X)} \spec(\C))^{an}.  \]
It is an isomorphism, if $\nG_\nSD$ (the reductive part of $\nP_\nSD$) is a torus.
These embeddings are compatible with the embeddings in 4., resp. (\ref{BOUNDARYCOMPONENTS2}, 1.)\footnote{going in different directions!}. Furthermore they are compatible with morphisms of Shimura data.
\end{enumerate}
\end{HAUPTSATZ}

\begin{BEM}
For a boundary component $\nSDB \Longrightarrow \nSD$, 
there is also an embedding 
\[ \nShD(\nSDB/\nW_\nSDB) \hookrightarrow \nShD(\nSD) \]
into (in general not onto) the complement of the image
of $\nShD(\nSDB)$. The image of the composition with the Borel embedding $h(\nX_\nSDB/\nW_\nSDB) \hookrightarrow \nShD(\nSDB/\nW_\nSDB)(\C)$ is
the boundary component of $h(\nX_\nSD)$ associated with $Q$ (the parabolic associated with $\nSDB$) in the sense of [AMRT], if $\nSD$ is pure.
\end{BEM}

The following Main Theorem \ref{MAINTHEOREM3STACKS} comprises the arithmetic theory of automorphic vector bundles on toroidal compactifications
(see \cite[Theorem 3.5.2]{Thesis} for a more concrete formulation of this). We will later derive a very concrete classical $q$-expansion principle (\ref{QEXPANSION}). The theorem includes many kinds of classical results on CM-values of integral (at $p$ of good reduction) modular forms on mixed Shimura varieties.

First, by taking quotients, we have a pseudo-functor $\nSD \rightarrow \left[ \nShD(\nSD) / \nP_{\nSD,\OOO_\nSD} \right]$ from $p$-MSD to 
the 2-category of Artin stacks over reflex rings (defined analogously to above). It can be extended to $p$-ECMSD by forgetting $K$ and $\nRPCD$.

\begin{HAUPTSATZ}\label{MAINTHEOREM3STACKS}Let $p\not=2$ be a prime.
{\em Assume Main Conjecture \ref{MAINCONJECTURE}}.

There is a {\em unique} (up to unique isomorphism) pseudo-natural transformation 
\[ \nMorphSt({}^K_\nRPCD \nSD): \nSh({}^K_\nRPCD \nSD) \rightarrow \left[ \nShD(\nSD) / \nP_{\nSD,\OOO_\nSD} \right], \]
and 2-isomorphism (or modification) of its restriction to $p$-EMSD:
\[ (\nMorphSt({}^K \nSD) \times_{\spec(\OOO_\nSD)} \spec(\C))^{an} \rightarrow \nMorphSt_\C({}^K \nSD), \]
where $\nMorphSt_\C({}^K \nSD)$ is the complex analytic morphism described by the diagram:
\[ \xymatrix{ [ \nP_\nSD(\Q) \backslash \nX_\nSD \times \nP_\nSD(\Af) / K ]  \\ 
[ \nP_\nSD(\Q) \backslash \nX_\nSD \times \nP_\nSD(\C) \times \nP_\nSD(\Af) / K ] \ar[u] \ar[d]  \\ 
(\nShD(\nSD) \times_{\spec(\OO_X)} \spec(\C))^{an}, }
\]
satisfying:

\begin{enumerate}
\item
Let $[\iota, \rho]: {}^{K'}_{\nRPCD'}\nSDB \Longrightarrow {}^K_\nRPCD \nSD$ be a boundary map and $\sigma \in \nRPCD'$ such that $\sigma \in C(\nX^0_\nSD, \nP_\nSDB)$, 
as in \ref{BOUNDARYSTRATA}.

There is a 2-isomorphism $[\iota, \rho]_\nMorphSt$ fitting into the diagram
\[\xymatrix{
\widehat{ \left[ \Stab_{\Gamma}([\sigma])\backslash \nSh({}^{K'}_{\nRPCD'} \nSDB) \right]  } \ar[rr]^-{\nMorphSt({}^{K'}_{\nRPCD'}\nSDB)} \ar[dd]_{\nSh(\iota, \rho)} \ar@/^2pc/[rdrd]_{}="1" \ar@/_2pc/[rdrd]^{}="2"
&& \left[ \nShD(\nSDB) / \nP_{\nSDB} \right] \ar[dd]^{\nShD(\iota)} \\
& \\
\widehat{\nSh({}^{K}_{\nRPCD} \nSD)}  \ar[rr]_-{\nMorphSt({}^{K}_{\nRPCD} \nSD)} && 
\left[ \nShD(\nSD) / \nP_{\nSD} \right] 
\ar@{=>}^-{[\iota, \rho]_\nMorphSt}"2";"1"
}\]
where the formal completion is taken along
\[ \left[ \Stab_{\Gamma}([\sigma]) \backslash \nSh( {}^{K'_{[\sigma]}}_{\Delta'_{[\sigma]}} \nSDB_{[\sigma]} ) \right], \]
such that
$[\iota, \rho]_{\nMorphSt,\C}$ converges in the interior and is the obvious 2-morphism over $\C$.

\item (canonicity) For $\nSD = \nH_{g_0}[(\nI\otimes\nI)^s,\nL_0\otimes \nI]$ (and $\nRPCD=0$),
$\nMorphSt({}^K \nSD)$ is described by the diagram:
\[  \cat{${}^K \nSD$-$\nL$-mot} \leftarrow \cat{${}^K \nSD$-$\nL$-mot-triv} \rightarrow \nShD(\nSD), \]
(see \ref{DEF1MOTIVECATEGORY}) given by the forgetful left arrow (which is a right $\nP_\nSD$-torsor) and the right arrow given by transport of
Hodge filtration. The identification with $\nMorphSt_\C({}^K \nSD)$ is the one given by the period construction (\ref{DOUBLEQUOTIENT}).

\item For $\nRPCD$ concentrated in the unipotent fibre (\ref{DEFCONCENTRATIONUNIPOTENTFIBRE}), $\nMorphSt$ is given as follows:
Recall from \ref{COMPUNIPOTENTFIBRE} that 
$\nSh({}^{K'} \nSD) \rightarrow \nSh({}^{K'} \nSD/\nU)$ is a torsor under the torus $T=\nSh({}^{K} \nSD/\nW_\nSD[\nU,0])$ over $\nSh({}^{K} \nSD/\nW_\nSD)$ for $K$ and $K'$ chosen as in \ref{COMPUNIPOTENTFIBRE}. 
$\Xi({}^{K'} \nSD)$ may etale locally (say on $S \rightarrow \nSh({}^{K'/\nU} \nSD/\nU_\nSD)$) be trivialized $T$-invariantly.
$\nMorphSt({}^{K'}_\nRPCD \nSD)$ is then the unique extension of $\nMorphSt({}^{K'} \nSD)$ to the torus embedding $\nSh({}^{K'}_\nRPCD \nSD)$ by means of extending any such $T$-invariant trivialization. 
\end{enumerate}
\end{HAUPTSATZ}

\section{Mixed Shimura varieties of symplectic type}

\begin{PAR}
If ${}^K \nSD$ is $p$-EMSD and $\nL$ is a representation of $\nP_\nSD$, there should exist a category fibred in groupoids
(over $\spec(\OO_\nSD)$) 
\[ \cat{${}^K \nSD$-$\nL$-mot} \]
of motives with ``$(\nP_\nSD, \nL)$-structure'' (see e.g. the introduction to \cite{Thesis}). It should
yield (up to taking several connected components) a canonical integral model of the Shimura variety described by
\ref{MAINTHEOREM1}. This can be made precise for P.E.L.-type Shimura data and the associated standard representation. We will need only the case of the symplectic Shimura data and certain representations (sometimes different ones, however) of their groups.

For each of the symplectic mixed Shimura data (\ref{SYMPLECTICSHIMURADATA}) we will explicitly describe
a moduli problem in terms of 1-motives which belongs to the standard representation of the underlying group scheme.
These moduli problems are representable by a Delige-Mumford stack and give an explicit description of
the canonical integral models. In the case of the morphism $\nMorphSt$ encoding the theory of automorphic vector bundles, the corresponding explicit model was already used to normalize it for all mixed Shimura varieties of Abelian type (cf. Main Theorem \ref{MAINTHEOREM3STACKS}).

For the language of 1-motives, as used here, we refer the reader to \cite[Chapter 5]{Thesis}.
We will denote a motive $M$ by $[\xymatrix{\underline{Y} \ar[r]^\mu &  G}]$ or by the symmetric description
$(A, A^\vee, \underline{X}, \underline{Y}, \alpha, \alpha^\vee, \nu)$ interchangingly. 
\end{PAR}

\begin{LEMMA}
The reflex field of the symplectic mixed Shimura data (\ref{SYMPLECTICSHIMURADATA}) is $\Q$.
\end{LEMMA}

\begin{DEF}\label{DEF1MOTIVECATEGORY}
We define the following stacks over $\Zpp$:
\begin{enumerate}
\item 
Recall the mixed Shimura datum $\nSD=\nH_{g_0}[(\nI\otimes \nI)^s, \nL_0 \otimes \nI]$ from \ref{SYMPLECTICSHIMURADATA}.
Let $\nL = \nI^* \oplus \nL_0 \oplus \nI$ with standard representation (cf. \ref{USP}) of $\nP_\nSD$ on it, where $\nL_{0,\Zpp}$ is a lattice of dimension $2 g_0$ with unimodular symplectic form (if $g_0\not=0$), $\nI_\Zpp$ is any lattice, and
$K \subset \PSp(\nL_\Af)$ is an admissible compact open subgroup. Assume that $\nL\not=0$.
We define
\[ \cat{${}^K \nH_{g_0}[(\nI\otimes \nI)^s, \nL_0 \otimes \nI]$-$\nL$-mot} \]
as the stack of the following data:
\begin{enumerate}
\item a $\Zpp$-scheme $S$,
\item a 1-motive $M=(A,A^\vee,\underline{X},\underline{Y},\alpha,\alpha^\vee,\nu)$ over $S$, with $\dim(A)=g_0$,
\item a $p$-polarization of some degree $d \in \Zpp^*$
\[ \psi_1: A \rightarrow A^\vee \]
and an isomorphism
\[ \psi_2: \underline{X}\otimes \Zpp \rightarrow \underline{Y}\otimes \Zpp \]
of the same degree $d$, such that they give a $p$-polarization of $M$,
\item an isomorphism $\rho: (\nI_{\Zpp}^*)_S \cong \underline{Y} \otimes \Z_{(p)}$,
\item a $K^{(p)}$-level structure (as above)
\[ \xi \in \underline{\Iso}_{(\USp(\nL_{0},\nI),\nL)} (\nL_\Afp, H^{et}(M,\Afp))/K^{(p)}, \]
where the isomorphisms in the etale sheaf on the right have to be compatible with the $(\USp(\nL_{0,\Zpp},\nI_{\Zpp}),\nL_{\Zpp})$-structure 
(see below) on both parameters.

Here the $(\USp(\nL_{0,\Afp}, \nI_{\Afp}),\nL_{\Afp})$-structure is given as follows:
The $p$-polarization induces an isomorphism
\[ H^{et}(M, \Afp) \rightarrow H^{et}(M, \Afp)^*(1). \]
Choosing some isomorphism $\Z_{(p)}^{et}(1) \cong \Afp$, we get an alternating form
on $H^{et}(M, \Afp)$ up to a scalar. The weight filtration satisfies $W_{-2}$ totally isotropic
and $W_{-1} = (W_{-2})^\perp$. Furthermore, we have (via $\rho$) an isomorphism
\[ \gr^0 (H^{et}(M, \Afp)) = \underline{Y} \tensor_\Z \Afp = \nI^*_{\Afp} \]
(i.e. a $\USp$-structure, see \ref{SYMPLECTICGSTRUCTURES}).

\end{enumerate}
Morphisms are $\Z_{(p)}$-morphisms of schemes $S \rightarrow S'$, together with an
$\Zpp$-isomorphism of the pullback of $M$ with the pullback of $M'$ to $S$, 
compatible with polarization (up to scalar), $\rho$'s and level structures.

\item 
With the same notation as in 1. we define
\[ \cat{${}^K \nH_{g_0}[(\nI\otimes \nI)^s, \nL_0 \otimes \nI]$-$\nL$-mot-triv} \]
as the stack consisting of the same data $(S, M, \dots)$ as before, but in addition with a trivialization
\[ \gamma: \nL \otimes \OO_S \rightarrow H^{dR}(M)  \]
of the de Rham realization of $M$ compatible with the $\USp(\nL_0, \nI)$-structure on both sides.
$\nP_{\nH_{g_0}[(\nI\otimes \nI)^s, \nL_0 \otimes \nI]} = \USp(\nL_0, \nI)$ acts on this stack
{\em on the right} by translation of $\beta$.

This renders the stack
 $\cat{${}^K \nH_{g_0}[(\nI\otimes \nI)^s, \nL_0 \otimes \nI]$-$\nL$-mot-triv}$
into a right $\USp(\nL_0, \nI)$-torsor over $\cat{${}^K \nH_{g_0}[(\nI\otimes \nI)^s, \nL_0 \otimes \nI]$-$\nL$-mot}$.

\item
Consider the Shimura datum $\nSD=\nH_0[\nU,0]$, a simple unipotent extension of $\nH_0$,
where $U$ is any lattice of dimension $k$. 
We have $\nP_\nSD = \nGa(\nU) \rtimes \Gm$. $\nP_\nSD$ acts on $\nL' = \Zpp \oplus \nU^*$ as in \ref{SIMPLEGSTRUCTURES}  
(If $\nU = (\nI \otimes \nI)^s$ we get one of the basic symplectic Shimura data but $\nL'$ is {\em not} the standard
representation). Let $K \subset \PSp(\nL_\Af)$ be an admissible compact open subgroup.
We define
\[ \cat{${}^K \nH_0[\nU,0]$-$\nL'$-mot} \]
as the stack of the following data:
\begin{enumerate}
 \item a $\Zpp$-scheme $S$,
 \item a 1-motive of the form $M=[\xymatrix{\underline{Y} \ar[r] & T}]$ (without part of weight -1),
 \item an isomorphism $\rho: (\nU_{\Zpp}^*)_S \cong \underline{Y} \otimes \Z_{(p)}$,
 \item an isomorphism $\rho': (\Zpp)_S \rightarrow \underline{X} \otimes \Z_{(p)}$,
 \item a $K^{(p)}$-level structure (as above),
\[ \xi \in \underline{\Iso}_{(\nP_\nSD,\nL_{\Zpp})} (\nL'_\Afp, H^{et}(M,\Afp))/K^{(p)}, \]
where the isomorphisms in the etale sheaf on the right have to be compatible with the obvious $(\nP_\nSD,\nL'_{\Zpp})$-structure on both parameters.
\end{enumerate}
Isomorphisms are morphisms over $\spec(\Zpp)$ of schemes $S \rightarrow S'$
together with an isomorphism of $M$ with the pullback of $M'$ compatible with $\rho$'s and level structures.

\item 
With the same notation as in 3. we define
\[ \cat{${}^K \nH_0[\nU, 0]$-$\nL'$-mot-triv} \]
consisting of the same data $(S, M, \dots)$ as before, but in addition with a trivialization
\[ \gamma: \nL' \otimes \OO_S \rightarrow H^{dR}(M)  \]
of the de Rham realization of $M$ compatible with the $(\nP_{\nH_0[\nU, 0]}, \nL')$-structure on both sides.

$\nP_{\nH_0[\nU, 0]}$ acts on this stack {\em on the right} by translation of $\beta$.

This renders the stack
 $\cat{${}^K \nH_{0}[\nU, 0]$-$\nL'$-mot-triv}$
into a right $\nP_\nSD$-torsor over $\cat{${}^K \nH_{0}[\nU, 0]$-$\nL'$-mot}$.
\end{enumerate}
\end{DEF}

Actually there is a similar description of stacks 
\[ \cat{${}^K \nSD$-$\nL$-mot-triv} 
 \qquad \cat{${}^K \nSD$-$\nL$-mot} \] 
where $\nSD$ is any of the symplectic Shimura data.

\begin{DEF}
\label{FUNCTORIALITY1MOT}
Let $\nSD$ be one of the symplectic mixed Shimura data of Definition \ref{DEF1MOTIVECATEGORY} with natural 
representation of $\nP_\nSD$ on $\nL_{\Zpp}$.
Furthermore for each $\rho \in \nP_\nSD(\Afp)$ and admissible compact open subgroups $K_1,K_2$,
such that $K_1^\rho \subseteq K_2$
we have a map
\[ \cat{${}^{K_1}\nSD$-$\nL$-mot} \rightarrow \cat{$^{K_2}\nSD$-$\nL$-mot} \]
by multiplication of the level-structure by $\rho$ from the right.
\end{DEF}

\begin{PAR}\label{DOUBLEQUOTIENT}
For $\nSD=\nH_{g_0}[(\nI\otimes\nI)^s, \nL_0\otimes \nI]$ with $g_0\not=0$,
 $\cat{${}^K \nSD$-$\nL$-mot}$ yields a model of the analytic
 mixed Shimura variety:
 \[ [\nP_\nSD(\Q) \backslash \nX_\nSD \times \nP_\nSD(\Af) / K].  \]
The identification is given as follows. 
Let an object 
\[ (M, \dots) \in \cat{$\spec(\C)$-${}^K \nSD$-$\nL$-mot} \] 
be given. Choose an isomorphism
$\beta: H^B(M, \Q) \rightarrow \nL_\Q $ (compatible with $(\nP_\nSD(\Q),\nL_\Q)$-structures on both sides). 
Consider the mixed Hodge structure $F^\bullet, W_\bullet$ on 
$H^{B}(M, \C)$, where $F^\bullet$ comes via the comparison isomorphism 
\[ \rho_{dR, B}: H^{dR}(M, \C) \rightarrow H^{B}(M, \C) \] from the Hodge filtration of $H^{dR}$.
The Hodge structure $\beta(F^\bullet), \beta(W_\bullet)$ is given by a morphism $h: \SSS_\C \rightarrow \GL(\nL_\C)$ which factors through the used representation of $\nP_{\nSD, \C}$ on $\nL_\C$ because the isomorphism
was chosen to be compatible with the $\nP_\nSD$-structures.
The level structure $\xi$ gives an element $\beta \rho_{et, B} \xi \in \nP_\nSD(\Af)$. The class of 
$(h, \beta \rho_{et, B} \xi)$ is well defined. 

The identification of a moduli point of the triv-category fibre over $\spec(\C)$ with
\[ \nP_\nSD(\Q) \backslash \nX_\nSD \times \nP_\nSD(\C) \times \nP_\nSD(\Af) / K  \]
sends a pont $(M, \dots,\gamma)$, where $(M, \dots)$ is as before and $\gamma: \nL_\C \rightarrow H_1^{dR}(M)$ is a trivialization to 
the (well-defined) point $(h, \beta \rho_{dR, B} \gamma, \beta \rho_{et, B} \xi)$.

If $g_0=0$, but $\nI \not= 0$, i.e. in the case of the mixed Shimura datum 
$\nSD=\nH_{0}[(\nI\otimes\nI)^s, 0]$, we have two descriptions as a moduli problem, according to the
representations $\nL=\nI \oplus \nI^*$ or $\nL'=\Zpp \oplus (\nI\otimes\nI)^s$.

In the first case (representation $\nL$)
$\nX_\nSD$ is given as the product of the set of Hodge structures as above and $\nX_{\nH_0} = \Hom(\Z, \Z(1))$.
We get the point $(\alpha, h, \beta \rho_{dR, B} \gamma, \beta \rho_{et, B} \xi)$ where $\alpha$ is determined as follows:
The chosen isomorphism $\beta: H_1(M, \Q) \cong \nL_\Q$, pins down a symplectic form on 
$H_1(M,\Q)$. The comparison with the natural $\Q(1)$-valued form on $H_1(M,\Q)$ (which coincides up to scalar) induces an isomorphism $\Q \rightarrow \Q(1)$ of
the form $q \alpha$ with $q \in \Q_{>0}$. 

In the second case (representation $\nL'$) --- and this even includes $\nI=0$ ---
the isomorphism $W_{-1}(\beta)^{-1}: \Q \cong H_1(\Gm, \Q)=\Q(1)$ is  of
the form $q \alpha$, as before. 

This determines in both cases an $\alpha \in \Hom(\Z, \Z(1))$. 
\end{PAR}

The following theorem is due to Riemann, Mumford, Artin, Deligne, etc. in the pure case, due to Pink in the mixed case (rational case) and was extended to the integral mixed case in the thesis of the author \cite{Thesis}:

\begin{SATZ}\label{REPR}
For each of the $p$-integral mixed Shimura data of symplectic type $\nSD$ and each admissible compact open subgroup $K \subset \nP_\nSD(\Af)$
\[ \cat{${}^K \nSD$-$\nL$-mot}, \] 
considered as a model of the analytic mixed Shimura variety via the construction \ref{DOUBLEQUOTIENT}, 
is representable by the canonical model $\nSh({}^K \nSD)$ (compatible with the $\nP_\nSD(\Afp)$-action)
which is a smooth, Deligne-Mumford stack of finite type over $\spec(\Zpp)$.
It is a quasi-projective scheme, if $K$ is neat.

The stack
\[ \cat{${}^K \nSD$-$\nL$-mot-triv} \] 
is representable by a right $\nP_\nSD$-torsor over $\cat{${}^K \nSD$-$\nL$-mot}$.
\end{SATZ}

\begin{PAR}
Note that the torsor $\cat{${}^K \nSD$-$\nL$-mot-triv}$ was used to ``normalize''
the theory of the of the morphism $\nMorphSt$ in \ref{MAINTHEOREM3STACKS}.
Involved in the construction of $\nMorphSt$ there is a $\nP_\nSD$-equivariant morphism
$\nSPB({}^K \nSD) \rightarrow \nShD(\nSD)$ which in the modular description is given by transport of
the Hodge filtration to $\nL_{\OO_S}$ by means of the trivialization $\beta$ (and then taking its fixing quasi-parabolic, cf. \ref{MAINTHEOREM2}, 2.).
\end{PAR}

We have the following `integral' variants of the foregoing stacks:

\begin{DEF}
Fixing an $N \in \N$, we define the following stacks over $\spec(\Z[1/N])$. 
\begin{enumerate}

\item
Let $\nL_{0,\Z}$ be a with unimodular symplectic form of dimension $2g_0$. Let $\nI_\Z$ be another lattice and $\nL_\Z = \nI_\Z^* \oplus \nL_{0,\Z} \oplus \nI_\Z$ as above. Assume again $\nL\not=0$. We define
\[ \cat{${}^N \nH_{g_0}[(\nI_\Z\otimes \nI_\Z)^s,\nL_{0,\Z} \otimes \nI_\Z]$-$\nL_\Z$-mot} \]
as the stack of the following data:
\begin{enumerate}
\item a scheme $S$ over $\spec(\Zpp)$,
\item a 1-motive $M=(A,A^\vee,\underline{X},\underline{Y},\alpha,\alpha^\vee,\nu)$ over $S$, where $\dim(A)=g$,
\item a principal polarization
\[ \psi_1: A \rightarrow A^\vee \]
and an isomorphism
\[ \psi_2: \underline{X} \rightarrow \underline{Y} \]
such that they induce a polarization of $M$,
\item an isomorphism $\rho: (\nI_\Z^*)_S \cong \underline{Y}$,
\item a level-$N$-structure
\[ \xi \in \Iso_{(\USp(\nL_{0,\Z/N\Z},\nI_{\Z/N\Z}),\nL_{\Z/N\Z})} (\nL_{\Z/N\Z}, H^{et}(M,\Z/N\Z)). \]
This means that the isomorphisms have to be compatible with the
$(\USp(\nL_{0,\Z/N\Z},\nI_{\Z/N\Z}),\nL_{\Z/N\Z})$-structure (see the adelic case) on both parameters.
\end{enumerate}
Morphisms are as in the adelic case with the difference, that the $\Zpp$-morphism of the 1-motives has to be an isomorphism.
(Note: In particular $\psi_Y$ is determined by the $\rho$'s and $\psi_X$ by them and the polarization.)

\item We define the stack
\[ \cat{${}^N \nH_{g_0}[(\nI_\Z\otimes \nI_\Z)^s,\nL_{0,\Z} \otimes \nI_\Z]$-$\nL_\Z$-mot-triv} \]
as before with a trivialization of de Rham.

\item Let $\nL'_\Z = \Z \oplus \nU_\Z^*$ be as in \ref{DEF1MOTIVECATEGORY}, 3. We define the stack
\[ \cat{${}^N \nH_{0}[\nU_\Z,0]$-$\nL'_\Z$-mot} \]
as the category of the following data:
\begin{enumerate}
\item a scheme $S$ over $\spec(\Zpp)$,
\item a 1-motive of the form $M=[\xymatrix{\underline{Y} \ar[r] & T}]$,
\item an isomorphism $\rho: (\nU_{\Z}^*)_S \cong \underline{Y}$,
\item an isomorphism $\rho': (\Z)_S \cong \underline{X}$,
 \item a level-$N$-structure 
\[ \xi \in \underline{\Iso}_{(\nP_\nSD,\nL'_{\Z/N\Z})} (\nL'_{\Z/N\Z}, H^{et}(M,\Z/N\Z)), \]
where the isomorphisms in the etale sheaf on the right have to be compatible with the obvious $(\nP_\nSD,\nL'_{\Z/N\Z})$-structure on both parameters.
Morphisms are as in the adelic case, with the difference that the $\Zpp$-morphism of the 1-motives has to be an isomorphism.
\end{enumerate}

\item We define the stack
\[ \cat{${}^N \nH_{0}[\nU_\Z,0]$-$\nL'_\Z$-mot-triv} \]
as before with  a trivialization of de Rham.

\item We define the stack
\[ \cat{${}^{N,0} \nH_{0}[\nU_\Z,0]$-$\nL'_\Z$-mot} \]
where $\nL'_\Z = \Z \oplus \nU_\Z^*$
as the category of the following data ($S, M, \dots, \xi$) as before, but with
\[ \xi \in \underline{\Iso} (\Z/N\Z, H^{et}(\gr_{-2} M,\Z/N\Z)), \]
only.

\item We define the stack
\[ \cat{${}^{N,0} \nH_{0}[\nU_\Z,0]$-$\nL'_\Z$-mot-triv} \]
as before with  a trivialization of de Rham.
\end{enumerate}
\end{DEF}

\begin{BEM}\label{H0REPR}
For $\nU=(\nI \otimes \nI)^s \not= 0$, 
we explicitly describe the ``change-of-representation equivalences'':
\[ \cat{${}^K \nH_0[ (\nI \otimes \nI)^s,0]$-$\nL=\nI^* \oplus \nI$-mot} 
\rightarrow  \cat{${}^K \nH_0[\nU,0]$-$\nL'=\Zpp \oplus \nU^*$-mot}. \]
and
\[  \cat{${}^K \nH_0[ (\nI \otimes \nI)^s,0]$-$\nL=\nI^* \oplus \nI$-mot-triv} 
 \rightarrow \cat{${}^K \nH_0[\nU,0]$-$\nL'=\Zpp \oplus \nU^*$-mot-triv}. \]
as follows:
Let a ($M=(0, 0, \underline{X}, \underline{Y}, 0, 0, \nu), 0, \psi_2, \rho, \xi$) over $S$ be given.  
$\nu$ can via polarization and $\rho$ be considered as a morphism $\Sym^2(\nI^*)_S \rightarrow \G_{m,S}$.
This gives rise to a motive $M''$ and a $\rho$ required for the second category ($\rho'$ is given by the canonical identification $\Hom(\Gm,\Gm) \cong \Z$). The level structures
translate as follows. It is easy to check that the realizations $H^{et}(M''), H^{dR}(M'')$ and $H^B(M'')$ (when defined), together with their 
$(\USp(0, \nI), \Zpp \oplus \Sym^2(\nI^*))$-structures, are obtained from
the corresponding realizations of $M$ with their $(\USp(0, \nI), \nI \oplus \nI^*)$-structures by the procedure described in \ref{SIMPLEGSTRUCTURES}.
Since the procedure described in \ref{SIMPLEGSTRUCTURES} applied to $\nI \oplus \nI^*$ gives $\Zpp \oplus \Sym^2(\nI^*)$ 
canoncially, we may translate the
$K$-level structure or trivialization of $H^{dR}$ accordingly.
\end{BEM}

\begin{PROP}\label{INTEGRAL1MOTIVES}
Let $S$ be a scheme over $\spec(\Zpp)$.
For each of the $p$-integral mixed symplectic Shimura data as above, we have equivalences
\[ \cat{${}^N \nSD$-$\nL_\Z$-mot} \rightarrow  \cat{${}^{K(N)}\nSD$-$\nL$-mot}, \]
where 
\[ K(N) = \{ g \in \nP_\nSD(\Af) \where g\nL_\Zh = \nL_\Zh, g \equiv \id\ (N) \}. \]
We have also a similar equivalences for $\nSD = \nH_0[\nU,0]$ and the representation $\nL'$, as well as 
\[ \cat{${}^{N,0} \nSD$-$\nL_\Z'$-mot} \rightarrow  \cat{${}^{\nU_\Zh \rtimes K(\Gm,N)} \nSD$-$\nL'$-mot}. \]
\end{PROP}

\begin{BEM}
Since the $K(N)$ form a cofinal system of compact open groups, the integral categories give a complete description of the canonical models. 
In the case $\nSD = \nH_0[\nU,0]$, each $K(N)$ is conjugated to $\nU_\Zh \rtimes K(\Gm,N)$, therefore
already the integral categories with $(N,0)$-level structure suffice to describe the canonical models. 
\end{BEM}

\begin{proof}[Proof of the Proposition.]
We will prove this for the case $\nH_{g_0}[(\nI_\Z \otimes \nI_\Z)^s, \nL_{0,\Z} \otimes \nI_\Z]$, where $\nI_\Z \not=0$, $\nL_{0,\Z} \not=0$
--- the other cases are degenerate special cases of this construction. We may assume that $S$ is connected.

We first describe the functor.
Let $[M,\psi_1,\psi_2,\rho,\xi]$ 
be an object of $\cat{${}^N\nSD$-$\nL_\Z$-mot}$ over $S$.
Choose a geometric point $\overline{s}$.
$\xi$ can be considered as an isomorphism
\[ \xi : \nL_{\Z/N\Z} \rightarrow H^{et}(M_{\overline{s}},\Z/N\Z ) \]
invariant under the action of $\pi_1^{et}(S, \overline{s})$ ($H^{et}(M, \Z/N\Z)$ has to be constant).

Choose some isomorphism
\[ \delta : H^{et}(M_{\overline{s}},\widehat{\Z}^{(p)} ) \rightarrow \nL_{\widehat{\Z}^{(p)}}  \]
compatible with $\USp$-structures. Composing the reduction mod $N$ with $\xi$, we get an
element of $\USp(\nL_{0,\Z/N\Z}, \nI_{\Z/N\Z})$. Since $\USp$ is a smooth group scheme over $\Z$, by Hensel's Lemma
we get a lift to $\USp(\nL_{\widehat{\Z}^{(p)}})$. Taking composition again with the chosen isomorphism
we get a
\[ \xi' : \nL_{\widehat{\Z}^{(p)}} \rightarrow  H^{et}(M_{\overline{s}},\widehat{\Z}^{(p)} ) ,  \]
reducing mod $N$ to $\xi$. It is well-defined mod $K(N)$ and the class is, by construction, invariant under
$\pi_1^{et}(S, \overline{s})$.

This functor is faithful. It is full because an $p$-morphism of 1-motives which
induces an isomorphism
$H^{et}(M_{\overline{s}},\widehat{\Z}^{(p)}) \cong H^{et}(M_{\overline{s}},\widehat{\Z}^{(p)})$
must be a morphism. This follows from \cite[Theorem 5.1.7]{Thesis}.

Let on the other hand
$[M,\psi_1,\psi_2,\rho,\xi']$ be an object of 
$\cat{${}^{K(N)} \nSD$-$\nL$-mot}$ over $S$.
Choose a geometric point $\overline{s}$.

If $\xi'$ is represented by an isomorphism
\[ \nL_{\widehat{\Z}^{(p)}} \rightarrow H^{et}(M_{\overline{s}},\widehat{\Z}^{(p)})   \]
then the object is in the image of the functor because of the following:

\begin{enumerate}
\item $\xi'$ can be given as a lift of a $\xi$ as above.

\item There is a principal polarization in the class of $\psi'$.
For, there is a $d \in \Afp$, such that $\Psi$ induces an isomorphism
\[ H^{et}(M_{\overline{s}}, \Zh^{(p)}) \rightarrow d H^{et}(M_{\overline{s}}^\vee, \Zh^{(p)}) \]
because $\xi'$ is a morphism of $\USp$-structures, hence a symplectic similitude.
So we get a principal polarization by changing $\psi'$ by $+d$ or $-d$, which must lie in $\Zpp^*$. Only one
sign leads to a {\em polarization}.

\item $\rho$ maps $\nI_\Z^*$ to $\underline{Y}$.
\end{enumerate}

If $\xi'$ is not represented by an isomorphism as above, we have to show that there exists an 
isogenous object with this property.

Composing with a $p$-isogeny $\in \Z_{(p)} \setminus \{0\}$, we may assume that there is an $c, p\nmid c$, and
a diagram
\[ \xymatrix{
0 \ar[r] & H^{et}(M_{\overline{s}}, \Zh^{(p)}) \ar@{^{(}->}[r]^{\xi^{-1}} \ar@{=}[d]& \nL_{\widehat{\Z}^{(p)}} \ar@{^{(}->}[d] \ar[r] & K_{\overline{s}} \ar@{^{(}->}[d] \ar[r] & 0 \\
0 \ar[r] & H^{et}(M_{\overline{s}}, \Zh^{(p)}) \ar[r]^{[c]} & H^{et}(M_{\overline{s}}, \Zh^{(p)}) \ar[r] & \ker([c]) \ar[r] & 0
}\]
Furthermore the operation of $\pi_1^{et}(S, \overline{s})$ induces one on $K_{\overline{s}}$, hence there
is a finite etale group scheme $K \subset \ker([c])$ with fibre $K_{\overline{s}}$, and we have 
\cite[Theorem 5.1.8]{Thesis}
an isogeny $\psi: M \rightarrow M'$ with $\ker(\psi)=K$, hence a diagram
\[ \xymatrix{
0 \ar[r] & H^{et}(M_{\overline{s}}, \Zh^{(p)}) \ar@{^{(}->}[r]^{\xi^{-1}} \ar@{=}[d]& \nL_{\widehat{\Z}^{(p)}} \ar[r] & K_{\overline{s}} \ar@{=}[d]  \ar[r] & 0 \\
0 \ar[r] & H^{et}(M_{\overline{s}}, \Zh^{(p)}) \ar[r]^{\psi} & H^{et}(M'_{\overline{s}}, \Zh^{(p)})
\ar[u]_{(\xi')^{-1}} \ar[r] & K_{\overline{s}} \ar[r] & 0
}\]
There is an $\rho'$ because $\xi$ is a morphism of $\USp$-structures and hence $\Psi_Y$ has to be an isomorphism of
$\Z$-lattices! Therefore we get an isogenous object with the property that $\xi'$ is an isomorphism
$\nL_{\widehat{\Z}^{(p)}} \rightarrow H^{et}(M_{\overline{s}},\widehat{\Z}^{(p)})$.
\end{proof}

\begin{PAR}\label{SECTIONBOUNDARY}
Let $\nU$ be any $\Zpp$-lattice. 
Consider the mixed Shimura datum $\nSD = \nH_0[\nU,0]$. It plays the mayor role in the description of point-like boundary components
of many pure Shimura varieties, in particular of those symplectic and also orthogonal type (cf. \ref{SYMPLECTICBOUNDARYCOMP} and \ref{ORTHBOUNDARYCOMP}).

Consider the category (cf. Definition \ref{DEF1MOTIVECATEGORY}, 3.):
\[ \cat{ ${}^{K}\nSD$-$\nU^* \oplus \Zpp$-mot }, \]
where $K$ is any admissible compact open subgroup of $\nP_{\nH_0[\nU,0]}(\Af)$.

Recall that, by definition, 
\begin{equation} \label{split}
 \nP_{\nH_0[\nU,0]} = \nGa(\nU) \rtimes \Gm
\end{equation}
and that $K$ contains the subgroup $\nU_\Zh \rtimes K(M)$ for some $M \in \Z$ and some (saturated at $p$) lattice $\nU_\Z \subset \nU$.
The splitting (\ref{split}) determines and is determined by a $\Zpp$-valued point $v$ of the
dual $\nShD(\nSD)$ (The preimage of $\Gm$ under the splitting is the quasi-parabolic corresponding to the point).

The category 
\[ \cat{ ${}^{\nU_\Zh \rtimes K(M)}\nH_0[\nU,0]$-$\nU^* \oplus \Zpp$-mot } \]
is isomorphic to the integral category
\[ \cat{ ${}^{(M,0)}\nH_0[\nU,0]$-$\nU^*_\Z \oplus \Z$-mot }, \]
 see \ref{INTEGRAL1MOTIVES}.
An object in the latter is isomorphic to $(M, \xi)$, where $M$ is of the form
\[ [ \mu: \nU_\Z^* \rightarrow \Gm ], \]
with a level structure $\xi: \Z/M\Z \rightarrow H_1^{et}(\Gm, \Z/M\Z) \cong \Gm[M]$ only.

We may identify 
\[  \nX_{\nH_0} \times \nU_\C \rightarrow \nX_{\nH_0[\nU,0]}  \]
such that an element $(\alpha, u)$ is mapped to $(\alpha, u \circ  h_0 \circ (-u))$, where $u \circ h_0 \circ (-u)$ is the $u$-conjugate (from the left) of the morphism
$h_0: \SSS_\C \rightarrow \nP_{\nH_0[\nU,0]} = \nGa(\nU) \rtimes \Gm$ mapping
$z$ to $z\overline{z}$ in the second factor. The associated mixed Hodge structure is the $u$-translate of the one given by $F^0=\nU_\C$.
\end{PAR}

The following Proposition --- which is quite tautological --- describes the associated canonical models. It will be used in the proof of the $q$-expansion principle (Theorem \ref{QEXPANSION}). 

\begin{PROP}\label{SATZBOUNDARYPOINTS}
The $S$-valued points of the dual $\nShD(\nSD)(S)$ are in natural bijection with splittings $\nP_{\nSD,S} \cong \nGa(\nU_S) \rtimes \G_{m,S}$.

We have a commutative diagram (where all morphisms are functorial in $S$)
\begin{equation}\label{diagm1}
\xymatrix{
\cat{$S$-${}^{(M,0)} \nSD$-$\Z \oplus \nU_\Z^*$-mot-triv} \ar[r]^-\sim_-{i_3} \ar@<-1ex>[d]_p & (\nP \times_\Zpp \G_{m,\Zpp(\zeta_M)} \otimes \nU_\Z)(S) \ar@<-1ex>[d]_p\\
\cat{$S$-${}^{(M,0)} \nSD$-$\Z \oplus \nU_\Z^*$-mot} \ar@<-1ex>[u]_{can.} \ar[d]_p  \ar[r]^-\sim_-{i_2}  & (\G_{m,\Zpp(\zeta_M)} \otimes \nU_\Z)(S) \ar@<-1ex>[u]_{can.} \ar[d]_p  \\
\cat{$S$-${}^{M} \nH_0$-$\Z$-mot} \ar[r]^-\sim_-{i_1}  & \spec(\Zpp(\zeta_M))(S)
}
\end{equation}
where $\zeta_M$ an (abstract) primitive $M$-th root of unity. Here the middle horizontal morphisms
give morphisms of tori, and the topmost horizontal morphisms give morphisms of right $\nP_S$-torsors. 

Over $\C$, the identifications \ref{DOUBLEQUOTIENT} render this into the following commutative diagram
\begin{equation}\label{diagm2}
 \xymatrix{
  \nP(\Q) \backslash \nX_{\nH_0} \times \nU_\C \times \nP(\C) \times \nP(\Af) / K    \ar[r]^-\sim_-{j_3} \ar@<-1ex>[d]_p & (\nP \times_\Zpp \G_{m,\Zpp(\zeta_M)} \otimes \nU_\Z)(\C) \ar@<-1ex>[d]_p\\
   \nP(\Q) \backslash \nX_{\nH_0} \times \nU_\C \times \nP(\Af) / K    \ar@<-1ex>[u]_{can.} \ar[d]_p  \ar[r]^-\sim_-{j_2}  & (\G_{m,\Zpp(\zeta_M)} \otimes \nU_\Z)(\C) \ar@<-1ex>[u]_{can.} \ar[d]_p  \\
 \Q^* \backslash \nX_{\nH_0} \times (\Af)^* / K(M)    \ar[r]^-\sim_-{j_1}  & \spec(\Zpp(\zeta_M))(\C)
} 
\end{equation}
where $K = \nU_\Zh \rtimes K(M)$ for the splitting (\ref{split}). The morphism can. on the left hand side is given as follows:
\[ [\alpha, \mu_{\log}, \rho ] \mapsto [\alpha, \mu_{\log}, \mu_{\log}\alpha^{-1}(1), \rho]  \]
where $\rho$ must be chosen to be in $\nU_\Zh \rtimes K(1)$ or even in $K(1) \cong \Zh^*$ (this amounts to the same). 

This yields a description of the canonical integral models $\nSh({}^K \nSD), \nMorphSt({}^K \nSD)$ and $\nSh({}^{K(M)} \nH_0)$, respectively. 
\end{PROP}

\begin{PAR}\label{detailsrpcd}
The lower morphisms $p$ describe in all diagrams obvious structures of a relative group object, which are the same as those induced by the unipotent extension morphism of Shimura data $\nH_0[\nU, 0] \rightarrow \nH_0$, which is a group object (with the unit section given by the splitting $\gamma$, chosen above). 
The trivialization $can.$ extends to partial compactifications, which are given by torus embeddings of
$\G_{m,\Zpp[\zeta_M]}  \otimes \nU_\Z$ over $\spec(\Zpp[\zeta_M])$.

More precisely: By the arithmeticity condition and left invariance under conjugation by $\nP_\nSD(\Q)$ a rational (w.r.t. $\nU_\Q$) polyhedral cone decomposition of $\nU_\R$ into cones $\{\sigma\}$ determines a rational polyhedral cone decomposition $\nRPCD$ of the conical complex (here $\nH_0 \times \nU_\R(-1)$).
On $\alpha \times \nU_\R(-1)$, the decomposition is given by $\{ \alpha^{-1}(\sigma) \}$. 

Hence, for each such $\sigma$ the corresponding torus embedding is
\[ \spec(\Zpp[\zeta_M][\nU_\Z^*]) \hookrightarrow \spec(\Zpp[\zeta_M][\nU_\Z^* \cap \sigma^\vee]).  \]
They are glued in the usual way using the topology of the decomposition.
\end{PAR}

\begin{proof}[Proof of the Proposition]
We begin by describing the horizontal morphisms in the diagram (\ref{diagm1}): 

$i_1$: Up to isomorphism any element in the 
category 
\[ \cat{$S$-${}^{M} \nH_0$-$\Z$-mot} \]
is of the form $X=(\G_{m,S}, 0, \id, \xi)$
with  $\xi: (\Z/ N\Z)_S \rightarrow \Gm[M]_S \cong \spec(\Zpp(\zeta_M))_S$.
$i_1(X)$ is the morphism determined by $1\in \Z/N\Z$, which determines the level-structure uniquely. 

$i_2$: Up to isomorphism any element in the 
category 
\[ \cat{$S$-${}^{(M,0)} \nH_0[\nU_\Z, 0]$-$\nL_\Z$-mot} \] 
is of the form $X=(\mu: (\nU_\Z)_S \rightarrow \G_{m,S}, \id, \id, \xi)$.
$i_2$ is defined as the fibre product of $i_1(X)$ (defined as before) with $\mu$.

$i_3$: Up to isomorphism any element in the 
category 
\[ \cat{$S$-${}^{(M,0)} \nH_0[\nU_\Z, 0]$-$\nL_\Z$-mot-triv}  \]
is of the form $X=(\mu: (\nU_\Z)_S \rightarrow \G_{m,S}, \id, \id, \xi, \gamma)$,
where $\gamma: \nL_S \rightarrow H_1^{dR}(M)$ is a trivialization (compatible with the $\nP_{\nSD,S}$-structure).
We have a canonical trivialization $can.:  F^0(H^{dR}(M)) \oplus \OO_S (\frac{\D z}{z})^* \cong H^{dR}(M)$,
where $\OO_S (\frac{\D z}{z})^*$ is $\Lie(\Gm)$ and $\nU_S^* \cong F^0(H^{dR}(M))$ (canonically via $\rho$) over any $S$.
Mapping $1$ to $(\frac{\D z}{z})^*$ the domain of $can.$ can be identified with $\nL_S$. In other words, we have 
\[ can.:  \nL_S \rightarrow H^{dR}(M).  \] 
Equipping a point  
\[ X=(\mu: (\nU_\Z)_S \rightarrow \G_{m,S}, \id, \id, \xi) \in \cat{$S$-${}^{(M,0)} \nH_0[\nU_\Z, 0]$-$\nL_\Z$-mot} \]
 with the trivialization $can.$
defines the vertical morphism $can.$ on the left in (\ref{diagm1}). The morphism $can.$ on the right hand side is the obvious trivialization. The morphism $i_3$ is defined
as the only $\nP_{\nSD,S}$-equivariant morphism that makes the diagram commutative (it maps $X$ to $(can.^{-1} \circ \gamma, i_2(X)) \in \nP_{\nSD}(S) \times (\G_{m,\Zpp[\zeta_M]} \otimes \nU_\Z)(S)$). 

We now describe the resulting horizontal morphisms in diagram (\ref{diagm2}):

$j_1$:
We choose an isomorphism 
\[ \beta: H^{B}(\Gm, \Z) \rightarrow \Z.  \]
Since $H^B_1(\Gm, \Z) = \Z(1)$, we have an $\alpha \in \nX_{\nH_0}$ with $\alpha^{-1} = \beta$.

Let $\xi$ be a level structure. Its image in $\Zh^*$ is given mod $K(M)$ by the composite (see \ref{DOUBLEQUOTIENT}): 
\[ \xymatrix{ \Z/ N\Z \ar[r]^-\xi & H^{et}_1(\Gm, \Z/N\Z) = \Gm[N](\C) \ar[r]^-{\rho_{et,B}} & H^{B}_1(\Gm, \Z/N\Z) \ar[r]^-\beta &\Z/N\Z }. \]
Note that $\rho_{et,B}$ is given by the logarithm.

We get an element
\[ [\alpha, \rho]  \in \Q^* \backslash \nX_{\nH_0} \times \Af^* / K(M) \]
where $\rho \in K(1)$ 
and have 
\[ j_1([\alpha, \rho]) = \exp(\alpha(\frac{1}{M}\rho)) \in \Gm(\C)[M]. \]

$j_2$:
Choose an isomorphism $\beta: H^B(M) \rightarrow \Z \oplus \nU^*_\Z$ (in the $\nP_\nSD(\Z)$-structure) 
such that $W_{-1}(\beta) = \alpha^{-1}$ as morphism $H^B_1(\Gm)=\Z(1) \rightarrow \Z$.

Recall that $\nL_\Z = \Z \oplus \nU_\Z^*$. The composite
\[ \xymatrix{ \nL_\C \ar[r]^-{can.} & H^{dR}(M) \ar[r]^-{\rho_{dR,B}} & H^{B}(M) \otimes \C \ar[r]^-\beta & \nL_\C }. \]
maps $(1, 0)$ to $(\alpha^{-1}(1), 0)$ and $(0, u^*)$ to $(-\alpha^{-1}(\log(\mu(u^*))), u^*)$. This is the same as the operation of $\mu_{\log} \alpha^{-1}(1) \in \nP_\nSD(\C)$
under the action defined in \ref{DEF1MOTIVECATEGORY}, 4., where  $\mu_{\log} \in \nU_\C/\nU_\Z$ is such that $u^*(\mu_{\log}) = \alpha^{-1}(\log(\mu(u^*)))$ for all $u^* \in \nU_\Z^*$.

We have then 
$\beta(F^0(M)) = \mu_{\log} U_\C^*$.

Hence the moduli point of $([\mu: \nU_\Z^* \rightarrow \Gm], \id, \id, \xi)$ in the complex analytic Shimura variety is:
\[ [\alpha, \mu_{\log}, \rho] \in \nP(\Q) \backslash \nX_{\nH_0} \times \nU_\C \times \nP_\nSD(\Af) / K \]

$j_3$: The image of the point $ [\alpha, \mu_{\log}, \rho]$ via $can.$ in the standard principal bundle is given by
\[Ê[\alpha, \mu_{\log}, \mu_{\log} \alpha(1)^{-1}, \rho] \in \nP(\Q) \backslash \nX_{\nH_0} \times \nU_\C \times \nP_\nSD(\C) \times \nP_\nSD(\Af) / K. \]

Therefore the trivialization $can.$ on the left in diagram (\ref{diagm2}) is given by
\begin{eqnarray*} 
 \nP(\Q) \backslash \nX_{\nH_0} \times \nU_\C \times \nP(\Af) / K &\rightarrow&
 \nP(\Q) \backslash \nX_{\nH_0} \times \nU_\C \times \nP(\C) \times \nP(\Af) / K \\
{ [\alpha, \mu_{\log}, \rho]} & \mapsto &Ê[\alpha, \mu_{\log}, \mu_{\log} \alpha(1)^{-1}, \rho] 
 \end{eqnarray*}
 for $\rho \in \nU_\Zh K(1)$ (!), as claimed. This also defines the morphism $j_3$ by $\nP_{\nSD}(\C)$-equivariance.
 
It is easy to see that this morphism $can.$ is compatible with the action of
the torus multiplication in the relative torus $\nSh({}^K\nSD)$ over $\nSh({}^{K(N)} \nH_0)$. It is hence by the characterization \ref{MAINTHEOREM3STACKS}, 3. the trivialization which
extends to every partial compactification $\nSh({}^K_\nRPCD \nSD)$.

The model of the Shimura variety (resp. standard principal bundle) described by the schemes representing
\[ \cat{${}^K \nH_0[\nU, 0]$-$\nL$-mot}, \text{ resp. } \cat{${}^K \nH_0[\nU, 0]$-$\nL$-mot-triv} \]
are the canonical ones by \ref{REPR} and \ref{H0REPR}.
\end{proof}

\section{$q$-expansion principle}

\begin{PAR}
From the
abstract theory of canonical models of toroidal compactifications as described by the Main Theorems \ref{MAINTHEOREM1a} and \ref{MAINTHEOREM3STACKS}, a classical ``$q$-expansion principle'' can be derived.
In this section we explain this for a wide class of Shimura varieties with boundary.
\end{PAR}

\begin{PAR}\label{QEXPANSIONPREP}
Resume the discussion from \ref{SECTIONBOUNDARY}.
Let $\mathcal{E}$ be a $\nP_\nSD$-equivariant vector bundle on $\nShD(\nSD) = \nGa(\nU)v$. Here $v$ is the
$\Zpp$-valued point corresponding to the splitting (\ref{split}). $H^0(\nSh({}^K \nSD), \Xi^* \mathcal{E})$ is identified with the
$\nP_\Q$-invariant functions on 
\[  \nX_\nSD \times \nP(\Af) / K \rightarrow \mathcal{E} \]
Choose a basis $v_1, \dots, v_n$ of the fibre $\mathcal{E}_v$. 

The canonical trivialization ($can.$ of the Proposition \ref{SATZBOUNDARYPOINTS}) of the standard principal bundle trivializes also (the algebraic) $\Xi^* \mathcal{E}$.
For this consider the following diagram

\[\xymatrix{
 \nSh({}^K\nSD) \ar[r]^{can.} & \nSPB({}^K\nSD) \ar@<1ex>[l]^p \ar[r]^\nMorphD & \nShD(\nSD), } \]

$\nMorphSt^* \mathcal{E}$ is defined as the $\nP$-invariants in the pullback $\nMorphD^* \mathcal{E}$.
$can.$ identifies those with $can.^*  \nMorphD^* \mathcal{E} = (\mathcal{E}_v)_{\nSh({}^K\nSD)}$.
This last identification comes from the fact, that $\nMorphD \circ can.$ is constant with image $v$ by construction.
By Proposition \ref{SATZBOUNDARYPOINTS} this trivialization, base changed to $\C$, is given by the function
\begin{eqnarray*}
 s_i: \nX_{\nH_0} \times \nU_\C \times \nP_\nSD(\Af) / K &\rightarrow& \mathcal{E} \\
 {[\alpha, \mu_{\log}, \rho]} &\mapsto&  \mu_{\log} \alpha(1)^{-1} v_i \qquad \text{for } \rho \in K_\nU K(1),
\end{eqnarray*}
where $\mathfrak{B} = \{v_i\}$, and the $v_i$ are considered as points in the fibre over $v$.
(Extend the function to other $\rho$ by requiring $\nP_\nSD(\Q)$-invariance).

A general section may be written as
\[ [\alpha, \mu_{\log}, \rho] \mapsto \sum_{k \in \nU_\Z^*, i} a_{k,\alpha,\rho,i} \exp(\alpha(\mu_{\log} k)) s_i \]
($\rho \in K(1)=\Zh^*$),
and the corresponding $i$-th coordinate-function on $\G_{m,\C} \otimes \nU_\Z \times_\C \C[\zeta_M] = \spec( \bigoplus_{\zeta_{M,\C}} \C[\nU_\Z^*])$
is given by 
\[ \sum_{k \in \nU_\Z^*} a_{k, \alpha,\rho,i} [k] \]
 in the fibre above $\zeta_{M,\C} = \exp(\alpha(\frac{1}{M}\rho))$. 

Let $\sigma$ now be a maximal dimensional rational polyhedral cone of $\nU_\R$. The completion at
the corresponding boundary stratum is given by
\[ \spf( \bigoplus_{\zeta_{M, \C}} \C \llbracket \nU_\Z^* \cap \sigma^\vee \rrbracket). \]
If a formal function $f$ given by
\[ \sum_{k \in \nU_\Z^* \cap \sigma^\vee} a_{k, \alpha,\rho}(f) [k] \]
in the fibre above the respective $\zeta_{M,\C} = \exp(\alpha(\frac{1}{M}\rho))$
converges in a neighborhood of the boundary points, it has the form
\[ [\alpha, \mu_{\log}, \rho] \mapsto \sum_{k \in \nU_\Z^*} a_{k,\alpha,\rho}(f) \exp(\alpha(\mu_{\log} k)), \]
there, too.

We make the following observation which is important for the $q$-expansion principle:
The function ${}^\tau f$ for an automorphism $\tau \in \Aut(\C)$ has the expansion
\[ \sum_{k \in \nU_\Z^* \cap \sigma^\vee} {}^\tau a_{k, \alpha,\rho}(f) [k] \]
in the fibre over ${}^{\tau} (\zeta_{M,\C}) = \exp(\alpha(\frac{1}{M} \rho \kappa_\tau))$. Here $\kappa_\tau \in K(1)/K(M)$ is determined by this (image of $\tau$ under the class field theory isomorphism).

Therefore the Fourier coefficients satisfy:
\begin{equation} \label{galoisactionfourier}
  a_{k,\alpha,\rho}({}^{\tau}f) = {}^\tau a_{k, \alpha, \rho k_\tau^{-1}}(f).
\end{equation}
\end{PAR}

\begin{PAR}
We continue by stating a formal $q$-expansion principle, stating roughly that --- under certain hypotheses --- a regular function is integral, if after completion at some point, the corresponding formal function is integral.

For a commutative ring $R$ and ideal $I \subseteq R$, we denote by $C_I(R)$ the completion of $R$ w.r.t.
the $I$-adic topology. For an $I$-adically complete ring $R$ and a multiplicatively closed subset $S\subset R$, we denote by $R\{S^{-1}\}$ the completion w.r.t. $I[S^{-1}]$ of
the localization $R[S^{-1}]$.

Let $R$ be an excellent normal integral domain and $I$ an ideal, such that $\spf(C_I(R))$ is connected (or equivalently,
such that $R/\sqrt{I}$ has no nontrivial idempotents).

Let $s$ be a prime element of $R$, neither a zero divisor nor a unit of $R/I$. Let $M$ be a projective $R$-module.
\end{PAR}

\begin{LEMMA}\label{LEMMA_FORMAL}
The diagram
\[\xymatrix{M \ar@{^{(}->}[r] \ar@{^{(}->}[d] & M [s^{-1}] \ar@{^{(}->}[d] \\ C_I(M) \ar@{^{(}->}[r]& C_I(M)\{s^{-1}\} = C_{I[s^{-1}]} (M[s^{-1}]) } \]
is Cartesian. 
\end{LEMMA}
\begin{proof}
Because $M$ is projective this follows from the case $M=R$, which we assume from now on.

The top horizontal map is injective because $R$ is integral. The vertical maps are injective by Krull's Theorem. 
Because of the assumptions, the ring $C_I(R)$ is an integral domain.
Hence the bottom horizontal arrow is injective. 
Therefore also the map $C_I(R)[s^{-1}] \rightarrow C_I(R)\{s^{-1}\}$ is injective, and we are left to show that for
arbitrary $n$, the left square in the following diagram with exact lines is Cartesian:
\[ \xymatrix{
   0 \ar[r] & R \ar[r] \ar@{^{(}->}[d] & \frac{1}{s^n}R \ar[r]^{\cdot s^n} \ar@{^{(}->}[d] & R / s^n R \ar[d] \ar[r] & 0 \\
   0 \ar[r] & \widehat{R} \ar[r]  & \widehat{\frac{1}{s^n}R} \ar[r] & \widehat{R / s^n R} \ar[r] & 0 } \]
where in the bottom line we mean the $I$-adic completions of the respective f.g. $R$-modules.
(The rows are exact because exactness of the completion functor on f.g. modules. This is, however, not used in the sequel.)
We have to show that the right vertical map is injective. This is the case if $0 \in R / s^n R$ is closed w.r.t. the $I$-adic topology. Its closure is (by an extension of Krull's Theorem \cite[Chap. III, \S 3, 2., Prop. 5]{Bourbaki}) the set of
$x \in R/s^nR$ for which there exists an $m \in I$, such that $(1-m)x = 0$ holds true. Therefore the above map is injective, if
no $m \in I$ and $x \in R$ exist, such that $s^n|(1-m)x$ and $s^n \nmid x$. Since $s$ is prime, this is the case, if no
$m \in I$ exists, such that $s |1-m$. Since by assumption $s$ is not a unit modulo $I$, this is indeed impossible.
\end{proof}

\begin{PAR}\label{QEXPANSIONNOT}
Let ${}^K_\nRPCD \nSD$ be $p$-ECMSD such that an integral canonical model $\nSh({}^K_\nRPCD \nSD)$ exists.
We consider $\nRPCD$ ``up to refinements'' in this section, i.e. we feel free to refine it whenever necessary.

Assume that there is a boundary morphism:
\[ \iota: \nH_0[\nU, 0] \Longrightarrow \nSD,  \]
which we fix once and for all in this section.

By \cite[Proposition 12.1]{Pink}, its existence forces the reflex field of $\nSD$ to be $\Q$. We
write also shorthand $\nSDB = \nH_0[\nU, 0]$.
It induces a dual boundary morphism (open embedding):
\[ \nShD(\iota): \nShD(\nSDB) \hookrightarrow \nShD(\nSD).  \]
between $\Zpp$-schemes and therefore a base point $\nShD(\iota)(v) \in \nShD(\nSD)(\Zpp)$ which we denote by $v$, too, $\iota$ being fixed.

Let $\mathcal{E}$ be a $\nP_\nSD$-equivariant line bundle on $\nShD(\nSD)$. 

For any $\rho \in \nP_\nSD(\Af)$ we have an associated boundary morphism
\[ (\iota,\rho): {}^{K_1}_{\nRPCD_1} \nSDB\Longrightarrow {}^K_\nRPCD \nSD,  \]
where $K_1 = K^\rho \cap \iota(\nP_{\nSDB}(\Af))$ and $\nRPCD_1$ is defined as pullback (\ref{DEFFUNCTORIALITYRPCD}).
Refining the original $\nRPCD$, if necessary, we can assume that $\nRPCD_1$ is smooth.

$\nSh(\iota,\rho)^* \nMorphSt^* \mathcal{E} \cong \nMorphSt^* \nShD(\iota)^*  \mathcal{E} $ (\ref{MAINTHEOREM3STACKS}, 1.) may be trivialized as in \ref{QEXPANSIONPREP}, 
using a basis vector $z'$ of $(\nShD(\iota)^* \mathcal{E})_v=\mathcal{E}_v$.
$K_1$ always contains a group of the form $\nU_\Zh K(M)$ for a saturated (at $p$) lattice $\nU_\Z \subseteq \nU$ and $p \nmid M$.

Let $s_{z'}$ denote the corresponding trivializing section (\ref{QEXPANSIONPREP}).
Thus every section $f \in H^0(\nSh({}^K_\nRPCD \nSD)_\C, \nMorphSt^* \mathcal{E})$ yields
Fourier coefficients (see \ref{QEXPANSIONPREP}):
\[ a_{\iota, z', k, \alpha, \rho}(f) := a_{k, \alpha, 1}(\frac{\nSh(\iota, \rho)^* f}{s_{z'}}). \]
This definition includes all the Fourier expansions defined before
because we have for $\rho' \in \Zh^*=K(1) \subseteq \Af^*$:
\[ a_{\iota, z', k, \alpha, \rho \rho'}(f) = a_{k, \alpha, \rho'}(\frac{\nSh(\iota, \rho)^* f}{s_{z'}}). \]
\end{PAR}

\begin{LEMMA}\label{LEMMAQEXP}
If $f \in H^0(\nSh({}^K_\nRPCD \nSD)_\C, \nMorphSt^* \mathcal{E}_{merom.})$ is a meromorphic section (any $\nRPCD$, including $\{0\}$ --- the meromorphic sections are the same), we get Fourier coefficients
\[ a_{\iota, z', k, \alpha, \rho, \sigma}(f) := a_{k, \alpha, \rho, \sigma}(\frac{\nSh(\iota, \rho)^* f}{s_{z'}}), \]
which are defined the same way (\ref{QEXPANSIONPREP}) but valid analytically only on a set of the form: 
\[ \{\alpha\} \times (x + \nU_{\R} + \sigma) \times \{\rho\} \subset \nX_{\nSDB} \times \nP_{\nSDB}(\Af), \] 
where $\sigma$ is a maximal dimensional rational polyhedral cone in $\nU_{\R}(-1)$.
\end{LEMMA}
\begin{proof}
We show that there is a boundary point on {\em some} toroidal compactification over $\C$ at which locally the divisor of $f$ consists only of components of the boundary divisor itself. From this the statement follows because a small neighborhood around such a boundary point is mapped to the set given in the statement of the lemma.

It suffices to check this condition in a formal neighborhood of the boundary point. Writing $f$ as a quotient of formal functions, we are reduced to the
case, where $f$ is a formal function in the formal completion at
some boundary point associated with $\iota, \rho, \alpha \in \nX_{\nH_0}$ and a r.p.c. $\sigma$ of maximal dimension (now considered in $\nU_\R$, see \ref{detailsrpcd}). 
Let the cone be generated by $\{\lambda_i \in \nU_\Z\}$ and call $\{z_i\}$ the variables corresponding to the vertices of the dual cone.

This formal completion is
\[ \spf(\bigoplus_{\zeta_{M,\C}} \C \llbracket U_\Z^* \cap \sigma^\vee \rrbracket). \]
Hence the formal function is an element in
\[ \bigoplus_{\zeta_{M,\C}} \C \llbracket U_\Z^* \cap \sigma^\vee \rrbracket = \bigoplus_{\zeta_{M, \C}} \C \llbracket z_1, \dots, z_n \rrbracket. \]
We consider its components separately and call $f$ any of them.

We may now refine the polyhedral cone decomposition in a way that it induces barycentric subdivisions of faces of $\sigma$. 
The faces $\tau$ of $\sigma$ correspond to subsets $S \subseteq \{1,\dots,n\}$ (the vertices of $\tau$). 
There are then $|S|$ new maximal dimensional cones $\sigma'$ in the corresponding subdivison, indexed by the respective vertex $j$ which does not belong to $\sigma'$ anymore.
Let $\lambda_1', \dots, \lambda_n'$ be the generators of the corresponding cone. We have
$\lambda_j' = \sum_{i \in S} \lambda_i$ and $\lambda_i' = \lambda_i$, $i \not =j $. We get a corresponding homomorphism:
\[  \C \llbracket z_1, \dots, z_n \rrbracket  \rightarrow \C \llbracket z_1', \dots, z_n' \rrbracket .  \]
given by $z_j \mapsto \sum_{i \in S} z_i'$ and $z_i \mapsto z_i'$, $i \not= j$.

Since both rings have the same form, we may consider this homomorphism as automorphism $T_{j,S}$ of the ring on the left, in the obvious way. 
We have to show that repeatedly applying those to $f$ and clearing factors of $z_i$, $i=1,\dots, n$ one arrives at an element $\widetilde{f}$ containing a constant term.
Let $S_j$ be the product $T_{i,\{i,j\}}$ for $i\not=j$. It is given by mapping $z_j \mapsto z_j$ and $z_i \mapsto z_i z_j$ for $i \not=j$.

We show by induction on the minimal degree of the occurring monomials in $f$, that by repeatedly applying the $S_i$ to $f$ and clearing factors of $z_i$ one arrives at a function $\widetilde{f}$ containing a constant term. Let $N = \sum_i n_i$ be the minimal degree of a monomial $g = \prod z_i^{n_i}$ occurring in $f$. 
Let $j$ be  such that $n_j \not=0$. By applying $S_j$ to $f$, a monomial $h=\prod z_i^{m_i}$ gets transformed to $ z_j^{\deg(h)}  \prod_{i \not= j} z_i^{m_i }$.
Since $\deg(g)$ is minimal, we may define $\widetilde{f}$ by $\frac{S_j(f)}{ z_j^{\deg(g)}}$. We have
$\deg(\frac{S_j(g)}{ z_j^{\deg(g)}}) = \sum_{i \not= j} n_i < \deg(g)$, hence $\widetilde{f}$ has a term of smaller minimal degree.
\end{proof}

\begin{SATZ}[$q$-expansion principle]\label{QEXPANSION}
Assume \ref{MAINCONJECTURE}.

With the notations of \ref{QEXPANSIONNOT}, 
there is an $M \in \N$, $p \nmid M$ and a finite subset $\{\rho_1, \dots, \rho_n\} \subset \nP_\nSD(\Af)$, such that for $R:=\Zpp[\zeta_{N}]$, $M|N$,
we have
\begin{gather*}
 H^0(\nSh({}^K_{\nRPCD} \nO)_R, \nMorphSt^*\mathcal{E}) = \\
 \left\{ f \in H^0(\nSh({}^K_{\nRPCD} \nO)_\C, \nMorphSt^*\mathcal{E}) \middle| \forall i \
 \forall k \in U^*_\Q, \forall \alpha\  a_{\iota, z', k, \alpha, \rho_i}(f) \in R \right\}
\end{gather*}
and similarly\footnote{$horz.$ means that we inverted all functions, which have no component of the fibre above $p$ in its divisor.}
\begin{gather*}
 H^0(\nSh({}^K_\nRPCD \nO)_R, \nMorphSt^*\mathcal{E}_{horz}) = \\
 \left\{ f \in H^0(\nSh({}^K\nO)_\C, \nMorphSt^*\mathcal{E}_{merom.}) \middle| \substack{ \forall i\ \forall \alpha \ \exists \sigma \subset U_\R \text{ open r.p. cone}, \\
 \forall k \in U_\Q^*,\ a_{\iota, z', k, \alpha, \rho_i, \sigma}(f) \in R } \right\}.
\end{gather*}
Moreover, we have
\[ a_{\iota, z', k, \alpha, \rho, \sigma}({}^\tau f) = {}^\tau (a_{\iota, z', k, \alpha, \rho k_\tau^{-1}, \sigma}(f)) \qquad \forall \tau \in \Gal(\Q(\zeta_N)|\Q), \]
hence $f$ is (in both cases) defined over $\Zpp$ instead of $R$, if it satisfies
\[ a_{\iota, z', k, \alpha, \rho, \sigma}(f) = {}^\tau (a_{\iota, z', k, \alpha, \rho k_\tau^{-1}, \sigma}(f)), \]
for all $\tau \in \Gal(\Q(\zeta_N)|\Q)$, all $\rho, \alpha, k$ and some $\sigma$ (depending on $\rho$ and $\alpha$).
Here $k_\tau$ is the image of $\tau$ under the natural isomorphism $\Gal(\Q(\zeta_N)|\Q) \cong K(1)/K(N)$.
\end{SATZ}

\begin{proof}
We may show the theorem for a particular $\nRPCD$ and all of its refinements. A posteriori, it will then be true for any $\nRPCD'$, using a common refinement of $\nRPCD$ and $\nRPCD'$.
Let $f$ be a complex section of the bundle $\nMorphSt^*(\mathcal{E})$ over $\nSh({}^K_\nRPCD \nSD)$.
By the abstract $q$-expansion principle \ref{LEMMA_FORMAL}, a section of a locally free sheaf is defined over $\nSh({}^{K}_{\nRPCD} \nSD)_R$ if on
any connected component (over $R$), there is a point $p$, defined over $R$, such that after passing to the completion at $p$, the corresponding formal series is defined over $R$.
Now on any connected component $i$ of $\nSh({}^K \nSD)_\C$, there is a boundary point corresponding to a boundary component $(\iota, \rho_i)$ and a $\sigma \subset \nU_\R$.
Using the boundary isomorphism
$\nSh(\iota, \rho_i)$ (\ref{MAINTHEOREM1a}), defined over $\Zpp$, it suffices to show that the pullback of $f$ to $\nMorphSt^*(\nShD(\iota)^* \mathcal{E})$ over
$\widehat{\nSh({}^{K_i}_{\nRPCD_i}\nH_0[\nU, 0])}$ (completion with respect to the boundary point corresponding to $\sigma$) is defined over $R$. We may choose the integer $M$ such that all connected components of
$\nSh({}^{K_i}_{\nRPCD_i}\nH_0[\nU, 0])$ are defined over $R$, for all $i$ and such that we may test $f$ by pulling it further back to $\widehat{\nSh({}^{K_i'}_{\nRPCD_i}\nH_0[\nU, 0])}$, where $K_i' = \nU_{\Zh,i} K(M)$ for some lattice $\nU_{\Z,i} \subset \nU$ as in \ref{QEXPANSIONPREP}. We refine $\nRPCD$ such that $\nRPCD_i$ is smooth
w.r.t. $K_i'$ for all $i$.

By the consideration in \ref{QEXPANSIONNOT} above, the formal function $f/s_{z'}$  on
\[ \widehat{\nSh({}^{\nU_{\Zh,i} K(M)}_{\nRPCD_i}\nH_0[\nU, 0])} \]
for $\nU_{\Zh,i} K(M) \subseteq K_i$ is given by

\[ \left(\frac{f}{s}\right)_{\zeta_{M,\C}} = \sum_{\lambda \in \nU_\Z^*} \alpha_{\iota,z', k,\alpha,\rho_i k,\lambda,\sigma}(f). \]
Here $[\alpha, k] \in \Q^* \backslash \nX_{\nH_0} \times \Af^* / K(M)$, with $k\in K(1)$ is such that $\zeta_{M} = \exp(\alpha(\frac{k}{M}))$.
(It suffices of course to look at the fibre over the $\zeta_{M,\C}$ corresponding to $[\alpha, 1]$).
Then, since by assumption these coefficients lie in $R$, $f/s_{z'}$ is defined over $R$ and hence $f$ itself.
The first statements follow because we investigated at least one boundary point on each connected component.
If $f$ was a meromorphic section, we may refine $\nRPCD$ in a way such that for all $i$ its pullback $\nRPCD_i$ contains a cone determined by $\sigma \subset \nU(\R)$ such
that there exists a Fourier expansion in a neighborhood of the corresponding boundary point (see \ref{LEMMAQEXP}). 
The statement about the Fourier coefficients of the Galois conjugates was shown in \ref{QEXPANSIONPREP}.
\end{proof}

\section{Shimura varieties of orthogonal type, special cycles}

\begin{PROP}
If $\dim(\nL) \ge 3$, the reflex field of $\nS(\nL)$ is $\Q$.
The compact dual is the zero quadric, i.e.
\[ \nShD(\nS(\L))(R) = \{ <z> \in \PP(\nL_R) \where \langle z, z \rangle = 0 \} \]
for any ring $R$ over $\Zpp$.

If $\dim(\nL) = 2$, the reflex field of $\nS(\nL)_+$, resp. $\nS(\nL)_-$ is the imaginary quadratic field (in $\C$) associated with the 
(negative) definite binary quadratic form $\nL$.
The compact dual is a point given by the isotropic vector corresponding to {\em the} filtration in $\nX_{\nS(\nL)_+}$, respectively $\nX_{\nS(\nL)_-}$.
\end{PROP}
\begin{proof}
Let $D$ be a negative definite subspace $D$ defined over $\Zpp$.
Over an imaginary quadratic field $F$, unramified at $p$, there is a morphism
\[ \mu: \Gm \hookrightarrow \GSpin(\nL), \]
acting with exponent $2$ on one isotropic vector $z$, with exponent $0$ on $\overline{z}$ and with exponent $1$ on the orthogonal complement. Over $\C$ it yields the associated morphism $u_h$ (\ref{DEFREFLEXRING}), where $h$ is one of the 2 morphisms corresponding to $D$. The morphism $\overline{u_h}$ is also 
of the form $u_{h'}$ and hence conjugated to $u_h$
if $n\ge 3$. The conjugacy class of these morphisms is therefore defined over $\Q$.
If $n=2$, $\GSpin(\nL)$ is a torus, and hence the reflex field is $F$.
Furthermore, the morphism $u$ defines the filtration $0 \subset <z> \subset <z>^\perp \subset \nL$ which is
completely determined by $<z>$.
\end{proof}

\begin{BEM}\label{HODGEEMBEDDINGDUAL}
The Hodge embedding from Lemma \ref{LEMMAGSPINEMBEDDING} induces a morphism (cf. Main Theorem \ref{MAINTHEOREM2})
\[ \nShD(\nS(\nL)) \hookrightarrow \nShD(\nH(\Cliff^+(\nL), \langle\cdot ,\cdot \rangle_\delta)). \]
Here $\nShD(\nS(\nL))$ is the space of isotropic lines in $\nL$ and $\nShD(\nH(\Cliff^+(\nL), \langle\cdot ,\cdot \rangle_\delta))$
is the space of Lagrangian saturated sublattices of  $\Cliff^+(\nL)$ with respect to the form $\langle\cdot ,\cdot \rangle_\delta$. The proof of Lemma \ref{LEMMAGSPINEMBEDDING} shows that this map is given (for any $R$)
by sending a $<z>$ in $\nL_R$, with $z$ assumed to be primitive, to
the subspace $zw\Cliff^+(\nL_R)$, where $w$ is a primitive vector with $\langle z, w \rangle=1$ (this subspace actually does not depend on $w$).
One also immediately sees that this is compatible with the Borel embedding because $w=\frac{\overline{z}}{\langle z, \overline{z}\rangle}$ leads to
the projector considered in the proof of Lemma \ref{LEMMAGSPINEMBEDDING}.
\end{BEM}

\begin{LEMMA}\label{COMPONENTS}
Let $K$ be a neat admissible compact open subgroup of $\SO(\nL_\Af)$ and $\nRPCD$ be a smooth $K$-admissible rational polyhedral cone decomposition.
The (geometric) connected components of $\nSh({}^K_\nRPCD \nO)_{\overline{\Q}}$ and $\nSh({}^K_\nRPCD \nO)_{\overline{\Fp}}$ are in bijection and defined over
$\spec(\Zpp[\zeta_N])$, where $\zeta_N$, $p\nmid N$ is an $N$-th root of unity.
\end{LEMMA}
\begin{proof}
This follows directly from \ref{MAINTHEOREM1}, \ref{MAINTHEOREM1a} using [Milne4, Theorem 5.17].
\end{proof}

\begin{LEMMA}
Let $\nL_\Z$ be a lattice with quadratic form of discriminant $D \not=0$.
Let $K$ be a compact open subgroup of $\SO(\nL_\Af)$, and pick a set of representatives
\[ g_i \in \SO(\Q) \backslash \SO(\Af) / K. \]
For each $g_i$ there is the lattice $\nL_\Z^{(i)}$ characterized by $\nL_\Zp^{(i)} = g_i \nL_\Zp$.
We have
\[ \left[ \SO(\Q) \backslash \nX_{\nO} \times (\SO(\Af) / K) \right] = \bigcup_i \left[ \Gamma_i \backslash \nX_\nO \right], \]
where $\Gamma_i$ is the subgroup of $\SO(\nL^{(i)}_\Z)$ defined by $\SO(\Q) \cap g_i K g_i^{-1}$.

If $K$ is the discriminant kernel of $\nL_\Z$, each $\Gamma_i$ is the discriminant kernel of the respective $\nL^{(i)}_\Z$.
\end{LEMMA}
\begin{proof}{[H, 1.10]} or [Milne4].
\end{proof}

\begin{PAR}
Let $\nL_\Zpp$ be a lattice with unimodular
quadratic form, i.e. inducing an isomorphism $\nL_\Zpp \cong (\nL_\Zpp)^*$, and $\nL'_{\Zpp}$ be a saturated sublattice with unimodular form, of signature $(m-2,2)$ and $(m-2-n,2)$, respectively, where $0 < n \le m-2$.

Recall the embedding \ref{LEMMAGSPINEMBEDDING}, here used mod $\Gm$
\[ \iota: \nO(\nL') \hookrightarrow \nO(\nL), \]
in the case $m-n>2$ and the 2 embeddings 
\[ \iota_\pm: \nO(\nL')_\pm \hookrightarrow \nO(\nL), \]
in the case $m-n=2$.

Let an admissible compact open subgroup $K$ be given.
For each $g \in \SO(\nL_\Afp)$ we get a conjugated embedding:
\[ (\iota, g): {}^{K'_g}\nO(\nL') \hookrightarrow {}^K \nO(\nL) \]
for $K'_g = K \cap g^{-1} \iota(\SO(\nL'_\Af)) g$.

In addition, we may find smooth (w.r.t. $K$, resp. $K_g'$) complete and projective $\nRPCD$, and $\nRPCD'_g$ for each $g$ (\ref{PROPERTIESRPCD}),
such that models exist (\ref{MAINTHEOREM1a}), and we have, in the end, embeddings of $p$-ECMSD:
\[ (\iota, g): {}^{K'_g}_{\nRPCD'_g} \nO(\nL') \hookrightarrow {}^K_\nRPCD \nO(\nL) \]
(resp. with $\pm$). 

Let $\nM_{\Zpp}$ be another lattice with unimodular quadratic form and define 
\[ I(\nM, \nL)_R := \{ \alpha: \nM_R \rightarrow \nL_R \where \alpha  \text{ is an isometry}\}. \]
We have $I(\nM,\nL)_\Q \not= \emptyset \Leftrightarrow I(\nM,\nL)_\Af \not= \emptyset$ by Hasse's principle.
If these sets are nonempty, we may then find even an $x \in I(\nM, \nL)_{\Zpp}$.

For a $K$-invariant admissible\footnote{by this we mean that $\varphi$ is the tensor
product of a Schwartz function $\varphi \in S((\nM^*\otimes \nL)_\Afp)$
with the characteristic function of $(\nM^*\otimes \nL)_\Zp$} Schwartz function $\varphi \in S((\nM^*\otimes \nL)_\Af)$
\[ I(\nM, \nL)_\Af \cap \supp(\varphi) = \coprod_j K g_j^{-1} x \]
with finitely many $g_j \in \SO(\nL_{\Afp})$. We assume that $\nRPCD$ has been refined such that $\nRPCD'_{g_j}$ exists, with the
properties claimed above for all $j$.
\end{PAR}

\begin{DEF}\label{DEFSPECIALCYCLEP}
If $I(\nM,\nL)_\Q \not= \emptyset$,
we define the $p$-integral {\bf special cycle} associated with $\nM_{\Z_{(p)}}$ and $\varphi$
on $\nSh({}^K_\nRPCD \nO(\nL))$ as
\[ \nZ(\nL, \nM, \varphi; K) := 
\sum_j \varphi(g_j^{-1}x) \im( \nSh(\iota, g_j) ) ,
 \]
where $\iota$ is associated with the embedding $x^\perp \hookrightarrow \nL_{\Zpp}$.
If $I(\nM,\nL)_\Q = \emptyset$ we define $\nZ(\nL, \nM, \varphi; K):=0$.
\end{DEF}
Note that if $m-n=2$, $\nSh({}^{K_g'}_{\nRPCD_g'} \nO(x^\perp))$ is actually equipped with a morphism to
$\spec(\OO_\nSD)$ which is an extension of $\spec(\Zpp)$. 
In the definition above, we consider, however, the structural morphism to $\spec(\Zpp)$. (For the $\C$-points this means, that we get in any case the special cycles as defined e.g. in \cite{Borcherds1} or \cite{Kudla4}.)
The sum is  finite and independent of the choice of the $g_i$.
The formation is compatible with Hecke operators (\ref{DEFPECMSD}) in the obvious sense.

If $I(\nM,\nL)_\Q \not= \emptyset$ and $\nM_{\Q}$ is only a $\Q$-vector space with non-degenerate quadratic form and $\varphi$ any Schwartz function, we write
\[ I(\nM, \nL)_\Af \cap \supp(\varphi) = \coprod_j K g_j^{-1} x \]
for any $x \in I(\nM, \nL)_\Q$. 
\begin{DEF}\label{DEFSPECIALCYCLE}
We define the rational {\bf special cycle} associated with $\nM_{\Q}$ and $\varphi$
on $\nSh({}^K_\nRPCD \nO(\nL))$ as 
\[ \nZ(\nL, \nM, \varphi; K) := \sum_j \varphi(g_j^{-1}x) \overline{\im( \nSh(\iota, g_j) )}^{Zar}. \]
Here $\iota$ is associated with the embedding $x^\perp \hookrightarrow \nL_{\Q}$ and
$\nSh(\iota, g_j)$ is only a morphism of rational Shimura varieties. 
Here even $\nRPCD$ may be chosen arbitrary. In that case we understand the Zariski closure of
the embedding of the uncompactified rational Shimura variety.
\end{DEF}
From the extension property (cf. Main Theorem \ref{MAINTHEOREM1}, 3.) and the properties of
toroidal compactifications (cf. Main Theorem \ref{MAINTHEOREM1a}) follows that this definition coincides with the previous one, if there exists
a lattice $\nM_{\Zpp}$, such that the restriction of the 
quadratic form is unimodular, $\varphi$ is admissible, and $\nRPCD$ is chosen in such a way that we have embeddings of $p$-ECMSD.

\section{The arithmetic theory of automorphic forms}
  
\begin{PAR}\label{INPUTDATAAUTOMORPHICVB}
Let ${}^K_\nRPCD \nSD$ be pure $p$-ECMSD (\ref{DEFPECMSD}) with smooth (w.r.t. $K$), projective, and complete $\Delta$, $E \subset \C$ be the reflex field and $\OOO_\nSD = \OOO_{E,(p)}$ the reflex ring. 
We have an isomorphism
\[ \nSh({}^K_\nRPCD \nSD) \times_{\spec(\Zpp)} \C = \coprod_{\sigma \in \Hom(E, \C)} \nSh({}^K_\nRPCD \nSD) \times_{\OO_{\nSD}, \sigma} \C.  \]

Conjecture A of \cite[p. 58]{MiShi} which is proven in \cite[Corollary 18.5]{MiShi} states in particular that
for each $\sigma \in \Hom(E, \C)$ there
is a (rational) Shimura datum $\nSD^\sigma$, and we have an isomorphism
\begin{equation} \label{compos}
 (\nSh({}^K \nSD) \times_{\OO_{\nSD}, \sigma} \C)^{an} \rightarrow [\nP_{{}^\sigma\nSD} \backslash \nX_{{}^\sigma \nSD} \times \nP_{{}^\sigma \nSD}(\Af) / K]  
\end{equation}
which renders the system $\nSh({}^K \nSD) \times_{\OO_{\nSD}, \sigma} {}^\sigma \OO_{\nSD}$ with a twisted adelic action into a rational canonical model for the Shimura datum ${}^\sigma \nSD$. 

The compact dual $\nShD(\nSD)$, considered as scheme over $\spec(\Zpp)$, can be identified also with 
a closed subscheme of $\mathcal{QPAR}$ (\ref{QUASIPARABOLICS}) and carries an action of
$\nP_\nSD$.

Let $\mathcal{E}$ a locally free sheaf on the stack (over $\spec(\Zpp)$!)
\[ \left[ \nShD(\nSD)/P_\nSD \right] \]
(i.e. a $\nP_\nSD$-equivariant sheaf on $\nShD(\nSD)$)
and $h_E=\{h_{E,\sigma}\}$ for $\sigma \in \Hom(E, \C)$ be a collection of $\nP_{{}^\sigma \nSD}(\R)\nU_{{}^\sigma \nSD}(\C)$-invariant Hermitian metric on $E_\C|_{h(\nX_{{}^\sigma \nSD})}$
where $\bigcup_\sigma h(\nX_{{}^\sigma \nSD})$ is embedded into $(\nShD(\nSD) \times_\Zpp \C)(\C)$ via the Borel embeddings --- \ref{MAINTHEOREM2} --- taking into account the isomorphism (\ref{compos}).

Recall from \ref{MAINTHEOREM3STACKS} the 1-morphism 
\[ \nMorphSt: \nSh({}^K_\nRPCD\nSD) \rightarrow \left[ \nShD(\nSD)/\nP_\nSD \right]. \]
\end{PAR}

\begin{PAR}\label{AUTOMORPHICVBC}
Let $D$ be the boundary divisor (\ref{MAINTHEOREM1a}). It is a divisor of normal crossings.
We will define a Hermitian metric $\nMorphSt^* h_E$ on $(\nMorphSt^*\mathcal{E})_\C$, singular along $D_\C$,
as follows\footnote{This is a slight abuse of notation, since the construction seems to depend on the particular presentation of 
the target stack as a quotient of $\nShD$}: We do this for each ${}^\sigma \nSD$ separately, so let $\nSD$ be one of them.
By definition
$\nMorphSt^*E_\C(U)$, for any open $U \subset \nSh({}^K_\nRPCD \nSD)(\C)$, is given by the
$\nP(\Q)$-invariant sections of the pullback of $E_\C$ to the standard principal bundle $\nSPB({}^K_\nRPCD \nSD)$ 
(the right $\nP_\nSD$-torsor describing the morphism $\nMorphSt$) via the compatibility isomorphism of $\nMorphSt$ over $\C$ with $\nMorphSt_\C$ (\ref{MAINTHEOREM3STACKS}).
Now consider the diagram of analytic manifolds (we assumed $K$ to be neat):
{ \footnotesize 
\[
\xymatrix{ \nP_\nSD(\Q) \backslash \nX_\nSD \times (\nP_\nSD(\Af)/K) \ar@{=}[d] & \nX_\nSD \times (\nP_\nSD(\Af)/K) 
\ar@{^{(}->}[d] \ar[l] \ar[r]^p & h(\nX_\nSD) \ar@{^{(}->}[d] \\
\nP_\nSD(\Q) \backslash \nX_\nSD \times (\nP_\nSD(\Af)/K)  & \nP_\nSD(\Q) \backslash \nX_\nSD \times \nP_\nSD(\C) \times (\nP_\nSD(\Af)/K) 
\ar[l] \ar[r] & \nShD(\nSD)(\C)
} \]}Let $U \cong \nSh({}^K \nSD)(\C)$ be the complement of $D_\C$.
The diagram shows that $\nMorphSt^*E_\C|_U$ is canonically identified with the
$\nP_\nSD(\Q)$-invariant sections of the pullback $p^*(E_\C|_{h(\nX_\nSD)})$ to the standard local system $\nX_\nSD \times \nP_\nSD(\Af)/K$.
But this pullback carries the pullback of the Hermitian metric $h_E$. 
For the induced $\nP_\nSD(\Q)$-action on $p^*(E_\C|_{h(\nX_\nSD)})$
compatible with the usual action on the standard local system, $h_E$ is invariant as well and descends
to a smooth Hermitian metric on $U$.
\end{PAR}

\begin{SATZ}\label{MUMFORD}
Let ${}_\nRPCD^K \nSD$ be {\em pure} $p$-ECSD, let $(\mathcal{E}, h_E)$ as in \ref{INPUTDATAAUTOMORPHICVB} which is fully-decomposed, i.e. the unipotent radical of 
a stabilizer group under the action of $\nP_\nSD(\C)$ on $\nShD(\nSD)(\C)$ acts trivially on $\mathcal{E}$. 

Then the vector bundles $\nMorphSt^*\mathcal{E}_\C$ are the unique extensions to $\nSh({}^K_\nRPCD \nSD)$ of these bundles on $\nSh({}^K \nSD)$ 
such that the metric $\nMorphSt^* h_{\mathcal{E}}$ is good along the boundary divisor $D_\infty$ in the sense of \cite[p. 242]{Mumford1}.
\end{SATZ}
\begin{proof}
Follows from the construction of the extension in [loc. cit.], cf. also \cite[\S 3.8]{HaZu1}.
\end{proof}

\begin{DEF}\label{DEFAUTOMORPHICVB}
Let $(\mathcal{E}, h_E)$ be as in \ref{INPUTDATAAUTOMORPHICVB}. The locally free sheaf $\nMorphSt^*\mathcal{E}$, together with the
singular (along $D_\C$) Hermitian metric $\nMorphSt^*h_E$ is called the \textbf{Hermitian automorphic vector bundle} associated
to $(\mathcal{E}, h_E)$ on $\nSh({}^K_\nRPCD \nSD)$.
\end{DEF}

\begin{BEM}\label{FUNCTORIALITYAUTOMVECTB}
It follows directly from the construction of metrics in \ref{AUTOMORPHICVBC}
that for a morphism of $p$-ECMSD
\[ (\alpha, \rho): {}^{K_1}_{\nRPCD_1} \nSD_1 \rightarrow {}^{K_2}_{\nRPCD_2} \nSD_2 \]
and given $(\mathcal{E}, h_E)$, with $\mathcal{E}$ a locally
free sheaf on $\left[ \nShD(\nSD_1)/\nP_{\nSD_1} \right]$ and $h_E=\{ h_{E,\sigma}\}$ a collection of invariant Hermitian metrics, we have a canonical isomorphism
\[ \nSh(\alpha, \rho)^* (\nMorphSt^* \mathcal{E}, \nMorphSt^* h_E) \cong (\nMorphSt^* \nShD(\alpha)^* \mathcal{E}, \nMorphSt^* \nShD(\alpha)^*h_E), \]
in particular these isomorphisms are compatible with composition.

For the metrics this follows from the {\em 2-isomorphism} 
\[ (\nMorphSt({}^K \nSD) \times_{\spec(\OOO_\nSD)} \spec(\C))^{an} \rightarrow \nMorphSt_\C({}^K \nSD), \]
given by \ref{MAINTHEOREM3STACKS}.

Similarly, in the case of a boundary component 
\[ (\iota,\rho): {}^{K_1}_{\nRPCD_1} \nSDB \Longrightarrow {}^K_\nRPCD  \nSD, \]
and, say, the case of $\OO_\nSD = \Zpp$ for simplicity,
if $h_E$ is the restriction
of a $\nP_{\nSDB}(\R)\nU_\nSD(\C)$-invariant metric $h_{E,1}$ on $\nX_{\nSDB}$ to $\nX_{\nSDB \Longrightarrow \nSD} \subset \nX_{\nSDB}$, 
we have on the formal completion a canonical isomorphism
\[ \nSh(\iota,\rho)^* (\nMorphSt^*\mathcal{E}, \nMorphSt^*h_E) \cong (\nMorphSt^* \mathcal{E}|_{\nShD(\nSDB)}, \nMorphSt^* h_{E,1}) \]
taking into account that the formal isomorphism converges over $\C$. However, in general, the metrics
$h_E$ do not come from $h_{E,1}$'s as above (they are not even defined everywhere on $\nX_\nSDB$), but
of course an $\nMorphSt^* h_E$ is always defined in a neighborhood of $D_\C$ on
$\nSh({}^{K_1}_{\nRPCD_1}\nSDB)(\C)$ and the above statement makes sense.
\end{BEM}

\begin{PAR}
In the reminder of this section we will describe the necessary tools from Arakelov theory that will be needed to define uniquely determined heights on (resp.
arithmetic volumes of) Shimura varieties w.r.t. automorphic line bundles. We will be as minimalist as possible, stating only the properties of the objects involved that are needed in the sequel and which uniquely characterize the resulting values. 
We assume familiarity with classical Arakelov theory as presented in \cite{SABK}.
\end{PAR}

\begin{PAR}
We consider the category $(X, D)$, where $X$ is a smooth complex algebraic variety equipped with a divisor of normal crossings $D$. Morphisms $\alpha: (X, D_X) \rightarrow (Y,D_Y)$ are morphisms $\alpha: X \rightarrow Y$ which
satisfy $\alpha^{-1}D_Y \subseteq D_X$. 

We consider an association with maps a pair $(X,D)$ to a subspace of differential forms 
\[ A^{i,j}(X, D) \subset A^{i,j}(X-D), \]
for every $i,j$, stable w.r.t. the differentials $\dd$ and $\dd^c$. In addition, for any vector bundle $E$
on $X$, we consider a class $\mathcal{M}$ of Hermitian metrics on $E|_{X-D}$, such that 
\begin{enumerate}
\item[(F1)] For any morphism $f: (X,D) \rightarrow (Y,D')$, $f^*$ restricts to a map of $A^{i,j}$'s and $f^*$ of
a vector bundle with metric in $\mathcal{M}$ is again a vector bundle with metric in $\mathcal{M}$.
\item[(F2)] Each $\xi \in A^{i,j}(X, D)$ defines a current on $X$, and we have $\dd [\xi] = [\dd \xi]$, resp. $\dd^c [\xi] = [\dd^c \xi]$.
\item[(F3)] The Chern character form of any vector bundle with metric in $\mathcal{M}$ is in $\bigoplus_p A^{p,p}(X,D)$.
\item[(F4)] There is a unique way to associate with each $(X,D)$ and exact sequence 
\[ E_\bullet:  0 \rightarrow E_{-1} \rightarrow  E_0  \rightarrow  E_1 \rightarrow 0\]
of vector bundles on $X$ and metrics $h_{-1},h_0,h_1$ in $\mathcal{M}$
a Bott-Chern character $\chernch(E_{\bullet},h_{-1},h_{0},h_1) \in \bigoplus_p A^{p,p}(X,D)/(\im \dd + \im \dd^c)$ such that
\begin{enumerate}
\item $\ddc \chernch(E_\bullet, h_{-1}, h_0, h_1) = \sum_i (-1)^i \chernch(E_i, h_i)$,
\item the formation is compatible with pullback,
\item $\chernch(E_\bullet,h_{-1},h_0,h_1)=0$ if $E_\bullet$ splits and $h_0 = h_{-1} \oplus h_1$,
\item for an exact square consisting of vector bundles $E_{i,j}$, $i,j=-1,\dots,1$ equipped with metrics $h_{i,j}$ in $\mathcal{M}$, we have
\[ \sum_i (-1)^i \chernch(E_{i,\bullet},h_{i,-1},h_{i,0},h_{i,1}) = \sum_j (-1)^j \chernch(E_{\bullet,j},h_{-1,j},h_{0,j},h_{1,j}). \]
\end{enumerate}
\end{enumerate}

Often convenient is the following property, which we will not need, however:
\begin{enumerate}
\item[(F5)] If $Y$ is a cycle of codimension $p$ which intersects $D$ properly then for any Greens function $g$ of logarithmic type along $Y$ \cite[2.1, Definition 3]{SABK}, and for all $\xi \in A^{p,p}(X,D)$, $g \xi$ is locally integrable on $X$.
\end{enumerate}
\end{PAR}

Note that by the properties (F4,a--c) the Bott-Chern characters involving only smooth metrics have to be necessarily the same as those
considered in \cite[3.1, Theorem 2]{SABK}.
We will write also $\chernch(E, h_{-1}, h_0)$ for $\chernch(E_\bullet, h_{-1}, h_0, 0)$ where $E_\bullet$ is the sequence $0\rightarrow E \rightarrow E \rightarrow 0 \rightarrow 0$.

\begin{PAR}\label{STARPRODUCT}
From closedness under $\dd$ and $\dd^c$ follows in particular that for $e_1 \in A^{p,p}(X,D)$ and $e_2 \in A^{q,q}(X,D)$ the (star-)product $e_1 \ast e_2 := (\ddc e_1) e_2$ is associative and commutative mod image $\dd,\dd^c$.
If $g$ is a form on $(X-D)(\C)$ which can be expressed as a sum $\widetilde{g} + g_0$, where $\widetilde{g}$ is
a Greens form of logarithmic type \cite[Definition 3, p. 43]{SABK} and $g$ is in $\mathcal{M}$ one can define
$g \ast h \in A^{p+q+1,p+q+1}(X-D)$ to be $\widetilde{g} \ast \widetilde{h} + g_0 \ddc \widetilde{h} + h_0 \ddc \widetilde{g} + h_0 \ddc g_0$, which
is clearly (mod image $\dd, \dd^c$) independent of the way of expressing the sum.
\end{PAR}

\begin{PAR}\label{COMPSTARPRODUCT}
If $Y_1$ and $Y_2$ are cycles on $X$ of codimension $p$ and $q$, respectively, {\em without intersection}, and
$\widetilde{g}_1$, $\widetilde{g}_2$ be Greens forms of logarithmic type for them, we can represent their
star-product (modulo image $\dd, \dd^c$) by
\[ \widetilde{g}_1 \ast \widetilde{g}_2 =  \sigma_{1} \widetilde{g}_{1} (\ddc \widetilde{g}_2) + \ddc (\sigma_2 \widetilde{g}_1) \widetilde{g}_2, \]
where $\sigma_1 + \sigma_2 = 1$ is a partition of unity, such that $\sigma_1=0$ on $Y_1$ and $\sigma_2=0$ on 
$Y_2$. If $g_1$ and $g_2$ are the sums of the previous Greens forms with
forms $g_{1,0} \in A^{p,p}(X,D)$, resp. $g_{2,0} \in A^{q,q}(X,D)$, for the star product as defined in \ref{STARPRODUCT}, we have the same formula
\[ g_1 \ast g_2 =  \sigma_{1} g_{1} (\ddc g_2) + \ddc (\sigma_2 g_1) g_2. \]
For this note that (e.g.)
\[  \sigma_{1} g_{1} (\ddc g_{2,0}) + \ddc (\sigma_2 g_1) g_{2,0}
=  g_{2,0} \ddc g_1 +  \sigma_{1} g_{1} (\ddc g_{2,0}) - \ddc (\sigma_1 g_1) g_{2,0}. \]
However in
\[ \sigma_{1} g_{1} (\ddc g_{2,0}) - \ddc (\sigma_1 g_1) g_{2,0} \]
$\sigma_{1} g_{1}$ and $g_{2,0}$ are in $A^{q,q}(X,D)$. Therefore this
is contained in the image of $\dd, \dd^c$. Using an embedded resolution of singularities
and properties (F1, F2), one can extend this to non-empty intersections. We will not need this, however.
\end{PAR}

\begin{SATZ}[Burgos, Kramer, K\"uhn]\label{BURGOS}
A class $A^{i,j}(\cdot, \cdot)$ of differential forms and class of metrics $\mathcal{M}$ satisfying the axioms (F1)-(F4) such that $\mathcal{M}$ includes the natural metrics $\nMorphSt^* h_\mathcal{E}$ of fully decomposed automorphic vector bundles $\nMorphSt^* \mathcal{E}$ on pure Shimura varieties exists.
\end{SATZ}
\begin{proof}The class of differential forms is the class of $\log$-$\log$-forms \cite[section 2.2]{BKK2}. Those are in particular pre-$\log$-$\log$ forms \cite[section 7.1]{BKK1}. The class of metrics is the class of $\log$-singular ones \cite[4.29]{BKK1}.
The properties above are shown respectively at the following places:
(F1) \cite[Proposition 2.24]{BKK2} for forms, \cite[Proposition 4.35]{BKK2} for the metrics,
(F2) \cite[Proposition 7.6]{BKK1},
(F3) is part of their definition of $\log$-singular metric,
(F4) \cite[Proposition 4.64]{BKK2},
(F5) (not needed here) follows from the proof of \cite[Lemma 7.35]{BKK1}. The statement about fully decomposed automorphic vector bundles is proven in \cite[Theorem 6.3]{BKK2}.
\end{proof}

\begin{PAR}\label{DEFCHOW}
Let $X$ be a projective regular and flat scheme over $\spec(R)$, where $R$ is any localization of $\Z$. Let $D$ be a divisor on $X$ such that $D_\C$ is a divisor of normal crossings.
From now on we assume that classes $A^{i,j}(\cdot, \cdot)$ and $\mathcal{M}$ satisfying (F1)-(F4) have been chosen and define in analogy to \cite[p. 55]{SABK}:
\begin{eqnarray*}
A^{p,p}(X,D) &:=& \{\omega \in A^{p,p}(X(\C), D(\C)) \ |\ \omega \text{ real}, F^*_\infty \omega = (-1)^p \omega \} \\
\widetilde{A}^{p,p}(X,D) &:=& A^{p,p}(X,D)/( \im \partial + \im \overline{\partial} ).
\end{eqnarray*}

We define arithmetic Chow groups $\aCH^p(X,D)$ by the following pushout
\[
\xymatrix{ 
 \CH^{p-1,p}(X)  \ar@{=}[d] \ar[r]&  \widetilde{A}^{p-1,p-1}(X,D) \ar[r]^a  & \aCH^p(X,D) \ar[r] & \CH^p(X) \ar@{=}[d] \ar[r] & 0\\  
 \CH^{p-1,p}(X) \ar[r] &\widetilde{A}^{p-1,p-1}(X) \ar@{^{(}->}[u]  \ar[r]^{a_{GS}} & \aCH_{GS}^p(X) \ar@{^{(}->}[u] \ar[r] & \CH^p(X)  \ar[r] & 0
}
\]
where an index $GS$ means the corresponding object as defined in \cite{SABK}.

We define a product
\[ \aCH^p(X,D) \times \aCH^q(X,D) \rightarrow \aCH^{p+q}_\Q(X,D) \]
by defining:
$a(b_1) a(b_2)$ to be $a(b_1 \ddc b_2)$ and
$a(b) \widehat{z}$ to be $a(b \omega(\widehat{z}))$ (for the map $\omega$ of [loc. cit.]). This is well defined because the formulas
hold for $a_{GS}$.
This product is obviously associative and commutative.
We also form the direct sum $\aCH^\bullet(X,D)_\Q := \bigoplus_p \aCH^p(X,D)_\Q$, which yields a contravariant functor (w.r.t. pullback) from the category of pairs (X,D), where $X$ is projective, regular, and flat over $\spec(R)$ and $D$ is a divisor on $X$ such that $D_\C$ has only normal crossings, to the category of graded $\Q$-algebras.
\end{PAR}

\begin{PAR}
If $X$ is projective over $\spec(R)$ of relative dimension $d$, where $R$ is any localization of $\Z$, we have a push forward
\[ \pi_*: \aCH^p(X,D) \rightarrow \aCH^{p-d}(\spec(R)), \]
which on the image of $a$ is defined by integration (see F2).
\end{PAR}

\begin{DEF}
For any vector bundle $\mathcal{E}$ with metric $h$ in the class we define an arithmetic Chern character by
\[ \achernch(\mathcal{E}, h) := \achernch_{GS}(\mathcal{E}, \widetilde{h}) + a(\chernch(E, h, \widetilde{h})), \]
where $\widetilde{h}$ is any smooth metric on $E=\mathcal{E}_\C$. In particular the graded components define
Chern classes $\achern_i(\mathcal{E},h)$ of every degree.
\end{DEF}

\begin{LEMMA}
This is well defined.
\end{LEMMA}
\begin{proof}
We have the property $\achernch_{GS}(E, \widetilde{h}_2)-\achernch_{GS}(E, \widetilde{h}_1) = a(\chernch(E, \widetilde{h}_2,\widetilde{h}_1))$ \cite[Proposition 1, p. 86]{SABK}. From properties (F4,c,d) follows that
$\chernch(E, \widetilde{h}_2,\widetilde{h}_1) = \chernch(E, \widetilde{h}_2,h)-\chernch(E, \widetilde{h}_1, h)$.
\end{proof}

The formation of Chern character (hence Chern classes) is compatible with pullback because this is true for $\achernch_{GS}(\mathcal{E}, \widetilde{h})$.

\begin{PAR}\label{RN}
For $R$ a localization of $\Z$ we have a map 
\[ \adeg: \aCH^1(\spec(R))_\Q \rightarrow \R_R, \]
where $\R_R$ is $\R$ modulo rational multiples of $\log(p)$ for all $p$ with $p^{-1} \in R$.
We denote $\R_{\Zpp}$ also by $\R^{(p)}$ and $\R_{\Z[1/N]}$ by $\R_N$.

Obviously $\R_N$ {\em injects} into the
fibre product of all $\R^{(p)}$ $p \nmid N$ over $\R_\Q$.
\end{PAR}

\begin{DEF} \label{CHOWHEIGHT}
Let $X$ be as in \ref{DEFCHOW}.
If $\overline{\mathcal{E}} = (\mathcal{E}, h)$ is a line bundle with metric in $\mathcal{M}$ and $Y$ is a cycle on $X$ such that $Y_\Q$ intersects $D_\Q$ properly, we define the {\bf height} with respect to $\overline{\mathcal{E}}$ of $Y$ as follows:
Choose a Greens form of logarithmic type for $Y_\C$ (exists by \cite[Theorem 3, p.44]{SABK}) and a smooth metric $\widetilde{h}$ on $\mathcal{E}$.
\begin{eqnarray*} 
\hght_{\overline{\mathcal{E}}} (Y) &:=& \adeg \pi_* ( (Y, g) \cdot \achern_1(E,\widetilde{h})^{n+1-d})  -\frac{1}{2} \int_{X} g \chern_1(E,\widetilde{h})^{n+1-d} + \frac{1}{2}\int_{Y} \phi \\
&=& \hght_{(\mathcal{E},\widetilde{h}), GS} (Y) + \frac{1}{2} \int_{Y} \phi,
\end{eqnarray*}
where $\phi \in \widetilde{A}^{n-d,n-d}(X,D)$ is such that $a(\phi) = \achern_1(E,h)^{n+1-d} - \achern_1(E, \widetilde{h})^{n+1-d}$.
The integral over $Y$ exists because of property (F2) (consider a desingularization of $Y$). Here it is crucial that $Y$ and $D$ intersect properly. 
\end{DEF}
\begin{LEMMA}
This is well defined.
\end{LEMMA}
\begin{proof}The first line is independent of $g$ by means of the definition of star product \cite[Definition 4, p.50]{SABK}. We have to see that it is independent of
the choice of smooth metric $\widetilde{h}$. Let $\widetilde{h}_i$, $i=1,2$ be two different ones. Let $\phi_i$ be such that $a(\phi_i) = \achern_1(\mathcal{E}, h)^{n+1-d} - \achern_1(\mathcal{E},\widetilde{h}_i)^{n+1-d}$. Note that
$\phi_2-\phi_1$ is smooth.
We have
\begin{gather*}
\adeg \pi_* \left( (Y, g) \cdot \achern_1(\mathcal{E},\widetilde{h}_2)^{n+1-d} -  (Y, g) \cdot \achern_1(\mathcal{E},\widetilde{h}_1)^{n+1-d}\right) = \frac{1}{2} \int_{X(\C)} (\phi_1-\phi_2) \ddc g \\
 -\frac{1}{2}\int_{X(\C)} g \chern_1(E,\widetilde{h}_2)^{n+1-d} + \frac{1}{2} \int_{X(\C)} g \chern_1(E,\widetilde{h}_1)^{n+1-d} = \frac{1}{2}  \int_{X(\C)} g \ddc (\phi_2-\phi_1) \\
 \frac{1}{2} \int_{Y(\C)} \phi_2 - \frac{1}{2} \int_{Y(\C)} \phi_1 =  \frac{1}{2} \int_{X(\C)}  (\phi_2-\phi_1) \ddc g - \frac{1}{2} \int_{X(\C)} g \ddc(\phi_2-\phi_1)
\end{gather*}
so the sum over these differences is zero.
\end{proof}

From (F1) and (F2) together with (F5) the following follows formally --- using an embedded resolution of singularities (see the proof of \cite[Theorem 7.33]{BKK1} with $\sigma_{YZ}=0$):
\[ \int_X g \ddc g_2 = \int_Y g_2 + \int_X (\ddc g) g_2 \]
for any form $g_2$ in our class. 
Hence, if (F5) holds, we could simply define
\[ \hght_{\overline{\mathcal{E}}} (Y) := \adeg ( \pi_* ((Y, g) \cdot \achern_1(\mathcal{E},h)^{n+1-d})  - \frac{1}{2}\int_{X(\C)} g  \chern_1(E,h)^{n+1-d}  \]
which is literally the same definition as $\hght_{\overline{\mathcal{E}},GS}$.

\begin{PAR}\label{HEIGHTPROJECTION}
In any case, if there is a desingularization $(\widetilde{Y},\widetilde{D})$ of $(Y,Y\cap D)$ which is 
itself a pair as considered in \ref{DEFCHOW}, we get from the projection formula \cite[Theorem 3 (iii), p. 64]{SABK} that
\[ \hght_{\overline{\mathcal{E}}} (Y) =  \adeg \pi_* (f^*\achern_1(\mathcal{E},h))^{n+1-d}   \]
for the corresponding morphism $f: (\widetilde{Y},\widetilde{D}) \rightarrow (X,D)$.
\end{PAR}

\begin{PAR}We will apply the constructions above to (pure) Shimura varieties now.
Assume first that all occurring $K$'s are neat, see \ref{ARITHMVOLNOTNEAT} for the general case.
For a pure Shimura variety (considered as defined over $\Zpp$ as in \ref{INPUTDATAAUTOMORPHICVB}) $\nSh({}^K_\nRPCD \nSD)$ 
with complete, smooth w.r.t. $K$ and projective $\nRPCD$, the exceptional divisor $D$ has normal crossings (\ref{MAINTHEOREM1a}) and
we define $\aCH^\bullet(\nSh({}^K_\nRPCD \nSD))$ as $\aCH^\bullet(\nSh({}^K_\nRPCD \nSD), D)$ as defined in \ref{DEFCHOW} w.r.t. any class of forms and metrics satisfying (F1--4), such that the metrics
of fully decomposed automorphic vector bundles are in the class (exists by Theorem \ref{BURGOS}).
We assume that $\nSh({}^K_\nRPCD \nSD)$ exists --- 
this requires that $\nRPCD$ is sufficiently fine and that Conjecture \ref{MAINCONJECTURE} is true (cf. \ref{MAINTHEOREM1a}).
\end{PAR}

\begin{DEF}\label{DEFVOLUME}
Let ${}^K_\nRPCD \nSD$ be pure $p$-ECMSD with complete, smooth w.r.t. $K$ and projective $\nRPCD$,
and $\overline{\mathcal{E}} = (\mathcal{E}, h_E)$ be an {\em invertible} sheaf $\mathcal{E}$ on
$\left[ \nShD(\nSD) / \nP_\nSD \right]$ (stack over $\Zpp$!) and $h_E=\{h_{E,\sigma}\}$ a
collection of invariant Hermitian metrics as in \ref{INPUTDATAAUTOMORPHICVB}.

The \textbf{geometric volume}
\[ \vol_{E} (\nSh({}^K_\nRPCD \nSD )) \]
is defined as the volume of $\nSh({}^K_\nRPCD \nSD )(\C)$
with respect to the volume form $(\chern_1(\nMorphSt^* \overline{E}_\C))^d$. 

The \textbf{arithmetic volume} at $p$
\[ \widehat{\vol}_{\overline{\mathcal{E}}, p} (\nSh({}^K_\nRPCD \nSD)) \]
is defined as
\[ \pi_* (\achern_1(\nMorphSt^*\mathcal{E}, \nMorphSt^*h_E)^d) \] 
in $\aCH^1(\spec(\Zpp)) = \R^{(p)}$, where $\pi: \nSh({}^K_\nRPCD \nSD) \rightarrow \spec(\Zpp)$ is the structural morphism

If $\nP_\nSD$ is a group scheme of type (P) over $\spec(\Z[1/N])$
and $K$ is admissible {\em for all} $p \nmid N$, then we define the
global \textbf{arithmetic volume}
\[ \widehat{\vol}_{\overline{\mathcal{E}}} (\nSh({}^K_\nRPCD \nSD)) \]
in $\R_N$ as the value determined by the arithmetic volumes at $p$ (compare \ref{RN}, see also \ref{RMK}).

If $(\alpha,\rho): {}^{K'}_{\nRPCD'} \nSDi_\Q \hookrightarrow {}^K_\nRPCD \nSD_\Q$
is an embedding of {\em rational} pure Shimura data, we define
the \textbf{height} at $p$:
\[ \hght_{\overline{\mathcal{E}}, p} (\nSh({}^{K'}_{\nRPCD'} \nSDi)) 
=  \hght_{\nMorphSt^* \overline{\mathcal{E}}}( \overline{\nSh(\alpha, \rho)(\nSh({}^{K'}_{\nRPCD'}\nSDi_\Q))}^{Zar}),  \]
where $\hght_{\nMorphSt^* \overline{\mathcal{E}}}$ is the height (\ref{CHOWHEIGHT}).
Similarly for the (global) height.
\end{DEF}

\begin{FRAGE}
We mention that there is a \textbf{proportionality principle} \cite{Mumford1} in the geometric case. 
This means that {\em all} polynomials of degree $\dim(\nSh({}^K_\nRPCD \nSD)_\C)$ in
the Chern classes of an automorphic vector bundle $(\nMorphSt^*\mathcal{E})_\C$, considered as a number, are proportional to the same expression in the
Chern classes of the original bundle $\mathcal{E}$, computed on $\nShD(\nSD)$.

Is there an analogue in the arithmetic case? --- cf. also \cite{Koehler}.
\end{FRAGE}

\begin{BEM}
If $(\alpha,\rho)$ was in fact an embedding of {\em $p$-integral} data:
${}^{K'}_{\nRPCD'} \nSDi \hookrightarrow {}^{K}_{\nRPCD} \nSD$,
we have, by the projection formula (cf. \ref{HEIGHTPROJECTION}) and \ref{FUNCTORIALITYAUTOMVECTB}:
\[ \hght_{\overline{\mathcal{E}},p} (\nSh({}^{K'}_{\nRPCD'} \nSDi))
= \widehat{\vol}_{\nShD(\alpha)^*\overline{\mathcal{E}},p} (\nSh({}^{K'}_{\nRPCD'} \nSDi)). \]
\end{BEM}

\begin{LEMMA}\label{LEMMAINDEPRPCD}
The geometric and arithmetic volumes do not depend on the rational polyhedral cone decomposition $\nRPCD$, i.e. not on
the chosen toroidal compactification.
\end{LEMMA}
\begin{proof}
For each pair $\nRPCD_i$, $i=1,2$, there is a common refinement $\nRPCD$, cf. \ref{PROPERTIESRPCD}, and we have two projections (\ref{MAINTHEOREM1}):
\[ \xymatrix{
 & \nSh({}^K_{\nRPCD} \nSD) \ar[dl] \ar[dr] & \\
\nSh({}^K_{\nRPCD_1} \nSD) & & \nSh({}^{K}_{\nRPCD_2} \nSD)
}\]
Furthermore this diagram is compatible with formation of $\nMorphSt^*\mathcal{E}, \nMorphSt^*h_E$ by \ref{FUNCTORIALITYAUTOMVECTB}.
Its Chern forms are therefore transported into each other by pullback along the arrows in the diagram. 
Since the forms are integrable on every $\nSh({}^{K}_{\nRPCD} \nSD)$ (see \cite{Mumford1}, cf. also \ref{MUMFORD}), the geometric volume agrees.

Now, there is pullback map 
\[ p^*: \aCH^\bullet(\nSh({}^K_{\nRPCD_i}\nSD))_\Q \rightarrow \aCH^\bullet(\nSh({}^K_{\nRPCD}\nSD))_\Q \]
as well, which is a ring homomorphism and compatible with Chern classes (see (F4,b) and \cite[Theorem 3, p. 84]{SABK}).

Therefore the assertion for the arithmetic volume boils down to the fact that
for a class $x = (z,g) \in \aCH^{n+1}_{GS}({}^K_{\nRPCD} \nSD)$ or $x=a(g')$, where $z$ is a zero-cycle in the fibre above $p$ and $g$ is a smooth
top-degree Greens form and $g'$ is a top-degree form on $\nSh({}^K \nSD)(\C)$ which is integrable on the whole $\nSh({}^K_\nRPCD \nSD)(\C)$, we have $\widehat{\deg}(\pi_* p^*(x)) = \widehat{\deg}(\pi_* x)$, where the $\pi$ are the respective structural morphisms.
This is true because the push-forwards of $g$, $g'$ are
computed by an integral. 
\end{proof}

\begin{BEM}\label{ARITHMVOLNOTNEAT}
If $K$ is not neat, $\nSh$ is only a Deligne-Mumford stack. Instead of extending the theory of \cite{BKK1, BKK2} to stacks like e.g. in \cite{Gillet}, we define the arithmetic
volume in this case as follows: We take some neat admissible $K' \subset K$ and define
\[ \widehat{\vol}_{\overline{\mathcal{E}},p} (\nSh({}^{K}_\nRPCD \nSD))
:= \frac{1}{[K:K']} \widehat{\vol}_{\overline{\mathcal{E}},p} (\nSh({}^{K'}_\nRPCD \nSD)) \]
as value in $\R^{(p)}$. Similarly for the global arithmetic volume in $\R_N$, assuming that both $K$ and $K'$ are admissible for all $p \nmid N$.
It follows from \ref{FUNCTORIALITYAUTOMVECTB}, the definition of pullback, and the same reasoning as in 
Lemma \ref{LEMMAINDEPRPCD}.
that this is independent of the choice of $K'$. In particular, if $K$ is neat --- or using any reasonable extension of Arakelov theory to stacks --- this definition agrees with the previous one.

In this sense the geometric and arithmetic volume do not depend essentially on $K$, as they are both multiplied by the index, if $K$ is changed. For the arithmetic case this is of course only true for admissible $K$, and the more primes
are considered the less admissible groups there are.
\end{BEM}

\begin{BEM}\label{RMK}
Since the canonical models can be defined as the normalization of a Zariski closure in a variety 
defined over some $\Z[1/N]$ (compactification of the moduli space of Abelian varieties), it follows that there is even a
model of $\nSh({}^{K}_\nRPCD \nSD)$ defined over some $\OOO_\nSD[1/N]$, yielding the canonical ones over the various reflex rings $\OOO_{\nSD,(p)}$ for $p|N$,
such that the arithmetic volume of this model is the value determined above. This justifies a priori that
our value lies actually in $\R_N$. Otherwise we will not care about this global model.
\end{BEM}

\section{The arithmetic of Borcherds products}

\begin{PAR}\label{WEILFORMULAS2}
Let $\nL$ be a f.d. $\Q$-vector space with non-degenerate quadratic form $Q_\nL \in \Sym^2(\nL^*)$
and $\nM$ be any f.d. $\Q$-vector space. Denote $\mathfrak{\nM} = \nM^* \oplus \nM$ with
standard symplectic form. Consider the following elements of $\Sp(\mathfrak{\nM})$:
\begin{align*}
 g_l(\alpha) &:= \left(\begin{matrix}\alpha &0\\0&^t\alpha^{-1} \end{matrix}\right)  \\
 u(\beta) &:= \left(\begin{matrix}1 & \beta \\ 0 & 1 \end{matrix}\right) \\
 d(\gamma) &:= \left(\begin{matrix}0 & -^t \gamma^{-1} \\ \gamma & 0 \end{matrix}\right) 
\end{align*}
for $\alpha \in \Aut(\nM)$, $\beta \in \Hom(\nM^*, \nM)$ and $\gamma \in \Iso(\nM, \nM^*)$.
We denote the image of $g_l$ and $u$ by $G_l$ and $U$ respectively, and their product by $P$, which is a maximal parabolic of $\Sp(\mathfrak{\nM})$.

Let $R$ be one of $\R, \Q_p, \Af$ or $\A$ and $\chi$ its standard additive character (depends on a choice of $i \in \C$).
Let $\Sp'(\mathfrak{\nM}_R)$ be $\Sp(\mathfrak{\nM}_R)$ if $m$ is even and
 $\Mp(\mathfrak{\nM}_R)$ (metaplectic 
double cover) if $m$ is odd. The elements above have lifts to $\Sp'(\mathfrak{\nM}_R)$, denoted $r''(\cdots)$ 
such that the {\bf Weil representation} of $\Sp'(\mathfrak{\nM}_R)$ on $S((\nM^* \otimes \nL)_R)$ (Schwartz functions) is given by:
\begin{align*}
\omega(r''(g_l(\alpha)))\varphi: &x^* \mapsto \widetilde{\WeilGamma}(\gamma_0\alpha)\widetilde{\WeilGamma}(\gamma_0)^{-1} |\alpha|^{\frac{m}{2}} \varphi(^t \alpha x^*) \\
\omega(r''(u(\beta)))\varphi: &x^* \mapsto \chi((x^*)^!Q_L \cdot \beta) \varphi(x^*) \\
\omega(r''(d(\gamma)))\varphi: &x^* \mapsto \widetilde{\WeilGamma}(\gamma) |\gamma|^{-\frac{m}{2}} \int_{\nM \otimes \nL^*} \varphi({}^t\gamma x) \chi(-\langle x^*, x \rangle) \dd x.
\end{align*}
$\widetilde{\WeilGamma}$ is a certain eigth root of unity. 
For $n=1$ we have just $\widetilde{\WeilGamma}(\gamma) = \WeilGamma(\gamma \otimes Q_L)$, where $\WeilGamma$ is the Weil index \cite{Weil2}. In the third formula $|\gamma|$ is computed w.r.t. the measure $\dd x$, and so it is independent of this measure.

If a global lattice $\nL_\Z \subset \nL_\Q$ is given, with $Q_\nL \in \Sym^2(\nL_\Z^*)$ and $\nM_\Z \subset \nM_\Q$,
as in Borcherds' paper, we often restrict to the subspace of $S((\nM^* \otimes \nL)_\Af)$ given by
characteristic functions of cosets in $(\nL_\Zh^*/\nL_\Zh) \otimes \nM_\Z^*$. 
We denote the space --- as in \cite{Borcherds1} --- by $\C[(\nL_\Z^*/\nL_\Z) \otimes \nM_\Z^*]$.
\end{PAR}

\begin{PAR}\label{EXPLICITWEIL}
In the $n=1$ case ($\nM=\Z$) we denote it also by $\Weil(\nL_\Z^*/\nL_\Z)$, and we consider it as a representation $\Sp'_2(\Z)$, $\Sp'_2(\Zh)$ or a suitable $\Sp'_2(\Z/N\Z)$. 
$\Sp'_2(\Z)$ is generated by 2 elements $S$ and $T$, satisfying $Z=S^2=(ST)^3$ and ($Z^4=1$ or $Z^2=1$ according to whether $m$ is odd or even), mapping down to $\M{0&-1\\1&0}$ and $\M{1&1\\0&1}$, respectively, in $\Sp_2(\Z)$
and acting explicitly as:
\begin{align*}
T \chi_\kappa &= \exp(2\pi i Q_L(\kappa)) \chi_\kappa, \\
S \chi_\kappa &= \frac{\WeilGamma_\infty(Q_\nL)^{-1} } {\sqrt{D}} \sum_{\delta \in \nL_\Z^*/\nL_\Z} \exp(- 2\pi i \langle \delta, \kappa \rangle) \chi_\delta.
\end{align*}
Note that the occurring correction factor $\widetilde{\WeilGamma}_f(1) = \WeilGamma_f(Q_\nL)$ 
is the product over the local $\WeilGamma_p(Q_\nL)$'s which is, however, the
same as $\WeilGamma_\infty(Q_\nL)^{-1}$ and we have $\WeilGamma_\infty(Q_\nL) = \exp(2\pi i (p-q))$, if the signature of $Q$ is $(p,q)$.
These are the formulas used in \cite{Borcherds1}.
We will need the following facts about the Weil representation, too:
The Weil representation of $\Sp'_2(\Zh)$, or $\Sp'_2(\Z)$, always factors through $\Sp_2'(\Z/N\Z)$, where $N$ is related to the exponent of $\nL_\Z^*/\nL_\Z$.
If $m$ is even, the Borel subgroup $\M{*&*\\ & *} \subset \Sp_2'(\Z/N\Z)$ acts on 
$\chi_{\nL_\Zh}$ by the {\em character} $\WeilGamma(\alpha Q_\nL)\WeilGamma(Q_\nL)^{-1}$.

By construction of the Weil representation, we have
\begin{equation}\label{weildecomp} \Weil(\nL_\Z^*/\nL_\Z) = \bigotimes_{p|D} \Weil(\nL_\Zp^*/\nL_\Zp), \end{equation}
where $\Weil(\nL_\Zp^*/\nL_\Zp)$ is actually a representation of $\Sp_2'(\Zp)$.
If $\nL_\Zp^*/\nL_\Zp$ is cyclic of prime order or of order $4$, we have
\[ \Weil(\nL_\Zp^*/\nL_\Zp) = \Weil(\nL_\Zp^*/\nL_\Zp)^+ \oplus \Weil(\nL_\Zp^*/\nL_\Zp)^- \]
(decompositions into odd and even functions), and both summands are {\em irreducible}. (If $p=2$ and the order is 2, $\Weil(\nL_\Zp^*/\nL_\Zp)^- = 0$.)
Together with (\ref{weildecomp}) this yields the decomposition of the Weil representation of $\nL$ into irreducible representations, if $D$ is square-free, 
resp. $\frac{D}{2}$ square-free and $\nL_{\Z_2}^*/\nL_{\Z_2}$ cyclic of order 4.
\end{PAR}

\begin{PAR}\label{BORCHERDSLIFT}
Let $\nL_\Q$ be a non-degenerate quadratic space of signature $(p,q)$.
Consider a one-dimensional space $\nM_\Q = \Q$. We canonically identify $\nM^* \otimes \nL = \nL$.

Let $N$ be a maximal negative definite subspace of $\nL_\R$.
We have the Gaussian $\varphi_\infty^0 \in S(\nL_\R) \otimes \mathcal{C}(\nX_\nO)$, 
defined by
\[ \varphi_\infty^0(x, N) := e^{- \pi \langle x, \alpha_N x \rangle }, \]
where $\alpha_N \in \End(M)$ acts as -1 on $N$ and as +1 on the orthogonal complement.
It satisfies
\[ \varphi_\infty^0(hx, hN) = \varphi_\infty^0(x, N) \qquad \forall h \in \SO(\nL_\R). \]

For any $\varphi \in S(\nL_\Af)$, define the theta function
\[ \Theta(g', N, h; \varphi) = \sum_{v \in \nL_\Q} \sum \omega(g')(\varphi_\infty(\cdot, N) \otimes \omega(h) \varphi)(x) \]
as a function of $g' \in \Sp'(\mathfrak{\nM}_\A), N \in \nX_\nO, h \in \SO(\nL_\Af)$.
Here $\omega$ is the Weil representation (\ref{WEILFORMULAS2}).

We have (cf. \cite{Weil1}) for any $\gamma \in \SO(\nL_\Q)$ and $\gamma' \in \Sp'(\mathfrak{\nM}_\Q)$ (canonical lift of $\Sp(\mathfrak{\nM}_\Q)$ to $\Sp'(\mathfrak{\nM}_\A)$),
\[ \Theta(\gamma' g', \gamma N, \gamma h; \varphi) = \Theta(g', N, h; \varphi) \]
and for $g_1' \in \Sp'(\mathfrak{\nM}_\Af)$ and $h_1 \in \SO(\nL_\Af)$:
\[ \Theta( g' g_1', N, h h_1; \varphi) = \Theta(g', N, h; \omega(g_1') \omega(h_1) \varphi). \]
By these invariance properties, we may consider $\Theta$ as a function
\[ \Sp'(\mathfrak{\nM}_\Q) \backslash \Sp'(\mathfrak{\nM}_\Af) \times \nSh({}^K \nO(\nL) )(\C) \rightarrow \left(S(\nL_\Af)^K\right)^*. \]
We have (cf. \cite[p. 10]{Kudla4}):
\[ \omega(k_\infty') \varphi_\infty^0(\cdot, N) = \chi_l(k_\infty') \varphi_\infty^0(\cdot, N) \]
for $l=\frac{p-q}{2}$ and $k_\infty \in K'_\infty$ (inverse image of $K_\infty=\SO(2)$ in $\Sp'_2(\R)$). 
It follows that
\[ \Theta(g' k_\infty' k', N, h) = \chi_l(k_\infty') ({^t \omega(k')})^{-1} \Theta(g', N, h) \]
for all $k_\infty' \in K'_\infty$ and $k' \in K'$ (inverse image of $\Sp_2(\Zh)$ in $\Sp'_2(\Af)$).
Now consider an automorphic function
\[ F: \Sp(\mathfrak{\nM}_\Q) \backslash \Sp'(\mathfrak{\nM}_\Af) \rightarrow S(\nL_\Af), \]
satisfying
\[ F(g'k'_\infty k') = \chi_{-l}(k'_\infty) \omega(k')^{-1} F(g') \]
and invariant under some compact open subgroup $K \subset \SO(\Af)$. 

Define the \textbf{Borcherds lift} of $F$ as
\[ \Phi(N, h; F) = \int_{\Sp_2'(\Q) \backslash \Sp_2'(\A)}^\bullet \Theta(g', N, h; F(g')) \dd g', \]
where $\bullet$ denotes Borcherds' regularization \cite{Borcherds1}, described in \cite{Kudla4} in the adelic language.
$\dd g'$ is the Tamagawa measure. Note that this formation does not depend on $K$.
By definition, we have
\begin{equation} \label{transfborcherds}
\Phi(N, h\rho; F) = \Phi(N, h; \omega(\rho)F).
\end{equation}
It compares to the language of Borcherds as follows:
If $K \subset \SO(\Af)$ is the discriminant kernel with respect to a lattice $\nL_\Z$, we may consider
$\Theta$ as valued in $\C[(\nL_\Z^*/\nL_\Z)]^*$ and $F$, valued in $\C[(\nL_\Z^*/\nL_\Z)]$, relates to a usual
modular form (if $l$ is odd, of half integer weight) by
\[ y^{l/2} F(g'_\tau) = f(\tau) \qquad \tau = x + iy \in \HH. \]
\end{PAR}

\begin{LEMMA}
We have
\[ \Phi(\iota(N), 1; F) = \Phi_{classic}(\iota, \nL_\Z, N, f) \]
for $N \in \Grass^-(\nL_\R)$ and according to $\iota: \Grass^-(\nL_\R) \hookrightarrow \nX_\nO$. $\iota$ is a parameter in Borcherds' theory (cf. \cite[p. 46]{Borcherds1}).
\end{LEMMA}
\begin{proof}
In \ref{EXPLICITWEIL}, we saw that the Weil representation composed with the inclusion $\Sp_2'(\Z) \rightarrow \Sp_2'(\Zh)$ acts on the subspace
$\C[\nL_\Z^*/\nL_\Z] \subset S(\nL_\Af)$ by the same formulas used in \cite{Borcherds1}.
\end{proof}
We call $f$, resp. $F$, nearly holomorphic, if $F$ is holomorphic on $\HH$ and has a Fourier expansion
\[ f(\tau) =  \sum_{k \in \Q} a_k \exp(2\pi i k \tau)   \]
with $a_k=0$ unless $k \gg -\infty$.

\begin{PAR}\label{DEFLAMBDAMU}
We recall some definitions from \cite{Paper1}. Let $\nL_\Z$ be a lattice of dimension $m$ with
non-degenerate quadratic form in $\Sym^2(\nL_\Z^*)$. 
Let $\nM_\Z$ be a lattice of dimension $n$ with quadratic form, not necessarily integer valued. We assume $m-n \ge 1$ throughout.  
(It suffices even to have $\Zh$-lattices.) 

Let $\kappa$ be a coset in $(\nL_\Z^*/\nL_\Z) \otimes \nM_\Z^*$.

We defined the following Euler products \cite[10.2, 10.3]{Paper1}:
\begin{eqnarray*}
 \lambda(\nL; s) &:=& \prod_{\nu} \lambda_\nu(\nL; s), \\
 \mu(\nL, \nM, \kappa; s) &:=& \prod_{\nu} \mu_\nu(\nL, \nM, \kappa; s), 
\end{eqnarray*}
where \cite[Defintion 5.9]{Paper1}:
\begin{eqnarray*}
\lambda_p(\nL_\Z; s) &:=& \frac{\vol(\SO'(\nL_\Zp \oplus H_\Zp^s))}{\prod_{i=1}^{s}(1-p^{-2i})}, \\
\mu_p(\nL_\Z, \nM_\Z, \kappa; s) &:=& \vol(\nIsome(\nM, \nL \oplus H_\Zp^s)(\Qp) \cap \kappa \oplus H_\Zp^s\otimes \nM_\Zp^*) \\
&=& |d(\nL_\Zp)|^{\frac{n}{2}} |d(\nM_\Zp)|^{-s+\frac{1+n-m}{2}}  \beta_p(\nL_\Zp, \nM_\Zp, \kappa; s),
\end{eqnarray*}
and \cite[Defintion 6.1]{Paper1}:
\begin{eqnarray*}
\lambda_\infty(\nL_\Z) &:=&  \Gamma_{m-1,m}(s) \\
\mu_\infty(\nL_\Z, \nM_\Z, \kappa) &:=& \Gamma_{n,m}(s).  
\end{eqnarray*}

The $\mu$ functions are related to Whittaker integrals occurring as Euler factors of Eisenstein series. The
coefficient $\mu_p$ is equal to the corresponding $p$-factor (\cite[Theorem 4.5]{Paper1}), whereas the infinite factor is only indirectly related to $\mu_\infty$
(\cite[Theorem 4.7]{Paper1}). In the case $n=1$ it occurs as a limit of the Archimedean factor up to a factor $-|2 d(\nM_\Z)|^{s-s_0}$, where $d(\nM_\Z)$ is
the discriminant of the associated bilinear form (a number in this case) on $\nM_\Z$. It is therefore convenient to define \cite[Section 10]{Paper1}:
\[ \widetilde{\mu}(\nL_\Z, \nM_\Z, \kappa; s) = |2 d(\nM_\Z)|_\infty^{-\frac{1}{2}s} \mu(\nL_\Z, \nM_\Z, \kappa; s). \]
because the derivative in $s$ of this quantity is actually is related to the integral of a Borcherds form (cf. Theorem \ref{BPADELIC}).

We define also:
\[ \widetilde{\lambda}(\nL; s) := |d(\nL_\Z)|_\infty^{\frac{1}{2}s} \lambda(\nL_\Z; s). \] 

The functions $\mu_p$ and $\lambda_p$ can be computed explicitly. See \cite[Theorem 8.1, 8.2]{Paper1}. For explicit computations 
of $\widetilde{\mu}$ and $\widetilde{\lambda}$ from this, see \cite[Example 11.10--11.15]{Paper1}.
We will continue the discussion of these functions in \ref{DISCORBIT}.
\end{PAR}

\begin{PAR}
Assume from now on that the signature of $\nL$ is $(m-2,2)$.
Borcherds proved in \cite[Theorem 13.3]{Borcherds1} that $\Phi(F)$ is roughly the logarithm of
the Hermitian norm of a modular form on $\nSh({}^K_\nRPCD \nO(\nL))$, which has a striking product expansion. The task of the rest of this section is a translation of this expansion into the adelic language.
\end{PAR}
\begin{DEF}\label{BORCHERDSBUNDLE}
Let $\mathcal{E}$ be the restriction of
the tautological line bundle on $\PP(\nL)$ to the zero-quadric $\nShD(\nO)$.
It carries a $\SO(\nL)$ action. $E_\C|_{h(\nX_\nO)}$ carries a natural Hermitian metric $h_E$, too,
which we normalize as follows:
\[ h_E(v) := - \frac{1}{2} e^{-C} \langle v, \overline{v} \rangle, \]
where $C=\gamma + \log(2\pi)$ and $\gamma = -\Gamma'(1)$ is Euler's constant.

Denote $\overline{\mathcal{E}} = (\mathcal{E}, h_E)$. 
\end{DEF}
\begin{PAR}
Let $\nI_\Q$ be an isotropic line in $\nL_\Q$ and $z' \in \nL_\Q$ a vector with $\langle z', \nI \rangle \not=0$.
We fix a measure $\dd x$ on $\Af$ which has the property, 
that the measure of the set
\[ \{ \langle \delta, z'  \rangle \ |\ \delta \in \nI_\Af, a_k(x+\delta) = a_k(x) \text{ for all } k \in \Q \} \]
is an integer for all $x \in \nL_\Af$. Since $f$ is a finite linear combination of input modular forms considered
by Borcherds, whose Fourier coefficients lies in the subspace $\C[\nL_\Z^*/\nL_\Z]$ for certain lattices $\nL_\Z$, this requirement does indeed make sense. We let $\chi_\alpha: \Af \rightarrow \C^*$, $\alpha \in \Iso(\Z, \Z(1))$, be the corresponding additive character, i.e. such that $\dd \delta$ is self-dual with respect to it, and such that $\chi_\alpha |_\Q = \exp( \alpha(q \cdot))$ for a $q>0$.

We let $\nU := (\nI^\perp /\nI) \otimes \nI \subseteq \Lambda^2 \nL$. It acts (Lie algebra action, see proof of Proposition \ref{ORTHBOUNDARYCOMP}) on $\nL$, and we have a (rational) boundary morphism
\[ \nH_0[\nU_\Q, 0] \Longrightarrow \nO(\nL_\Q). \]
We may identify $\nU^*$ with $(\nI^\perp/\nI) \otimes \nI^*$ by means of $Q_\nL$.

$z'$ induces the following quadratic form on $\nU$:
\[ Q_{z'}(k \wedge z_1) =  Q_\nL((k \wedge z_1) \circ z') = (\langle z_1, z' \rangle)^2 Q_\nL(k). \]
Here $\circ$ denotes the Lie algebra action of $\nU$ on $\nL$.
Its inverse on $\nU^*$ is:
\[ Q_{z'}(k \wedge z_1') = \left(\frac{\langle z,z_1' \rangle}{\langle z,z' \rangle}\right)^2 Q_\nL(k) \]
for an arbitrary  $z \in \nI$.

We may also embed $\nU^*$ in $\nI^\perp$ by associating
\begin{equation}\label{eqlambdazprime}
  \lambda = k \wedge z_1' \mapsto \lambda_{z'} := \frac{\langle z,z_1' \rangle}{\langle z,z' \rangle}  \widetilde{k}, 
\end{equation}
where $\widetilde{k} \in \nI^\perp$ is the unique lift of $k$ with $\langle z',\widetilde{k}\rangle =0 $.
The embedding may be seen as the composition
\[ \xymatrix{ \nU^* \ar[r]^-{\langle \cdot, \cdot \rangle_{z'}} & \nU \ar[r]^-{\circ z'}  & <\nI, z'>^\perp }.  \]
The pullback of $Q_\nL$ is the quadratic form $Q_{z'}$.
\end{PAR}

\begin{DEFLEMMA}\label{DEFWEYLVECTOR}
If $\sigma \subset \nU_\R$ is a sufficiently small cone, for example a Weyl chamber \cite[\S 10]{Borcherds1},
the Weyl vector $\rho(F, \nI, \sigma) \in \nU^*_\Q$ is characterized by
\[ 8 \pi \rho(F, \nI, \sigma) ( \nu )  = \sqrt{Q_{z'}(\nu)} \Phi(<\nu>, 1, F^{U,z'})
\qquad \text{for } \nu \in \sigma \subset U_\R, \]
where $F^{U,z'}$ as a $S(\nU_\Af)$-valued function is defined as:
\[ F^{U, z'}(g')(w) := C^{-1} \int_{\delta \in \nI_\Af} F(g')(w \circ z' + \delta) \dd \langle z', \delta \rangle \]
(any invariant measure). It does not depend on the choice of $z'$.
\end{DEFLEMMA}
\begin{proof}
Compare \cite[p. 40]{Borcherds1} for this characterization of the Weyl vector. 
First it is independent of $z'$:
Scaling $z'$ to $\beta z'$ changes the $w \circ z'$ to $\beta w \circ z'$.
$F^{U,z'}(w,z')$ changes to $\beta^{-1} F^{U,z'}(\beta w, g')$.
Since the argument of the quadratic form on $U_\Q$ and the argument in $K_\Af$ of $F^{U,z'}$ are both scaled by $\beta$,
$\Psi(\dots)$ is only scaled by $\beta^{-1}$ coming from scaling $F^{U,z'}$, hence invariance of the Weyl vector. (Observe that $\sigma$ is invariant under scaling.)
Replacing $z'$ by $\exp(w)z'$ for $w \in \nU$ does neither change $F^{U, z'}$ nor $Q_{z'}$. 
\end{proof}

\begin{DEFLEMMA}\label{DEFBP}
Let $F$ be a modular form as above, $\nI$, $z'$, $\alpha$ be chosen and $W \subset \nU_\R$ be a Weyl chamber. 
The associated adelic Borcherds product 
is the formal function
\begin{gather*} \BP(F, \nI, z', \alpha, W) :=  C^{-\frac{a_0(0)}{2}} \prod_{\substack{\delta \in \nI_\Af \\ \chi_\alpha(\langle z', \delta \rangle)\not=1}} (1-\chi_\alpha(\langle z', \delta \rangle))^{\frac{a_0(\delta)}{2}\dd \langle z' , \delta \rangle}  \\
 \cdot
 \prod_{\substack{\lambda \in U_\Q^*\\ \langle \lambda, W \rangle>0}} \prod_{\delta \in \nI_\Af} (1-\chi_\alpha(\langle z', \delta\rangle)[\frac{\lambda}{C}])^{a_{Q_{z'}(\lambda)}(\lambda_{z'} + \delta) \dd \langle z', \delta\rangle}
\end{gather*}
in $\C\llbracket \nU^*_\Z \cap \nC^\vee \rrbracket [\nU^*_\Z]$ for a certain appropriate lattice $\nU_\Z \subseteq \nU_\Q$.  Here $C$ is the conductor, i.e. the volume of $\widehat{\Z}$ with respect to $\dd x$.

$\BP(F, \nI, z', \alpha, W)$ is independent of the chosen measure.
\end{DEFLEMMA}
\begin{proof}[Proof of the independence.]
Rewrite the first integral as
\[ \left. \frac{\prod_{\delta \in \nI_\Af} (1-\chi_\alpha(\langle z', \delta\rangle)[C^{-1} \lambda])^{\frac{a_0(\delta)}{2}\dd \langle z', \delta \rangle}}{(1-[C^{-1} \lambda])^{\frac{a_0(0)}{2}}} \right|_{\lambda=0}. \]

Now consider the multiplicative integral
\begin{equation}\label{multint1}
\prod_{\delta \in (z')^{-1}(C\Zh) + a} (1 - \chi_\alpha(\langle z', \delta \rangle)[\lambda C^{-1}])^{\dd \langle z', \delta \rangle} = (1 - \chi_\alpha(\langle z', a \rangle)[C^{-1}\lambda]).
\end{equation}
All multiplicative integrals occurring in $\BP(F, I, z', \alpha, W)$ split into a
product of these integrals by hypothesis on the measure.
Now changing the measure to $N \dd \delta$, we get the expression
\begin{align*}
&  \prod_{\delta \in (z')^{-1}(C\Zh)} (1 - \chi_\alpha(N^{-1} \langle z', a \rangle) \chi_\alpha( N^{-1} \langle z', \delta \rangle)[ (NC)^{-1} \lambda ])^{\dd N^{-1} \langle z', \delta \rangle}  \\
=& \prod_{i \in \Z/N\Z} (1 - \chi_\alpha(N^{-1} \langle z', a \rangle) \zeta^i [(NC)^{-1} \lambda ])  \\
=& (1-\chi_\alpha(\langle z', a \rangle)[ C^{-1} \lambda])
\end{align*}
here $\zeta$ is some primitive $N$-th root of unity in $\C$.
This shows invariance of the multiplicative integral (\ref{multint1}). We have to check what happens to
\[ \frac{1}{(1-[\lambda C^{-1}])^{\frac{a_0(0)}{2}}}. \]
It changes to $\frac{1}{(1-[\lambda (NC)^{-1}])^{\frac{a_0(0)}{2}}}$ hence gets multiplied by
\[ (1 + [\lambda (CN)^{-1}] + [2 \lambda (CN)^{-1}] + \dots + [(N-1) \lambda (CN)^{-1}])^{\frac{a_0(0)}{2}}, \]
which for $\lambda = 0$ is equal to $N^{\frac{a_0(0)}{2}}$. $C^{-\frac{a_0(0)}{2}}$ gets $(NC)^{-\frac{a_0(0)}{2}}$.
Therefore the whole expression $\BP(F, I, z', \alpha, W)$ is invariant.
\end{proof}

We start by observing the following lemmas about the adelic Borcherds products:

\begin{LEMMA}\label{LEMMABP1}
$\BP(F, \nI, z', \alpha, W)$ multiplied with the trivialization $(s_{z'})^{\otimes \frac{a_0(0)}{2}}$ (\ref{QEXPANSIONPREP})
is invariant under $z' \mapsto \beta z'$ for $\beta \in \Q_{>0}$, i.e.
we have
\[ \BP(F, \nI, \beta z', \alpha, W) = \beta^{-\frac{a_0(0)}{2}} \BP(F, \nI, z', \alpha, W).  \]
\end{LEMMA}
\begin{proof}
Changing $z'$ in the various $\langle z', \delta \rangle$ has the same effect as scaling the measure,
hence (by the last lemma) as multiplying $C$ by $\beta$. The expression $\lambda_{z'}$ gets divided by $\beta$ 
(\ref{eqlambdazprime}), $Q_{z'}(\lambda)$ gets divided by $\beta^2$. Substituting $\beta^2 \lambda$ for $\lambda$, we get the claimed expression.
\end{proof}

\begin{LEMMA}\label{LEMMABP2}
For $w=k \wedge z_1 \in \nU_\Q$, 
$\BP(F, \nI, \exp(w)z', \alpha, W)$
is equal to $\BP(F, \nI, z', \alpha, W)$ transformed by the homomorphism 
$[\lambda] \mapsto \exp(\alpha( \lambda (w)))[\lambda]$.

\end{LEMMA}
\begin{proof}Let $w = k \wedge z_1$ and
\[ z'' := \exp(w)z' = z' + \langle z_1, z' \rangle k - (\langle k, z' \rangle  +  \langle z_1, z' \rangle Q(k) )z_1. \]
First $Q_{z'}=Q_{z''}$ because $\langle z, z'\rangle = \langle z, z''\rangle$.

We have (see (\ref{eqlambdazprime}) for the definition):
 \[ \lambda_{z''} + \delta = \lambda_{z'} + \delta - \frac{\lambda( w)}{\langle z_1, z' \rangle}z_1. \]
 If we substitute $\delta \mapsto  \delta + \frac{\lambda( w)}{\langle z_1, z' \rangle}z_1$, we
 get the same by applying $[\lambda C^{-1}] \mapsto \chi_\alpha( \lambda w) [\lambda C^{-1}]$, i.e.
 by  $[\lambda] \mapsto \exp( \alpha( \lambda w)) [\lambda]$.
\end{proof}

\begin{PAR}
Using the $q$-expansion principle \ref{QEXPANSION}, we will show the following strengthening of
the main results \cite[Theorem 13.3]{Borcherds1} and \cite[Theorem 2.12]{Kudla4} (resp. \cite[Theorem 4.11]{BK}).

Let $p \not= 2$ be a prime\footnote{$\not=2$ in order to be able to apply the Main Theorem \ref{MAINTHEOREM1a}}. Let $\nL_\Zpp \subset \nL_\Q$ be a lattice such that the quadratic form is unimodular. For this to exist, it is necessary that $p$ does not divide a minimal discriminant of $\nL_\Q$.

Assume that $f$ has Fourier coefficients of the form $a_q = \chi_{\nL_\Zp} \otimes a^{(p)}_q$, $a^{(p)}_q$ as in \ref{BORCHERDSLIFT} with values in $S(\nL_{\Afp})$.
$f$ then is invariant under an (w.l.o.g.) {\em admissible} compact open subgroup $K \subseteq \SO(\Af)$,
with Fourier coefficients $a_q \in S(\nL_\Af)$, $q \in \Q$.
\end{PAR}

\begin{SATZ}\label{BPADELIC}
Assume \ref{MAINCONJECTURE}.

There is an $N$, $p \nmid N$ such that for $R=\Zpp[\zeta_{N, \C}] \subset \C$, and
up to multiplication of $F$ by a large constant $\in \Z$,
we have
\begin{enumerate}
\item
For any r.p.c.d. $\Delta$, there is a section
\[ \Psi(F) \in H^0(\nSh({}^K_{\nRPCD}\nO)_R, (\nMorphSt^*\mathcal{E})^{\otimes a_0(0)/2}_{horz.} ) \]
such that on $\nSh({}^K \nO)(\C)$, we have (see \ref{BORCHERDSBUNDLE} for the definition of $h$) 
\[ \Phi(\nu, h; F) = - 2 \log h( \Psi(\nu, h; F) ).  \]

There is a locally constant, invertible function $\Lambda$ on $\nSh({}^K_{\nRPCD}\nO)_R$
such that 
\[ \Lambda \Psi(F) \in H^0(\nSh({}^K_\nRPCD\nO), (\nMorphSt^*\mathcal{E})^{\otimes a_0(0)/2}_{horz.} ) \]
(i.e. $\Lambda \Psi(F)$ is defined over $\Zpp$).

\item
\[ \Div(\Psi(\nu, h, F)^2) = \sum_{q<0} \nZ(\nL_\Z, <-q>, a_q; K) + D_\infty \]
(cf. \ref{DEFSPECIALCYCLEP}), where $D_\infty$ has support only within the support of the exceptional divisor.

\item
\begin{gather*}
  \int_{\nSh({}^K\nO)_\C} \log \nMorphSt^* h( \Lambda \Psi(\nu, h; F) ) \chern_1(\nMorphSt^*\mathcal{E}, \nMorphSt^* h)^{m-2}  \\
 \equiv 
   \sum_{q<0} \vol_{\mathcal{E}}(\nSh({}^K\nO)) \cdot \left. \frac{\dd}{\dd s}\widetilde{\mu}(\nL, <-q>_\Z, a_q; s)\right|_{s = 0} 
\end{gather*}
in $\R^{(p)}$ (see \ref{DEFLAMBDAMU} for the definition of $\widetilde{\mu}$).

\item
For any saturated isotropic line $\nI \subset \nL_\Zpp$ and primitive isotropic vector $z'$ in $\nL_\Zpp$ with $\langle \nI, z' \rangle \not=0$ consider the associated boundary morphism (\ref{ORTHBOUNDARYCOMP})
\[ \iota: \nH_0[(\nI^\perp/\nI) \otimes \nI, 0] \rightarrow \nO(\nL).  \]
Using the trivialization $(s_{z'})^{\otimes a_0(0)/2}$ (\ref{QEXPANSIONPREP}),
we get Fourier coefficients (\ref{QEXPANSIONNOT}, \ref{LEMMAQEXP}) which are valid on a set of the form
$\{\alpha\} \times (x + \nU_\R + \sigma) \times \{1\}$, where $\sigma \subset \nU_\R(-1)$ is some r.p.cone.
The formal function
\[ \sum_{\lambda \in U_\Q^*} a_{\iota,z',\lambda,\alpha,1,\sigma}(\Psi(F)) [\lambda] \]
is, up to multiplying $\Psi(F)$ by a constant of absolute value 1, given by:
\begin{equation}\label{bpexpansion}
 [\rho(F, \nI, W)] \cdot \BP(F, \nI, z', \alpha, W).
\end{equation}
\end{enumerate}
\end{SATZ}
Remark: The expansion in 4. suffices to know every other expansion, i.e. for $\rho\not=1$, by property (\ref{transfborcherds}).

\begin{proof}
First observe that $F$ is a finite linear combination of functions valued in $\C[\nL_\Z/\nL_\Z^*]$ for various lattices $\nL_\Z$ of $\nL_\Zpp$.
Hence, we may assume that $F$ is of this kind, observing that all properties are stable under taking finite linear combinations, resp. change of $K$.
Hence w.l.o.g. $F$ corresponds (as above) to an $f$ considered in \cite{Borcherds1}.

Recall the normalization of the Hermitian metric $h$ (\ref{BORCHERDSBUNDLE}).
Consider the boundary morphism as in 4. associated with some $I$ and a $<z'>$:
\[ \iota: \nH_0[(\nI^\perp/\nI) \otimes \nI, 0] \rightarrow \nO(\nL)  \]
$\frac{\Psi}{s_{z'}}$ may be seen as a function on $\nL_{0,\C}= <z', I>^\perp$. We denote the
corresponding variable by $Z$.
We have
\[ h_\mathcal{E}(\Psi) = |\frac{\Psi}{s}|^2 h_\mathcal{E}(s), \]
and
\[ h_\mathcal{E}(s) = ((2\pi)^{-1} e^{-\gamma} \langle\Im(Z),\Im(Z)\rangle (2\pi)^2 )^{\frac{a_0(0)}{2}}.  \]

Hence we get 
\[ \log h_\mathcal{E}(\Psi) = \log |\frac{\Psi}{s}|^2 + \frac{a_0(0)}{2} ( \log(\langle\Im(Z),\Im(Z)\rangle) + \log(2\pi) - \gamma ) \]
and 1. follows from \cite[Theorem 13.3]{Borcherds1}, 3. is essentially \cite[Theorem 2.12 (ii)]{Kudla4}, using \cite[Theorem 4.7 and section 10.2]{Paper1}. See \cite[Theorem 4.11]{BK} for
an independent proof (use \cite[Lemma 7.3]{Paper1} for the comparison).
Observe that the function $\Psi$ differs from the one in \cite{Kudla4} or \cite{Borcherds1} by $(2\pi i)^{\frac{a_0(0)}{2}}$. 
The constant of the Hermitian metric has been adjusted, such that
3. holds without an additive constant. Compare also to STEP 1 of the proof of Main Theorem \ref{RESULTGLOBAL}, where the same constant appears naturally, too.

That some multiple of $F$ exists, such that {\em all} coefficients are integral and even is due to
the classical $q$-expansion principle. Further multiplication with the order of the character of $\SO(\Q)$ occurring in \cite{Borcherds1}, we get a section of the {\em complex} bundle $\nMorphSt^* \mathcal{E}_\C$.

It suffices to prove 4. for any choice of $z'$ and $\alpha: 1 \mapsto 2 \pi i$. 
For note that change of $z'$ by $\beta'$ changes $\BP$ by
$\beta^{-a_0(0)/2}$ (\ref{LEMMABP1}) and $s_{z'}^{\otimes a_0(0)/2}$ by $\beta^{a_0(0)/2}$.
Changing $z'$ by $\exp(w)z'$ has the effect of translating by $w$ in the $\nU_\C$-domain. Hence
a formal function gets transformed by $[\lambda] \mapsto \exp(2\pi i \lambda(w))[\lambda]$. This
is the same effect as observed by $\BP$ (\ref{LEMMABP2}). The Weyl vector is independent of $z'$
so the above transformation has the effect of multiplication by a constant of absolute value 1.

Choose $z=z_{Bor} \in \nL_\Z \cap \nI$ and $z_{Bor}' \in \nL_\Z^*$ as in \cite{Borcherds1}, i.e. $z_{Bor}$ primitive and isotropic and $\langle z_{Bor}, z_{Bor}' \rangle=1$.

We may now show 4. for 
$z':=z'_{Bor} - Q(z'_{Bor})z_{Bor}$, which is now isotropic, 
but satisfies still $\langle z_{Bor}, z' \rangle = 1$. $z'$ may not lie in $\nL^*_\Z$, however.

As measure we may choose the standard invariant measure on $\Af$, giving $\Zh$ the volume one.
The trivialization now coincides (up to the factor $(2\pi i)^{\frac{a_0(0)}{2}}$ that we incorporated in
the characterization of $\Psi(F)$) with the one given in \cite{Borcherds1}, up to
identification of $\nI^\perp/\nI$ with $\nI \otimes (\nI^\perp/\nI)$ with by $k \mapsto k \otimes z_{Bor}$. This interchanges $Q$ and $Q_{z'}$.

Recall from \cite[Theorem 13.3 (v)]{Borcherds1} that the Fourier expansion with respect to this particular
choice is (up to a constant of absolute value 1) given by
\begin{gather}
\prod_{\delta \in \Z/N\Z} (1-e(\delta/N))^{a_{\delta z/N}(0)/2} \cdot [\rho(U, W, F_U)]\cdot \nonumber \\
\prod_{\substack{\lambda \in (<z>^\perp/<z>)^* \\ \langle \lambda, W\rangle >0}}
\prod_{\substack{\delta \in (\nL_\Z)^*/\nL_\Z \\ \delta|_{z^\perp} = \lambda }} (1-e(\langle \delta, z' \rangle)[\lambda])^{a_{\lambda^2/2}(\delta)}. \label{fcborcherds}
\end{gather}

The first product in $\BP$ may be rewritten as
\[ \prod_{\delta \in \frac{1}{N}\Z/\Z} (1 - e(\delta))^{\frac{a_0(\delta z)}{2}}, \]
and coincides with the product in Borcherds' expression. The Weyl vector is the same, too.

Claim: If $\langle z,z' \rangle =1$, there is a bijection between
\[ \{ \lambda \in U_\Q^*, \delta \in \Q/\Z = \Af/\Zh \where \lambda_{z'} + \delta z \in \nL_\Z^* \}. \]
and
\[  \{ k_{Bor} \in (<z>_\Z^\perp/<z>_\Z)^*, \delta_{Bor} \in \nL_\Z^*/\nL_\Z, \delta_{Bor}|_{z^\perp} = k_{Bor} \} \]

Proof: The map is given sending $(\lambda,\delta)$ to 
$(k_{Bor},\delta_{Bor}) = (\lambda z, \lambda_{z'} + \delta z)$ 
For $v \in <z>^\perp_\Z$, we have
\[ \langle k_{Bor}, v \rangle = \langle \lambda z, v \rangle = \langle \lambda_{z'}, v \rangle = \langle \lambda_{z'} + \delta z, v \rangle \in \Z, \] 
hence  $k_{Bor} \in (<z>_\Z^\perp/<z>_\Z)^*$.

In the other direction, we map $(k_{Bor}, \delta_{Bor})$ to 
$(\lambda, \delta) = (k_{Bor} \otimes z', \langle \delta_{Bor}, z' \rangle)$.
These maps are inverse to each other. One direction is clear; for the other we have to show:
\[ \delta_{Bor} \equiv \widetilde{k}_{Bor} +  \langle \delta_{Bor}, z' \rangle z \quad (\text{mod }\nL_\Z)§ , \]
where $\widetilde{k}_{Bor}$ is the representative in $<z>^\perp_\Q$ that satisfies $\langle \widetilde{k}_{Bor}, z'\rangle =0$. This follows from the condition $\delta_{Bor}|_{z^\perp} = k_{Bor}$.

The claim shows that
the second product occurring in Borcherds' Fourier expansion is equal to the second product appearing in $\BP$. This finishes the proof of 4.

We now show the integrality properties stated in 1. and 2.

We can multiply $\Psi(F)$ by an appropriate complex number of absolute value 1 on every connected component
and hence may assume that {\em one} Fourier expansion for each connected component is given
{\em precisely} by 
\begin{equation}\label{bpexp} 
[\rho(\omega(\rho) F, \nI,  W)] \BP(\omega(\rho) F, \nI, z', \alpha, W).
\end{equation} 
We get a finite number of different expansions with coefficients in some ring $R = \Zpp[\zeta_{N,\C}]$, $p \nmid N$ which we may assume to contain the ring described by \ref{QEXPANSION}.
By multiplying $F$ again by a large integer, we may also assume that all Weyl vectors involved are integral.

Therefore by \ref{QEXPANSION} $\Psi(F)$ extends to an (integral) section of $H^0(\nSh({}^K_{\nRPCD}\nO)_R, \nMorphSt^*(\mathcal{E})^{\frac{a_0(0)}{2}}_{horz.})$ for any $\Delta$.

Furthermore, the expansion (\ref{bpexp}) is invertible integrally, hence we may infer that $\Psi(F)^{-1}$ is a
section\footnote{Up to the problem of constants, this is $\Psi(-F)$.} of
\[ H^0(\nSh({}^K_{\nRPCD}\nO)_R, \nMorphSt^*(\mathcal{E})^{-\frac{a_0(0)}{2}}_{horz.}).\] 
It follows that $\Psi(F)$  cannot
have any connected (=irreducible) component of the fibre above $p$ in its divisor.

To prove that $\Psi(F)$ is actually defined over $\Zpp$, i.e. lies in
\[ H^0(\nSh({}^K_{\nRPCD} \nO), \nMorphSt^*(\mathcal{E})^{\frac{a_0(0)}{2}}_{horz.}), \]
we have to check the requested invariance property of the Fourier coefficients.

We compute
\[ a_{\iota, z', \lambda, \alpha, \rho, \sigma}({}^\tau \Psi(F))
= {}^\tau a_{\iota, z', \lambda, \alpha, \rho \iota(k_\tau)^{-1}, \sigma} (\Psi(F))   \]
for all $\tau \in \Gal(\Q(\zeta_N)|\Q)$.

By the transformation property (\ref{transfborcherds}), we are reduced to compute
\[ a_{\iota, z', \lambda, \alpha, 1, \sigma} ({}^\tau \Psi(F'))
= {}^\tau a_{\iota, z', \lambda, \alpha, 1, \sigma} (\Psi(\omega(\iota(k_\tau))^{-1} F')  \]
for $F'= \omega(\rho) F$. We will, however, keep the letter $F$.

We have seen that, up to a constant of absolute value 1, the coefficients 
\[ a_{\iota, z', \lambda, \alpha, 1, \sigma} (\Psi(\omega(\iota(k_\tau))^{-1} F) \] 
are the ones forming 
$\BP(\omega(\iota(k_\tau))^{-1} F, \nI, z', \alpha, W)$. 
We have by definition
\[ a_m(\omega(\iota(k_\tau))^{-1} F)(\rho) = a_m(F)(\iota(k_\tau) \rho). \] 

However
$\iota(k_\tau)$ acts trivial on $\lambda_{z'}$ because the latter is
perpendicular to $\nI \oplus \mbox{$<z'>$}$, and it acts by multiplication by $k_\tau$ on $\delta$. Substituting
$k_\tau^{-1} \delta$ 
for $\delta$ and using invariance of the measure ($k_\tau \in \Zh^*$!), we get the
same expression, but all $\chi_\alpha(\langle z', \delta\rangle)$ changed to
$\chi_\alpha(\langle z', k_\tau^{-1} \delta \rangle) = \chi_\alpha(k_\tau^{-1} \langle z', \delta \rangle) =
\chi_\alpha(\langle z', \delta \rangle)^{\tau^{-1}}$. 
This shows:
\[
  a_{\iota, z', \lambda, \alpha, 1, \sigma} (\Psi(\omega(\iota(k_\tau))^{-1} F))  
= \Lambda_\tau  {}^{\tau^{-1}} a_{\iota, z', \lambda, \alpha, 1, \sigma} (\Psi(F))  \]

with $|\Lambda_\tau| = 1$ and $\Lambda_\tau \in R^*$ (because $\Psi(F)$ is defined over $R$ and
invertible in the sense above).

Therefore
\[ {}^\tau f = {}^\tau \Lambda_\tau f \]
for all $\tau \in \Gal(\Q(\zeta_N)|\Q)$, where $\Lambda_\tau$ locally constant invertible on $\nSh({}^K\nO)_R$ with $|\Lambda_\tau|_\infty \equiv 1$ constant.
From Hilbert 90 follows that there exists a $\Lambda$, locally constant invertible on $\nSh({}^K\nO)_R$
such that
\[ {}^\tau (\Lambda f) = \Lambda f \]
(i.e. $({}^\tau \Lambda) / \Lambda = {}^\tau \Lambda_\tau$).
That is, $\Lambda f$ is defined over $\Zpp$.

It remains to show 3. Without the $\Lambda$ this is proven in \cite{Kudla4} or \cite{BK} as an identity in $\R$, cf. beginning of this proof. The evaluation of
\[ \int \log \parallel \Lambda f \parallel \chern_1( \nMorphSt^*(\overline{\mathcal{E}}) )^{m-2} \]
differs from the same expression for $f$ instead of $\Lambda f$ by
\[ \sum_{\alpha \in \pi_0(\nSh({}^K\nO))} \vol(\alpha) \log(|\Lambda(\alpha)|_\infty). \]

Since $|\Lambda(\tau^{-1}\alpha)^\tau|_\infty = |\Lambda(\alpha)|_\infty$ we may write this as
\begin{align*}
& \frac{1}{\phi(M)} \sum_{\tau \in \Gal(\Q(\zeta_M)|\Q), \alpha \in \pi_0(\nSh({}^K\nO))}
\vol(\alpha) \log(|\Lambda(\tau^{-1}\alpha)^\tau|_\infty) \\
=& \frac{1}{\phi(M)} \sum_{\tau \in \Gal(\Q(\zeta_M)|\Q), \alpha \in \pi_0(\nSh({}^K\nO))}
\vol(\tau \alpha) \log(|\Lambda(\alpha)^\tau|_\infty) \\
=& \frac{1}{\phi(M)} \sum_{\alpha \in \pi_0(\nSh({}^K\nO))}
\vol(\alpha) \log(|N_{\Q(\zeta_M)|\Q}( \Lambda(\alpha))|_\infty) \equiv 0 \qquad \text{in } \R^{(p)}. \\
\end{align*}
Here we used $\vol(\tau \alpha)=\vol(\alpha)$, which is true because $\chern_1(\nMorphSt^*(\mathcal{E}))$ is Galois invariant.
\end{proof}

We will later need some general lemmas about the line bundle of orthogonal modular forms.

\begin{LEMMA}\label{LEMMABUNDLEEMBEDDING}
With respect to the Hodge embedding (\ref{LEMMAGSPINEMBEDDING})
\[ \nS(\nL) \hookrightarrow \nH(\Cliff^+(\nL), \langle, \rangle_\delta), \]
the pullback of the canonical bundle $\Lambda^g \mathcal{L}$ via the `dual' embedding
\[ \nShD(\nS(\nL))  \rightarrow \nShD(\nH(\Cliff^+(\nL), \langle, \rangle_\delta)) \]
is, up to a bundle coming from a character $\GSpin(\nL) \rightarrow \Gm$, and up to a tensor-power,
the bundle $\mathcal{E}$ (\ref{BORCHERDSBUNDLE}).
\end{LEMMA}
\begin{proof}
Recall that the induced map on compact duals (\ref{HODGEEMBEDDINGDUAL})
is given by $<z> \mapsto zw\Cliff^+(\nL)$, where $z$ is assumed to be primitive in $\nL_\Zpp$ and $w$ is {\em any}
primitive vector with $\langle w, z \rangle = 1$. The fibre of $\mathcal{L}$ at the image of $<z>$ is given by the space $zw\Cliff^+(\nL)$ and the 
fibre of $\mathcal{E}$ is given by
$<z>$. By an explicit determination of the action of the stabilizer group in both cases the result follows.
\end{proof}

\begin{KOR}\label{BORCHERDSBUNDLEAMPLE}
The pullback of the bundle of Siegel modular forms $\nMorphSt^* (\Lambda^g \mathcal{L})^{\otimes n}$ of some positive weight $n$ is
equal to some positive power of the bundle of orthogonal modular forms $\nMorphSt^* \mathcal{E}^{\otimes n'}$. 
In particular $(\nMorphSt^* \mathcal{E})$ is ample on every $\nSh({}^K\nS(\nL))$ and $\nSh({}^K\nO(\nL))$ and a suitable power has no base points on 
$\nSh({}^K_\nRPCD \nO(\nL))$.
\end{KOR}
\begin{proof}
Follows from \ref{LEMMABUNDLEEMBEDDING} and \ref{FUNCTORIALITYAUTOMVECTB}. The bundle of Siegel modular forms $\nMorphSt^* (\Lambda^g \mathcal{L})$ is ample \cite{FC}.
The ampleness of $\nMorphSt^* \mathcal{E}$ follows from the fact that the Hodge embedding induces an
embedding (up to a normzlization) of the corresponding Shimura varieties (\ref{MAINTHEOREM1}) for suitable levels and the fact that 
all $\nSh({}^K\nS(\nL))$ and $\nSh({}^K\nO(\nL))$ are related by finite etale maps that pullback the $\nMorphSt^* \mathcal{E}$'s into each other.
The assertion on base points is true, because a suitable power of $\nMorphSt^* (\Lambda^g \mathcal{L})$ defines a morphism to the minimal compactification on every
$\nSh({}^{K'}_{\nRPCD'} \nH_g)$ \cite[Theorem 2.3, (1)]{FC}.
\end{proof}

We will compare the measure given by the volume form $c_1(\nMorphSt^* \overline{\mathcal{E}})^{m-2}$, the
highest power of the Chern form of the canonical Hermitian line bundle defined above with
the quotient $\mu$ of the canonical volume form \cite[2.4]{Paper1} on $\SO(\nL_\R)$ by the one giving $K_\infty = $ stabilizer of some positive definite subspace $N$ in $\nL_\R$ the volume one.
\begin{LEMMA}\label{BORCHERDSBUNDLECOMP}
\[  2 \lambda_\infty^{-1}(\nL; 0) \mu = (c_1(\nMorphSt^* \mathcal{E}, \nMorphSt^* h))^{m-2}. \]
\end{LEMMA}
\begin{proof}
We will ignore henceforth all signs of volume forms, whenever they have no influence because we know that the 2 volume forms are positive and invariant.
Choose an orthogonal basis of $\nL_\R$ with $\langle e_i, e_i \rangle = -1, i = 1,2$ and
$\langle e_i, e_i \rangle = 1, i > 2$. Set $z = \frac{1}{2}(e_3-e_2), z' = \frac{1}{2}(e_3+e_2)$. We have $\langle z, z \rangle = 0$ and $\langle z, z' \rangle = \frac{1}{2}$.
Choose a base point $<\xi> \in \nX_\nO \subset \PP^1(\nL_\C)$, $\xi = i e_1 + e_2$. It induces an isomorphism
$\SO(\nL_\R) / K_\infty \cong \nX_\nO$. The tangent space at $\xi$ in $\nX_\nO \subset \nShD(\nO(\nL))(\C)$ is canonically identified with
\[ \Hom(<\xi>_\C, <\xi>^\perp_\C / <\xi>). \]
The induced map
\[ \Lie(\SO(\nL_\R)) \rightarrow \Hom(<\xi>_\C, <\xi>^\perp_\C / <\xi>_\C) \]
is given by
\[ \Lambda \mapsto \{ \xi \mapsto \Lambda \xi \mod <\xi>_\C \}. \]
On $\Lie(\SO(\nL_\R))$ we have the basis:
\[ L_{ij} = \{e_{i} \wedge e_{j}\}, \]
whose associated volume form is the canonical $\mu$ \cite[Lemma 3.6]{Paper1}. The basis elements $L_{12}$, $L_{ij}$, $i,j \ge 3$ are mapped to 0.
Its associated volume form is the canonical \cite[2.4]{Paper1} on $K_\infty$ \cite[Lemma 3.6]{Paper1}.
The induced volume form on $\nX_\nO$ hence is associated to
\[ \bigwedge_{i=3}^m \widetilde{L}_{1i}^* \wedge \bigwedge_{i=3}^m \widetilde{L}_{2i}^* \in \bigwedge^{2(m-2)} \Hom(<\xi>_\C, <\xi>^\perp_\C / <\xi>_\C)^*,  \]
where $\widetilde{L}_{1i} = \{ \xi \mapsto i e_i \}$ and $\widetilde{L}_{2i} = \{ \xi \mapsto e_i \}$ (wedge product over $\R$!).

Now let $K = <e_1, e_4, \dots e_m>_\C = <z,z'>^\perp_\C$.
We calculate
\begin{align*}
 4 \partial \overline{\partial} \log Y^2 &= \sum_{i,j \in I} \frac{\partial^2}{\partial y_i \partial y_j} \log(Y^2) \dd z_i \wedge \dd \overline{z}_j \\
&= \sum_{i,j \in I} -\frac{4\delta_i \delta_j y_iy_j}{Y^4} \dd z_i \wedge \dd \overline{z}_j + \sum_i \frac{2\delta_i}{Y^2} \dd z_i \wedge \dd \overline{z}_i
\end{align*}
where $K \ni Z = X+iY$, $I = \{1, 4, \dots, m\}$ and $\delta_1 = -1$ and $\delta_i = 1$ otherwise.

At the point $ie_1$ this yields:
\[ (\dd \dd^c \log Y^2)^{m-2} = \left(\frac{1}{2\pi}\right)^{m-2} (m-2)! \dd x_1 \wedge \dd y_1 \wedge \cdots \wedge \dd x_m \wedge \dd y_m \]
(observe $\dd \dd^c = -\frac{1} {2 \pi i} \partial \overline{\partial}$ and $\dd z_i \wedge \dd \overline{z}_i = 2 i \dd x_i \wedge \dd y_i$).

Under the parametrization
\begin{align*}
 K_\C &\rightarrow \nShD(\nO(\nL))(\C) \\
 Z &\mapsto Z + z' - \langle Z, Z \rangle z
\end{align*}
$ie_1$ is mapped to $\xi$. The tangent map $T$ at this point is
\[ Z \mapsto \{ \xi \mapsto Z - 2 i \langle Z, e_1 \rangle z \mod <\xi>_\C \}, \]
and we have: $T(e_1) = \widetilde{L}_{23}, T(ie_1) = -\widetilde{L}_{13}$ and
$T(e_i) = \widetilde{L}_{1i}, T(ie_i) = \widetilde{L}_{2i}$ for $i\ge 4$.

This yields the volume form associated to
\[ (m-2)!\left(\frac{1}{2\pi}\right)^{m-2} \bigwedge_{i=3}^m  \widetilde{L}_{1i}^* \wedge \bigwedge_{i=3}^m \widetilde{L}_{2i}^* \in \bigwedge^{2(m-2)} \Hom(<\xi>_\C, <\xi>^\perp_\C / <\xi>_\C)^*.  \]

Now observe that $K_\infty = \SO(N) \times \SO(N^\perp)$ because $K_\infty$ is not
allowed to change the orientation of $N$ (the $\C$-oriented negative definite plane corresponding to $<\xi>$).
Hence the volume of $K_\infty$ is $2\pi \cdot \frac{1}{2} \prod_{j=1}^{m-2} 2 \frac{\pi^{j/2}}{\Gamma(j/2)}$ \cite[Lemma 3.7, 2.]{Paper1}. This gives the factor
$2 \prod_{j=2}^{m} \frac{1}{2} \frac{\Gamma(\frac{j}{2})}{\pi^{\frac{j}{2}}}$ which is $2 \lambda_\infty^{-1}(\nL; 0)$.
\end{proof}

\section{The arithmetic of Borcherds input forms}

\begin{PAR}
In this section, we present and refine some results of \cite{Borcherds2}. We begin by recalling the notation of 
[loc. cit.]. Let $\Gamma \subset \Sp'_2(\R)$ be a subgroup, commensurable with $\Sp'_2(\Z)$ and with values in 
any representation $V$ of $\Sp'_2(\Z)$. Here $\Sp'_2$, as usual, is either $\Sp_2=\SL_2$ or $\Mp_2$. Let $k$ be
a weight, half integral, if $\Sp'=\Mp$ and an integer, if $\Sp'=\Sp$.

Let
\[ \ModForm(\Gamma, V, k) \]
be the space of modular forms of weight $k$ for $\Gamma$ of representation $V$,
meromorphic at the cusps. Let
\[ \HolModForm(\Gamma, V, k) \]
be the space of modular forms of weight $k$ for $\Gamma$ of representation $\rho$, which are holomorphic at the cusps.
\end{PAR}
\begin{PAR}\label{BORCHERDSGKZ}

Denote
\begin{align*}
 \PowSer(\Gamma) &= \bigoplus_\kappa \C \llbracket q_\kappa \rrbracket,   \\
 \Laur(\Gamma)   &= \bigoplus_\kappa \C \llbracket q_\kappa \rrbracket [q_\kappa^{-1}],  \\
 \Sing(\Gamma)   &=  \bigoplus_\kappa \C \llbracket q_\kappa \rrbracket [q_\kappa^{-1}] / q_\kappa \C \llbracket q_\kappa \rrbracket, 
\end{align*}
where $\kappa$ runs through all cusps for $\Gamma$, and $q_\kappa$ is a uniformizer at the the cusp $\kappa$.
$\Sp_2'(\Z)$ acts on these sets.

Let $\Gamma' \subset \Gamma$ be a subgroup of finite index, which acts trivial on $V$. We define
\[ \PowSer(\Gamma, V) = \left( \PowSer(\Gamma') \otimes V \right)^\Gamma, \]
where $\Gamma$ acts on both factors.

Similarly we define $\Laur(\Gamma, V)$ and $\Sing(\Gamma, V)$.
We have still (abstractly) a 1:1 correspondence between $f \in \PowSer(\Gamma, V)$ (say) and a collection $f_\kappa$ for each cusp of
$\Gamma$ (not $\Gamma'$!) where
\begin{equation}\label{eqnexpansion}
f_\kappa = \sum_{m \in \Q_{\ge 0}} a_m q_\kappa^m,
\end{equation}
and $a_m$ lies in a sub-vector space $V_m$ of $V$, depending on $m$. 

For $\Gamma = \Sp_2'(\Z)$, we have explicitly 
\begin{equation}\label{eqdefvm}
V_m = \{ v \in V \where \M{1&1\\0&1}v = e^{2 \pi i m}v, \}
\end{equation}
and in particular $V_m$ is zero if $m \notin \frac{1}{N}\Z$ for some $N$, depending on $\Gamma'$.

There are maps
\begin{align*}
 \lambda: \HolModForm(\Gamma, V^*, k) &\rightarrow \PowSer(\Gamma, V^*), \\
 \lambda: \ModForm(\Gamma, V, 2-k) &\rightarrow \Sing(\Gamma, V),
\end{align*}
which in the representation (\ref{eqnexpansion}) for the right hand side, this corresponds to taking Fourier series.

There is a non-degenerate pairing between $\PowSer(\Gamma, V)$ and $\Sing(\Gamma, V^*)$
given by
\[ \langle f, \phi \rangle = \sum_\kappa \res( \langle f_\kappa, \phi_\kappa \rangle q_\kappa^{-1} \dd q_\kappa ). \]
This pairing is invariant under the operation of $\Mp_2(\Z)$, and hence gives a well defined pairing between
the spaces in question. 

For example if $\Gamma=\Mp_2(\Z)$ and standard parameter $q$, up to a scalar, this pairing is the same as 
\[ \langle f, \phi \rangle = \sum_m \langle a_{-m}, b_m \rangle, \]
for $f = \sum_m a_m q^m$ and $\phi = \sum_m b_m q^m$.

In [loc. cit., Theorem 3.1] it is sketched that one can deduce from Serre duality that
\begin{equation}\label{SERREDUALITY}
 \lambda( \ModForm(\Sp_2'(\Z), V, 2-k) ) = \left( \lambda(\HolModForm(\Sp_2'(\Z), V^*, k) \right)^\perp, 
\end{equation}
and in [loc. cit., Lemma 4.3] this is refined to
\[ \lambda( \ModForm(\Sp_2'(\Z), V, 2-k)_\Z ) \otimes \C = \left( \Gal(\overline{\Q}|\Q) \lambda(\HolModForm(\Sp_2'(\Z), V^*, k) \right)^\perp. \]
It is also shown that the last space has finite index in $\Sing(\Sp_2'(\Z), V)$.

Let $\nL$ be a $\Z$-lattice of dimension $m$ with non-degenerate quadratic form $Q_\nL \in \Sym^2(\nL^*)$.
We let $\Weil(\nL^*/\nL) = \C[\nL^*/\nL]$ be the Weil representation of $\Sp_2'(\Z)$ and $2k=m$ modulo 2. 
If $V^* \subseteq \Weil(\nL^*/\nL)^*$ is a sub-representation, which, as a linear subspace, is defined over $\Q$,
Proposition \ref{LEMMAWEILGALOIS} below states that 
\[ \lambda(\HolModForm(\Sp_2'(\Z), V^*, k)) \]
 is actually Galois stable.
\end{PAR}

\begin{PAR}We start by a general discussion about the arithmetic of `vector valued modular forms' for $\GL_2$ here.
We denote by $\Q^{cycl}$ the direct limit of $\OO_{\G_{m,\Q}[\frac{1}{N}]}$. It is isomorphic to $\Q^{ab}$ but equipped with a
compatible system of primitive $N$-th roots of unity $\{ \zeta_N \}$.
Define 
\[  \tau: \GL_2(\Zh) \rightarrow \Gal(\Q^{cycl}/\Q) \]
\[ \tau: k \mapsto \tau_k:= \rec \circ \det(k) \]
 which is characterized by ${}^{\tau_k} \zeta_N = (\zeta_N)^{\det(k)}$.
\end{PAR}

\begin{PROP}\label{SATZCONSTRMODFORMVB}
For each $n \in \Z$ and an open subgroup $K \subseteq \GL_2(\Zh)$, we have a natural exact functor from the category of $\tau$-semi-linear representations of $K$  (\ref{SEMILINEAR}) to vector bundles over $\nSh({}^K \nH_1)_\Q$
denoted 
\[ V \mapsto \nMorphSt^*\mathcal{L}^{\otimes n}_V.  \]
For $K_1 \subseteq K_2 \subseteq \GL_2(\Zh)$ open subgroups,  we have a natural isomorphism:
\[ H^0(\nSh({}^{K_1} \nH_1)_\Q, \nMorphSt^*\mathcal{L}^{\otimes n}_V) \cong H^0(\nSh({}^{K_2} \nH_1)_\Q, \nMorphSt^*\mathcal{L}^{\otimes n}_{\ind_{K_1}^{K_2}V}). \]
\end{PROP}
\begin{proof}
The functor is defined as follows:
Let $N$ be such that the $\tau$-semi-linear representation $(V, \omega)$ has
a $\Q[\zeta_N]$-structure $V_N$ and $K(N)$ acts just via $\tau$ on $V \cong V_N \otimes_{\Q(\zeta_N)} \Q^{cycl}$.

We have the following obvious commutative diagram of extended Shimura data:
\[  \xymatrix{ {}^{K(N)}\nH_1 \ar[rr]^{(\id, k)} \ar[d]^{(\det,1)} && {}^{K(N)}\nH_1 \ar[d]^{(\det,1)} \\
{}^{K(N, \Gm)} \nH_0 \ar[rr]^{(\id, \det(k))} && {}^{K(N, \Gm)} \nH_0
}\]
Applying the canonical model functor $\nSh$ to this diagram, we see, that $\nSh(\det,1)$ is compatible with the
(Hecke) action of $\GL(\Z/N\Z)$. 

From the explicit description of the canonical model (cf. e.g. \ref{SATZBOUNDARYPOINTS}), we know that
\[ \nSh({}^{K(N,\Gm)} \nH_0)_\Q \cong \spec(\Q(\zeta_N)), \]
and the Hecke action transports to the natural isomorphism
\[ (\Z/N\Z)^* \rightarrow \Gal(\Q(\zeta_N)/\Q). \]

The $\tau$-semi-linear representation $V$ thus gives rise to a vector bundle $\mathcal{V}$ with {\em right} $K/K(N)$-action on
\[ \spec(\Q(\zeta_N)), \] 
which is compatible with the Hecke action (acting via $\tau=\rec \circ \det$). 
In detail: we let 
$k \in K$ act on $V$ as $v \mapsto \omega(k)^{-1} v$ (right action). This action is compatible with the action on $\spec(\Q(\zeta_N))$
given by letting $k \in K$ act on a function $\alpha \in \Q(\zeta_N)$ as $\tau_k(\alpha)$, (i.e. the action $k \mapsto \spec(\tau_k)$ on $\spec(\Q(\zeta_N)$), if we have
\[ \omega(k) (\alpha v) = \tau_{k}(\alpha) \omega(k) v. \]
This is precisely the condition of $\omega$ being $\tau$-semi-linear. 
 
We define $\nMorphSt^*\mathcal{L}^{\otimes n}_V$ as a quotient of the bundle on $\nSh({}^{K(N)} \nH_1)_\Q$
\[ \nMorphSt^*\mathcal{L}^{\otimes n} \otimes \nSh(\det, 1)^* \mathcal{V}, \]
by the action of $K/K(N)$ on both factors. Since the action is compatible with the Hecke-action on $\nSh({}^{K(N)} \nH_1)_\Q$, it descends along 
the quotient $\nSh({}^{K(N)} \nH_1)_\Q \rightarrow \nSh({}^{K} \nH_1)_\Q$. 
Up to isomorphism, this bundle does not depend on the $N$ and $V_N$ chosen.

The induction formula is a variant of Shapiro's Lemma.
\end{proof}

\begin{PAR}
Fourier coefficients:
We fix the boundary component $\nSDB$ (corresponding to $i\infty$ in the classical picture) with
\[ \nP_\nSDB = \M{*&*\\0&1} \qquad \nX_\nSDB = \nH_0 \times \C \]
with its standard splitting
\[ \nP_\nSDB = \M{1&*\\0&1} \rtimes \M{*&0\\0&1} \]
 (corresponding to $0\in \C$ in the classical picture). It gives rise to a boundary morphism
\[ \nH_0[0, \Q] \Longrightarrow \nH_1 \]
with respect to which each $f \in H^0(\nSh({}^{K_1} \nH_1)_\C, \nMorphSt^*\mathcal{L}^{\otimes n}_V)$ has a Fourier expansion,
depending on some fixed rational trivializing section $s \in H^0(\nSh({}^{K}\nH_0[0, \Q])_\Q, \nMorphSt^* \mathcal{L})$ 
as in \ref{QEXPANSIONPREP}:
\[ g(z, \rho) := \frac{f}{s}(z, \rho) = \sum_{k \in \frac{1}{N}\Z} a_{k,\alpha,\rho}(f) \exp(\alpha(k z)), \]
where $a_{k,\alpha,\rho} \in (V_\C)_{k,\alpha}$, where
\[ (V_\C)_{k,\alpha} = \{ v \in V_\C\ |\ \M{1&1\\0&1}v = \exp(\alpha(k))v \}, \]
and $\alpha$ is determined by the connected component of $\nH^{\pm}$ that contains $z$.
They satisfy (\ref{QEXPANSIONPREP})
\[  a_{k,\alpha,\rho}({}^\tau f) = {}^\tau a_{k, \alpha, \rho {\tiny \M{k_\tau^{-1}&\\&1} }}(f). \]

Now choose a $\Q$-structure $V_\Q$ on $V$ such that $\omega$ gives rise to 
a cocycle $\omega' \in H^1(K, \GL_{\Q^{cycl}}(V))$, where $K$ acts via $\tau$ (see \ref{EXSEMILINEAR}).
Any holomorphic function $g$ 
\[  \HH^{\pm} \times (\GL_2(\Af)) \rightarrow V_\C \]
which satisfies
\[ g(\gamma z, \gamma \rho) = j(\gamma, z)^n g(z, \rho), \]
where $j$ is the usual automorphy factor and
\begin{equation} \label{eqtransg}
g(z, \rho k) = \omega'(k^{-1})^{(\exp(\alpha(\det(\rho k)/N))} g(z, \rho) 
\end{equation}
for $\rho \in \GL_2(\Zh)$, 
(where $\omega'(k^{-1})^{(\exp(\alpha(\det(\rho k)/N)))}$ means that $\omega'(k^{-1})$ acts on $V_\C$ 
by mapping $\zeta_N$ to $\exp(\alpha(\det(\rho k)/N))$),
gives rise to a section 
\[ f=gs \in H^0(\nSh({}^{K(1)} \nH_1)_\C, \nMorphSt^*\mathcal{L}^{\otimes n}_V). \]

The condition (\ref{eqtransg}) can be expressed purely in terms of Fourier coefficients:
 \[ a_{k,\alpha, \rho k}(g) = \omega'(k^{-1})^{(\exp(\alpha(\det(\rho k)/N))} a_{k,\alpha,\rho}(g). \]
 
 In particular, the Galois operation on $f$ is determined by
\[ a_{k, \alpha^+, 1}(\frac{{}^\tau f}{s}) =  \omega'( \M{k_\tau&\\&1} )^{(\exp(\alpha^+(\frac{1}{N})))}\cdot {}^\tau (a_{k,\alpha^+,1}(\frac{f}{s})), 
\]
where $\alpha^+$ is one chosen isomorphism $\Z \rightarrow \Z(1)$.
 
We have proven:
\end{PAR}
\begin{PROP}\label{SATZINTERPRETCLASSMODFORM}
If $\det: K \rightarrow \Zh^*$ is surjective and $(V, \omega)$ 
is a $\tau$-semi-linear representation, the association $g \mapsto f:=gs$ determines a bijection
\[ \ModForm(\Gamma, V, k) \rightarrow H^0(\nSh({}^K \nH_1)_\C, (\nMorphSt^* {\mathcal{L}})^{\otimes k}_V ),  \]
where $\Gamma = K \cap \SL_2(\Z)$. A section is determined by the Fourier coefficients 
\[ a_k(g) = a_{k, \alpha^+, 1}(\frac{f}{s}). \]
Here $\alpha^+$ corresponds to the connected component $\HH^+$ chosen as domain for the classical modular form $g$.
$f$ is rational, if we have:
\[ a_{k}(g) =  \omega'(\M{k_\tau&\\&1})^{(\exp(\alpha^+(\frac{1}{N})))} \cdot {}^\tau ( a_{k}(g) ). \] 
\end{PROP}

It is crucial that, if $(V, \omega)$ is given by a $\Q(\zeta_N)$-representation of $\SL(\Z/N\Z)$, say, for the theorem to make sense, the representation must extend $\tau$-semi-linearly to $\GL(\Z/N\Z)$. 

\begin{BEISPIEL}\label{EXAMPLEGAMMA0P_1}
Consider the special case $K=K_0(p)$, $p$ an odd prime.
Let $\omega$ be the $\tau$-semi-linear representation
$\M{a&*\\0&d} \mapsto \chi(d) \cdot \rec(ad)$ defined on $\GL_2(\Z/p\Z)$ for the unique non-trivial character
$\chi: (\Z/p\Z)^* \rightarrow \{\pm 1\}$ of order 2. The action of Galois on the Fourier coefficients is the natural, untwisted one.

We have a decomposition (\ref{EXABSTRACTGAMMA0P}) 
\[ \ind_{K_0(p)}^{K(1)}(\omega) = V^+ \oplus V^-  \]
as $\tau$-semi-linear representations, and hence Propositions \ref{SATZCONSTRMODFORMVB} and \ref{SATZINTERPRETCLASSMODFORM}
yield
\[  \ModForm(\Gamma_0(k), \chi, k) = H^0(\nSh({}^{K(1)}\nH_1)_\C, (\nMorphSt^*\mathcal{L})^{\otimes k}_{V^+}) \oplus H^0(\nSh({}^{K(1)}\nH_1)_\C, (\nMorphSt^*\mathcal{L})^{\otimes k}_{V^-}).   \]
This decomposition has been described by Hecke. We get here a priori that it is actually defined over $\Q$.
See also \ref{EXAMPLEGAMMA0P_2} below.
\end{BEISPIEL}

Let $\nL$ be a lattice of {\em even} dimension with non-degenerate quadratic form.

\begin{LEMMA}The Weil representation $\Weil(\nL^*/\nL)$, identifying $\exp(2\pi i/N)$, as chosen in the underlying characters, with the distinguished $\zeta_N \in \Q^{cycl}$, gives rise to a $\tau$-semi-linear representation, where the matrix $\M{\alpha & \\& 1} \in \GL_2(\Zh)$ acts as $\rec(\alpha)$.
The dual of the Weil representation gives rise to a $\tau$-semi-linear representation, too.
\end{LEMMA}
Note that this $\tau$-semi-linear representation does not involve any choices anymore and does depend only on $\nL$.
\begin{proof}
That the above defines indeed a $\tau$-semi-linear representation boils down to the equation:
\begin{equation}\label{compweil}
 \omega(\M{\alpha&0\\0&1} k \M{\alpha^{-1}&0\\0&1} ) = \rec(\alpha)(\omega(k))
\end{equation}
for the Weil representation (cf. \ref{WEILFORMULAS2}). This is checked as follows:
The operators $\varphi \mapsto (x^* \mapsto \varphi({}^t \alpha x^*))$ are Galois invariant. 
If we act on the operator $\varphi \mapsto (x^* \mapsto \exp(2\pi i Q_\nL(x^*) \beta) \varphi(x^*))$ by $\tau_\alpha$, we get the same as by substituting $\beta$ by $\alpha\beta$.
If we act on the operator $\varphi \mapsto (x^* \mapsto \int_{\nL_\Af} \varphi(\gamma x) \exp(-2\pi i x^* x) \dd x)$  by $\tau_\alpha$, we get the same as by substituting $\gamma$ by $\alpha^{-1}\gamma$, provided $\dd x$ is chosen Galois invariant (e.g. such that $\nL_\Zh$ has a volume in $\Q$).
Now an actual operator $\omega(r''(g_l(\alpha)))$ differs by those considered by a $\frac{\WeilGamma_f(\alpha Q_L)}{\WeilGamma_f(Q_L)}$ in the first case, 
which is a sign because $m$ is even, hence Galois invariant. In the third case $\omega(r''(d(\gamma)))$ differs by $c(\gamma):=\frac{\WeilGamma_f(\gamma Q_L)}{|\gamma|^{\frac{1}{2}}}$, where $|\gamma|$ is calculated with respect to the chosen measure and its dual. Now consider the equation $\left( u(-\gamma^{-1}) d(\gamma) \right)^3 = 1$. Since $c(\gamma)^2 \in \Q$, it shows
$c(\alpha^{-1}\gamma) = {}^{\tau_\alpha} c(\gamma)$.
\end{proof}

\begin{PAR}\label{WEILREPGALOIS}
Consider the Weil representation $\Weil(\nL_\Z^*/\nL_\Z)$ as a $\tau$-semi-linear representation of $\SL_2(\Zh)$.
$K_0(N)$ acts on $\chi_{\nL_\Zh}$ by
\[ \M{a&*\\&d} \chi_{\nL_\Zh} = \WeilGamma(d Q_\nL)\WeilGamma(Q_\nL)^{-1} \chi_{\nL_\Zh} \] 
which is a character, valued in $\mu_2$. Call this character $\varepsilon$.

We have hence maps
\[ \alpha: \varepsilon \rightarrow \Weil(\nL_\Z^*/\nL_\Z), \]
and we have also the evaluation at 0 map:
\[ \beta: \Weil(\nL_\Z^*/\nL_\Z) \rightarrow \varepsilon. \]

Using the adjunctions, we get maps of $\tau$-semi-linear representations
\[\widetilde{\alpha}: \ind_{K_0(N)}^{K(1)} \varepsilon \rightarrow \Weil(\nL_\Z^*/\nL_\Z),  \]
and
\[ \widetilde{\beta}:   \Weil(\nL_\Z^*/\nL_\Z) \rightarrow \ind_{K_0(N)}^{K(1)} \varepsilon.  \]
We have $\widetilde{\beta} \circ \widetilde{\alpha} = \id$.

Using the functor, we get maps
\[ \alpha': H^0(\nSh({}^{K(1)}  \nH_1)_\Q, \nMorphSt^* \mathcal{L}^{\otimes k}_{\ind_{K_0(N)}^{K(1)} \varepsilon}) \rightarrow H^0(\nSh({}^{K(1)}  \nH_1)_\Q, \nMorphSt^* \mathcal{L}^{\otimes k}_{\Weil(\nL_\Z^*/\nL_\Z)}) \]
\[ \beta': H^0(\nSh({}^{K(1)}  \nH_1)_\Q, \nMorphSt^* \mathcal{L}^{\otimes k}_{\Weil(\nL_\Z^*/\nL_\Z)}) \rightarrow H^0(\nSh({}^{K(1)}  \nH_1)_\Q, \nMorphSt^* \mathcal{L}^{\otimes k}_{\ind_{K_0(N)}^{K(1)} \varepsilon}). \]

We get (this is true more generally) that the operator
$\widetilde{\beta} \circ \widetilde{\alpha}$ is given by convolution with the function
\[ \varphi(g) := \beta \circ \omega(g) \circ \alpha,\]
that is
\[ \widetilde{\beta}(\widetilde{\alpha}(f))(h) = \sum_{a\in K_0(p) \backslash K(1)}  \beta ( \omega(ha^{-1}) \alpha( f(a) ) ). \] 
\end{PAR}

\begin{BEISPIEL}\label{EXAMPLEGAMMA0P_2} The following was inspired by \cite{BB}.

Resuming example \ref{EXAMPLEGAMMA0P_1} above, it follows directly from the formul\ae\ of the Weil-representation (\ref{EXPLICITWEIL}), that this function $\varphi$ corresponding to $\widetilde{\beta}\circ \widetilde{\alpha}$ is determined by
\[ \varphi(1) = 1 \qquad \varphi(S) = \frac{\sqrt{-1}^{2-\frac{m}{2}}}{\sqrt{p}}. \]
Using the dichotomy between $p$ mod 4, the type of $\nL_\Z^*/\nL_\Z$, and $\frac{m}{2}$ mod 4, see \cite[\S 2]{BB}, we get  
\[ \varphi(S) = \begin{cases} \frac{1}{\sqrt{p^*}} & \text{if } \nL_\Z^*/\nL_\Z  \text{ represents the squares mod } p, \\ - \frac{1}{\sqrt{p^*}} & \text{if } \nL_\Z^*/\nL_\Z  \text{ represents the non-squares mod } p, \end{cases} \]
where $p^*=p$ if $p \equiv 1\ (4)$ and $p^* =-p$, otherwise. Here $\sqrt{-1}$ is the distinguished primitive 4-th root of unity in $\Q^{cycl}$.

Therefore $\widetilde{\beta} \circ \widetilde{\alpha}$ is a projector (\ref{EXABSTRACTGAMMA0P}) onto $V^+$, resp. $V^-$, and thus $\widetilde{\alpha}$ 
yields an isomorphism
\[  V^\pm \cong \Weil(\nL^*/\nL)^+ \]
according to the isomorphism class of $\nL^*/\nL$.
We therefore see from condition (\ref{eqdefvm}) and the explicit formul\ae\ for the Weil-representation that for $\lambda \not\in \Z$, we have
\begin{eqnarray*}
(V^+_\C)_\lambda &\cong& \begin{cases} \C & pk \text{ is a square mod } p, \\
 0 & p\lambda \text{ is a non-square mod } p, \end{cases} \\
(V^-_\C)_\lambda &\cong& \begin{cases} \C & pk \text{ is a non-square mod } p, \\
 0 & p\lambda \text{ is a square mod } p. \end{cases} 
\end{eqnarray*}

Recall that the Fricke involution $W_q$ is given by
\[ W_q(f)(\tau):= q^{\frac{k}{2}}(q\tau)^{-k}f(\M{&-1\\q&}\tau). \]
In terms of the corresponding $\widetilde{f} \in \ModForm(\SL_2(\Z), \ind_{\Gamma_0(0)}^{\SL_2(\Z)} \chi, k)$, we get 
\[ W_q(f)(\tau) = ( q^{\frac{k}{2}}\rho(\M{&-1\\1&}) \widetilde{f}(q\tau) )(1). \] 
Hence 
$\ModForm(\SL_2(\Z), V^\pm_\C, k)$ is the space of those 
$f \in \ModForm(\Gamma_0(q), \chi, k)$ such that
the Fourier coefficients $a_\lambda$ of $W_q(f)$, for $(\lambda,q)=1$ are 0 if $\lambda$ is a non-square (resp. square) mod $p$.
This is (for $W_q(f)$) the classical condition.
\end{BEISPIEL}

\begin{PROP}\label{LEMMAWEILGALOIS}
Let $k$ be a (half\nobreakdash-)integral weight and $\Sp_2'$ either $\Sp_2$ or $\Mp_2$ according to whether $k \in \Z$ or $k \in \frac{1}{2}+\Z$ and suppose $2k \equiv m$ modulo $2$.

For $f \in \ModForm(\Sp_2', \Weil(\nL^*_\Z/\nL_\Z)^*, k)$ with Fourier expansion
\[ f = \sum_{k\in \Q} a_k q^k, \]
$a_k \in \Weil(\nL^*_\Z/\nL_\Z)_k^*$, we have for $\tau \in \Aut(\C)$ and
\[ {}^\tau f := \sum_{k \in \Q} {}^\tau a_k q^k, \]
that ${}^\tau f \in \ModForm(\Sp_2', \Weil(\nL_\Z^*/\nL_\Z)^*, k)$, too.
\end{PROP}
\begin{proof}
We first reduce to the case $k \in \Z$. For half integral weight, i.e. $m$ odd, consider the space
$\nL'_\Z = \nL_\Z \oplus <1>$. We have 
\[ \Weil((\nL'_\Z)^*/\nL'_\Z)^* = \Weil(\nL_\Z^*/\nL_\Z)^* \otimes \Weil(<1>^*/<1>)^*. \]
We have the classical theta function
\[ \theta \in \HolModForm(\Mp_2, \Weil(<1>^*/<1>)^*, \frac{1}{2}) \]
(with integral Fourier coefficients!).
Multiplication with it yields a map
\[ \ModForm(\Mp_2, \Weil(\nL_\Z^*/\nL_\Z)^*, k) \rightarrow \ModForm(\Mp_2, \Weil((\nL'_\Z)^*/\nL'_\Z)^*, k+\frac{1}{2}), \]
which commutes with the Galois action on Fourier coefficients (because $\theta$ has integral Fourier coefficients).
Now, if $f \in \ModForm(\Sp_2', \Weil(\nL^*_\Z/\nL_\Z)^*, k)$ is given, we get (assuming the integral case) that 
\[ {}^\tau (f \otimes \theta) = ({}^\tau f) \otimes \theta \in \ModForm(\Sp_2', \Weil(\nL_\Z^*/\nL_\Z)^*,  k+\frac{1}{2}), \] 
or, in other words:
\[  ({}^\tau f)(gz)\otimes \theta(g\tau) = j(g, \tau)^{2k+1} (\omega_\nL(g) ({}^\tau f)(g\tau)) \otimes (\omega_{<1>}(g)\theta(g \tau)), \]
where $j(g, \tau)$ is the fundamental cocycle of weight $\frac{1}{2}$ on $\Mp_2(\Z) \times \HH$.

For a dense set of $\tau$, $\theta(\tau)$ and $\theta(g\tau)$ are not zero; hence ${}^\tau f$ must converge, too, and
\[  ({}^\tau f)(gz) = j(g, \tau)^{2k} (\omega_\nL(g) ({}^\tau f)(g\tau)), \]
using that $\theta$ is in $\HolModForm(\Mp_2, \Weil(<1>^*/<1>)^*, \frac{1}{2})$.
Hence ${}^\tau f$ lies in $\ModForm(\Mp_2, \Weil(\nL_\Z^*/\nL_\Z)^*, k)$.

Now we may assume that $k\in \Z$ and $\Sp'_2(\Z) = \Sp_2(\Z)$. 
Then the statement of the theorem follows from Lemma \ref{WEILREPGALOIS} and Propositions \ref{SATZCONSTRMODFORMVB} and \ref{SATZINTERPRETCLASSMODFORM}.
\end{proof}

\begin{PAR}We resume the notation from section \ref{BORCHERDSGKZ}.
Assume now that $n \ge 3$ --- or --- $n=2$ and Witt rank of $\nL_\Q = 1$.
With the results above, we are able to construct Borcherds products $\Psi(F)$ whose divisor
\[ \Div(\Psi(F)^2) = \sum_{m \in \Q_{< 0}} \nZ(\nL_\Z, <-m>, a_m; K) + D_\infty \]
satisfies certain special properties (see Theorems \ref{THEOREMBORCHERDSPREP1}--\ref{THEOREMBORCHERDSPREP4} below).
\end{PAR}

\begin{LEMMA}\label{LEMMABORCHERDSPREP2}
Let $\nL_\Z$ be a lattice of signature $(m-2,2)$, $m \ge 4$.
Let $\Weil_0$ be a sub-representation of the Weil representation $\Weil(\nL_\Zh^*/\nL_\Zh)$, which is, as a subspace, defined over $\Q$. 
Let $\{M(m)\}_{m\in\Q_{> 0}}$ be a collection of sub-vector spaces $M(m) \subseteq \Weil_{0,m}^*$ (for the notation, see \ref{WEILREPGALOIS}), such that
$\sum_m \dim(M(m)^\perp) \rightarrow \infty$.

Assume that any modular form in $\HolModForm(\Sp'_2(\Z), \Weil_0^*, \frac{m}{2})$ whose Fourier coefficients satisfy $a_m \in M(m)$ for all $m>0$, vanishes.

\begin{enumerate}
 \item There is an $F \in \ModForm(\Sp'_2(\Z), \Weil_0, 2-\frac{m}{2})$ with Fourier expansion:
\[  \lambda(F) = \sum_{m \in \Q} a_{m} q^m  \]
with $a_0(0) \not= 0$ and all $a_{m} \perp M(-m), m<0$,

 \item Let $0<l \in \Q$ and $a_{-l} \in \Weil_{0,-l}$ with $a_{-l} \not\perp M(l)$ be given. 

There is an $F \in \ModForm(\Sp'_2(\Z), \Weil_0, 2-\frac{m}{2})$ with Fourier expansion:
\[  \lambda(F) = \sum_{m \in \Q} a_{m} q^m   \]
with $a_0(0) \not= 0$, all $a_{-m} \perp M(m), m>0, m\not=l$ and $a_{-l}$ is the given one.
\end{enumerate}

In both cases we may assume $a_m \in \Z[\nL^*_\Z/\nL_\Z]$ for all $m\in \Q$.

\end{LEMMA}
\begin{proof}
(Compare \cite[Lemma 4.11]{BBK})

1. Let $\Sing_M(\Sp'_2(\Z), \Weil_0) \subset \Sing(\Sp'_2(\Z), \Weil_0)$ be the subspace, defined by the conditions $a_{-m} \perp M(m), m>0$.
Obviously $\Sing_M(\Sp'_2(\Z), \Weil_0)^\perp$ is the subspace $\PowSer^M(\Sp'_2(\Z), \Weil_0^*)$, defined by the conditions $a_m^* \in M(m)$ for all $m>0$ and $c^*_0 = 0$.

Since $\lambda(\ModForm(\Sp'_2(\Z), \Weil_0, 2-\frac{m}{2})_\Q) \otimes \C$ has f. ind. in $\Sing(\Sp'_2(\Z), \Weil_0)$ 
and $\sum_m \dim(M(m)^\perp) \rightarrow \infty$, we have
\[ \lambda(\ModForm(\Sp'_2(\Z), \Weil_0, 2-\frac{m}{2})_\Q) \cap \Sing_M(\Sp'_2(\Z), \Weil_0)_\Q \not= 0. \]
We want to show that the application $[a_0(0)]: \Sing(\Sp'_2(\Z), \Weil_0)_\Q \rightarrow \Q$ does not vanish on this intersection.

Now the duality between $\PowSer(\cdots)$ and $\Sing(\cdots)$ is non-degenerate, therefore:
\begin{align*}
 & \left( \lambda(\ModForm(\Sp'_2(\Z), \Weil_0, 2-\frac{m}{2})_\Q) \cap \Sing_M(\Sp'_2(\Z), \Weil_0)_\Q \right)^\perp \\
 =&  \lambda(\HolModForm(\Sp'_2(\Z), \Weil_0^*, \frac{m}{2}))^{\Aut(\C)} + \PowSer^M(\Sp'_2(\Z), \Weil_0^*)_\Q.
\end{align*}
Here we used Proposition \ref{LEMMAWEILGALOIS}, stating that $\lambda(\HolModForm(\Sp'_2(\Z), \Weil_{0}^*, \frac{m}{2}))$ is a Galois invariant subspace 
because by assumption $\Weil_0^*$ is, as a subspace, defined over $\Q$.

Let us assume that there is a relation
\[ [a_{0}(0)] = \lambda(f) + \phi, \]
where $f \in \HolModForm(\Sp'_2(\Z), \Weil_0^*, \frac{m}{2})$  and $\phi$ is in $\PowSer^M(\Sp'_2(\Z), \Weil_0^*)$.

From this follows that $f$ is a modular form, whose coefficients $a_m^*$, $m>0$, satisfy $a_m^* \in M(m)$. Therefore $f=0$ by assumption, a contradiction.

2. Since $a_{-l} \not\perp M(l)$, we find a $a_{l}^* \in M(l)$ such that $a_{l}^* a_{-l} \not=0$.
We have to show that the element $[a_{l}^* q^{l}]: \Sing(\Sp'_2(\Z), \rho)_\Q \rightarrow \Q$
does not vanish on the intersection 
\begin{equation*}
 \lambda(\ModForm(\Sp'_2(\Z), \Weil_0, 2-\frac{m}{2})_\Q) \cap \Sing_{\widetilde{M}}(\Sp'_2(\Z), \Weil_0)_\Q,
\end{equation*}
where $\widetilde{M} = M$ except that $\widetilde{M}(l)$ has been set to 0.

For if this is not the case, we can add an appropriate element not being in its kernel to our original $F$ to get the result.
Suppose that $[a_{l}^* q^{l}]$ vanishes. Then we get a relation:
\[ [a_{l}^* q^{l}] = \lambda(f) + \phi, \]
where $f \in \HolModForm(\Sp'_2(\Z), \Weil_0^*, \frac{m}{2})$ and $\phi \in \PowSer^{\widetilde{M}}(\Sp'_2(\Z), \Weil_0^*)$.
Therefore $f=0$ as above, a contradiction.

The statement about integrality follows from the fact that all components of the constructed $f$s are
in fact modular forms for $\Gamma(N)$ for some $N$. For integral weight it is well-known that their Fourier coefficients have a bounded denominator. Multiplication with the common denominator yields a form with integral coefficients. For half-integral weight, use the reduction to integral weight method of the proof of Proposition \ref{LEMMAWEILGALOIS}.
\end{proof}

We will now show several theorems, stating the existence of Borcherds lifts, whose divisor has special properties. This is an essential
ingredient in the calculation of arithmetic volumes later. To not interrupt the discussion, we refer to the appendix \ref{LEMMATAQUADRATICFORMS}
for several elementary lemmas on quadratic forms which are needed and to appendix \ref{LACUNARITY} for several facts about vanishing of modular forms with
sparse Fourier coefficients.

\begin{SATZ}\label{THEOREMBORCHERDSPREP1}
\begin{enumerate}
\item
Let $\nL_\Zpp$ be an isotropic lattice \textbf{of signature $(m-2,2)$, $m \ge 4$}.
There is a saturated lattice $\nL_\Z \subset \nL_\Zpp$ and a 

$F \in \ModForm(\Sp'_2(\Z), \Weil(\nL_\Z^*/\nL_\Z), 2-\frac{m}{2})$ with integral Fourier coefficients, such that $\Psi(F)$ has non-zero weight,
all occurring $\nZ(\nL_\Z, <-m>, a_m; K)$, $m<0$ in $\Div(\Psi(F))$ are $p$-integral, i.e.
consist of canonical models of Shimura varieties $\nSh({}^K \nO(\nL'_\Zpp))$,
for various lattices $\nL'_\Zpp$ with unimodular (at $p$) quadratic form.

\item We find an $F$ above, such that $\Div(\Psi(F))$ contains, in addition, precisely one $\nZ(\nL_\Z, <m>, \kappa; K)$, $p | m$ with non-zero multiplicity.
\end{enumerate}
\end{SATZ}
\begin{proof}
1. Choose any saturated lattice of the form $\nL_\Z = H \oplus \nL_\Z'$ in $\nL_\Zpp$.
Let
\[ M(m) = \begin{cases} 0 & |m|_p = 1, \\ \Weil_{m}^* & |m|_p \not= 1. \end{cases} \]

Then all $\nZ(\nL_\Z, <-m>, a_m; K)$ with $a_m \perp M(m)$
consist of $p$-integral canonical models of Shimura varieties $\nSh({}^K \nO(\nL'_\Zpp))$ because $v^\perp$, for any
$v \in \nL_\Zp$ with $Q_\nL(v)=m$, $|m|_p=1$, is unimodular.

In view of Theorem \ref{BPADELIC}, to construct the required Borcherds form, by Lemma \ref{LEMMABORCHERDSPREP2}, 1., we have to show that any modular form
\[ f \in \HolModForm(\Sp'_2(\Z), \Weil^*, \frac{m}{2}), \]  
whose Fourier coefficients are supported only on $M$, vanishes.
But every component of $f$ is (in particular) a modular form for some $\Gamma(N)$, $p\nmid N$. It vanishes by Lemma \ref{LEMMAMODFORM1}.

For 2., use Lemma \ref{LEMMABORCHERDSPREP2}, 2.
\end{proof}

\begin{SATZ}\label{THEOREMBORCHERDSPREP2}
\begin{enumerate}
 \item Let $\nL_\Zpp$ be a lattice \textbf{of signature $(3,2)$ and Witt rank 1}.
 Up to multiplication of $Q_\nL$ with a scalar $\in \Z_{(p)}^*$,
there is a saturated lattice $\nL_\Z \subset \nL_\Zpp$ and a $F \in \ModForm(\Mp_2(\Z), \Weil, 2-\frac{m}{2})$ with integral Fourier coefficients, such that $\Psi(F)$ has non-zero weight,
all occurring $\nZ(\nL_\Z, <-m>, a_m; K)$ in $\Div(\Psi(F))$ are $p$-integral, i.e.
consist of canonical models of Shimura varieties $\nSh({}^K \nO(\nL'_\Zpp))$,
for various lattices $\nL'_\Zpp$ with unimodular (at $p$) quadratic form (i.e. with $K$ admissible), such that
$\nL'_\Q$ has signature $(2, 2)$ and Witt rank 1.

\item Up to multiplication of $Q_{\nL'}$ with a scalar $\in \Z_{(p)}^*$, for every $\nL_\Zpp'$ \textbf{of signature $(2,2)$, Witt rank 0}, we find a lattice $\nL_\Zpp$ of signature $(3, 2)$, Witt rank 1, and an $F$ as in 1., such that in $\Div(\Psi(F))$ occurs, in addition to the subvarieties above, precisely one \mbox{$\nZ(\nL, <l>, \kappa; K)$} with non-zero coefficient,
consisting of canonical models of Shimura varieties $\nSh({}^{K_i} \nO(\nL'_\Zpp))$ (with the given $\nL'_\Zpp$) for various different admissible $K_i$'s.
\end{enumerate}
\end{SATZ}
\begin{proof}
1. By Lemma \ref{LEMMA210REPN}, we may multiply $Q_\nL$ by a scalar in such a way that 
there is a saturated lattice $\nL_\Z \subset \nL_\Zpp$ with cyclic $\nL_\Z^*/\nL_\Z$ of order $2D'$, where $D'$ is square-free, of the form
\[ \nL_\Z = H \oplus \nL_\Z'. \]

By \ref{EXPLICITWEIL}, we know that $\Weil(\nL_\Z^*/\nL_\Z)$ decomposes
\[ \Weil(\nL_\Z^*/\nL_\Z) = \bigotimes_{l|D} \Weil(\nL_\Zl^*/\nL_\Zl), \]
and for each $l$ we have a decomposition 
\[ \Weil(\nL_\Zl^*/\nL_\Zl) = \Weil(\nL_\Zl^*/\nL_\Zl)^+ \oplus \Weil(\nL_\Zl^*/\nL_\Zl)^-, \]
into irreducible representations, here $\nL_{\Z_l}^*/\nL_{\Z_l}$ cyclic of odd prime order (respectively of order 2 or 4) is used.
(If the order is 2, $\Weil(\nL_{\Z_2}^*/\nL_{\Z_2})^-$ is zero).

We will work with the {\em irreducible} representation 
\[ \Weil(\nL_\Z^*/\nL_\Z)^+ := \bigotimes_l \Weil(\nL_\Zl^*/\nL_\Zl)^+, \]
with product basis build from the bases $(\chi_{\kappa}+\chi_{-\kappa}) \in \Weil(\nL_\Zl^*/\nL_\Zl)^+$
for $\kappa \in (\nL_\Zl^*/\nL_\Zl)/(\pm 1)$.
There will be a basis vector of the form 
\[ \gamma := \sum \chi_{\kappa}, \]
where the sum runs only over {\em primitive} $\kappa \in \nL_\Z^*/\nL_\Z$.

Let 
\[ M(m) = \begin{cases} (\C \gamma)^\perp & |m|_p = 1, \\ (\Weil(\nL_\Z^*/\nL_\Z)^+)^*_m & |m|_p \not = 1. \end{cases} \]

Now for $v \in \nL_\Q \cap \kappa$, with $\kappa$ such that $\chi_\kappa$ occurs in $\gamma$ (in particular $\kappa$ primitive), and $|Q_\nL(v)|_p=1$, $v^\perp$ is unimodular. It is also isotropic because by Lemma \ref{LEMMA210REPN} 
$m$ is already represented by $\nL_\Zh'$.

In view of Theorem \ref{BPADELIC} again, to construct the required Borcherds form, by Lemma \ref{LEMMABORCHERDSPREP2}, 2., we have to show that any modular form
\[ f \in \HolModForm(\Mp_2(\Z),(\Weil(\nL_\Z^*/\nL_\Z)^+)^*,\frac{m}{2}) \]
with Fourier coefficients supported only on $M$ vanishes.
Note that $\Weil(\nL_\Z^*/\nL_\Z)^+$ is defined over $\Q$ as a subvectorspace.

Now $\gamma \circ f$ vanishes by \ref{LEMMAMODFORM1}. Since $\Weil(\nL_\Z^*/\nL_\Z)^+$ is irreducible, Lemma \ref{LEMMAMODFORM3} tells us 
$f=0$.

2. Use \ref{LEMMA220INCLUSION} to construct the lattice $\nL_\Z$, and $\nZ(\nL_\Z, <x>, \chi_{\nL_\Zh}; K)$ consists obviously of the required models.
Since $\chi_{\nL_\Zh} \not\perp \gamma^\perp$, apply Lemma \ref{LEMMABORCHERDSPREP2}, 2.
\end{proof}

\begin{SATZ}\label{THEOREMBORCHERDSPREP3}
\begin{enumerate}
 \item 
Let $\nL_\Zpp$ be an \textbf{isotropic lattice of signature $(2,2)$ and discriminant $q$, $q$ a prime $\equiv 1 \enspace (4)$}.

Up to multiplication of $Q_\nL$ with a scalar $\in \Z_{(p)}^*$,
there is a lattice $\nL_\Z \subset \nL_\Zpp$ and a $F \in \ModForm(\Sp_2(\Z), \Weil(\nL^*_\Z/\nL_\Z), 2-\frac{m}{2})$ with integral Fourier coefficients, such that $\Psi(F)$ has non-zero weight,

all occurring $\nZ(\nL, <-m>, a_m; K)$ in $\Div(\Psi(F))$ are $p$-integral, i.e.
consist of canonical models of Shimura varieties $\nSh({}^K \nO(\nL'_\Zpp))$,
for various lattices $\nL'_\Zpp$ with unimodular quadratic form, such that
$\nL'$ has signature $(1, 2)$ and Witt rank 1.
\item 
For every $\nL_\Zpp'$ \textbf{of signature $(1,2)$, Witt rank 0}, up to multiplication of $Q_{\nL'}$ with a scalar $\in \Z_{(p)}^*$, we find a lattice $\nL_\Z$ of signature $(2, 2)$, Witt rank 1 and discriminant $q$, $q$ a prime $\equiv 1 \enspace (4)$ (as in i) and an $F$ as in 1. such that in $\Div(\Psi(F))$ occurs, in addition to the subvarieties above, precisely one $\nZ(\nL, <l>, \varphi; K)$ with non-zero coefficient,
consisting of canonical models of Shimura varieties $\nSh({}^{K_i} \nO(\nL'_\Zpp))$ (with the given $\nL_\Zpp$) for various different admissible $K_i$'s.
\end{enumerate}

\end{SATZ}
\begin{proof}
1. Let $\Weil^+$ be as in the last theorem.

By Lemma \ref{LEMMA210REPN}, we may multiply $Q_\nL$ by a scalar such that 
there is a lattice $\nL_\Z \subset \nL_\Zpp$, with cyclic $\nL_\Z^*/\nL_\Z$ of order $q$, of the form
\[ \nL_\Z = H \oplus <x^2+xy+\frac{1-q}{4}y^2>. \]

A $\nZ(\nL, <m>, \kappa; K)$ (for admissible $K$) consists of Shimura varieties of the required form, if $v^\perp$ for $v \in \pm\kappa+\nL_\Zh$, $Q_\nL(v)=m$ is isotropic. 
This is the case, if and only if $m$ is represented by $<x^2+xy+\frac{1-q}{4}y^2>$. Hence define 

\[ M(m) := \begin{cases} (\Weil^+)_{m}^* & |m|_p \not= 1, \\ 
(\Weil^+)_{m}^* \cap R(m)^\perp  & |m|_p = 1, \end{cases} \]
where $R(m) = \{ f \in \Weil_{0,m} \where \exists v \in <x^2+xy+\frac{1-q}{4}y^2>_\Af: f(v)\not=0, Q_\nL(v)=m \}$.

We know (\ref{EXAMPLEGAMMA0P_2}) that $(\Weil^+)_{\frac{j}{q}}^*$ is zero, if $\chi_q(j)=-1$. The definition of $M(m)$ and \ref{LEMMA110REPN} imply that for {\em primes} $p' \not = p$ with $\chi_q(p')=1$, $M(\frac{p'}{q})$ is also zero.

In view of Theorem \ref{BPADELIC} again, to construct the required meromorphic modular form, by Lemma \ref{LEMMABORCHERDSPREP2}, we have to show that any modular form
$f \in \HolModForm(\Sp_2(\Z), (\Weil^+)^*, \frac{m}{2})$ with Fourier coefficients supported only on $M$ vanishes.
From the remarks in \ref{EXAMPLEGAMMA0P_2} it follows that $f'$ defined by
$f'(\tau):=\rho(\M{0&-1\\1&0})f(q\tau)(1)$ is a scalar multiple of $W_q(f(1))$, hence a modular form for $\Gamma_0(q)$ again.
For primes $p' \not= p$, $M(\frac{p'}{q})$ is zero, hence the Fourier coefficient $a_{p'}$ of $f'$ is zero.
Similarly for $(n,q)=1$, $\chi_p(n)=-1$, $a_{n}$ is zero. Therefore the vanishing of $f'$ now follows from Proposition \ref{LEMMAMODFORM2} and that of $f$ by Lemma \ref{LEMMAMODFORM3}.

2. Use \ref{LEMMA210INCLUSION} to construct the lattice $\nL_\Z$ and $\nZ(\nL_\Z, <x>,\chi_{\nL_\Zh}; K)$ consists obviously of the required models.
Since $\chi_{\nL_\Zh} \not\perp R(m)^\perp$, apply Lemma \ref{LEMMABORCHERDSPREP2}, 2.
\end{proof}

\begin{SATZ}\label{THEOREMBORCHERDSPREP4}
Let $\nL_\Zpp$ be a \textbf{lattice of signature $(3,2)$ of Witt rank 2}.
There is a lattice $\nL_\Z \subset \nL_\Zpp$, such that up to multiplication of $Q_\nL$ with a scalar $\in \Z_{(p)}^*$,
there is an $F \in \ModForm(\Mp_2(\Z), \Weil(\nL^*_\Z/\nL_\Z), 2-\frac{m}{2})$ with integral Fourier coefficients, such that $\Psi(F)$ has non-zero weight,
$\Div(\Psi(F))$ consist of exactly one $\nZ(\nL_\Z, <-l>, a_l; K)$ with non-zero coefficient,
which itself consists of canonical models of a Shimura varieties $\nSh({}^K \nO(\nL'_\Zpp))$,
for {\em any} (a priori) given lattice $\nL'_\Zpp$ of signature $(2,2)$, Witt rank $\ge 1$.
\end{SATZ}
\begin{proof}
Define $\nL_\Zpp = \nL_\Zpp' \perp <x>$, where $-x$ is represented by $\nL_\Zpp'$ and $|x|_p=1$.
Up to multiplication of $Q_\nL$ by a scalar, we find a saturated lattice $\nL_\Z = H^2 \perp <1> \subset \nL_\Zpp$.
$\Weil = \Weil(\nL^*_\Z/\nL_\Z)$ is irreducible in this case. W.l.o.g $x$ is assumed to be represented by $\nL_\Zh$.
Define 
\[ M(m) = \begin{cases} \Weil_m^* & m\not=x, \\ 0 & m =x. \end{cases} \]
In view of Theorem \ref{BPADELIC} again, to construct the required meromorphic modular form, by Lemma \ref{LEMMABORCHERDSPREP2}, we have to show that any modular form
$f \in \HolModForm(\Mp_2(\Z), \Weil^*, \frac{5}{2})$ with Fourier coefficients supported only on $M$ vanishes.
This follows because it is well-known that
\begin{gather*} 
\HolModForm(\Mp_2(\Z), \Weil^*, \frac{5}{2})
\end{gather*} 
is generated by the Eisenstein series whose coefficient $a_{x}$ is non-zero because $x$ is represented by $\nL_\Zh$.
\end{proof}

\section{Borcherds' products and Arakelov geometry}\label{BORCHERDS}

\begin{PAR}
Let $\nL_\Q$ be a f.d. vector space with non-degenerate quadratic form $Q_\nL$ of signature $(m-2,2)$ with $m \ge 4$, and assume the Witt rank to be 1 if $m=4$.
Let $\nL_\Z \subset \nL_\Q$ be a lattice of the form $\nL_\Z = H \perp \nL_\Z'$.

Take a modular form $F$ as in \ref{BORCHERDSLIFT} with Fourier expansion
\[  F(\tau) = \sum_{m \in \Q} a_m q^m, \]
where $a_m \in \Weil(\nL^*_\Z/\nL_\Z) \subset S(\nL_\Af)$, with $a_0(0) \not= 0$.

In this section we prove a relation (Theorem \ref{MAINTHEOREMAVERAGE}, 2.) between arithmetic volumes of different Shimura varieties of orthogonal type
which involves Borcherds lifts --- in particular their integral (\ref{BPADELIC}, 3). It will be the main ingredient in the proof of the main result of this article, Main Theorem \ref{RESULTGLOBAL}. As is explained there, the arithmetic formula expresses --- in some sense --- the special derivative of the orbit equation \ref{GLOBALORBITEQUATION}.
We mention also a rather well-known geometric analogue (Theorem \ref{MAINTHEOREMAVERAGE}, 1.) which expresses the special value of the orbit equation. Recall the $\SO$-equivariant
line bundle $\overline{\mathcal{E}}=(\mathcal{E},h_\mathcal{E})$ with Hermitian metric
on $\nX_\nO$. We denote $E=\mathcal{E}_\C$ (cf. \ref{BORCHERDSBUNDLE}).
\end{PAR}

\begin{SATZ}\label{MAINTHEOREMAVERAGE}
Assume \ref{MAINCONJECTURE}.

Under the conditions above, we have:
\begin{enumerate}
\item (geometric formula)
\begin{align*} 
 & \sum_{q} \widetilde{\mu}(\nL_\Z, <-q>, a_q; 0) \vol_E(\nSh({}^{K}_\nRPCD \nO(\nL))) \\
 =& \sum_{q} \vol_E(\nZ(\nL_\Z, <-q>, a_q; K))
\end{align*}
\item (arithmetic formula)
\begin{align*} 
 & \sum_{q} \widetilde{\mu}'(\nL_\Z, <-q>, a_q; 0) \vol_E(\nSh({}^{K}_\nRPCD \nO(\nL)))  \\
 +& \sum_{q} \widetilde{\mu}(\nL_\Z, <-q>, a_q; 0) \avol_{\overline{\mathcal{E}},p}(\nSh({}^{K}_\nRPCD \nO(\nL))) \\
 =& \sum_{q} \avol_{\overline{\mathcal{E}},p}(\nZ(\nL_\Z, <-q>, a_q; K))
\end{align*}
in $\R^{(p)}$. 
\end{enumerate}
\end{SATZ}

\begin{proof}
Let $f_0:= \Psi(F)$ be the Borcherds lift of $F$ --- cf. section \ref{BORCHERDSLIFT} ff. and especially Theorem \ref{BPADELIC}.

We sketch the proof of 1. The assumption on $m$ and the Witt rank imply that the
Eisenstein series associated with $\nL_\Z$ (see \cite[Section 4]{Paper1}) is holomorphic, i.e. there is a
\[ G \in \HolModForm(\SL'_2(\Z),\Weil(\nL_\Z^*/\nL_\Z)^*,\frac{m}{2}) \]
with $a_n(G)(\varphi) = -\widetilde{\mu}(\nL_\Z, <n>, \varphi; 0)$ for $n>0, n\in \Q$ and $a_0(G)(\varphi) = \varphi(0)$.
Hence relation (\ref{SERREDUALITY}) yields:
\begin{equation} \label{bcleq1}
\sum_{q} \widetilde{\mu}(\nL_\Z, <-q>, a_q; 0) = a_0(0).
\end{equation}
In addition, we have the relation
\begin{gather}
 a_0(0) \int_M  \chern_1(\nMorphSt^*E, \nMorphSt^*h)^m = 2\int_M \ddc \log h(f_0) \wedge \chern_1(\nMorphSt^*E,\nMorphSt^*h)^{m-1} \nonumber \\
  = 2 \int_{\Div(f_0) \cap M}  \chern_1(\nMorphSt^*E,\nMorphSt^*h)^{m-1}.   \label{bcleq2}
\end{gather}
The justification that here everything is integrable and that the second equality is true without contributions from the boundary, can be found in \cite{Br2}. From (\ref{bcleq1}) and (\ref{bcleq2}), 1. follows, taking the description of the divisor of $f_0$ (Theorem \ref{BPADELIC}) into account.

2. We may calculate
$\avol_{\overline{\mathcal{E}},p}(\nSh({}^{K}_\nRPCD\nO(\nL)))$
in the following way:

First assume w.l.o.g. (e.g. by taking a lattice with large discriminant in the construction of
$F$ or by just pulling back $\Psi(F)$ afterwards) that $K$ is neat.

We know by \ref{BORCHERDSBUNDLEAMPLE} that $\nMorphSt^*(\mathcal{E})$ is ample on $\nSh({}^K \nO)$ 
and that some power of it has no base points on 
$\nSh({}^K_\nRPCD \nO)$. Choose some sufficiently fine, smooth w.r.t. $K$, complete, and projective $\nRPCD$
(by Lemma \ref{LEMMAINDEPRPCD} the arithmetic volume does not depend on this choice) such that all models exist 
and all special cycles involved embed closedly as projective schemes (cf. \ref{DEFSPECIALCYCLEP}).
$\nMorphSt^*(\mathcal{E})^{\otimes k}$, for some positive $k$, defines a morphism 
\[ \alpha: \nSh({}^K_\nRPCD \nO) \rightarrow \PP^r(\Zpp), \]
whose restriction to $\nSh({}^K \nO)$ is an embedding.
Let $D$ be the boundary divisor. By Main Theorem \ref{MAINTHEOREM1a}, $\nSh$ and $D$ are defined over $\Zpp$.
We have $\dim(\alpha(D)_\C) = \text{Witt rank of $\nL_\Q$-1}$, hence $\dim(\alpha(D)) = \text{Witt rank of $\nL_\Q$}$,
because $\nMorphSt^*(\mathcal{E})_\C^{\otimes k}$ induces the Baily-Borel compactification.
In addition, we may also choose $k$ and $F$ such that $f_0 \in H^0(\nSh({}^K_\nRPCD \nO), \nMorphSt^*(\mathcal{E})^{\otimes k}_{horz})$ (Theorem \ref{BPADELIC}).

Up to increasing $k$ again, we may find hyperplanes $H_1, \dots, H_{m-2} \subset \PP^r(\Zpp)$
which intersect properly with $\alpha(\Div(f_0))$, and we
may avoid that they intersect simultaneously in any closed subset $Z$ of dimension $<m-2$. 

Our assumptions on $m$ and the Witt rank imply that we can take $Z$ to be the union of 
$\alpha(D)$ (dim=Witt rank$<m-2$) and the locus, where $\alpha_{\F_p}$ is not finite (dim $< m-2$). 

We may clearly assume that
already $H_{1,\C}$ and $H_{2,\C}$ do not intersect in $\alpha(D)_\C$. 

Let $f_1, \dots, f_{m-2}$ be the corresponding sections.

The construction assures that also $\Div(f_0), \Div(f_1), \dots, \Div(f_{m-2})$ intersect properly and not in $D$ simultaneously.
Therefore by Theorem \ref{BPADELIC}, 2., 
\[ \Div(f_0) \cdot \Div(f_1) \cdots \Div(f_{m-2}) = \frac{1}{2} \sum_{q<0} \nZ(\nL_\Z, <-q>, a_q; K) \cdot \Div(f_1) \cdots \Div(f_{m-2}). \]

Note that the arithmetic volume is the sum of this expression and
\begin{equation}\label{eqsp}
  \frac{1}{2} \int g_0 \ast g_1 \ast \cdots \ast g_{m-2},  
\end{equation} 
where
\[ g_i = -\log \nMorphSt^*h(f_i) \]
(and the star product has to be interpreted in the sense of \ref{STARPRODUCT}).
In the following we will denote 
\[ \Omega := \chern_1(\nMorphSt^*(\mathcal{E}), \nMorphSt^*(h_\mathcal{E})) \]
to be the first Chern form of the bundle $\nMorphSt^*(\mathcal{E})$ with respect to the ($\log$-singular) Hermitian metric $\nMorphSt^* h_\mathcal{E}$. 

On any parametrization defined by a point-like boundary component as in \ref{ORTHBOUNDARYCOMP} it is given by
\[ -\dd \dd^c \log(Q_{z'}(Y)) \]
where $Z = X + iY \in \nU_\C \cong \nX_{\nH_0[\nU,0]}$.

By Theorem \ref{THEOREMSTARPRODUCT} below, we may write (\ref{eqsp}) as
\begin{gather*}
 = \frac{1}{2}\int_{\nSh({}^K_\nRPCD \nO)_{\C}} g_0 k\Omega \wedge \cdots \wedge k \Omega \\  
 + \frac{1}{2} \int_{\Div(f_0) \cap \nSh({}^K_\nRPCD \nO)_\C} g_1 \ast \cdots \ast g_{m-2},  
\end{gather*}

hence we get the equation ($\in \R^{(p)}$)
\begin{align*}
 & k^{m-1} \avol_{\overline{\mathcal{E}},p}(\nSh({}^K_\nRPCD \nO(\nL))) \\
 =& \frac{1}{2} k^{m-2} \int_{\nSh({}^K_\nRPCD \nO)_{\C}} g_0 \Omega \wedge \cdots \wedge \Omega   \\  
 &+ \frac{1}{2} k^{m-2} \sum_{q} \avol_{\overline{\mathcal{E}},p}(\nZ(\nL_\Z, <-q>, a_q; K)),
\end{align*}
(cf. also \ref{CHOWHEIGHT}) and therefore the required one, taking into account that
\[ \frac{1}{2} \sum_{q} \widetilde{\mu}(\nL_\Z, <-q>, a_q; 0) = k = \frac{a_0(0)}{2} \]
(geometric formula --- part 1.) and by Theorem \ref{BPADELIC}, 3.:
\begin{gather*}
 \int_{\nSh({}^K_\nRPCD \nO)_{\C}} g_0 \Omega \wedge \cdots \wedge \Omega 
\equiv \\
- \vol_{\mathcal{E}}(\nSh({}^K_\nRPCD \nO)_{\C}) \sum_{q} \widetilde{\mu}'(\nL_\Z, <-q>, a_q; 0)
\end{gather*}
in $\R^{(p)}$.
\end{proof}

\begin{SATZ}\label{THEOREMSTARPRODUCT}
Let $\nL_\Q$ be as before, i.e. of signature $(m-2,2)$ with $m\ge 5$, or $m=4$ and $\nL_\Q$ of Witt rank 1.

Let $f_0, \dots, f_{m-2}$ be sections of $\nMorphSt^*(\mathcal{E})^{\otimes k}$ on $\nSh({}^K_\nRPCD \nO)$, such that $\Div(f_0), \dots, \Div(f_{m-2})$ intersect properly:
\[ \bigcap_i \supp(\Div(f_i)) = \emptyset, \] 
such that 
\[ \supp(\Div(f_1)) \cap \supp(\Div(f_2)) \cap \supp(D) = \emptyset\]
(Witt rank 2), resp.
\[ \supp(\Div(f_1)) \cap \supp(D) = \emptyset \]
(Witt rank 1).

Then we have
\begin{gather*}
\int_{\nSh({}^K_\nRPCD \nO)_{\C}} g_0 \ast g_1 \ast \cdots \ast g_{m-2} \\ 
=\int_{\nSh({}^K \nO)_{\C}} g_0 k \Omega \wedge \cdots \wedge k \Omega \\ 
+  \int_{\Div(f_0) \cap \nSh({}^K \nO)_\C} g_1 \ast \cdots \ast g_{m-2}.  
\end{gather*}
All occurring integrals exist.
\end{SATZ}
\begin{proof}
According to \cite[Theorem 1.14]{BBK}, we have
\begin{align*}
 & \int_{\nSh({}^K_\nRPCD \nO)_\C} g_0 \ast \cdots \ast g_{m-2} \\
 = &\lim_{\varepsilon \rightarrow 0} \left( \int_{\nSh({}^K_\nRPCD \nO)_\C-B_\varepsilon(D)} g_0 (k\Omega) \wedge \cdots \wedge (k \Omega)  \right. \\
& - \left. \int_{\partial B_\varepsilon(D)} g \wedge \dd^c g_{0} - g_{0} \wedge \dd^c g \right) \\
& +  \int_{\Div(f_0) \cap \nSh({}^K \nO)_\C} g, 
\end{align*}
where 
\[ g=g_1 \ast \cdots \ast g_{m-2}. \]

The integral
\[ \int_{\nSh({}^K \nO)} g_0 (k\Omega) \wedge \cdots \wedge (k \Omega) \] 
exists by \cite[Theorem 2]{Br2}
because we excluded the cases $m=3$, Witt rank 1 (modular curves) and $m=4$, Witt rank 2 (product of
modular curves).

\begin{PAR}
We will need special neighborhoods of points on the boundary divisor of $\nSh({}^K_\nRPCD \nO)$. By Main Theorem
\ref{MAINTHEOREM1a}, 3. any such point lies on a stratum
\[ \left[ \Stab_{\Gamma([\sigma])} \backslash \nSh({}^{K_{[\sigma]}} \nSDB_{[\sigma]}) \right]_\C, \]
where $\nSDB$ corresponds (\ref{ORTHBOUNDARYCOMP}) to an isotropic line or an isotropic surface in $\nL_\Q$.
(The corresponding boundary stratum in the Baily-Borel compactification is 0, resp. 1-dimensional.) 
We will prepare special neighborhoods of these points and call them of \textbf{first (resp. second) type} for the rest of this section. By the boundary isomorphism \ref{MAINTHEOREM1a}, 4., we transfer the neighborhood to the boundary
of the mixed Shimura variety (compactfied only along the unipotent fibre):
\[ \nSh({}^{K'}_{(\nRPCD')^0} \nSDB)_\C,   \]
which is a torus embedding constructed from the (relative) torus
$\nSh({}^{K'} \nSDB)_\C$
over
$\nSh({}^{K''} \nSDB/\nU_\nSDB)_\C$ (\ref{COMPUNIPOTENTFIBRE}). 
\end{PAR}

\begin{PAR}\label{nbhdfirsttype}
For a point of the first type, we have  
\begin{equation}\label{splitting}
\nSDB \cong \nH_0[0, \nU_\Zpp],
\end{equation}
where $\nU_\Zpp = \nI \otimes (\nI^\perp/\nI)$
(\ref{ORTHBOUNDARYCOMP}) 
and
the trivialization $s$ of $\nMorphSt^*\mathcal{E}$ described in \ref{QEXPANSIONPREP}.
It is of the form
\[ s(z_1,\dots, z_{m-2}) = (\log(z_1)\lambda_1 + \cdots + \log(z_n)\lambda_n + \psi(z_{k+1}, \dots ,z_{m-2}) ) z' \]
for $z_1,\dots,z_{m-2} \in B_\epsilon(0)^{m-2}$, $z'$ is a vector in the fibre of $\mathcal{E}$ above the point corresponding to the splitting (\ref{splitting}), i.e. $z' \in \nI^*$.
The $\lambda_i$ span the rational polyhedral cone $\sigma$ corresponding to our boundary point and form part of a basis of $(\nU_\Q \cap K')(-1)$ ($\nRPCD$ is smooth w.r.t. $K$), and $\psi$ is a holomorphic embedding $B_\epsilon(0)^{m-2-k} \rightarrow \nU_\nSDB(\C) = \nU_\C$.

Let $r_i:=|z_i|$. The norm of $s$ is given by
\[ \sum_{i \le n,j \le n} \langle \lambda_i, \lambda_j \rangle_{z'} \log(r_i) \log(r_j) + \sum_{i \le n} \log(r_i) \Psi_i(z_1,\dots,z_{m-2}) + \Psi_0(z_1,\dots,z_{m-2}),  \]
where the $\Psi_i, i>0$ are smooth functions on $B_\epsilon(0)^{m-2}$ satisfying $\ddc \Psi_i = 0$, and $\Psi_0$ is smooth.
\end{PAR}

\begin{PAR}\label{nbhdsecondtype}
For a point of the second type, we have  
\begin{equation}\label{splitting2}
\nSDB \cong  \nH_1(\nI)[\Lambda^2 \nI, \nI \otimes (\nI^\perp/\nI)],
\end{equation}
(\ref{ORTHBOUNDARYCOMP}) 
and
the trivialization $s$ of $\nMorphSt^*\mathcal{E}$ described in \ref{QEXPANSIONPREP}.
It is of the form
\[ s(z_1, z_2, \cdots, z_{m-3}, z_{m-2}) \mapsto (\lambda \log(z_{1})) \circ (\psi_2(z_2,\cdots, z_{m-3})) (\psi_1(z_{m-2})). \]
Here $\psi_1$ is a local trivialization of $\mathcal{E}_\C$ over a small neighborhood over the image
of $\nX_{\nH_1(\nI)}$ in $\nX_{\nH_1(\nI)[\Lambda^2 \nI, \nI \otimes (\nI^\perp/\nI)]}$ such that $p(\psi_1(0))$ is
a representative of the projection of the boundary point. Since $\mathcal{E}$ is the tautological bundle, 
we may express this by saying that $\psi_1$ maps to $\nI^*$ inducing an embedding of $B_\epsilon(0)$ into $\PP(\nI^*)$.
$\psi_2$ is a holomorphic map $B_\epsilon(0)^{m-4} \rightarrow \nW_{\nSDB}(\C)$ which is modulo $\nU_\nSDB(\C)$ an open embedding into $\nV_\nSDB(\C)/\Stab(x)$ for all $x$ in the image of $p(\psi_1(z_{m-2}))$. $\lambda$ is a basis
of $(\nU_\nSDB(\Q)\cap K')(-1)$ and corresponds to the rational polyhedral cone $\sigma$ (here there is only one possible).

Its norm is of the form
\[ h_{\mathcal{E}}(s) = \log(r_1) \Psi_1(z_{m-2}) + \Psi_0(z_2,\dots,z_{m-3},z_{m-2}), \]
where $\ddc \Psi_1 =0$, and $\Psi_0$ is smooth. Note that $\Psi_1$ does only depend on $z_{m-2}$.
\end{PAR}

\begin{PAR}
We will prepare an $\varepsilon$-tube neighborhood of $D$ as follows.
We prepared neighborhoods around every point of the boundary above. Take a finite cover $\{U_i\}$ consisting of these.
Let $\tau_i$ be a partition of unity defined on $\nSh({}^K_\nRPCD \nO(\nL))_\C$ for the chosen cover, consisting of the neighborhoods considered in \ref{nbhdfirsttype} and \ref{nbhdsecondtype}.
Consider a neighborhood $U_i$ with coordinates $z_{i,1}, \dots, z_{i,m-2}$. $D$ has the
equation $(z_{i,1})^{a_1} \cdots (z_{i,m-2})^{a_{m-2}} = 0$. 
On intersections $U_i \cap U_j$ we have
\[ z_{j,k_j} = z_{i,k} (f_{i,k,0}^j(z_{i,1}, \dotsop^{\widehat{k}}, z_{i,m-2}) + z_{i,k} f_{i,k}^j(z_{i,1}, \dots, z_{i,m-2})),  \]
where $f_{i,k,0}^j(z_{i,1}, {\displaystyle \dotsop^{\widehat{k}}}, z_{i,m-2})$ does not depend on $z_{i,k}$ and is everywhere non-zero on the overlap in question and where $k_j$ is the index of the $k$-th divisor component in the
neighborhood $U_j$.
We define 
\[ r'_{i,k} = \sum_j \tau_j  | z^{j,k_j} | \quad \text{on } U_i. \]
If the $k$-komponent does not exist in $U_j$, we omit the term in the sum.

Define $r' = \min_{k| a_k\not=0} r'_{i,k}$. This does not depend on $i$ anymore and is a well-defined function in
a neighborhood of the divisor.
A global $\varepsilon$-tube neighborhood around $D$ may
therefore be described as 
\[ B_\varepsilon(D) := \{ z \in \nSh({}^K_\nRPCD \nO(\nL))_\C \ |\ r'(z) \le \varepsilon \}. \]

In a local neighborhood $U_i$ write
\[ r_{i,k}' = r_{i,k} \left(  \sum_j \tau_j |f_{0,i,k}^j(z_{i,1}, \dotsop^{\widehat{k}}, z_{i,m-2})| + r_{i,k} g \right), \]
where $g$ is a bounded $C^\infty$ function. If $\varepsilon$ is small enough, $0<r_i'\le \varepsilon$ will ensure
$r_{i,k}=|z_{i,k}|>0$. 

On $U_i$ we may cover $\partial B_\varepsilon(D)$ by 
(we omit the index $i$ because it will be fixed from now on)
\[ St_k'(\varepsilon) = \{ z \in U_i \where r_{k}'=\varepsilon, r_j'\ge \varepsilon \text{ if } a_j \not= 0 \}, \]
for all $k$, where $a_k \not= 0$. Note that on overlaps the $r_{i,k}'$ and $r_{j,k}'$ are the same up
to permutation of the indices $k$.

We have for small $r_k'$ and some bounded $C^\infty$-function $h$:
\begin{align}
r_k =& r_k' \left( \left(  \sum \tau_j |f_{0,k}^j(z_1, \dotsop^{\widehat{k}}, z_{m-2})|\right)^{-1} + r_k' h \right), \nonumber
\end{align}
\begin{align}
\D{r_k} =& \D{r_k}'  \left( \sum \tau_j |f_{0,k}^j(z_1, \dotsop^{\widehat{k}}, z_{m-2})| \right)^{-1}  \nonumber \\
 &+ r_k' \D{r_k}' h  \nonumber \\
 &+ r_k' \D{ \left( \left( \sum \tau_j |f_{0,k}^j(z_1, \dotsop^{\widehat{k}}, z_{m-2})|\right)^{-1} + r_k' h \right)}.
\label{eqnrprime}
\end{align}
\end{PAR}

\begin{PAR}In any of the neighborhoods $U_i$,
we will write the quantities in question in the basis
\[ \D{\varphi}_1, \cdots, \D{\varphi}_n, \D{r}_1, \dotsop^{\widehat{k}}, \D{r}_n, \D{r}_k'. \]
Only the term
\[ \omega := \D{\varphi}_1 \wedge \cdots \wedge \D{\varphi}_n \wedge \D{r}_1 \wedge \dotsop^{\widehat{k}} \wedge \D{r}_n \]
gives a non-zero contribution in integrals over $St_k'(\varepsilon)$. In what follows, we understand:
\[ |f(z)  \omega| := |f(z)| \omega. \]
We number the coordinates in such a way that $D$ is given by the equation $z_1\cdots z_n=0$ on $U_i$. 

Recall the notion of $\log$-$\log$-growth of infinite order \cite[2.2]{BKK2}.
A {\bf $\log$-$\log$-form} \cite[2.22]{BKK2} is a differential form generated (over smooth forms) by functions of $\log$-$\log$-growth of infinite order and the differentials
\[ \frac{\D{\varphi}_1}{ \log(r_1)}, \dots, \frac{\D{\varphi}_{n}}{ \log(r_{n})}, \frac{\D{r_1}}{r_1 \log(r_1)}, \dots, \frac{\D{r_n}}{r_n \log(r_n)} \]
such that $\dd, \dd^c$ and $\ddc$ of it are again of this type.

Call a form a \textbf{(*)-form} for an index $k$, if it is generated by functions of $\log$-$\log$-growth 
of infinite order and the differentials
\[ \D{\varphi}_1, \dots, \D{\varphi}_n, \frac{\D{r_1}}{r_1 \log(r_1)}, \dotsop^{\widehat{k}}, \frac{\D{r_n}}{r_n \log(r_n)} \]
and $f \D{r_k'}$ for any $f$, which is smooth outside $D$.
\end{PAR}

\begin{LEMMA}\label{INTSTARFORM}
For every (*)-form $\sigma$ for the index $k$:
\[ \int_{St_k'(\varepsilon)} |\sigma| = O(\log(|\log \varepsilon |)^M)  \]
for some $M$.
\end{LEMMA}
\begin{proof}
It suffices to show that on each of the sets
\[ St'_k(\varepsilon, R) = \{ z \in B(R) \where r_k'=\varepsilon, r_j \ge \delta \varepsilon \text{ if } a_j \not= 0 \}, \]
for some small $\delta$ (note that we wrote $r_j$ instead of $r_j'$), we have the estimate:
\[ \int_{St_k'(\varepsilon, R)} |\sigma| = O(\log(|\log \varepsilon |)^M). \]
This is easy, see e.g. \cite{Br2}. The (maybe singular) terms $f \D{r_k'}$ play no role because terms involving $\D{r_k'}$
do not give any contribution to the integral.
\end{proof}

\begin{KOR}\label{INTLOGLOGGROWTH}
For every $\log$-$\log$ growth form $\xi$:
\[ \lim_{\varepsilon \rightarrow 0} \int_{St_k'(\varepsilon)} |\xi| = 0.  \]
\end{KOR}
\begin{proof}
We have by (\ref{eqnrprime})
\[ \frac{1}{\log(r_k)} \D{\varphi_i} \prec \frac{1}{\log(r_k')} \sigma, \]
and
\[ \frac{\D{r_k}}{r_k \log(r_k)} \prec \frac{1}{\log(r_k')} \sigma, \]
where the $\sigma$ are (*)-forms w.r.t. $k$.

If $\xi$ has correct degree, by definition, it involves at least either $\frac{1}{\log(r_k)} \D{\varphi_k}$
or $\frac{\D{r_k}}{r_k \log(r_k)}$.
Hence
\[ \xi \prec \frac{1}{\log(r_k')} \sigma, \]
for some (*)-form $\sigma$. The statement follows using \ref{INTSTARFORM} because, of course,
$\log$ grows faster than any power of $\log$-$\log$.
\end{proof}

\begin{PAR}
First of all $\supp(\Div(f_1)) \cap \supp(\Div(f_2)) \cap \supp(D) = \emptyset$ by assumption, hence we may represent
(cf. \ref{COMPSTARPRODUCT}):
\[ g = \sigma g_{1,2} \Omega^{n-2} + \ddc ((1-\sigma) g_{1,2}) g', \]
where $\sigma$ is equal to 1 in a neighborhood of $D$ and $g_{1,2} = g_1 \ast g_2$. 
We may represent this by
\[ g_{1,2} = \sigma_{1} g_{1} \Omega + \ddc (\sigma_2 g_1) g_2, \]
where $\sigma_1 + \sigma_2 = 1$ is a partition of unity of the following form:
The intersection of $\Div(f_1)$ with $D$ occurs precisely on the pre-image under the
projection on the Baily-Borel compactification of a set of isolated points in the
1-dimensional boundary stratum isomorphic to some $\nSh({}^K \nH(\nI))$. Choose a
$C^\infty$-function $p\sigma$ on $\nSh({}^K\nH(\nI))$ with support on some disc around one of these points,
which is 1 on a smaller neighborhood of the point.

We may assume that in $B_\varepsilon(D)$ for small $\varepsilon$,
$\sigma_1$ is just the pre-image of $p\sigma$ under the projection $\nSh({}^{K'}_{(\Delta')^\circ}\nSDB) \rightarrow \nSh({}^K\nH(\nI))$.

We have hence in a neighborhood of $D$:
\begin{eqnarray}\label{eqng}
g &=&  \sigma_{1} g_{1} \Omega^{n-1} + \sigma_2 g_2 \Omega^{n-1} \\
&& + \left( (\dd \sigma_2 \wedge \dd^c g_1) g_2  
- (\dd^c \sigma_2 \wedge \dd g_1) g_2 + (\ddc \sigma_2) g_1 g_2 \right) \wedge \Omega^{n-2}.\nonumber
\end{eqnarray}

In any of the neighborhoods considered in \ref{nbhdfirsttype} and \ref{nbhdsecondtype} we have
\begin{equation}\label{eqng1}
 g_0 = \sum_{j \le n} a_j \log(r_j) + \log(h_\mathcal{E}(s)) + \psi(z)
\end{equation}
for some harmonic function $\psi$, where $s$ is the corresponding trivializing section of $\nMorphSt^*(\mathcal{E})$.

Recall that we want to show that

\[ \lim_{\varepsilon \rightarrow 0} \int_{\partial B_\varepsilon(D)} g \wedge \dd^c g_{0} - g_{0} \wedge \dd^c g = 0.  \]

Inserting (\ref{eqng}) and (\ref{eqng1}) into this, 
since in any case (by \ref{INTLOGLOGGROWTH}) the limit of integrals of $\log$-$\log$ forms is 0, we are reduced to show, one the one hand (first line in \ref{eqng}), that
\[ \lim_{\varepsilon \rightarrow 0 } \int_{St_k'(\varepsilon)} | \D \varphi_j \sigma_{l} g_{l} \Omega^{m-3} |  = 0\]
and
\[ \lim_{\varepsilon \rightarrow 0 } \int_{St_k'(\varepsilon)} | \log(r_j) \dd^c (\sigma_{l}g_{l}) \Omega^{m-3} |  = 0, \]
where $l=1,2$, 
in neighborhoods of both types. Every other term in (\ref{eqng1}) inserted into 
\[ \lim_{\varepsilon \rightarrow 0} \int_{St_k'(\varepsilon)} | g \wedge \dd^c g_{0} - g_{0} \wedge \dd^c g | \]
yields a limit over an integral of a form of $\log$-$\log$ type, which is zero by Lemma \ref{INTLOGLOGGROWTH}. 

If $j = k$ we may rewrite $r_j$ by means of $r_k'$ and apply (\ref{eqnlemmaomega1}) of Lemma \ref{LEMMAESTIMATEOMEGA}.
In the other case we apply (\ref{eqnlemmaomega2}) of Lemma \ref{LEMMAESTIMATEOMEGA}.
(Note that $\D \varphi_j$ is of the form  $\log(r_j)$ times a $\log$-$\log$-form.)
After this, the vanishing of the limit $\varepsilon \rightarrow 0$ follows from Lemma \ref{INTSTARFORM}.

On the other hand (second line in \ref{eqng}), we have to show
\begin{align}\label{eqnint1}
 \lim_{\varepsilon \rightarrow 0 } \int_{St_1'(\varepsilon)}|  \D \varphi_1 \left( \dd \sigma_2 \wedge \dd^c g_1)  - \dd^c \sigma_2 \wedge \dd g_1 + (\ddc \sigma_2) g_1 \right) \wedge g_2 \wedge \Omega^{m-4} | &= 0\\
 \label{eqnint2}
 \lim_{\varepsilon \rightarrow 0 } \int_{St_1'(\varepsilon)} | \log(r_1) \left( \dd \sigma_2 \wedge \dd^c g_1)  - \dd^c \sigma_2 \wedge \dd g_1 + (\ddc \sigma_2) g_1 \right) \wedge \dd^c g_2 \wedge \Omega^{m-4} | &= 0
\end{align}
Here $j=k=1$ because this is non-zero only in neighborhoods of the second type, and there cannot be an intersection of components above $\nSh({}^K\nH(\nI))$.

The estimate (\ref{eqnlemmaomega3}) of Lemma \ref{LEMMAESTIMATEOMEGA}, we get vanishing of the limit $\varepsilon \rightarrow 0$ of (\ref{eqnint1}) by Lemma \ref{INTSTARFORM} again.
\end{PAR}
For (\ref{eqnint2})
note that we constructed $\sigma_2$ (in $B_\varepsilon(D)$!) as the pre-image (under the projection $\nSh({}^{K_1}_{\nRPCD_1^\circ} \nSDB) \rightarrow \nSh({}^K\nH(\nI))$ of a function supported on
a small disc around the zeros or poles of $f_1$ in the boundary of the Baily-Borel compactification. Let $X$ be the pre-image of these discs in some neighborhood of the boundary in $\nSh({}^K_{\Delta_1^\circ}\nSDB)$.
Hence the form
\[ \dd \left( g_0 (\dd \sigma_2 \wedge \dd^c g_1 - \dd^c \sigma_2 \wedge \dd g_1 + (\ddc \sigma_2) g_1) \wedge \dd^c g_2 \wedge \dd^c \log(h_\mathcal{E}(s)) \wedge \Omega^{m-3} \right) \]
is supported only on $X$.
Applying the Theorem of Stokes to this function on $\partial B_\varepsilon(D)$, we get
\begin{align*}
0 =& -\int_{\partial B_\varepsilon(D)}
 g_0 \left( \dd \sigma_2 \wedge \dd^c g_1 - \dd^c \sigma_2 \wedge \dd g_1 + (\ddc \sigma_2) g_1 \right) \\
& \qquad \wedge \dd^c g_2 \wedge \mathbf{\Omega} \wedge \Omega^{m-5}  \\
&+  \int_{\partial B_\varepsilon(D)}
 g_0 \left( \dd \sigma_2 \wedge \dd^c g_1 - \dd^c \sigma_2 \wedge \dd g_1 + (\ddc \sigma_2) g_1 \right) \\
& \qquad \wedge \textbf{d} \dd^c g_2 \wedge \dd^c \log(h_{\mathcal{E}}(s)) \wedge \Omega^{m-5}   \\
&+ \int_{\partial B_\varepsilon(D)}
 g_0 \left( - \dd \sigma_2 \wedge \textbf{d} \dd^c g_1 \underbrace{ - \textbf{d} \dd^c \sigma_2 \wedge \dd g_1 + \ddc \sigma_2 \wedge \textbf{d} g_1}_{=0} \right) \\
 & \qquad \wedge \dd^c g_2 \wedge \dd^c \log(h_{\mathcal{E}}(s)) \wedge \Omega^{m-5}   \\
&+  \int_{\partial B_\varepsilon(D)}
\textbf{d} g_0 \wedge \left( \dd \sigma_2 \wedge \dd^c g_1 - \dd^c \sigma_2 \wedge \dd g_1 + (\ddc \sigma_2) g_1 \right) \\
& \qquad \wedge \dd^c g_2 \wedge \dd^c \log(h_{\mathcal{E}}(s)) \wedge \Omega^{m-5}
\end{align*}
The limit of the first line is the one we have to estimate. We may hence instead consider the remaining three lines.
Those we may again restrict to small neighborhoods $U_i$ (all of the second type!) and estimate
the absolute value of the integral. Here all problematic terms are of the form
\[ \int_{St_1'(\varepsilon)} |\log(r_1) \xi' \wedge \xi \wedge \dd^c \log(h_{\mathcal{E}}(s)) \wedge \Omega^{m-4}| \]
or
\[ \int_{St_1'(\varepsilon)} |\frac{\D{r_1}}{r_1} \wedge \xi' \wedge \xi \wedge \dd^c g_2 \wedge \dd^c \log(h_{\mathcal{E}}(s)) \wedge \Omega^{m-5}|,  \]
where $\xi'$ is a smooth form of degree at least 1, generated by $r_{m-2} \dd \varphi_{m-2}$ and $\dd r_{m-2}$ and $\xi$ is log-log.
The latter poses no problem because $\frac{\D{r_1}}{r_1}$ is, up to terms involving $\D{r_1'}$, a smooth form (\ref{eqnrprime}).

Now, by the estimate (\ref{eqnlemmaomega4}) of Lemma \ref{LEMMAESTIMATEOMEGA}, we get vanishing of the limit $\varepsilon \rightarrow 0$ by Lemma  \ref{INTSTARFORM} again.
\end{proof}

\begin{LEMMA}\label{LEMMAESTIMATEOMEGA}
Let $U$ be a neighborhood of any type and let $i$ be the Witt rank of $\nL$.
Then for any $\log$-$\log$ form $\xi$ of rank $2(m-2-i)-1$, we have
\begin{equation}\label{eqnlemmaomega1}
 \xi \wedge \Omega^{i} \prec \frac{1}{\log(r_k')^2} \sigma,
\end{equation}
where $\sigma$ is a (*)-form w.r.t. to $k$.

Also we have for any $j \not= k$
\begin{equation}\label{eqnlemmaomega2}
 \xi \wedge \Omega^i \prec \frac{1}{\log(r_k')\log(r_j)} \sigma,
\end{equation}
where $\sigma$ is another (*)-form w.r.t. $k$.

If $U$ is a neighborhood of the second type (\ref{nbhdsecondtype}),
for any $\log$-$\log$ form $\xi' \wedge \xi$ of degree $2(m-2)-5$, where $\xi'$ is a smooth form of degree at least 1, involving only $r_{m-2}\dd\varphi_{m-2}$ and $\dd r_{m-2}$, where $r_{m-2}$, $\varphi_{m-2}$ are the polar coordinates of the projection to a $\nSh({}^K\nH(\nI))$ in a neighborhood of the point, we have
\begin{equation}\label{eqnlemmaomega3}
 \dd \varphi_1 \wedge \xi' \wedge \xi \wedge \Omega \prec \frac{1}{\log(r_1')} \sigma,
\end{equation}
and
\begin{equation}\label{eqnlemmaomega4}
\xi' \wedge \xi \wedge \dd^c \log(h_{\mathcal{E}}(s)) \wedge \Omega \prec \frac{1}{\log(r_1')^2} \sigma,
\end{equation}
where the $\sigma$ are (*)-forms w.r.t. the index 1.
\end{LEMMA}

\begin{proof}
We have in any case:
\begin{align*}
 \dd^c \log(r_i) &= \D{\varphi}_i,  \\
 \dd \log(r_i) &= \frac{\D{r}_i}{r_i},
\end{align*}
and
\begin{align*}
 \dd^c \log h_{\mathcal{E}}(s) &=  \frac{ \dd^c h_{\mathcal{E}}(s)}{h_{\mathcal{E}}(s)},  \\
 \dd \log h_{\mathcal{E}}(s) &=  \frac{ \dd h_{\mathcal{E}}(s)}{h_{\mathcal{E}}(s)},  \\
 \Omega \sim \ddc \log h_{\mathcal{E}}(s) &= \frac{ \dd h_{\mathcal{E}}(s)}{h_{\mathcal{E}}(s)}  \frac{ \dd^c h_{\mathcal{E}}(s)}{h_{\mathcal{E}}(s)} + \frac{ \ddc h_{\mathcal{E}}(s)}{h_{\mathcal{E}}(s)}.
\end{align*}

In a neighborhood {\em of the first type} (\ref{nbhdfirsttype}), we may write
\[ h_{\mathcal{E}}(s) = \sum_{i,j\le n} \langle \lambda_i, \lambda_j \rangle \log(r_i) \log(r_j) + \sum_{i\le n} \log(r_i) \psi_i(z) + \psi_0(z), \]
where $\psi_i$ are harmonic functions and $\psi_0$ is smooth.
Here the $\lambda_i$ are linearly independent and $\langle \lambda_i, \lambda_j \rangle > 0$ if $i \not= j$.
Hence for $i \not= j$ always
\begin{equation}\label{eqnestimateh1}
 h_{\mathcal{E}}(s) \gg \log(r_i)\log(r_j),
\end{equation}
and if $Q_\nL(\lambda_i)>0$
\begin{equation}\label{eqnestimateh2}
 h_{\mathcal{E}}(s) \gg \log(r_i)^2.
\end{equation}

Hence
\begin{align}
\label{eqnddc1} \dd^c h_{\mathcal{E}}(s) =&
  \log(r_k) (\sum_{j\le n} \langle \lambda_k, \lambda_j \rangle \D{\varphi_j} + \dd^c \psi_k)
  + \sum_{i \le n} \D{\varphi_i} \psi_i  \nonumber\\
  &+ \sum_{i \le n , j \le n, i\not=k} \langle \lambda_i, \lambda_j \rangle \log(r_i) \D{\varphi_j}
  + \sum_{i \le n, i\not=k} \log(r_i) \dd^c \psi_k + \dd^c \psi_0  
\end{align}
\begin{align}
\label{eqnddc2} \dd h_{\mathcal{E}}(s) =&
  \log(r_k) (\sum_{j \le n} \langle \lambda_k, \lambda_j \rangle \frac{\D{r_j}}{r_j} + \dd \psi_k)
  + \sum_{i \le n} \frac{\D{r_j}}{r_j} \psi_i  \nonumber\\
  &+ \sum_{i \le n , j \le n, i\not=k} \langle \lambda_i, \lambda_j \rangle \log(r_i) \frac{\D{r_j}}{r_j}
  + \sum_{i \le n, i\not=k} \log(r_i) \dd \psi_k + \dd \psi_0 \\
\label{eqnddc3} \ddc h_{\mathcal{E}}(s) =&
  \D{\varphi_k} (\sum_{j \le n} \langle \lambda_k, \lambda_j \rangle \frac{\D{r_j}}{r_j} + \dd \psi_k)
  + \sum_{i \le n} \frac{\D{r_j}}{r_j} \dd^c \psi_i  \nonumber\\
  &+ \sum_{i \le n , j \le n, i\not=k} \langle \lambda_i, \lambda_j \rangle \D{\varphi_i} \frac{\D{r_j}}{r_j}
  + \sum_{i \le n, i\not=k} \D{\varphi_i} \dd \psi_k + \ddc \psi_0.
\end{align}

We now substitute $\frac{\D{r_k}}{r_k}$ by a form of shape
\[ f \frac{\D{r_k'}}{r_k'} + \xi, \]
where $\xi$ is smooth. This is possible by means of formula (\ref{eqnrprime}).

First assume $i=2$.
$\Omega^2$ is proportional to
\begin{align}
\label{eqnomega1} & \frac{ \ddc h_{\mathcal{E}}(s)}{h_{\mathcal{E}}(s)} \wedge \frac{ \ddc h_{\mathcal{E}}(s)}{h_{\mathcal{E}}(s)}  \\
\label{eqnomega2} &+ \frac{ \dd h_{\mathcal{E}}(s)}{h_{\mathcal{E}}(s)}  \frac{ \dd^c h_{\mathcal{E}}(s)}{h_{\mathcal{E}}(s)} \wedge \frac{ \ddc h_{\mathcal{E}}(s)}{h_{\mathcal{E}}(s)}.
\end{align}

In (\ref{eqnomega1}), multiplying out the expression (\ref{eqnddc3}) squared, we have for every occurring summand $S$
\[ S \prec \frac{1}{\log(r_k)^2} \sigma \prec \frac{1}{\log(r_k')^2} \sigma, \]
where $\sigma$ is a (*)-form, using the estimate (\ref{eqnestimateh1}).

In (\ref{eqnomega2}), multiplying out the product of (\ref{eqnddc1}--\ref{eqnddc3}), we have
\[ S \prec \frac{1}{\log(r_k)^2} \sigma \prec \frac{1}{\log(r_k')^2} \sigma, \]
using the estimate (\ref{eqnestimateh1}) again,
except for the summands of the form
\begin{equation}\label{eqn1}
 \log(r_k) \log(r_k) (\sum_{j \le n} \langle \lambda_k, \lambda_j \rangle \D{\varphi_j} + \dd^c \psi_k) \wedge (\sum_{j \le n} \langle \lambda_k, \lambda_j \rangle \frac{\D{r_j}}{r_j} + \dd \psi_k) \wedge S,
\end{equation}
where $S$ is any summand of (\ref{eqnddc3}).

But now in the expression $\xi \wedge \Omega^2$ either occurs a $\frac{\D{\varphi}_k}{\log(r_k)}$ from $\xi$, hence the estimate (\ref{eqnlemmaomega1}) is true, or a $r_k \D{\varphi}_k$ occurs in $S$ which satisfies a much stronger estimate, or a $\D{\varphi}_k$ occurs in $S$. It occurs, however, multiplied with $(\sum_{j \le n} \langle \lambda_k, \lambda_j \rangle \frac{\D{r_j}}{r_j} + \dd \psi_k)$ so (\ref{eqn1}) is zero in that case.

The estimate (\ref{eqnlemmaomega2}) is more easy and left to the reader.

Now assume $i=1$. Then, by the estimate (\ref{eqnestimateh2}), already every summand $S$ in (\ref{eqnddc1}--\ref{eqnddc3}) divided by $h_\mathcal{E}(s)$ satisfies
\[ S \prec \frac{1}{\log(r_k)} \sigma \prec \frac{1}{\log(r_k')} \sigma, \]
with a (*)-form $\sigma$.

The estimates (\ref{eqnlemmaomega3}, \ref{eqnlemmaomega4}) do not involve neighborhoods of the first type.

In a neighborhood {\em of the second type} (\ref{nbhdsecondtype}), we may write
\[ h_{\mathcal{E}}(s) = \log(r_1) \psi_1(z_{m-2}) + \psi_0(z), \]
where $\psi_1$ is harmonic and depends only on $z_{m-2}$ and $\psi_0$ is smooth.

Clearly in this case:
\begin{equation}\label{eqnestimateh3}
 h_{\mathcal{E}}(s) \gg \log(r_1).
\end{equation}

Also
\begin{align*}
 \dd^c h_{\mathcal{E}}(s) &=
 \log(r_1) \dd^c \psi_1(z_{m-2}) + \D{\varphi_1} \psi_1(z_{m-2}) + \dd^c \psi_0(z), \\
 \dd h_{\mathcal{E}}(s) &=
 \log(r_1) \dd \psi_1(z_{m-2}) + \frac{\D{r_1}}{r_1} \psi_1(z_{m-2}) + \dd \psi_0(z), \\
 \ddc h_{\mathcal{E}}(s) &=
 \D{\varphi_1} \dd \psi_1(z_{m-2}) + \frac{\D{r_1}}{r_1} \dd^c \psi_1(z_{m-2}) + \ddc \psi_0(z).
\end{align*}

We now substitute again $\frac{\D{r_1}}{r_1}$ by a form of shape
\[ f \frac{\D{r_1'}}{r_1'} + \xi, \]
where $\xi$ is smooth. This is possible by means of formula (\ref{eqnrprime}).

Here, for (\ref{eqnlemmaomega1}), we may argue exactly as before.
(\ref{eqnlemmaomega2}) is vacuous in this type of neighborhood.
For (\ref{eqnlemmaomega3}): Write once again $\Omega$ as
\begin{align}
\label{eqnomega3} & \frac{ \ddc h_{\mathcal{E}}(s)}{h_{\mathcal{E}}(s)}  \\
\label{eqnomega4} +& \frac{ \dd^c h_{\mathcal{E}}(s)}{h_{\mathcal{E}}(s)} \wedge \frac{ \ddc h_{\mathcal{E}}(s)}{h_{\mathcal{E}}(s)}.
\end{align}

In (\ref{eqnomega3}), using (\ref{eqnestimateh3}), we have for any summand
\[ S \prec \frac{1}{\log(r_1')} \sigma, \]
for a (*)-form $\sigma$.

In (\ref{eqnomega4}), using (\ref{eqnestimateh3}) again, we get the same for any summand except possibly for
\[ \frac{1}{h_\mathcal{E}(s)^2} \log(r_1)^2 \dd^c \psi_1(z_{m-2}) \wedge \dd \psi_1(z_{m-2}), \]
which cancels, however, with $\xi'$ because $\dd^c \psi_1(z_{m-2}) \wedge \dd \psi_1(z_{m-2})$ is proportional to $r_{m-2} \dd \varphi_{m-2} \wedge \dd r_{m-2}$.

For (\ref{eqnlemmaomega4}): Consider:
\begin{equation}
\label{eqnomega5} \dd^c \log(h_\mathcal{E}(s)) \wedge \Omega = \frac{\dd^c h_{\mathcal{E}}(s)}{h_{\mathcal{E}}(s)} \wedge \frac{ \ddc h_{\mathcal{E}}(s)}{h_{\mathcal{E}}(s)}.
\end{equation}

Using (\ref{eqnestimateh3}) again, we now have for any summand in 
\[ S \prec \frac{1}{\log(r_1')^2} \sigma, \]
for a (*)-form $\sigma$,
except possibly for
\[ \frac{1}{h_\mathcal{E}(s)^2} \log(r_1) \dd^c \psi_1(z_{m-2}) \dd \varphi_i \dd \psi_1(z_{m-2}), \]
which cancels with $\xi'$ because $\dd^c \psi_1(z_n) \wedge \dd \psi_1(z_n)$ is proportional to $r_n \D{\varphi_n} \wedge \D{r_n}$
and possibly for
\[ \frac{1}{h_\mathcal{E}(s)^2} \log(r_1) \dd^c \psi_1(z_{m-2})  \ddc \psi_0(z), \]
But then a $\frac{\dd \varphi_1}{\log(r_1)}$ must appear somewhere in the expression (\ref{eqnlemmaomega4}).
\end{proof}

\section{The arithmetic and geometric volume of Shimura varieties of orthogonal type}

This section contains the main results of this article. All results are conditional on the existence of a
good theory of canonical integral models of toroidal compactifications of Shimura varieties as outlined in
the first sections. In the thesis of the author, the existence of such a theory has been proven under
a missing technical hypothesis (cf. \ref{MAINCONJECTURE}). 

\begin{PAR}\label{DISCORBIT}
Recall the definitions of section \ref{DEFLAMBDAMU} and \ref{RN}.
The relation between $\widetilde{\mu}$ and $\widetilde{\lambda}$ 
is given by a global orbit equation:
\end{PAR}
\begin{SATZ}[{\cite[Theorem 10.5]{Paper1}}]\label{GLOBALORBITEQUATION}
Assume $m \ge 3$, $m-n \ge 1$. Let $D$ be the discriminant of $\nL_\Z$ and $D'$ be the $D$-primary part of the discriminant of $\nM_\Z$.
\[ \widetilde{\lambda}^{-1}(\nL_\Z; 0) \widetilde{\mu}(\nL_\Z, \nM_\Z, \kappa; 0) = \sum_{\alpha \SO'(\nL_\Zh) \subset \nIsome(\nM, \nL)(\Af) \cap \kappa} \widetilde{\lambda}^{-1}(\alpha^\perp_\Z; 0) \]
and
\[ \frac{\dd}{\dd s} \left. { \left(\widetilde{\lambda}^{-1}(\nL_\Z; s) \widetilde{\mu}(\nL_\Z, \nM_\Z, \kappa; s) \right) } \right|_{s=0} \equiv \sum_{\alpha\SO'(\nL_\Zh) \subset \nIsome(\nM, \nL)(\Af) \cap \kappa} \frac{\dd}{\dd s} \left.  \widetilde{\lambda}^{-1}(\alpha^\perp_\Z; s) \right|_{s=0} \]
in $\R_{2DD''}$, where $D''$ is the product of primes such that $p^2 \nmid D'$.
\end{SATZ}

There is also a relation between different $\widetilde{\mu}$'s in the style of Kitaoka's recursion formula: 
\begin{SATZ}[{\cite[Theorem 10.6]{Paper1}}]\label{GLOBALKITAOKA}
Assume $m \ge 3$, $m-n \ge 1$. Let $D$ be the discriminant of $\nL_\Z$ and $\nM_\Z = \nM_\Z' \perp \nM_\Z''$. Let $D'$ be the $D$-primary part of the 
discriminant of $\nM_\Z'$ (not $\nM_\Z$ !).
Let $\kappa \in (\nL_\Z^*/\nL_\Z) \otimes \nM_\Z^*$ with a corresponding decomposition $\kappa = \kappa' \oplus \kappa''$.
We have
\[ \widetilde{\lambda}^{-1}(\nL_\Z; 0) \widetilde{\mu}(\nL_\Z, \nM_\Z, \kappa; 0) = \sum_{\substack{\alpha \SO'(\nL_\Zh) \subset \nIsome(\nM', \nL)(\Af) \cap \kappa' \\ \kappa'' \cap \alpha^\perp_\Af \otimes (\nM''_\Af)^* \not= \emptyset}} \widetilde{\lambda}^{-1}(\alpha^\perp_\Z; 0) \widetilde{\mu}(\alpha^\perp_\Z, \nM_\Zh, \kappa''; 0) \]
and
\begin{gather*} 
\frac{\dd}{\dd s} \left. { \left( \widetilde{\lambda}^{-1}(\nL_\Z; s) \widetilde{\mu}(\nL_\Z, \nM_\Z, \kappa; s) \right) } \right|_{s=0} \\ = 
\sum_{\substack{\alpha \SO'(\nL_\Zh) \subset \nIsome(\nM', \nL)(\Af) \cap \kappa' \\ \kappa'' \cap \alpha^\perp_\Af \otimes (\nM''_\Af)^* \not= \emptyset}} \frac{\dd}{\dd s} \left. { \left( \widetilde{\lambda}^{-1}(\alpha^\perp_\Z; s) \widetilde{\mu}(\alpha^\perp_\Z, \nM'', \kappa''; s)
\right) } \right|_{s=0} 
\end{gather*}
in $\R_{2DD'}$. Here $\kappa'' \in (\nL_\Zh^*/\nL_\Zh)\otimes (\nM_\Z'')^*$ is considered as an element of $((\alpha^\perp_\Zh)^* / \alpha^\perp_\Zh) \otimes (\nM''_\Zh)^*$ via 
$\kappa'' \mapsto \kappa'' \cap (\alpha^\perp_\Af \otimes (\nM''_\Af)^*)$.
\end{SATZ}

\begin{PAR} \label{SPECIALMODELS}
Let $\nL_{\Z}$ be a lattice with quadratic form of discriminant $D\not=0$ and signature $(m-2,2)$.
Let $K$ be the {\em discriminant kernel }
of $\nL_{\Zh}$. It is an admissible compact open subgroup for all $p \nmid D$. Let $\nRPCD$ be a
complete and smooth $K$-admissible rational polyhedral cone decomposition, and let $\nSh({}^K_\nRPCD \nO(\nL))$ the {\em global} canonical
model defined over $\Z[1/(2D)]$.

At primes $p\not=2$, with $p|D$ but $p^2\nmid D$, we can find a lattice $\nL'_\Z$ containing $\nL_\Z$ as a saturated sublattice of codimension one such that
$\nL'_\Zpp$ is unimodular. If $\nRPCD$ is chosen appropriate, we 
may define thus a local model over $\Zpp$ of $\nSh({}^K_\nRPCD \nO(\nL))_\Q$ at those $p$ by taking
normalization and Zariski closure in the model $\nSh({}^{K'}_{\nRPCD'} \nO(\nL'))$ over $\spec(\Zpp)$. We understand by $\nSh$ the so glued global model over $\Z[1/(2D')]$, where
$D'$ is the product of primes $p$ such that $p^2|D$. Let $\overline{\mathcal{E}}$ be as before.
In the same way, we glue $\nMorphSt^* \overline{\mathcal{E}}$ on $\nSh$ from the pullbacks of $\nMorphSt^* \overline{\mathcal{E}}$ on 
$\nSh({}^{K'}_{\nRPCD'} \nO(\nL'))$ over $\spec(\Zpp)$. Note that it is not clear from the properties of the theory of canonical models outlined in the first sections, whether this 
depends on the choice of lattices $\nL'_\Z$. The theorem below, however, is true for any of these choices.
\end{PAR}

\begin{HAUPTSATZ}\label{RESULTGLOBAL}
Assume \ref{MAINCONJECTURE}. We have
\[
\begin{array}{rrcl}
\text{1.} \qquad & \vol_{E}(\nSh) &=& 4 \widetilde{\lambda}^{-1}(\nL_\Z; 0)  \\
\text{2.} \qquad & \widehat{\vol}_{\overline{\mathcal{E}}}(\nSh) &\equiv& \frac{d}{ds} 4 \widetilde{\lambda}^{-1}(\nL_{\Z};s)|_{s=0} \qquad \text{ in $\R_{2D'}$}
\end{array}
\]

Let $\nM_\Z$ be a lattice of dimension $n$ with positive definite $Q_\nM \in \Sym^2(\nM_\Q^*)$. 
Let $D''$ be the product of primes $p$ such that $\nM_\Zp \not\subseteq \nM_\Zp^*$ or $\nM_\Zp^*/\nM_\Zp$ is not cyclic. 
Assume 
\begin{itemize}
\item $m-n \ge 4$, {\em or} 
\item $m=4, n=1$ and $\nL_\Q$ is has Witt rank 1. 
\end{itemize}
Then we have
\[
\begin{array}{rrcl}
\text{3.} &  \vol_{\mathcal{E}}(\nZ(\nL_\Z, \nM_\Z, \kappa; K)) &=& 4 \widetilde{\lambda}^{-1}(\nL_\Z; s) \widetilde{\mu}(\nL_\Z, \nM_\Z, \kappa; 0),  \\
\text{4.} &  \hght_{\overline{\mathcal{E}}}(\nZ(\nL_\Z, \nM_\Z, \kappa; K)) &\equiv& \frac{\dd}{\dd s} \left. { 4 \left( \widetilde{\lambda}^{-1}(\nL_\Z; s) \widetilde{\mu}(\nL_\Z, \nM_\Z, \kappa; s) \right) } \right|_{s=0} \\
&&& \text{in } \R_{2D D''}.
\end{array}
\]
3. is true without the restriction on $m$ and $n$. Note that for $n=1$ and integral $Q_\nM$, we have trivially always $D''=1$. 
\end{HAUPTSATZ}

\begin{proof}

1. and 3., which are concerned with geometric volumes, are
well-known but proofs are included to emphasize the analogies with the arithmetic volume case.

If formula 1. is true for a lattice $\nL$ then it is true for any lattice $\nL'$ with $\nL_\Q \cong \nL_\Q'$ respecting the quadratic form up to scalar. This is because in this case
the two associated Shimura varieties are related by a correspondence, and for the assertion we only have to show that 
\begin{equation}\label{eqvol}
 \frac{[\SO'(\nL'_\Zh):K]}{[\SO'(\nL_\Zh):K]} = \frac{\lambda^{-1}(\nL'_\Z; 0)}{\lambda^{-1}(\nL_\Z; 0)},  
\end{equation}
where $K \subseteq \SO'(\nL'_\Zh) \cap \SO'(\nL_\Zh)$ is any finite index subgroup. Observe that the morphism $\nMorphSt$ (\ref{MAINTHEOREM3STACKS}) is compatible with change of $K$ (see also \ref{ARITHMVOLNOTNEAT}).
Now $\lambda(\nL; 0)$ and $\lambda(\nL'; 0)$ are Euler products, and their factors at almost all $p$ are the same. The finite Euler factors are the volumes of 
$\SO'(\nL_\Zp)$ and $\SO'(\nL_\Zp')$, respectively, w.r.t. the canonical measures. This shows (\ref{eqvol}).

We continue by showing 1. for lattices $\nL_\Z$ with signature $(0,2)$. Let $\nL'_\Z$ be the 
lattice $\{ A \in M_2(\Z) \ |\ {}^t A = A\}$ with quadratic form $Q_\nL(A) = \det(A)$. By the remark before, we may assume w.l.o.g. that the discriminant of $\nL_\Z$ 
is fundamental, and there is a
saturated embedding $\nL_\Z \hookrightarrow \nL_\Z'$. The formula 1. is then basically just the classical Dirichlet class number formula. We prove it here as follows:
Let $E'(\tau, s)$ be the Eisenstein series \cite[12.4]{Paper1} und $E(\tau,s)=\frac{Z(2s)}{Z(s)}E'(\tau,s)$. 
We let $E(\nL_\Z; s)$ be its trace (in the stack sense) over the embedded Shimura variety.
By Kronecker's limit formula, we have (using our normalization of the metric $h_E$, cf. \ref{BORCHERDSBUNDLE}):
\[ 2 E(\nL_\Z, s+1) = \vol_E(\nSh({}^K \nO(\nL)) + \hght_{\overline{\mathcal{E}}}(\nSh({}^K \nO(\nL)) s + O(s^2).  \]
On the other hand, we have \cite[Theorem 12.6, 5.]{Paper1}:
\[ E(\nL_\Z, s+1) = 2 \widetilde{\lambda}^{-1}(\nL'_\Z; s) \]
(up to multiplication with a rational function in $2^{-s}$ with is 1 at $s=0$).
Taking the value at $s=0$, resp. the derivative at $s=0$, we get formula 1. --- respectively 2. for this special lattice and embedding --- in this case.

In the remaining case $m \ge 3$, we give two different proofs of formula 1:

{\em First proof:} 

Let $\nM_\Z = <q>$, $q\in \Q$ and $\kappa \in (\nL^*_\Z/\nL_\Z) \otimes \nM_\Z^* = \nL^*_\Z/\nL_\Z$ is a coset. 

We have the following identity:
\begin{equation}\label{MUGEOMETRY}
 \mu(\nL_\Z, <q>, \kappa; 0) = \frac{\vol_E(\nZ(\nL_\Z, <q>, \kappa; K))}{\vol_E(\nSh({}^K \nO(\nL)))}.
\end{equation}
This is an application of the Siegel-Weil formula, see \cite[4.17, 4.20, 4.21]{Kudla4} and also \cite[Theorem 7.6.8]{Thesis}.
A certain average over this equation, sufficient for proving 1. by induction follows also from Borcherds' theory --- see Theorem \ref{MAINTHEOREMAVERAGE}, 1. However, this
would exclude some special cases. We mention it because it is precisely the geometric analogue of 
the proof of 2. that we will give below. 

Comparing \ref{MUGEOMETRY} to the value at $s=0$ of the global orbit equation (\ref{GLOBALORBITEQUATION})
\[ \mu(\nL_\Z, <q>, \kappa; 0) \lambda(\nL_\Z; 0)^{-1} = \sum_j \lambda((g_j^{-1} x)^\perp; 0)^{-1}, \]

we establish that assertion 1. is true for lattices of signature $(m-2,2)$ if and only if it is true for lattices of signature $(m-3,2)$,
whereby it is proven by induction on $m$ because $\sum_i 4 \lambda((g_j^{-1} x)^\perp; 0)^{-1}$ is $\vol_E(\nZ(\nL_\Z, <q>, \kappa; K))$ by the induction
hypothesis.

{\em Second proof:} We assume again $m\ge 3$ and use $\tau(\SO(\nL_\Q)) = 2$, where $\tau$ means Tamagawa number.
We have the following elementary relation:
\[  \Lambda \tau = \vol(K) \vol_E(\nSh({}^K\nO(\nL))), \]
where $\Lambda = 2 \lambda_\infty^{-1}(\nL; 0)$ is the comparison factor from Lemma \ref{BORCHERDSBUNDLECOMP} and 
$\vol(K)$ is computed with respect to the product of the canonical volumes. Recall that $\Lambda$ involved the canonical volume form on $\SO(\nL_\R)$ and
the product over all $\nu$ of the canonical measures is a Tamagawa measure (\cite[Lemma 10.1]{Paper1}). 
We have $\vol(K) = \prod_p \lambda_p(\nL; 0)$ by definition of the $\lambda_p$. Everything 
put together yields $\vol_E(\nSh({}^K\nO(\nL))) = 4  \lambda^{-1}(\nL_\Z; 0)$ (which equals also $4  \widetilde{\lambda}^{-1}(\nL_\Z; 0)$).

3. follows from 1. using the value of the global orbit equation at $s=0$:
\[ \widetilde{\lambda}^{-1}(\nL_\Z; 0) \widetilde{\mu}(\nL_\Z, \nM_\Z, \kappa; 0) = \sum_{\alpha \SO'(\nL_\Zh) \subset \nIsome(\nM, \nL)(\Af) \cap \kappa} \widetilde{\lambda}^{-1}(\alpha^\perp_\Z; 0). \]

2., {\em as an identity $\R_{(2D)}$}, follows (using 1.) from the following local statement:
\begin{ASSERTION}\label{localstatement}
Let $p \not=2$.
Let $\nL_\Zpp$ be a unimodular (at $p$) lattice and signature $(m-2,2)$.
Let $K$ be any admissible compact open subgroup and $\nRPCD$ be
complete and smooth w.r.t. $K$, and let 
$\nSh = \nSh({}^K_\nRPCD \nO(\nL))$ the Shimura variety, considered as variety over $\spec(\Zpp)$.
We have
\[ \widehat{\vol}_{\overline{\mathcal{E}},p}(\nSh ) \equiv
\vol_{E}(\nSh_\C) \left. \frac{ \frac{d}{ds} \lambda^{-1}(\nL_\Z;s)}{\lambda^{-1}(\nL_\Z;s)}\right|_{s=0}  \]
in $\R^{(p)}$ and any $\Z$-model $\nL_\Z \subset \nL_{\Zpp}$. 
\end{ASSERTION}
First of all, to prove Assertion \ref{localstatement} we may w.l.o.g. multiply $Q_\nL$ by a scalar $\in \Z_{(p)}^*$ because the Shimura datum and 
$\mathcal{E}$ as a $\SO$-equivariant bundle are not affected by this and
the Hermitian metric changes by a factor in $\Z_{(p)}^*$ with does not change the arithmetic volume considered in $\R^{(p)}$.
Assertion \ref{localstatement} also may be shown for any specific $p$-{\em admissible} $K$ and will then be true for any such choice (see \ref{ARITHMVOLNOTNEAT}).

The strategy of induction, similar to the first proof of 1., is to walk through the set of (unimodular at $p$) lattices in a special way, starting from lattices $\nL_\Zpp$ with known
\[  \avol_{\overline{\mathcal{E}},p}(\nSh({}^{K}_\nRPCD \nO(\nL_\Zpp)))  \] 
(for admissible $K$) and then to construct special Borcherds products by the Theorems \ref{THEOREMBORCHERDSPREP1}--\ref{THEOREMBORCHERDSPREP4} in such a way that
all quantities in Theorem \ref{MAINTHEOREMAVERAGE}, 2. --- except one --- are already known. This path through the set of lattices will be mostly according to the dimension and Witt rank of the lattice. It is illustrated by Figure \ref{figind}, wherein an arrow indicates logical dependence, i.e. the reverse walking direction.

\begin{figure}
{\footnotesize
\begin{center}
\begin{sideways}
$ \xymatrix{ 
*+[F]\txt{$(0,2)^0$ \\ Heegner points \mbox{\textcircled{10}} } & *+[F]\txt{$(1,2)^0$ \\ Shimura curves (compact)} \ar[rd]^-{\mbox{\textcircled{9}}} & *+[F]\txt{$(2,2)^0$ \\ compact surface} \ar[rdd]^-{\mbox{\textcircled{8}}} & \\
& *+[F]\txt{$(1,2)^1$ \\ modular curve} & *+[F]\txt{$(2,2)^1 \enspace q \equiv 1 \enspace (4)$ \\ Hilbert mod. surface } \ar[l]_-{\mbox{\textcircled{2}}} & \\
& & *+[F]\txt{$(2,2)^1$ \\ Hilbert mod. surface} \ar[rd]^-{\mbox{\textcircled{4}}} & *+[F]\txt{$(3,2)^1$ \\ twisted Siegel threefolds} \ar[l]_-{\mbox{\textcircled{5}}} & *+[F]\txt{$(4,2)^1$} \ar[l]_-{\mbox{\textcircled{6}}} \\
& & *+[F]\txt{$(2,2)^2$ \\ product of. mod. curves} \ar[r]^-{\mbox{\textcircled{7}}} & *+[F]\txt{$(3,2)^2$ \\ Siegel mod. threefold} \ar[luu]_(.2){\mbox{\textcircled{3}}} & *+[F]\txt{$(4,2)^2$} \ar[l]_-{\mbox{\textcircled{6}}} \ar[lu]_-{\mbox{\textcircled{6}}} & \ar[l]_-{\mbox{\textcircled{6}}} \ar[lu]_-{\mbox{\textcircled{6}}} \cdots
} $
\end{sideways}
\end{center}
}
\caption{Strategy of induction on the type of lattice}
\label{figind}
\end{figure}
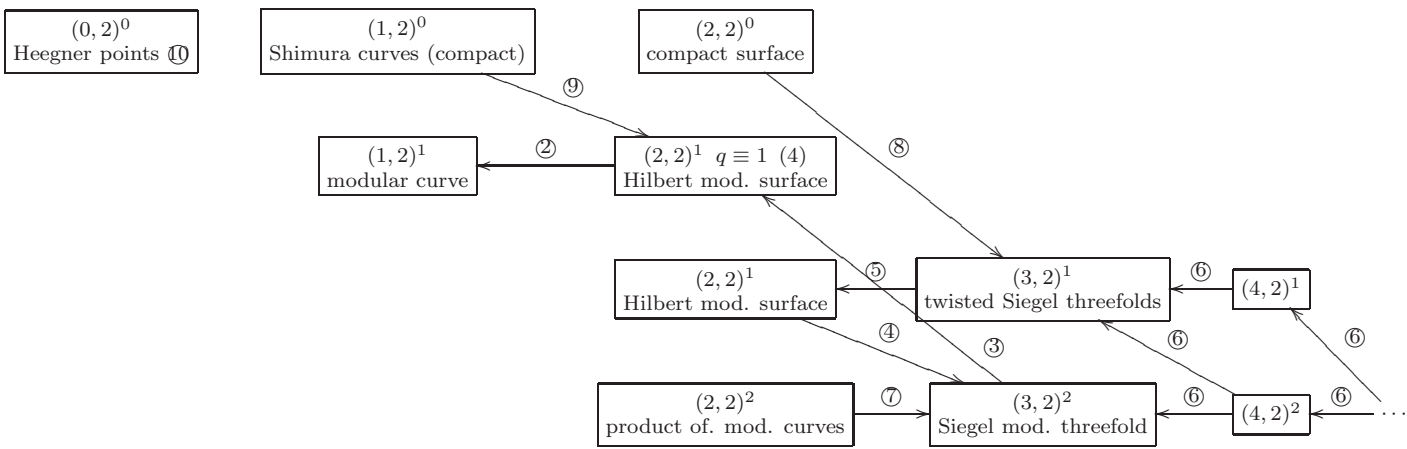

STEP 1:
We start with the case of a lattice $\nL_\Zpp$ of signature $(1,2)$ and Witt rank $1$. These are up to multiplication of $Q_\nL$ by scalars in $\Z_{(p)}^*$
of the form $Q_\nL(z) = z_1z_2-z_3^2$ which is equivalently the space $\{ X \in M_2(\Zpp) \where {}^tX = X\}$
with $Q_\nL(X) = \det(X)$. The arithmetic volume $\avol_{\overline{\mathcal{E}},p}(\nSh({}^{K}_\nRPCD \nO(\nL_\Zpp)))$ has been calculated in K\"uhn's thesis \cite{Kuehn}.
Since multiplying the norm (square root of the metric) by a scalar $\rho$ changes the arithmetic volume of an arithmetic variety $X$ of 
{\em arithmetic} dimension $\dim(X)$ by 
\[ - \dim(X) \log(\rho) \vol, \] 
comparing to \cite[Theorem 6.1]{Kuehn}, we see that 
\[ \avol_{\overline{\mathcal{E}}} (\Sh({}^{K(1)} \nO(\nL)) = \vol_{\mathcal{E}} (\Sh({}^{K(1)} \nO(\nL)) (-2\frac{\zeta'(-1)}{\zeta(-1)} - 1 + \log(2) + \log(2\pi) + \gamma ). \]
In 1. we obtained that
\[ \vol_{\mathcal{E}} (\Sh({}^{K(1)} \nO(\nL)) = 4 \lambda^{-1}(\nL_\Z;0) = -2\zeta(-1) = \frac{1}{6}, \]
which is of course well-known, taking into account that $\mathcal{E}$ is $\OO(-2)$ (identifying $\nShD(\nO(\nL))$ with $\PP^1$), hence $\nMorphSt^*(\mathcal{E})$ is the line-bundle of classical modular forms of weight 2, and $\SO(\nL)$ is $\PGL_2$.
On the other hand, we have (\cite[Example 11.11]{Paper1}):
\[ 4 \widetilde{\lambda}^{-1}(\nL, s) =  -2\zeta(-1) + -2\zeta(-1) \left( -2 \frac{\zeta'(-1)}{\zeta(-1)} - 1 + \frac{1}{2}\log(2) + \log(2\pi) + \gamma \right)s + O(s^2). \]
Hence formula 2. is true in $\R_2$.

From now on we proceed by induction along the arrows of the diagram according to the following scheme:
Choose a saturated $\nL_\Z \subset \nL_\Zpp$ such that its discriminant kernel $K \subset \SO(\nL_\Af)$ is neat (it's automatically $p$-admissible because
$\nL_\Zpp$ was assumed to be unimodular).
Assume that there is a modular form $F$ of weight $2-\frac{m}{2}$ associated with the  Weil representation of $\nL_\Z$ with $q$-expansion
\[ F(\tau) = \sum_{k \in \Q, k \gg -\infty} a_k q^k \]
with the following properties 
\begin{itemize}
\item $\nZ(\nL, <-k>, a_k; K)$ for all $k<0$ is a union of (images of) canonical models of sub-Shimura varieties (good reduction),
\item $\avol_{\overline{\mathcal{E}}, p}(\nSh({}^{K'}_\nRPCD \nO(\nL'))$ for all these sub-Shimura varieties is already known,
\item $a_0(0) \not= 0$.
\end{itemize}
Note that the first property is equivalent to $p|k \Rightarrow a_k=0$ for all $k<0$.

Then apply Theorem \ref{MAINTHEOREMAVERAGE}, 2. to get
\begin{align*} 
 & \sum_{q <0} \widetilde{\mu}'(\nL_\Z, <-q>, a_q; 0) \vol_E(\nSh({}^{K}_\nRPCD \nO(\nL)))  \\
 +& \sum_{q< 0} \widetilde{\mu}(\nL_\Z, <-q>, a_q; 0) \avol_{\overline{\mathcal{E}},p}(\nSh({}^{K}_\nRPCD \nO(\nL))) \\
 =& \sum_{q < 0} \avol_{\overline{\mathcal{E}},p}(\nZ(\nL_\Z, <-q>, a_q; K)).
\end{align*} 
Inserting the induction hypothesis and 1., we get:
\begin{gather}
 \sum_{q < 0} \left(  \widetilde{\mu}'(\nL_\Z, <-q>, a_q; 0) \widetilde{\lambda}^{-1}(\nL_\Z; 0) +  \widetilde{\mu}(\nL_\Z, <-q>, a_q; 0)  \avol_{\overline{\mathcal{E}},p}(\nSh({}^{K}_\nRPCD \nO(\nL))) \right. \nonumber \\ 
 \left.- \sum_{K\alpha \subseteq \nIsome(<-q>, \nL_\Af) \cap \supp(a_q)} a_q(\alpha) (\widetilde{\lambda}^{-1})'(\alpha^\perp_\Z; 0) \right)  = 0 \label{eqc1} 
\end{gather} 
The derivative of the global orbit equation \ref{GLOBALORBITEQUATION} summed up over all $\kappa \in \nL_\Zh^*/\nL_\Zh, q<0$ and weighted with $a_q(\kappa)$ yields:
\begin{gather}
 \sum_{q < 0} \left(  \widetilde{\mu}'(\nL_\Z, <-q>, a_q; 0) \widetilde{\lambda}^{-1}(\nL_\Z; 0) + \widetilde{\mu}(\nL_\Z, <-q>, a_q; 0) (\widetilde{\lambda}^{-1})'(\nL_\Z; 0) \right. \nonumber \\ 
 \left.- \sum_{K\alpha \subseteq \nIsome(<-q>, \nL_\Af) \cap \supp(a_q)} a_q(\alpha) (\widetilde{\lambda}^{-1})'(\alpha^\perp_\Z; 0) \right)  = 0. \label{eqc2}
\end{gather} 
The difference of (\ref{eqc1}) and (\ref{eqc2}) gives
\begin{gather*}
 \sum_{q < 0}  \widetilde{\mu}(\nL_\Z, <-q>, a_q; 0)\left(  \avol_{\overline{\mathcal{E}},p}(\nSh({}^{K}_\nRPCD \nO(\nL))) - (\widetilde{\lambda}^{-1})'(\nL_\Z; 0) \right) \\
 = a_0(0)\left(  \avol_{\overline{\mathcal{E}},p}(\nSh({}^{K}_\nRPCD \nO(\nL))) - (\widetilde{\lambda}^{-1})'(\nL_\Z; 0) \right) = 0 
\end{gather*} 
using $\sum_{q<0} \mu(\nL_\Z, <-q>, a_q; 0) = a_0(0)$ (see proof of \ref{MAINTHEOREMAVERAGE}, 1.).
Since we have $a_0(0) \not=0$, we get in particular that Assertion \ref{localstatement} is true for $\nL_\Z$.

This reduces the proof of Assertion \ref{localstatement}, and hence 2., to finding appropriate input forms for Borcherds products. This is done for the individual steps in Figure \ref{figind} as follows: 

STEP 2:
\textbf{Hilbert modular surfaces of prime discriminant $q \equiv 1 \enspace (4)$}. This is the case considered in \cite{BBK}. 
We reproduce their argument here as follows: Theorem \ref{THEOREMBORCHERDSPREP3}, 1. shows that a Borcherds lift of non-zero weight can be found such that
all occurring Shimura varieties in $\nZ(\nL_\Z, <-q>, a_q; K)$ for the occurring $a_q$ are of the form already treated in STEP 1. 
\\
STEP 3:
\textbf{Siegel modular threefold}. In Theorem \ref{THEOREMBORCHERDSPREP4} it is shown that a Borcherds lift of non-zero weight can be found, such that
every Shimura variety occurring in \mbox{$\nZ(\nL_\Z, <-q>, a_q; K)$} for the non-zero $a_q$'s is of the form considered in STEP 2 (even copies of the same). 
\\
STEP 4:
\textbf{general Hilbert modular surfaces}. In Theorem \ref{THEOREMBORCHERDSPREP4} it is shown that a Borcherds lift of non-zero weight can be found, such that
every Shimura variety occurring in \mbox{$\nZ(\nL_\Z, <-q>, a_q; K)$} for the non-zero $a_q$'s is a $\nSh({}^{K}_\nRPCD \nO(\nL_\Zpp))$ 
for a given lattice $\nL_\Zpp$ of signature $(2,2)$, Witt rank 1. Now $\avol_{\overline{\mathcal{E}},p}(\nSh({}^{K}_\nRPCD \nO(\nL)))$ is known by STEP 3, and $\avol_{\overline{\mathcal{E}},p}(\nZ(\nL_\Z, <-q>, a_q; K))$ may be deduced from the comparison of the formula in \ref{MAINTHEOREMAVERAGE} with the derivative of the orbit equation by reverting the argument above.
\\
STEP 5:
\textbf{twisted Siegel threefolds}. This is the case $\nL$ is of signature $(2, 3)$ and Witt rank 1. We have only to avoid $\nZ(\nL_\Z, <-q>, a_q; K)$'s
with occurring Shimura varieties for lattices with are not unimodular at $p$ {\em or} compact ones. This is achieved by Theorem \ref{THEOREMBORCHERDSPREP2}, 1.
\\
STEP 6:
\textbf{dimension 4 and higher}. We know all arithmetic volumes for orthogonal Shimura varieties of dimension 3. Hence we may proceed by induction on the 
dimension and have only to avoid $\nZ(\nL_\Z, <-q>, a_q; K)$'s containing Shimura varieties for lattices with are not unimodular at $p$. This is achieved by Theorem \ref{THEOREMBORCHERDSPREP1}.
\\
We are now left with a couple of cases that have been omitted in the above process. They can be treated analogously to STEP 4 above, by reverting the argument.
\\
STEP 7:
\textbf{product of modular curves}. This is the case of signature $(2,2)$, Witt rank 2. Use Theorem \ref{THEOREMBORCHERDSPREP4} again.
This can also be treated directly, using STEP 1 (K\"uhn's thesis), see \cite[section 7.8]{BKK1} for a proof.
\\
STEP 8:
\textbf{compact surfaces}. This is the case of signature $(2,2)$, Witt rank 0. Use Thereom \ref{THEOREMBORCHERDSPREP2}, 2.
\\
STEP 9:
\textbf{Shimura curves}. This is the case of signature $(1,2)$, Witt rank 0. Use Theorem \ref{THEOREMBORCHERDSPREP3}, 2.
\\
STEP 10:
\textbf{Heegner points}. These cases have been treated along the lines of the proof of 1. above.

This completes the proof of Assertion \ref{localstatement}, and hence of 2., as an identity in $\R_{2D}$. We use this, and Theorem \ref{MAINTHEOREMAVERAGE} again, to derive the following local assertion, from which 4. follows:
\begin{ASSERTION}\label{localstatement2}
Let $p \not=2$.
Let $\nL_\Zpp$ be a unimodular (at $p$) lattice and of signature $(m-2,2)$, Let $\nM_\Zpp$ be a positive definite lattice of dimension $n$ with cyclic 
$\nM_\Zp^*/\nM_\Zp$. 
Assume $m-n \ge 4$, or $m=4, n=1$ and $\nL_\Q$ is isotropic, or $m=5, n=2$ and $\nL_\Q$ has Witt rank 2.
Let $\nL_\Z$, $\nM_\Z$ any saturated lattices in $\nL_\Zpp$ and $\nM_\Zpp$, respectively, and let $K$ be any admissible compact open subgroup in the discriminant kernel of $\nL_\Z$.
We have:
\begin{align*}
 & \widetilde{\mu}'(\nL_\Z, \nM_\Z, \kappa; 0) \vol_E(\nSh({}^{K}_\nRPCD \nO(\nL)))  \\
 +& \deg_E(\nZ(\nL_\Z, \nM_\Z, \kappa; K)) \avol_{\overline{\mathcal{E}}}(\nSh({}^{K}_\nRPCD \nO(\nL))) \\
 \equiv& \hght_{\overline{\mathcal{E}}}(\nZ(\nL_\Z, \nM_\Z, \kappa; K))
\end{align*}
in $\R^{(p)}$. Here
\[ \deg_E(\nZ(\nL_\Z, \nM_\Z, \kappa; K)) := \frac{\vol_E(\nZ(\nL_\Z, \nM_\Z, \kappa; K))}{\vol_E(\nSh({}^K\nO(\nL)))} \]
 is the (relative) geometric degree.
\end{ASSERTION} 

We first prove this statement for $n=1$:
In this case it follows immediately from Theorem \ref{MAINTHEOREMAVERAGE}, 2., using a Borcherds product such that all $\nZ$'s in the divisor consist of Shimura varieties corresponding to lattices which are unimodular at $p$, {\em except} for $\nZ(\nL_\Z, <q>, \kappa; K)$ which shall occur, too, with non-zero multiplicity. Theorem \ref{THEOREMBORCHERDSPREP1}, 2. enables us to construct such a product. Note that the assumptions imply that $\nSh({}^K \nO(\nL))$ does have cusps.

Proof for $n\ge2$: Because $\nM_\Zpp^*/\nM_\Zpp$ is cyclic, we may find lattices $\nM_\Z$ and $\nM'_\Z$, unimodular at $p$, and $<q>_\Z$ such that 
\[ \nM_\Z = <q> \perp \nM_\Z' \]
is a saturated lattice in $\nM_\Zpp$. Let $\kappa = \kappa_q \oplus \kappa'$ be a corresponding decomposition.
Let $\SO'(\nL_\Zh)\alpha_i \subset \nIsome(\nL_\Zh, \nM_\Zh') \cap \kappa'$ be a decomposition into orbits. We may form the cycles
$\nZ(\alpha_i^\perp, <q>, \kappa_q)$ on $\nSh({}^{K_i}_{\Delta_i} \nO(\alpha_i^\perp))$, where the $K_i$ are the respective 
discriminant kernels (all admissible by construction of $\nM'_\Z$). The latter Shimura varieties are all equipped with morphisms into
$\nSh({}^K_\Delta \nO(\nL_{\Zpp}))$ (the union of their images is the cycle $\nZ(\nL, \nM', \kappa', K)$).
These Shimura varieties all have the same Shimura datum, and one could see their images as conjugated images of a single Shimura variety with varying $K$ as in the
definition of special cycle (\ref{DEFSPECIALCYCLE}).
In each of these Shimura varieties, we have cycles $\nZ(\alpha_i^\perp, <q>, \kappa_q; K_i)$. Identifying them with their image in $\nSh({}^K_\Delta \nO(\nL_\Zpp))$, we get 
\[ \nZ(\nL, \nM, \kappa; K) = \bigcup_{\SO'(\nL_\Zh)\alpha \subset \nIsome(\nL_\Zh, \nM_\Zh')\cap \kappa'} \nZ(\alpha_i^\perp, <q>, \kappa_q; K_i). \]
Using the $n=1$ case, we know the height of the right hand side. For this note that the assumptions imply that $\nSh({}^{K_i} \nO(\alpha_i^\perp))$ has 
cusps. 
Comparing this with global Kitaoka (Theorem \ref{GLOBALKITAOKA}), we get the general result,
and its truth does, of course, not depend on the models $\nL_\Z$, $\nM_\Z$ chosen and not on $K$ (cf. also \ref{ARITHMVOLNOTNEAT}).

Finally the previous Assertion \ref{localstatement} at $p|D$, where $p\not=2$ and $p^2\nmid D$, and therefore 2. as an identity in $\R_{2D'}$ (instead of $\R_{2D}$) now follows from \ref{localstatement2} applied to any of the embeddings used to construct the model $\nSh$ (see \ref{SPECIALMODELS}). 
For this, we have to use that the derivative of the global orbit equation \ref{GLOBALORBITEQUATION} is also valid for
$\widetilde{\lambda}$ and $\widetilde{\mu}$ in $\R^{(p)}$ for those $p$. Recall, that this followed from the fact
that the local orbit equation for $\widetilde{\lambda}_p$, and $\widetilde{\mu}_p$ remains true for those $p$. 
This is definitely wrong in general if $p^2|D$ (cf. also \cite[11.8]{Paper1}). 
\end{proof}

\begin{BEM}
Main Theorem \ref{RESULTGLOBAL}, 4. has been proven in full generality, also including information at $2D$ and $\infty$ in \cite{KRY1,KRY2,KRY3} for Shimura curves and
in \cite{Yang2} for the modular curve ($n=1$). For the equality at $\infty$, with $\mu_\infty$ replaced by the full $\infty$-factor of the Fourier coefficient of the Eisenstein series, the special cycles have to be complemented by the Kudla-Millson Greens functions (depending on the imaginary part as well),
cf. also the introduction to the author's thesis \cite{Thesis}. 

There are stronger results also for Hilbert modular surfaces \cite{BBK,KR1,KR3}, for Siegel threefolds \cite{KR3},
and for the product of modular curves (The $\nZ$'s for $n=1$ are the Hecke correspondences in this case) in \cite[section 7.8]{BKK1}.
\end{BEM}

\appendix

\def\thesection{\Alph{section}}

\section{Lemmas on quadratic forms}\label{LEMMATAQUADRATICFORMS}

\begin{LEMMA}\label{LEMMA221NORMAL}Let $p\not=2$ and
$\nL_\Zpp$ be an unimodular lattice of signature $(2,2)$ and Witt rank 1, with fundamental discriminant $D$.
Up to multiplication of $Q_\nL$ by a unit in $\Z_{(p)}^*$, 
there exists a saturated lattice $\nL_\Z \subset \nL_\Zpp$ of the form
\[ \nL_\Z = H \oplus <x^2+xy+\frac{1-D}{4}y^2> \]
if $D \equiv 1$ modulo 4 or 
\[ \nL_\Z = H \oplus <x^2-\frac{D}{4}y^2> \]
if $4|D$.
\end{LEMMA}
\begin{proof}
We may write $\nL_\Zpp = H \oplus \nL_\Zpp'$ and assume that $\nL_\Zpp'$ represents 1. $\nL_\Zpp'$ is anisotropic by assumption.
Then it is well-known and elementary that 
$\nL_\Zpp'$ contains a $\Z$-lattice of the required form.
\end{proof}

\begin{LEMMA}\label{LEMMA210INCLUSION}
Let $p\not=2$ and $\nL_\Zpp$ be a unimodular anisotopic lattice of signature $(1,2)$.
Up to multiplication of $Q_\nL$  with a unit in $\Z_{(p)}^*$,
there exists a unimodular lattice $\nL_\Zpp' = \nL_\Zpp \oplus <x>_\Zpp$, $x \in \Z_{>0}$ which is 
isotropic and of {\em prime} fundamental discriminant $q$, $q \not= p$ and a saturated lattice $\nL_\Z' \subset \nL_\Zpp'$ 
of the form 
\[ \nL_\Z' = H \oplus <x^2+xy+\frac{1-q}{4}y^2>. \]
\end{LEMMA}
\begin{proof}
Write $\nL_\Zpp = <\alpha_1, \dots, \alpha_3>_\Zpp$, with $\alpha_i \in \Z$ square-free and 
relatively prime. This is possible without multiplying $Q_\nL$ by a scalar.
Let $D$ be the square-free part of $\prod \alpha_i$. 
We may find a prime $q$, different from $p$, such that
$q \equiv 1\ (4)$ and 
for any $l\not=2$ with $\nu_l(\prod \alpha_i)=2$ (hence $l\nmid D$), we have $Dq \equiv -\alpha_j\ (l)$, where $\alpha_j$ is the one not divisible by $l$.

We may also prescribe its residue mod $8$, such that $\nL_{\Q_2} \oplus <Dq>$ is isotropic. For if they were 
anisotropic for $q,q'$ congruent to $1, 5$ modulo 8 respectively,
we would get $<Dq>_{\Q_2} \cong <Dq'>_{\Q_2}$ from Witt's Theorem and the uniqueness of the 4-dimensional anisotropic space, which is absurd.

Hence
\[ \nL_\Zpp \oplus <Dq>_\Zpp \]
is unimodular, of signature $(2,2)$ with square-free discriminant $q$. 
It is isotropic at all $l\not=2$, too, because of the congruence condition on $Dq$. Hence the Witt rank is 1 and
we may apply the previous lemma to it (this changes also the form on
$\nL_\Zpp$ by a scalar) to get the result.
\end{proof}

\begin{LEMMA}\label{LEMMA210REPN}
Let $p\not=2$ and $\nL_\Zpp$ be an unimodular anisotropic lattice of signature $(1,2)$.
Up to multiplication of $Q_\nL$ with a unit in $\Z_{(p)}^*$,
there is a saturated lattice $\nL_\Z \subset \nL_\Zpp$ of discriminant $D=2D'$, where $D'$ is square-free
and with $\nL_\Z^*/\nL_\Z$ cyclic.

It has the property that
a {\em primitive} $\kappa \in \nL_{\Zh}^*/\nL_{\Zh}$ represents an $m \in \Q$ if and only if
$Q_\nL(\kappa) \equiv m \enspace (1)$. 
\end{LEMMA}
\begin{proof}
First we prove the existence of the lattice. 
Write $\nL_\Zpp = <\alpha_1, \alpha_2, \alpha_3>_\Zpp$, with $\alpha_i \in \Z$ square-free. By multiplying
$Q_\nL$, if necessary, by a scalar, we may assume that the $\alpha_i$ are pairwise relatively coprime.

The lattice $\nL_\Z = <\alpha_1, \alpha_2, \alpha_3>_\Z$ then already satisfies the assumption locally for all $l \not= 2$.
It suffices to construct a lattice $\nL_{\Z_2}$ satisfying the assumption.
There are 2 cases: 

1. The space $\nL_{\Q_2}$  is isotropic, hence there is, up to multiplication of the quadratic form by 2, a lattice of the form
\[ \nL_{\Z_2} = H \perp <1>, \]
which has discriminant 2.

2. The space $\nL_{\Q_2}$ is anisotropic, hence, up to multiplication of the lattice by $\pm 2$, of the form
\[ \nL_{\Q_2} = <1, 3, 2\epsilon> \]
with $\epsilon \in \{1,3,-3,-1\}$. In it there exists the lattice
\[ \nL_{\Z_2} = <x^2+xy+y^2> \perp <2 \epsilon>  \]
of discriminant 4, with $\nL_{\Z_2}^*/\nL_{\Z_2}$ cyclic of order 4.

The only-if part of the claimed property is clear. For the if part, we have to show that $m$ is represented by $\nL_\Zl$ for all $l$. For $l \not| D$, $\nL_\Zl$ represents every $m \in \Zl$. 
For $l|D, l \not= 2$, $\kappa_l$ represents $m$ if
\[ \alpha_1 x_1^2 + \alpha_2 x_2^2 + \alpha_3((\frac{\beta}{l})^2 + 2\frac{\beta x_3}{l} + x_3^2) = \frac{n}{l} = m \]
has a solution with $x_i \in \Zl$, where $\beta \in \Zl^*$ is determined by $\kappa$ and $n \in \Zl^*$. Here w.l.o.g. $l|\alpha_3$.

Now the condition given implies $\frac{\alpha_3}{l} \beta^2 \equiv n\ (l)$. Since $\alpha_1 x_1^2 + \alpha_2 x_2^2$ represents $\Z_l^*$, $\kappa_l$ represents $m$ (choose $x_3$ appropriately in the above equation).

For $l=2$, case 1, $\kappa_2$ represents $m$ if
\[ x_1 x_2 + ((\frac{1}{2})^2 + x_3 + x_3^2) = \frac{n}{4} = m \]
has a solution with $x_i \in \Zl$. This follows from the condition which boils down to $n \equiv 1\ (4)$ in this case.

For $l=2$, case 2, $\kappa_2$ represents $m$ if
\[ x_1^2 + x_1x_2+x_2^2 + 2\epsilon((\frac{\beta}{4})^2 + 2\frac{\beta}{4}x_3 + x_3^2) = \frac{n}{8} = m \]
has a solution with $x_i \in \Zl$. Here $\beta=\pm 1$ distinguishes the two $\kappa_2$. The condition says
$\beta^2\epsilon \equiv n\ (8)$, hence solubility because $x_1^2 + x_1x_2+x_2^2$ represents $\Z_2^*$.
\end{proof}

The strong form of the lemma is wrong for lattices of dimension 2.
However, we have the following weaker form:
\begin{LEMMA}\label{LEMMA110REPN}
Let $p\not=2$ and $\nL_\Zpp$ be a unimodular anisotropic lattice of signature $(1,1)$ and of discriminant $-q$, where $q \equiv 1 \enspace (4)$ is prime.
Up to multiplication of $Q_\nL$ by a unit $\in \Z_{(p)}^*$,
there is a saturated lattice $\nL_\Z \subset \nL_\Zpp$ of discriminant $q$ of the form
\[ \nL_\Z = <x^2+xy+\frac{1-q}{4}y^2>. \]

It has the property that
a {\em primitive} $\kappa \in \nL_{\Zh}^*/\nL_{\Zh}$ represents a $m=\frac{l'}{q} \in \Q$, where $l'\not=2$ {\em is prime}, if and only if
$Q_\nL(\kappa) \equiv m \enspace (1)$. 
\end{LEMMA}
\begin{proof}
The existence of the lattice is well-known (compare also \ref{LEMMA221NORMAL}).
For the property, we have to show, that $\nL_\Zl$ represents $\frac{l'}{q}$ under the condition.

We can write $\nL_\Zl = <1,-q>$, hence we see (like in the proof of the last lemma) that the condition implies $-x^2\equiv l'\ (q)$ has a solution, 
i.e. $-l'$ is a quadratic residue mod q. Since $q \equiv 1\ (4)$ this is equivalent to $l'$ being a quadratic residue mod $q$. 

For $l \notin \{2,q\}$, 
\[ x_1^2 - q x_2^2 = \frac{l'}{q} \]
clearly has a solution for $x_1,x_2 \in \Zl$ if $l\not=l'$, and for $l=l'$ it does, if $q$ is a square mod $l'$, and the latter is true by the law of quadratic reciprocity.
For $l=q$, $\kappa_q$ represents $m$, if
\[ x_1^2 - q ((\frac{\beta}{q})^2 + 2\frac{\beta}{q}x_2 + x_2^2) = \frac{l'}{q} \]
has a solution with $x_i \in \Z_q$. Now $-\beta^2 \equiv l'\ (q)$, so this is clearly possible.

For $l=2$, we have $\nL_{\Z_2} = <x^2+xy+z^2>$, which represents $\Z_2^*$.
\end{proof}

\begin{LEMMA}\label{LEMMA220INCLUSION}
Let $p\not=2$ and $\nL_\Zpp$ be a unimodular anisotopic lattice of signature $(2,2)$.
Up to multiplication of $Q_\nL$  with a unit in $\Z_{(p)}^*$,
there exists a unimodular lattice $\nL_\Zpp' = \nL_\Zpp \oplus <x>$ of signature $(3,2)$ 
and a saturated $\Z$-lattice $\nL_\Z' \subset \nL_\Zpp'$ of discriminant $D=2D'$, where $D'$ is square-free of the form 
\[ \nL_\Z' = H \oplus \nL_\Z''. \]
such that $(\nL_\Z')^*/\nL_\Z'$ is cyclic.
\end{LEMMA}
\begin{proof}
Take $\nL'_\Zpp = \nL_\Zpp \oplus <x>$ for an arbitrary positive $x \in \Z_{(p)}^*$. It is automatically isotropic, hence of the form
\[ \nL'_\Zpp = H \oplus \nL_\Zpp''. \]
Applying Lemma \ref{LEMMA210REPN} to $\nL_\Zpp''$, we get the result (we multiply $Q_{\nL'}$ by the same scalar; this multiplies also $Q_\nL$ on $\nL_\Zpp$ by this scalar --- note, however, 
that multiplying the quadratic form on $H$ by a scalar does not affect its isomorphism class).
\end{proof}

\section{Lacunarity of modular forms}\label{LACUNARITY}

\begin{LEMMA}\label{LEMMAMODFORM1}
Let $N>0$ be an integer and $p$ a prime with $p\nmid N$.

Let $f$ be a holomorphic modular form of (half\nobreakdash-)integral weight $k \not= 0$ for the group $\Gamma(N)$.
If $f$ has a Fourier expansion of the form
\[  f = \sum_{n \in \Q_{\ge 0}} a_{n} q^{n}, \]
where $a_n$ is zero, unless $n$ is an integral multiple of $\frac{p}{N}$, then $f=0$.
\end{LEMMA}

\begin{proof}
The assumption implies that $f$ is periodic with period $\frac{N}{p}$.
It is hence a modular form for the group~$\Gamma$, generated by $\Gamma(N)$ and the matrix
$\M{1&\frac{N}{p}\\0&1}$. This group contains the product
\[ C =  
\M{1&\frac{N}{p}\\0&1} \M{1&0\\N&1} = 
\M{1+\frac{N^2}{p}&\frac{N}{p}\\N&1 }. \]
The trace of $C$ is equal to $2+\frac{N^2}{p}$. Since $p\nmid N$, its $p$-adic valuation is $>1$. Hence at least one of the eigenvalues of $C$ has p-adic valuation 
$>1$. (We choose some fixed extension of the $p$-adic valuation to $\overline{\Q}$). It follows that $\|C^i\| \rightarrow \infty $ for any chosen $p$-adic matrix norm $\|\cdot \|$.
From this, it follows that $[\Gamma:\Gamma(N)]=\infty$. For assume that there are finitely many representatives $\alpha_i$.
Let $\nu$ be the maximum of their $p$-adic matrix norms. Every element $\gamma \in \Gamma$ is of the form
\[ \alpha_i \gamma` \]
for $\gamma` \in \Gamma(N)$. The matrix norm of $\gamma$ is hence $\le \nu$. A contradiction.
Hence $\Gamma$ cannot be a discrete subgroup, and since $k\not=0$, we have $f=0$.
\end{proof}

\begin{PROP}\label{LEMMAMODFORM2}
Let $q \equiv 1 \enspace (4)$ be a prime and $S$ a finite set of primes.
Let $\chi_q(x) = \left(\frac{x}{q}\right)$ and $k \ge 2$ be an integer.

If $f \in \HolModForm(\Gamma_0(q), \chi_q, k)$ has a Fourier expansion of the form
\[  f = \sum_{n \in \Z_{\ge 0}} a_{n} q^{n}, \]
with algebraic $a_n$, where (i) $a_n=0$, whenever $\chi_q(n)=-1$, and (ii) $a_p=0$, whenever $p \not\in S$.

Then $f=0$.
\end{PROP}
\begin{proof}
This follows from an idea of \cite{OS}, see also \cite[Lemma 4.14]{BBK}:
Since $q \equiv 1 \enspace (4)$, there are no forms of type `CM' satisfying condition (i).

Since the $a_n$ are algebraic, we can write $f$ as a linear combination
\[ f = c_0 E_0 + c_\infty E_\infty + \sum_{i=1}^n c_i f_i, \]
where the $c_i$ are algebraic, the $f_i$ are cuspidal Hecke eigenforms and $E_0, E_\infty$ are the Eisenstein series. 

In \cite[Lemma 1]{OS} it is shown that for all sufficiently large primes $l$ the image of the mod $l$ representation 
$\rho := \rho_{1,l} \times \cdots \times \rho_{n,l}$ contains a subgroup $G$ conjugated to 
\[ \SL_2(\F_1) \times \cdots \times \SL_2(\F_n) \] 
and $\im(\rho)/G$ is Abelian, where the $\F_i$'s are defined in [loc. cit.].
Choose $l$ such that all $|c_i|_l=1$, whenever $c_i \not= 0$.

The mod $l$ representations $\rho_{0, l}$, resp. $\rho_{\infty,l}$ of $E_0$, resp. $E_\infty$ \cite[p. 28]{Ribet} are abelian and
$\rho' := (\rho_{0,l} \times \rho_{\infty, l}) \times (\rho_{1,l} \times \cdots \times \rho_{n,l})$ contains a subgroup of the form 
\[ H \times G \] 
because $\SL_2(\F_i)$ has no non-trivial Abelian quotient if $l>4$. Here $H \subset \im(\rho_{0,l} \times \rho_{\infty, l})$ is a subgroup that, at least, contains a pair of matrices with different traces.
Choose a pair $(M_1,M_2) \in H$ of matrices with different or equal traces, according to whether $c_0 \equiv -c_\infty \enspace (l)$ or not.

By \v{C}ebotarev, there is a positive density of primes $p$, such that the image of $\Frob_p$ is conjugated
to 
\[ A= M_1 \times M_2 \times \M{0&1\\-1&0} \times \cdots \times \M{0&1\\-1&0}. \]

Hence
$|a_p(c_0 E_0 + c_\infty E_\infty)|_l = 1$ and $|a_p(f_i)|_l<1$, We get $a_p(f)\not=0$ for infinitely many primes, a contradiction.
If $c_0=c_\infty=0$, and $c_1\not=0$ (say) choose 
\[ A=\M{1&0\\0&1} \times \M{1&0\\0&1} \times \M{1&0\\0&1} \times \M{0&1\\-1&0} \times \cdots \times \M{0&1\\-1&0} \] 
and argue as before.
\end{proof}

\begin{LEMMA}\label{LEMMAMODFORM3}
Let $\Gamma \subset \Mp_2(\R)$ be an arithmetic subgroup and 
$V_\rho, \rho$ an {\em irreducible} representation of $\Gamma$.
Let $f \in \HolModForm(\Gamma, \rho, k)$ and $\beta: V_\rho \rightarrow \C$ be a non-zero linear form.
If $\beta \circ f = 0$ then $f = 0$.
\end{LEMMA}
\begin{proof}
This follows immediately from the irreducibility:
Choose a basis $\{e_i\}$ of $V_\rho$ such that $\beta=e_0^*$.
Since $V_\rho$ is irreducible, for any $i$ there is an operator of the form
\[ O_i = \sum_j \alpha_j \rho(\gamma_j), \]
which interchanges $e_i$ and $e_0$.
Consider the form $O_i \circ f$. $e_0^* \circ O_i \circ f$ is equal to $e_i^* \circ f$ on the one hand. 
On the other hand it is equal to $\sum_i \alpha_i (e_0^* \circ f)|_k \gamma_i$, which is zero by assumption. Hence $f=0$.
\end{proof}

\section{Semi-linear representations}

Let $K$ be a pro-finite topological group and $\tau: K \rightarrow \Gal(F:\Q)$ be a continuous homomorphism, where $F$ is a (possibly infinite) algebraic extension of $\Q$.

\begin{DEF}\label{SEMILINEAR}
By a $\tau$-semi-linear representation of $K$ on a finite-dimensional $F$-vector-space $V$, 
we understand a continuous homomorphism
\[ \omega: K \rightarrow \Aut_\Q(V) \]
(where $V$ is equipped with the discrete topology) satisfying 
\[ \omega(g) (\alpha v ) = \tau_g(\alpha) \omega(g) v. \]
\end{DEF}

\begin{BEISPIEL}\label{EXSEMILINEAR}
For example, a $\tau$-semi-linear representation may be given by a $\Q$-vector space $V$ and a 
continuous homomorphism $\omega$, fitting into the following commutative diagram
\begin{equation}\label{commdiagcycl} 
\xymatrix{ K \ar[r]^-\omega \ar[dr]_\tau & \GL_{F}(V_{F}) \rtimes \Gal(F:\Q) \ar[d] \\
& \Gal(F:\Q)
} 
\end{equation}
Obvious any such homorphism determines and is determined by its restriction $\omega'$ to the first factor, which
is an element in $H^1(K, \GL_{F}(V_{F}))$, where $K$ acts via $\tau$.
Here $\GL_{F}(V_{F})$ is equipped with the discrete topology.
\end{BEISPIEL}

\begin{PAR}
$\tau$-semi-linear representations of $K$ form a category, 
morphisms being linear morphisms of $F$-vector-spaces respecting the representations. 

There is a notion
of induction for these representations, and we have for $K_1 \subseteq K_2$ (open subgroup) and $V_i$ a $\tau$-semi-linear representation of $K_i$:
\[ \Hom_{K_1}(V_2, V_1) \cong \Hom_{K_2}(V_2, \ind_{K_1}^{K_2}(V_1))  \]
naturally, and since $[K_1:K_2] < \infty$, we also have
\[ \Hom_{K_1}(V_1, V_2) \cong \Hom_{K_2}(\ind_{K_1}^{K_2}(V_1), V_2),  \]
naturally. The induction may be defined as
\[ \ind_{K_1}^{K_2}(V) = \{ f: K_2 \rightarrow V \where f(hg) = \omega(h) f(g)\ \forall h\in K_1, g \in K_2 \} \]
with the $F$-linear structure given by
\[ (\alpha f)(g):= \tau_g(\alpha) f(g). \]
We have also 
\begin{gather*} \Hom_{K_2}(\ind_{K_1}^{K_2}(V), \ind_{K_1}^{K_2}(W))  = \\
 \{ \varphi: K_2 \rightarrow \Hom_\Q(V,W) \where \varphi(h_1gh_2) = \omega(h_1) \circ \varphi(g) \circ \omega(h_2)\ \forall h_1,h_2 \in K_1, g \in K_2, \\
  \varphi(g) \alpha v = \tau_g(\alpha) \varphi(g) v\ \forall g \in K_2, \alpha \in F  \}.
\end{gather*}
Here a function $\varphi$ acts by convolution
\[ (\varphi \ast f)(g) := \frac{1}{[K_2:K_1]} \sum_{x \in K_2 \backslash K_1} \varphi(gx^{-1}) f(x). \]
\end{PAR}

\begin{BEISPIEL}
Any $\Q$-representation $\omega', V$ of $K$ gives rise to a $\tau$-semi-linear representation $\omega(k) = \omega'(k) \circ \tau_k$ on $V_{F}$.

For a $\tau$-semi-linear representation $V, \omega$ the dual is defined by $V^*$ (usual dual) acted on by $(\omega^*(g)v^*)(v) := \tau_g(v^*(\omega(g^{-1})v))$. It is a $\tau$-semi-linear representation, too.
\end{BEISPIEL}

\begin{BEISPIEL}\label{EXABSTRACTGAMMA0P}
Let $p$ be an odd prime, $\chi$ be the non-trivial character of order 2 of $(\Z/p\Z)^*$ and $\tau: (\Z/p\Z)^* \rightarrow \Gal(\Q(\zeta_p)/\Q)$ be the natural isomorphism.
It gives rise to a semi-linear representation 
\[  \M{a&*\\0&d} \mapsto \chi(a) \tau_{ad}.  \]
A function $\varphi \in \Hom(\ind_{K_1}^{K_2}\chi, \ind_{K_1}^{K_2} \chi)$ is determined by its values at
\[ \M{1&0\\0&1} \text{ and } \M{0&1\\-1&0}.  \]
The semi-linearity forces
\[\tau_\alpha(\varphi(\M{1&0\\0&1})) = \varphi(\M{1&0\\0&1}) \quad \text{ and } \quad \tau_\alpha(\varphi(\M{0&1\\-1&0})) = \chi(\alpha)\varphi( \M{0&1\\-1&0}), \]
therefore
\[ \varphi(\M{1&0\\0&1}) \in \Q \quad \text{ and } \quad \varphi(\M{0&1\\-1&0}) \in \Q \cdot \sqrt{p^*} , \]
where $p^* = p$ if $p \equiv 1$ mod 4 and $p^*=-p$ otherwise.
A calculation shows that $\End_{K_2}(\ind_{K_1}^{K_2}\chi) = \Q^2$ as algebras, hence the irreducible constituents
of $\ind_{K_1}^{K_2}\chi$ as $\SL_2(\Z/p\Z)$-representation are also irreducible for this semi-linear representation of $\GL_2(\Z/p\Z)$.
\end{BEISPIEL}

\bibliographystyle{abbrvnat}
\bibliography{hoermann}







\end{document}